\documentclass[11pt]{amsart}

\usepackage{lmodern}
\usepackage{stmaryrd}
\usepackage[utf8]{inputenc}
\usepackage{bm}
\newcommand{\VAN}[3]{#2}
\newcommand{\VANDEN}[3]{#2}

\usepackage{csquotes}
\usepackage[dvipsnames]{xcolor}
\usepackage{tikz-cd}
\usepackage[font=small,format=plain,labelfont=bf,up,textfont=it,up]{caption}
\usepackage{adjustbox}
\usepackage[margin=1.5in]{geometry}

\usepackage{fancyhdr}
\usepackage[shortlabels]{enumitem}
\usepackage{amsmath}
\usepackage{mathtools} 
\usepackage{amsfonts}
\usepackage{amssymb} 
\usepackage{amsthm}
\usepackage{mathrsfs}
\usepackage{calligra}
\usepackage{scalerel}
\usepackage{crossreftools}

\usepackage[linktocpage=true, hyperindex,pagebackref]{hyperref}
\usepackage{cleveref}
\usepackage[OT2,T1]{fontenc}
\DeclareSymbolFont{cyrletters}{OT2}{wncyr}{m}{n}
\DeclareMathSymbol{\Sha}{\mathalpha}{cyrletters}{"58}

\begingroup
    \makeatletter
    \@for\theoremstyle:=definition,remark,plain\do{%
        \expandafter\g@addto@macro\csname th@\theoremstyle\endcsname{%
            \addtolength\thm@preskip\parskip
            }%
        }
\endgroup

\newtheorem{Thm}[subsubsection]{Theorem}

\newtheorem{Lem}[subsubsection]{Lemma}
\newtheorem{Prop}[subsubsection]{Proposition}
\newtheorem{Cor}[subsubsection]{Corollary}
\newtheorem{Conj}[subsubsection]{Conjecture}
\newtheorem*{Thm*}{Theorem}
\newtheorem{mainThm}{Theorem}

\newtheorem{mainCor}[mainThm]{Corollary}

\theoremstyle{definition}
\newtheorem{Def}[subsubsection]{Definition}
\newtheorem{Example}[subsubsection]{Example}
\newtheorem{Rem}[subsubsection]{Remark}
\numberwithin{equation}{subsection}

\newcommand{\spec}{\operatorname{Spec}}
\newcommand{\spf}{\operatorname{Spf}}
\newcommand{\spa}{\operatorname{Spa}}
\newcommand{\spd}{\operatorname{Spd}}
\newcommand{\gal}{\operatorname{Gal}}
\newcommand{\Hom}{\operatorname{Hom}}
\newcommand{\Aut}{\operatorname{Aut}}

\newcommand{\GL}{\mathrm{GL}}
\newcommand{\id}{\mathrm{id}}
\newcommand{\im}{\mathrm{im}}

\newcommand{\ql}{{\mathbb{Q}_\ell}}
\newcommand{\zl}{{\mathbb{Z}_\ell}}
\newcommand{\ml}{\mathsf{L}}
\newcommand{\mlp}{\mathsf{L}^{+}}

\newcommand{\qlbar}{\overline{\mathbb{Q}}_\ell}
\newcommand{\zlbar}{\overline{\mathbb{Z}}_\ell}
\newcommand{\flbar}{{\overline{\mathbb{F}}_{\ell}}}

\newcommand{\qp}{\mathbb{Q}_p}
\newcommand{\zp}{\mathbb{Z}_{p}}
\newcommand{\fp}{\mathbb{F}_{p}}
\newcommand{\cp}{\mathbb{C}_p}
\newcommand{\Q}{\mathbb{Q}}
\newcommand{\R}{\mathbb{R}}
\newcommand{\qpbar}{\overline{\mathbb{Q}}_p}
\newcommand{\fpbar}{\overline{\mathbb{F}}_{p}}
\newcommand{\qpbr}{\breve{\mathbb{Q}}_{p}}

\newcommand{\ebreve}{\breve{E}}
\newcommand{\afp}{\mathbb{A}_f^p}
\newcommand{\af}{\mathbb{A}_f}
\newcommand{\ebar}{\overline{\mathsf{E}}}

\newcommand{\perf}{\operatorname{Perf}}
\newcommand{\shtgmu}{\mathrm{Sht}_{\mathcal{G},\mu}}

\newcommand{\bun}{\mathrm{Bun}}
\newcommand{\bung}{\bun_{G}}
\newcommand{\bungk}{\bun_{G,k}}
\newcommand{\bungmu}{\bun_{G,\mu^{-1}}}
\newcommand{\bungmuk}{\bun_{G,\mu^{-1},k}}
\newcommand{\bgmu}{B(G,\mu^{-1})}
\newcommand{\bunh}{\operatorname{Bun}_{H}}

\newcommand{\g}{\mathsf{G}}
\newcommand{\x}{\mathsf{X}}
\newcommand{\gx}{{(\mathsf{G}, \mathsf{X})}}
\newcommand{\gxp}{(\mathsf{G}', \mathsf{X}')}

\newcommand{\hy}{{(\mathsf{H}, \mathsf{Y})}}
\newcommand{\y}{\mathsf{Y}}
\newcommand{\h}{\mathsf{H}}
\newcommand{\ah}{\mathcal{A}(\h)}
\newcommand{\hyab}{{(\mathsf{H}^\mathrm{ab}, \mathsf{Y}^\mathrm{ab})}}

\newcommand{\scrs}{\mathscr{S}}

\newcommand{\Fl}{\mathscr{F}\!\ell}
\newcommand{\grg}{\mathrm{Gr}_G}
\newcommand{\grgmu}{\mathrm{Gr}_{G,\mu^{-1}}}
\newcommand{\grgmug}{\mathrm{Gr}_{G,\mu^{-1},G(\qp)}}
\newcommand{\grgtwomu}{\mathrm{Gr}_{H,\mu_{H}^{-1}}}
\newcommand{\grgtwomug}{\mathrm{Gr}_{H,\mu_{H}^{-1},H(\qp)}}

\newcommand{\igs}{\mathrm{Igs}}
\newcommand{\igsk}{\mathrm{Igs}_{K^p,k}}

\newcommand{\igsgx}{\operatorname{Igs}^?\gx}
\newcommand{\igsfingx}{\operatorname{Igs}^?_{K^p}\gx}

\newcommand{\IgsQuot}{q_\mathrm{Igs}}

\newcommand{\gxad}{(\mathsf{G}^{\mathrm{ad}},\mathsf{X}^{\mathrm{ad}})}
\newcommand{\gxab}{(\mathsf{G}^{\mathrm{ab}},\mathsf{X}^{\mathrm{ab}})}

\newcommand{\gxtwo}{{(\mathsf{G}_2, \mathsf{X}_2)}}
\newcommand{\gxthree}{(\mathsf{G}_3, \mathsf{X}_3)}
\newcommand{\gxthreeab}{(\mathsf{G}_3^{\mathrm{ab}},\mathsf{X}_3^{\mathrm{ab}})}

\newcommand{\gad}{G^{\mathrm{ad}}}
\newcommand{\Gad}{\g^{\mathrm{ad}}}
\newcommand{\gadqone}{\g^{\mathrm{ad}}(\mathbb{Q})^{1}}

\newcommand{\gafp}{\mathsf{G}(\afp)}
\newcommand{\gaf}{\mathsf{G}(\af)}
\newcommand{\ag}{\mathcal{A}(\g)}

\newcommand{\zg}{Z_{\mathsf{G}}}
\newcommand{\zgq}{Z_{\mathsf{G}}(\mathbb{Q})}
\newcommand{\zgqbar}{Z_{\mathsf{G}}(\mathbb{Q})^{-}}

\newcommand{\Div}{\mathrm{Div}^{1}_{E}}
\newcommand{\Divk}{\mathrm{Div}^{1}_{E,k}}

\def\shd{\mathbf{Sh}\gx^{?,\diamondsuit}}
\def\shtwod{\mathbf{Sh}\hy^{?,\diamondsuit}}
\def\shdinf{\mathbf{Sh}\gx^{?,\diamondsuit}}
\def\shdfin{\mathbf{Sh}_K\gx^{?,\diamondsuit}}
\def\shdpinf{\mathbf{Sh}_{K^p}\gx^{?,\diamondsuit}}

\def\dimtrg{\operatorname{dim.trg}}

\newcommand{\ul}[1]{\underline{#1}}

\newcommand{\parg}{\operatorname{Par}_{G}}
\newcommand{\LH}{{}^\mathsf{L} \widehat{\mathsf{H}}}
\newcommand{\lH}{{}^{L} \widehat{\mathsf{H}_v}}

\newcommand{\FS}{\mathrm{FS}}
\newcommand{\SO}{\operatorname{SO}}

\newcommand{\LL}{\operatorname{LL}}
\newcommand{\temp}{\mathrm{temp}}

\NewDocumentCommand\pair{o}{%
  \IfValueTF{#1}{(R^{\sharp_#1}, R^{\sharp_#1+})}{(R^\sharp, R^{\sharp+})}%
}
\def\upair{(R, R^+)}
\def\ffcurve{\mathcal{X}_\mathrm{FF}}
\def\apg{\ul{\mathcal{A}^p(\mathsf{G})}}
\def\aph{\ul{\mathcal{A}^p(\mathsf{H})}}

\newcounter{stepcounter}
\renewcommand{\thestepcounter}{\textbf{Step \arabic{stepcounter}}}
\newcommand{\step}{%
  \refstepcounter{stepcounter}
  \textbf{\thestepcounter:}
}

\makeatletter
\newcommand{\customlabel}[2]{%
   \phantomsection
   \protected@edef\@currentlabel{#1}
   \label{#2}
}
\makeatother

\newcommand\restr[2]{{
  \left.\kern-\nulldelimiterspace 
  #1 
  \vphantom{\big\vert} 
  \right\rvert_{#2} 
  }}

\iftrue
\newcommand{\commentDaniel}[1]{\textcolor{Blue}{Daniel: #1}}
\newcommand{\commentPatrick}[1]{\textcolor{violet}{Patrick: #1}}
\newcommand{\commentPol}[1]{\textcolor{red}{Pol: #1}}
\newcommand{\commentMingjia}[1]{\textcolor{ForestGreen}{Mingjia: #1}}
\else
\newcommand{\commentDaniel}[1]{}
\newcommand{\commentPatrick}[1]{}
\newcommand{\commentPol}[1]{}
\newcommand{\commentMingjia}[1]{}
\fi

\author{Patrick Daniels}
\address{Department of Mathematics and Statistics, Skidmore College, 815 N Broadway, Saratoga Springs, NY, 12866, USA}
\email{pdaniels@skidmore.edu}

\author{Pol van Hoften} 
\address{School of Mathematical Sciences, Zhejiang University, 866 Yuhangtang Rd, Hangzhou, 310058, P. R. China}
\email{pvhoften@zju.edu.cn}
\thanks{PvH is (partly) funded by the Dutch Research Council (NWO) under the grant VI.Veni.232.127.}

\author{Dongryul Kim}
\address{Department of Mathematics, Stanford University, 450 Jane Stanford Way
(Building 380), Stanford, California, USA}
\email{dkim04@stanford.edu}

\author{Mingjia Zhang}
\address{Department of Mathematics, Princeton university, Fine Hall, Washington Road,
Princeton, NJ, 08544-1000, USA}
\email{mz9413@princeton.edu}
\thanks{MZ is supported by the NSF through the Institute for Advanced Study, under the grant No. DMS-2424441.}

\title[Igusa stacks and the cohomology of Shimura varieties II]{Igusa stacks and the cohomology of Shimura varieties II}
\begin{document}

\begin{abstract}
    We construct Igusa stacks for all Shimura varieties of abelian type and derive consequences for the cohomology of these Shimura varieties. As an application, we prove that the Fargues--Scholze local Langlands correspondence agrees with the semi-simplification of the local Langlands correspondences constructed by Arthur, Mok and others, for all classical groups of type $A$, $B$ and $D$; this extends work of Hamann, Bertoloni Meli--Hamann--Nguyen and Peng. 
\end{abstract}

\maketitle

\tableofcontents

{\section{Introduction}

\subsection{Overview} 
Igusa stacks are geometric objects, akin to Shimura varieties, whose existence has been conjectured by Scholze, motivated by questions related to local-global compatibility in the Langlands program, see \cite[Conjecture 1.1]{ZhangThesis}. The program to construct Igusa stacks was initiated 
in \cite{ZhangThesis} and continued in \cite{DvHKZIgusaStacks} and \cite{KimFunctorial}, and their existence has led to several breakthroughs in the study of the cohomology of Shimura varieties. In particular, advances have been made in the study of the generic part of the cohomology with torsion coefficients, see \cite{Hamann-Lee} and \cite{YangZhuGeneric}, the generalized Ihara's lemma for definite unitary groups has been established, see \cite{YangIhara}, and a very general Eichler--Shimura relation has been proven for Shimura varieties of Hodge type, see \cite{vdH}. \smallskip

In this paper, we construct Igusa stacks for all Shimura varieties of abelian type, and we use these to prove new results about the cohomology of Shimura varieties. Our primary application is to the local Langlands correspondence for classical groups of types $A,B,D$, as well as $\operatorname{GSp}_4$, over arbitrary $p$-adic local fields. For these groups, we show that the Fargues--Scholze local Langlands correspondence agrees with the semisimplification of the local Langlands correspondences of Arthur, Mok and others, see \cite{Arthur}, \cite{Mok}, \cite{ChenZou}, \cite{KMSW}, \cite{Ishimoto}, \cite{GanTakeda}, \cite{GanTantono}. Our work extends \cite{HamannGSp4}, \cite{BMHN}, \cite{Peng}, where the analogous result is established for unramified groups over unramified extensions of $\qp$ with $p>2$.

In addition, we prove the Eichler--Shimura relation of Blasius--Rogawski \cite{BlasiusRogawski}, at Iwahori level and with torsion coefficients, generalizing \cite{vdH}. We furthermore show that the local plectic conjecture of Nekov\'ar--Scholl holds at split primes, extending results of \cite{Lee} and \cite{LiHuertaPlectic}; the general case is being studied by Feng--Tamiozzo--Zhang. 

\subsection{Main geometric results} \label{sub:GeometricResults} Let $\gx$ be a Shimura datum of abelian type with reflex field $\mathsf{E}$. Let $p$ be a prime number, let $v$ be a prime of $\mathsf{E}$ above $p$ and set $E=\mathsf{E}_v$ and $G=\mathsf{G} \otimes \qp$. Let $K \subset \gaf$ be a neat compact open subgroup and consider the Shimura variety $\mathbf{Sh}_K\gx$ of level $K$ over $E$. There is an open subspace (the ``good reduction locus'')
\begin{align*}
    \mathbf{Sh}_K\gx^{\circ, \mathrm{an}} \subset \mathbf{Sh}_K\gx^{\mathrm{an}}
\end{align*}
of the adic space $\mathbf{Sh}_K\gx^{\mathrm{an}}$ associated with $\mathbf{Sh}_K\gx$, which is compatible with changing the level $K$. Let $\perf$ denote the category of perfectoid spaces of characteristic $p$. We let $\mathbf{Sh}_{K^p}\gx^{\circ, \diamondsuit} \subset \mathbf{Sh}_{K^p}\gx^{\diamondsuit}$ be the corresponding objects with infinite level at $p$, considered as v-sheaves on $\perf$. These come equipped with a Hodge--Tate period map 
\begin{align*}
    \pi_{\mathrm{HT}}^\diamondsuit\colon \mathbf{Sh}_{K^p}\gx^{\diamondsuit} \to \operatorname{Gr}_{G, \mu^{-1}},
\end{align*}
where $\operatorname{Gr}_{G, \mu^{-1}}$ is the Schubert cell in the $\mathbf{B}_\mathrm{dR}^+$-affine Grassmannian corresponding to the inverse of the Hodge cocharacter $\mu$. Let $\bung$ denote the stack of $G$-bundles on the Fargues--Fontaine curve. Finally we need the Beauville--Laszlo map $\operatorname{BL}:\operatorname{Gr}_{G, \mu^{-1}} \to \bun_G$, whose image is an open substack $\bungmu$ of $\bung$.

\begin{mainThm}[Theorem \ref{Thm:MainThmIgusa}, \ref{Thm:FiniteLevelIgusa}] \label{Thm:IntroIgusaGeneric}
There is an open immersion of v-stacks $\mathrm{Igs}_{K^p}^{\circ}\gx \hookrightarrow \mathrm{Igs}_{K^p}\gx$ on $\perf$ sitting in a commutative diagram
    \[ \begin{tikzcd}
        \mathbf{Sh}_{K^p}\gx^{\circ,\diamondsuit} \arrow{d}{\IgsQuot^{\circ}}
          \arrow[hook]{r} & \mathbf{Sh}_{K^p}\gx^{\diamondsuit} \arrow{r}{\pi_{\mathrm{HT}}^\diamondsuit} \arrow{d}{\IgsQuot}& \operatorname{Gr}_{G, \mu^{-1}} \arrow{d}{\mathrm{BL}} \\
          \mathrm{Igs}_{K^p}^{\circ}\gx \arrow[hook]{r} & \mathrm{Igs}_{K^p} \gx \arrow{r}{\overline{\pi}_\mathrm{HT}} & \bungmu,
    \end{tikzcd} \]
with both squares Cartesian. Moreover, for $\ell \not=p$, both $\mathrm{Igs}_{K^p}^{\circ}\gx$ and $\mathrm{Igs}_{K^p}\gx$ are $\ell$-cohomologically smooth of dimension zero with constant dualizing sheaf. Furthermore, $\varprojlim_{K^p} \mathrm{Igs}_{K^p}^{\circ}\gx$ and $\varprojlim_{K^p} \mathrm{Igs}_{K^p} \gx$ satisfy \cite[Definition 5.4]{KimFunctorial}. 
\end{mainThm}

This confirms a conjecture of Scholze, see \cite[Conjecture 1.1.(4)]{ZhangThesis}. For PEL-type Shimura varieties at unramified primes $p$, the existence of $\mathrm{Igs}^{\circ}\gx$ was proved by one of us (MZ) as part of \cite[Theorem 1.3]{ZhangThesis}. For Shimura varieties of Hodge type, it is \cite[Theorem I]{DvHKZIgusaStacks} or alternatively \cite[Theorem D]{KimFunctorial}. For certain abelian type Shimura varieties closely related to the PEL cases of \cite[Theorem 1.3]{ZhangThesis}, it is \cite[Theorem 1.3]{Schnelle}. The existence of $\mathrm{Igs}_{} \gx $ for Shimura varieties of Hodge type was proved in \cite[Theorem D]{KimFunctorial}. Note that the Igusa stacks of Theorem \ref{Thm:MainThmIgusa} are functorial in morphisms of Shimura data by \cite[Theorem B]{KimFunctorial}, and in particular unique. We remark also that in \cite{LiHuertaGlobal}, Li-Huerta studies the global function field analog of the Igusa stack and establishes the analogous fiber product diagram, see \cite[Theorem D]{LiHuertaGlobal}.

We will discuss the proof of Theorem \ref{Thm:IntroIgusaGeneric} in Section \ref{sub:IntroGeometric}. First, we will discuss applications to the local Langlands correspondence for classical groups, and to the cohomology of Shimura varieties.

\subsection{Compatibility of local Langlands correspondences} Let $F$ be a $p$-adic local field and let $G$ be a connected reductive group over $F$. The (conjectural) local Langlands correspondence is a map
\begin{align*}
    \operatorname{LL}_G:\Pi_{G} \to \Phi(G),
\end{align*}
where $\Pi_{G}$ is the set of isomorphism classes of irreducible smooth $\mathbb{C}$-linear representations of $G(F)$, and $\Phi(G)$ is the set of $\hat{G}$-conjugacy classes of $L$-parameters $W_{F} \times \operatorname{SL}_{2}(\mathbb{C}) \to {}^LG(\mathbb{C})$, see Section \ref{subsub:LParameters}. Fargues and Scholze \cite{FarguesScholze} have recently constructed a map\footnote{This map depends a priori on a choice of $\ell$, a choice of $\sqrt{p} \in \qlbar$ and a choice of isomorphism $\iota:\qlbar \to \mathbb{C}$. By recent work of Scholze \cite[Theorem 1.1]{ScholzeMotivicGeometrization}, it is independent of these choices if we choose $\iota$ to take the fixed $\sqrt{p}$ to the positive square root of $p$ in $\mathbb{R}$.}
\begin{align*}
    \operatorname{LL}_G^{\mathrm{FS}}:\Pi_{G} \to \Phi^{\mathrm{ss}}(G),
\end{align*}
where the target is now the set of $\hat{G}$-conjugacy classes of semisimple $L$-parameters $W_F \to {}^LG(\mathbb{C})$. This map is expected to recover the composition of (the conjectural) map $\operatorname{LL}_{G}$ with the semisimplification map $(-)^{\mathrm{ss}}:\Phi(G) \to \Phi^{\mathrm{ss}}(G)$, see Section \ref{subsub:Semisimplification}.  

\subsubsection{} \label{subsub:ClassicalGroups} By work of Arthur \cite{Arthur}, Mok \cite{Mok}, and many others\footnote{The work of Arthur and hence the subsequent work of Mok and many others is conditional on several preprints of Arthur which have not yet appeared. By recent work of Atobe--Gan--Ichino--Kaletha--Minguez--Shin \cite{AGIKMS}, the work of Arthur is now only conditional on the twisted weighted fundamental lemma, see \cite[Assumption H1]{ShinWeakTransfer}.} a map $\operatorname{LL}_G$ has been constructed for many classical groups. More precisely, we consider the following cases.
\begin{itemize}[leftmargin=25pt]
    \item[$U$] The group $G$ is an inner form of a quasi-split unitary group over $F$, and the map $\operatorname{LL}_{G}$ is the one constructed by Mok \cite[Theorem 2.5.1, 3.2.1]{Mok} if $G$ is quasi-split and by Kaletha--Minguez--Shin--White \cite[Theorem 1.6.1]{KMSW} if $G$ is not quasi-split.

    \item[$B$] The group $G$ is an inner form of the split special orthogonal group $\operatorname{SO}_{2n+1,F}$, and the map $\operatorname{LL}_G$ is the one constructed by Arthur \cite[Theorem 1.5.1]{Arthur} if $G$ is quasi-split and by Ishimoto \cite[Theorem 1.2]{Ishimoto} if $G$ is not quasi-split. 
    
    \item[$C_2$] The group $G$ is an inner form of $\operatorname{GSp}_4$, and the map $\operatorname{LL}_G$ is the one constructed by Gan--Takeda \cite[Main theorem]{GanTakeda} if $G$ is split and Gan--Tantono \cite[Main Theorem]{GanTantono} if $G$ is nonsplit.  

    \item[$D$] The group $G$ is a pure inner form of a quasi-split special orthogonal group $\operatorname{SO}(V')$ for an even dimensional quadratic space $V'$ over $F$ (which implies that $G \simeq \operatorname{SO}(V)$ for some quadratic space $V$ over $F$ of the same dimension as $V'$). Noting that the dual group of $\hat{G}$ is $\operatorname{SO}_{2n}$, which has a conjugation action by $\operatorname{O}_{2n}$, we consider the map
    \begin{align*}
        \widetilde{\operatorname{LL}}_G:\Pi_G \to \widetilde{\Phi}(G)=\Phi(G)/\operatorname{O}_{2n},
    \end{align*}
    constructed by Arthur \cite[Theorem 1.5.1]{Arthur} if $G$ is quasi-split and by Chen--Zou \cite[Theorem A.2]{ChenZou} if $G$ is not quasi-split. In this case, we define $\widetilde{\operatorname{LL}}_{G}^{\operatorname{FS}}$ as the composition of $\operatorname{LL}_{G}^{\operatorname{FS}}$ with the map $ \Phi^{\mathrm{ss}}(G) \to  \Phi^{\mathrm{ss}}(G)/\operatorname{O}_{2n}$.
\end{itemize}
\begin{mainThm} \label{Thm:IntroCompatibility}
In cases $U,B,C_2$ listed in Section \ref{subsub:ClassicalGroups}, we have the equality
$(-)^{\mathrm{ss}} \circ \operatorname{LL}_{G} = \operatorname{LL}_G^{\mathrm{FS}}$.
In case $D$, we have the equality
$(-)^{\mathrm{ss}} \circ \widetilde{\operatorname{LL}}_G = \widetilde{\operatorname{LL}}_{G}^{\operatorname{FS}}$.
\end{mainThm}
\begin{Rem}
If $F$ is a finite unramified extension of $\qp$ for $p>2$ and $G$ is unramified over $F$, then Theorem \ref{Thm:IntroCompatibility} was previously known by work of Hamann \cite{HamannGSp4} (case $C_2$), and Peng \cite{Peng} (cases $U,B,D$). If moreover $F=\qp$ and $G$ is a unitary group in an odd number of variables, then case $U$ was dealt with in earlier work of Bertoloni Meli--Hamann--Nguyen \cite{BMHN}. A similar theorem for $\operatorname{GL}_N$ and its inner forms is proved in earlier work of Fargues--Scholze \cite{FarguesScholze} and Hansen--Kaletha--Weinstein \cite{HansenKalethaWeinstein}.
\end{Rem}
\begin{mainCor}[Theorem \ref{Thm:ActualEvenOrthogonalLanglands}]\label{Cor:IntroActualEvenOrthogonalLL}
Let $F$ be a $p$-adic local field and let $G$ be a pure inner form of an even orthogonal group over $F$. There is a unique local Langlands correspondence $\operatorname{LL}_{\operatorname{SO}_{2n}}:\Pi_{\mathrm{temp}} \to \Phi_{\mathrm{temp}}(G)$ lifting $\widetilde{\operatorname{LL}}_{\operatorname{SO}_{2n}}$ such that $(-)^{\mathrm{ss}} \circ \operatorname{LL}_{\operatorname{SO}_{2n}}= \operatorname{LL}_{\operatorname{SO}_{2n}}^{\FS}$. Moreover, it satisfies the conditions listed in \cite[Theorem 7.1.1]{Peng}.
\end{mainCor}
Corollary \ref{Cor:IntroActualEvenOrthogonalLL} is proved in \cite{Peng} if $p>2$ and $F$ is unramified, see \cite[Theorem 7.1.1]{Peng}. Given Theorem \ref{Thm:IntroCompatibility}, his proof goes through verbatim. 
\begin{Rem}
As observed by Peng and Hamann, the strategy of \cite{HamannGSp4}, \cite{BMHN}, and \cite{Peng} would work essentially verbatim for groups over ramified local fields, if basic uniformization were established for certain Shimura varieties associated to Shimura data $\gx$ of abelian type at primes $p$ where $G$ is ramified. Nevertheless, our focus in this work is on the construction of the Igusa stack and its applications, so we do not study basic uniformization in this article, and our proofs do not rely upon it.

Moreover, proving basic uniformization in full generality does not seem straightforward. With some work, the abelian type case likely reduces to the case of suitably chosen Hodge type Shimura varieties, but already in these Hodge type cases, the result is not known. To generalize the strategy of \cite[Proposition 5.2.2]{HeZhouZhu}, the missing ingredient is a CM lifting result for (basic) mod $p$ isogeny classes. For quasi-split groups this is \cite[Theorem 2.2.7]{KisinZhouII} under an assumption at $p=2$, but this assumption notably excludes ramified unitary groups. The authors have been informed that, in forthcoming work \cite{ShenWu}, Xu Shen and Peihang Wu will prove basic uniformization for Shimura varieties of abelian type under the restriction $p\ne 2$.
\end{Rem}

\subsubsection{} We now explain the proof of Theorem \ref{Thm:IntroCompatibility} in the quasi-split case; for notational simplicity, we focus on the case of odd orthogonal groups. We first reduce to the case of supercuspidal representations $\pi$ using induction on the rank of our classical group and compatibility of semisimplified $L$-parameters with parabolic induction; this part of our argument closely follows \cite{HamannGSp4, BMHN, Peng}. The case of supercuspidal $\pi$ moreover breaks up into two subcases depending on whether $\operatorname{LL}_G(\pi)$ is a supercuspidal $L$-parameter or not. \smallskip 

If $\phi=\operatorname{LL}_G(\pi)$ is not supercuspidal, then the $L$-packet of $\phi$ will contain a parabolically induced representation $\pi'$ of $G(\qp)$. By the induction hypothesis, we know that $(-)^{\mathrm{ss}} \circ \operatorname{LL}_{G}(\pi') = \operatorname{LL}_G^{\mathrm{FS}}(\pi')$, and we would like to deduce the same statement for the other elements of the $L$-packet. For this, we can apply
recent work of Hansen \cite[Theorem 1.1]{HansenStable} and Varma \cite[Theorem 4.4.2]{Varma}, to deduce the constancy on the $L$-packet of the Fargues--Scholze $L$-parameter, see \cite[Corollary 1.2]{HansenStable}. \smallskip 

If $\phi=\operatorname{LL}_G(\pi)$ is supercuspidal, then all the members $\pi$ of the $L$-packet of $\phi$ will be cuspidal, and we give a direct proof by a globalization argument: First we choose a totally real field $\mlp \not=\mathbb{Q}$ with a finite place $v$ of $\mlp$ such that $\mlp_{v}=F$. We then globalize $G$ to a connected reductive group $\h$ over $\mlp$ which is compact at all but one infinite place of $\mlp$, so that $\g=\operatorname{Res}_{\mlp/\mathbb{Q}}$ admits a Shimura datum $\x$ of abelian type.\footnote{The Shimura datum $\gx$ is essentially never of Hodge type.} We then globalize $\pi$ to a cohomological cuspidal automorphic representation $\Pi$ of $\h(\mathbb{A}_{\mlp})=\g(\mathbb{A})$ with good properties (e.g. cohomological of regular weight, isomorphic to a twist of Steinberg at an auxiliary place). It now follows from work of Arthur \cite{Arthur}, Kisin--Shin--Zhu \cite{KisinShinZhu} and others that we understand the $\Pi^{\infty}$-isotypic part of the cohomology of the Shimura varieties for $\gx$. More precisely, write 
\begin{align*}
    r_{\mu}: {}^L H \to \operatorname{GL}(V_{\mu})
\end{align*}
for the representation of the $L$-group of $H$ associated to $\x$ and $v$. For example, in the case of odd orthogonal groups $\operatorname{SO}_{2n+1}$ under consideration, this is just the standard representation
\begin{align*}
    \operatorname{Sp}_{2n} \xhookrightarrow{\mathrm{Std}} \operatorname{GL}_{2n}.
\end{align*}
Then by work of Arthur, Mok and others (see Theorem \ref{Thm:ExistenceLParameters}) there is an irreducible global Galois representation $R_{\Pi}:\gal_{\mlp} \to \operatorname{GL}_{2n}(\qlbar)$ such that (up to Tate-twist)
\begin{align*}
    \restr{R_{\Pi}}{W_{F}} \simeq r_{\mu} \circ \phi. 
\end{align*}
The complex $ R\Gamma(\mathbf{Sh}\gx, \mathbb{W})[\Pi^{\infty}]$, where $\mathbb{W}$ is an automorphic $\qlbar$ local system corresponding to $\Pi_{\infty}$, is concentrated in middle degree. By work of Kisin--Shin--Zhu \cite{KisinShinZhu}, its middle cohomology group is isomorphic up to semisimplification to $R_{\Pi}^{\oplus m}$ for some $m \in \mathbb{Z}_{\ge 0}$ (up to the same Tate-twist). It is at this point that we use the existence of Igusa stacks through Corollary \ref{Cor:StrongCompatibility}, to deduce that
\begin{align} \label{eq:IntroEqualityComposition}
    \left(\restr{R_{\Pi}}{W_{F}}\right)^{\mathrm{ss}} \simeq r_{\mu} \circ \operatorname{LL}_{G}^{\mathrm{FS}}(\pi).
\end{align}
To deduce this from Corollary \ref{Cor:StrongCompatibility}, we crucially use the supercuspidality of $\phi=\operatorname{LL}_{G}(\pi)$ which implies that all the irreducible $W_E$-subquotients of $\restr{R_{\Pi}^{\oplus m}}{W_{F}}$ occur with multiplicity $1$. To conclude, we note that the equality \cref{eq:IntroEqualityComposition} together with work of Gan--Gross--Prasad \cite[Theorem 8.1]{GGP} implies that the $L$-parameters $\phi$ and $\operatorname{LL}_G^{\mathrm{FS}}(\pi)$ are conjugate to one another.

\subsubsection{} The above argument does not work for non quasi-split groups, because it is not necessarily true that if $\phi$ is not supercuspidal, then its $L$-packet must contain a parabolically induced representation. Instead, we show that the conclusion of Theorem \ref{Thm:IntroCompatibility} propagates from quasi-split groups $G^{\ast}$ to their extended pure inner forms $G=G_b$. We deduce this from another recent result of Hansen \cite[Theorem 1.4]{HansenStable} together with the endoscopic character relations between $G$ and $G^{\ast}$ for $\operatorname{LL}_{G}$ and $\operatorname{LL}_{G^{\ast}}$. 

\begin{Rem}
It is crucial that we can take $r_{\mu}$ to be the standard representation of $\operatorname{Sp}_{2n}$, since otherwise $\operatorname{LL}_{G}^{\mathrm{FS}}(\pi)$ cannot be recovered from $r_{\mu} \circ \operatorname{LL}_{G}^{\mathrm{FS}}(\pi)$. This explains why we cannot prove Theorem \ref{Thm:IntroCompatibility} for $\operatorname{GSp}_{2n}$ with $n \ge 3$: The representation $r_{\mu}$ of the dual group $\operatorname{GSpin}_{2n+1}$ coming from the Hodge cocharacter $\mu$ of a Shimura datum for $\operatorname{GSp}_{2n}$ is given by the $2^{n}$-dimensional spin representation. For $\operatorname{GSp}_{2n}$, we do prove a statement after composition with the spin representation, under some assumptions, see Corollary \ref{Cor:StrongCompatibility}.
\end{Rem}

\subsection{Cohomological results}\label{Sec:CohomSummary} Let $\gx$ be a Shimura datum of abelian type with reflex field $\mathsf{E}$ and suppose that $\g=\operatorname{Res}_{\mathsf{F}/\mathbb{Q}} \h$ for a totally real field $\mathsf{F}$ and a connected reductive group $\h$ over $\mathsf{F}$; let $K \subset \gaf$ be a neat compact open subgroup. Nekov\'ar--Scholl \cite[Conjecture 6.1]{NekovarScholl} predict that the action of $\gal_{\mathsf{E}}$ on $R \Gamma_{\text{\'et}}(\mathbf{Sh}_L\gx_{\ebar}, \zl)$ extends to the action of a larger group, the plectic reflex Galois group. In particular, for a place $v$ of $\mathsf{E}$ with completion $E$, the action of $\gal_{E}$ should extend to the action of a local version of the plectic reflex Galois group. 

Suppose that $v$ lies above a prime $p$ that splits completely in $\mathsf{F}$. Then the local plectic reflex Galois group is a product $\prod_{i} \operatorname{Gal}_{E_i}$, see Section \ref{subsub:LocalPlecticReflex}, where the product is indexed by places of $\mathsf{F}$ above $p$ and each $E_i$ is a finite extension of $\qp$. There is moreover a natural map $\gal_{E} \to \prod_{i} \operatorname{Gal}_{E_i}$. Let $\ell$ be a prime distinct from $p$ and let $\Lambda$ be a $\mathbb{Z}/\ell^n \mathbb{Z}$-algebra containing a fixed square root of $p$.
\begin{mainThm}[Theorem \ref{Thm:Plectic}] \label{Thm:IntroPlectic}
Suppose that $\ell$ is coprime to the order of $\pi_0(Z(G))$. If $K=K_pK^p$ with $K^p$ sufficiently small $($see Definition \ref{Def:acceptable}$)$, then the cohomology complex
    \[R\Gamma(\mathbf{Sh}_{K}(\mathsf{G},\mathsf{X})_{\overline{E}}, \Lambda)\in D(W_{E},\Lambda)\]
    lifts canonically to an object in $D(\prod_i W_{E_i},\Lambda)$, compatibly with the Hecke actions. 
\end{mainThm}
\begin{Rem}
This is proved at unramified primes $p$ and with $\zlbar$ coefficients for the basic part of the cohomology of the Shimura variety by Li-Huerta \cite[Theorem A]{LiHuertaPlectic}.
\end{Rem}
If $K_p$ is hyperspecial, then the action of $\gal_{E}$ on the cohomology of the Shimura variety is unramified. Theorem \ref{Thm:IntroPlectic} then gives us commuting partial Frobenius operators acting on the cohomology. It turns out these partial Frobenii satisfy a version of the Eichler--Shimura relation of Blasius--Rogawski \cite[Section 6]{BlasiusRogawski}. 
\begin{mainCor}[Corollary \ref{Cor:PartialEichlerShimura}] \label{Cor:IntroPartialEichlerShimura}
Suppose that $\ell$ is coprime to the order of $\pi_0(Z(G))$. If $K=K_pK^p$ with $K_p$ hyperspecial and with $K^p$ sufficiently small, then there exist mutually commuting partial Frobenius operators $\sigma_{i}$ acting on $R\Gamma(\mathbf{Sh}_{K}(\mathsf{G},\mathsf{X})_{\overline{E}}, \Lambda)$, such that for any lift $\sigma$ of $\operatorname{Frob}_q$, we have
    \[\sigma=\prod_{i=1}^r \sigma_i\]
as endomorphisms of $R\Gamma(\mathbf{Sh}_{K}(\mathsf{G},\mathsf{X})_{\overline{E}}, \Lambda)$. Moreover, each $\sigma_i$ is annihilated by the renormalized Hecke polynomial $H_{\mu_i}(X)\in \mathcal{H}_{\mathcal{G}_i}[X]$ constructed in Section \ref{subsub:EichlerShimura}.
\end{mainCor}
Note that Corollary \ref{Cor:IntroPartialEichlerShimura} has previously been announced by Zhu \cite[Theorem 3.2.2]{zhu2025ICM}. When $E=\qp$, Corollary \ref{Cor:PartialEichlerShimura} is also proved by Lee \cite[Theorem 1.0.4]{Lee} by a completely different method.
\begin{Rem}
Corollary \ref{Cor:IntroPartialEichlerShimura} is new even when $\mathsf{F}=\mathbb{Q}$ and $E \not=\qp$. In this case, it verifies the conjectural Eichler--Shimura relation of Blasius--Rogawski at hyperspecial level, see \cite[Section 6]{BlasiusRogawski}. See \cite[Corollary 1.2]{WuSEqualsT} for the Hodge type case, and \cite[the discussion preceding Theorem III]{DvHKZIgusaStacks} for an overview of previous work on the conjecture.
\end{Rem}
\begin{Rem}
We also prove a version of Corollary \ref{Cor:IntroPartialEichlerShimura} at Iwahori level, see Theorem \ref{Thm:EichlerShimura}. This generalizes \cite[Theorem 1.1]{vdH}, which deals with the case of Hodge type Shimura varieties, generalizing \cite[Theorem III]{DvHKZIgusaStacks} which dealt only with the case that $G$ is unramified. Our proof uses the methods of van den Hove \cite{vdH} in combination with Theorem \ref{Thm:WeilCohShimVar}.
\end{Rem}

\subsubsection{} We now give a brief summary of our other cohomological results. For the sake of brevity, and because many are generalizations of results proved in \cite{DvHKZIgusaStacks}, we do not state these results in full here. \smallskip

Theorem \ref{Thm:WeilCohShimVar} expresses the (compactly supported) cohomology of the Shimura variety of level $K^p$ in terms of a Hecke operator applied to the sheaf on $\bun_{G,\fpbar}$, given by the (exceptional) pushforward of a constant sheaf along $\igs_{K^p} \gx_{\fpbar} \to \bun_{G,\fpbar}$. The Hodge type analogue is \cite[Theorem IV]{DvHKZIgusaStacks}, although we give a different proof. \smallskip

Theorem \ref{Thm:SV-FS-Compatibility} gives a weak compatibility result between the cohomology of Shimura varieties and the Fargues--Scholze local Langlands correspondence, it is the abelian type analogue of \cite[Theorem II]{DvHKZIgusaStacks} and is proved analogously. For compact Shimura varieties of abelian type, we use Matsushima's formula to prove a strengthening of Theorem \ref{Thm:SV-FS-Compatibility}, see Corollary \ref{Cor:WeakCompatibility}. This corollary is crucially used in the proof of Theorem \ref{Thm:IntroCompatibility}. \smallskip

Theorem \ref{Thm:MantovanProduct} is a version of Mantovan's product formula (see \cite[Theorem 1]{Mantovan}) expressing the compactly supported cohomology of the Shimura varieties in terms of the compactly supported cohomology of the fibers of $\pi_{\mathrm{HT}}$\footnote{The Mantovan product formula is usually stated in terms of the cohomology of the fibers of $\pi_{\mathrm{HT}}^{\circ}$, which can often be identified with the cohomology of perfect Igusa varieties.} and the cohomology of local Shimura varieties. In the non-compact case, this generalizes \cite[Proposition 3.17]{Hamann-Lee} and in the compact case it generalizes \cite[Theorem 8.5.7]{DvHKZIgusaStacks}. \smallskip

\begin{Rem}
We do not prove a generalization of \cite[Theorem VII]{DvHKZIgusaStacks}, which gives fiber product diagrams on the level of canonical integral models of Shimura varieties. However, see Conjecture \ref{Conj:IgusaMainInt} for a precise conjecture. 
\end{Rem}
\begin{Rem}
We do not address torsion vanishing in this paper, since the results of \cite{YangZhuGeneric} apply to most Shimura varieties of abelian type with parahoric level at $p$. For the same reason, we do not attempt to prove perversity of the Igusa sheaf, but see Conjecture \ref{Conj:ULAPerverse} for a precise conjecture. 
\end{Rem}

\begin{Rem}
We expect a version of Theorem \ref{Thm:WeilCohShimVar} to hold with coefficients in Berkovich motives as in \cite{scholze2026berkovichmotives} and \cite{ScholzeMotivicGeometrization}. Taking $\ell$-adic realizations for $\ell \not=p$ this should then recover Theorem \ref{Thm:WeilCohShimVar} as well as a version with $\zl$ or $\ql$-coefficients. More interestingly, there is a $p$-adic or Hyodo--Kato realization (as in \cite{HyodoKatoStacks}) of Berkovich motives, see \cite[Theorem 1.7]{AokiBerkovich}. The Hyodo--Kato realization of the Igusa sheaf is expected to encode the rigid cohomology of Igusa varieties, see \cite[Remark 1.4.7]{HyodoKatoStacks}.
\end{Rem}

\begin{Rem}
Recently, Scholze has suggested in a talk \cite[49:00 onwards]{IAS2024AnalyticPrismatization} that there should be a ``locally analytic'' version of the Igusa stack diagram which might be used to prove results about the locally analytic vectors in the completed cohomology of Shimura varieties; this is the subject of a work-in-progress of Ansch\"utz--Le Bras--Rodr\'iguez Camargo--Scholze. A variant of such a locally analytic Igusa stack diagram is worked out in \cite[Section 10]{howe2026inscriptiontwistorspadicperiods}, using the existence of an Igusa stack as an input. 
\end{Rem}

\subsection{The proof of Theorem \ref{Thm:IntroIgusaGeneric}} \label{sub:IntroGeometric} Let $\gxtwo$ be a Shimura datum of abelian type with reflex field $\mathsf{E}_2$. Work of Deligne and Lovering, \cite{DeligneVarietes}, \cite{LoveringAutomorphic}, tells us that we can find a diagram of Shimura data
\begin{equation*}
    \begin{tikzcd}
        & \gxthree \arrow{dr} \arrow{dl} \\
        \gxtwo && \gx,
    \end{tikzcd}
\end{equation*}
where $\gx$ is of Hodge type, where the left arrow is an ad-isomorphism and where the right arrow induces an isomorphism of derived groups. We moreover have fairly precise control over the reflex fields of $\gx$ and $\gxthree$. We will start with the Igusa stack $\igs \gx := \varprojlim_{K^p} \igs_{K^p}\gx$ and construct an Igusa stack for $\gxtwo$, the $\circ$-variants are treated in the same way. \smallskip

\step\label{step:tori} We construct Igusa stacks for tori using the geometric interpretation of class field theory of Fargues and Fargues--Scholze. More precisely, we consider the v-sheaf $[\ul{\mathsf{T}(\mathbb{Q})^-
\backslash \mathsf{T}(\af)} / \ul{T(\qp)}]$ over $\fpbar$ and explicitly write
down a Galois descent datum for $\fpbar/\fp$ using global class field theory and the Hodge cocharacter $\mu$. This defines the Igusa stack
$\mathrm{Igs}(\mathsf{T}, \{\mu\})$ over $\fp$. To check that this is an Igusa stack, we explicitly write down a presentation of $\bun_{T,\mu^{-1}}$ over $\spd \fp$, and give an explicit formula for the Beauville--Laszlo map
$\mathrm{BL} \colon \spd E \to \bun_T$ in this presentation. To do this, we reduce to the case that $T_{\qp}=\operatorname{Res}_{E / \qp} \mathbb{G}_m$
for some finite extension $E$ of $\qp$, and then use the geometric interpretation of Lubin--Tate theory (for $E$) of Fargues and Fargues--Scholze. Finally, we take the base change of
$\mathrm{Igs}(\mathsf{T}, \{\mu\})$ along $\mathrm{BL}$ and show that it recovers the canonical model of $\mathbf{Sh}(\mathsf{T},\{\mu\})$. This is carried out in \Cref{Sec:Tori}.

\step\label{step:ascend} To construct Igusa stacks for $\gxthree$, we note that following diagram is Cartesian (where $\mathrm{ab}$ means maximal abelian quotient)
\begin{equation*}
    \begin{tikzcd}
        \gxthree \arrow{r} \arrow{d} & \gxthreeab \arrow{d} \\ 
        \gx \arrow{r}&  \gxab.
    \end{tikzcd}
\end{equation*}
Thus $\gxthree$ sits inside $\gxthreeab \times \gx$. Since $\gx$ admits an Igusa stack by \cite[Theorem D]{KimFunctorial}, and $\gxthreeab$ admits an Igusa stack by Step 1, it is relatively straightforward to prove that their product admits an Igusa stack. We now apply \cite[Theorem C]{KimFunctorial} to deduce that $\gxtwo$ admits an Igusa stack. 

\step\label{step:descend} We descend the Igusa stack for $\gxthree$ down to an Igusa stack for $\gxtwo$. The Shimura variety $\mathbf{Sh}\gxthree$ has an action of Deligne's locally profinite group $\mathcal{A}(\g_3)$, and it is well known\footnote{We were not able to locate a proof in the literature, and so we have given a proof of this fact in Proposition \ref{Prop:PushoutShimura} using work of Deligne \cite{DeligneVarietes}.} that $\mathbf{Sh}\gxtwo$ is the pushout of $\mathbf{Sh}\gxthree$ along the natural morphism of locally profinite group schemes $\ul{\mathcal{A}(\g_3)} \to \ul{\mathcal{A}(\g_2)}$. 

It follows from the functoriality results of one of us (DK) \cite[Theorem B]{KimFunctorial} that $\ul{\mathcal{A}(\g_3)}$ acts on $\igs \gxthree$, and in fact it acts through the quotient\footnote{The group $G_3(\qp)$ injects into $\mathcal{A}(\g_3)$ but does not typically have closed image, thus our $\ul{\mathcal{A}(\g_3)}^{p}$ does not agree with the v-sheaf associated to the topological quotient $\mathcal{A}(G_3)/G_3(\qp)$.} $\ul{\mathcal{A}(\g_3)}^{p}:=\ul{\mathcal{A}(\g_3)}/\ul{G_3(\qp)}$. Roughly speaking, we then show that the pushout of $\igs \gxthree$ along $\ul{\mathcal{A}(\g_3)}^{p} \to \ul{\mathcal{A}(\g_2)}^{p}$ defines an Igusa stack for $\gxtwo$. Our actual argument proceeds in a different way, using the perspective of uniformization maps of \cite{KimFunctorial}, essentially because of subtleties of $2$-categorical nature. 

\step\label{step:finitelevel} We go from the Igusa stack $\igs \gxtwo$ to the finite level object $\igs_{K^p} \gxtwo$. This is straightforward if $\zg(\Q) \subset \zg(\af)$ is closed (i.e., Milne's axiom SV5 holds), where $\zg$ is the center of $\g$. However, if SV5 fails then this is somewhat painful, see \Cref{Cor:FiniteLevelIgusa1}. Here again we use the formalism from \cite{KimFunctorial}.

\step\label{step:dualizing} We prove that the Igusa stack $\igs_{K^p}\gxtwo$ is $\ell$-cohomologically smooth of $\ell$-dimension zero, with trivial dualizing sheaf. If SV5 holds, then the proof of \cite[Theorem 8.3.1]{DvHKZIgusaStacks} essentially goes through. In general, we need to compute the relative dualizing sheaves of gerbes for locally profinite groups, which we do in Lemma \ref{Lem:GerbeDualizing}.

\subsection{Organization of the paper} Let us give a brief guide to help the reader navigate this paper.

Section \ref{Sec:Prelims} contains some preliminaries. In particular, we recall the abstract definition of Igusa stack given in \cite{KimFunctorial}, and the equivalence between Igusa stacks and global uniformization maps, established in \textit{loc. cit.}.

Sections \ref{Sec:Tori} and \ref{Sec:IgusaStacksAbelianType} construct the Cartesian diagram in Theorem \ref{Thm:IntroIgusaGeneric}. In Section \ref{Sec:Tori}, we prove Theorem \ref{Thm:IntroIgusaGeneric} for toral-type Shimura data. In Section \ref{Sec:IgusaStacksAbelianType} we construct Igusa stacks for arbitrary Shimura data of abelian type, using as inputs the results of Section \ref{Sec:Tori} and the results of \cite{KimFunctorial}. In particular, the Cartesian diagram in Theorem \ref{Thm:IntroIgusaGeneric} is proved at the end of this section.

In Section \ref{Sec:Cohomology}, we first complete the proof of Theorem \ref{Thm:IntroIgusaGeneric} by proving that the Igusa stack is $\ell$-cohomologically smooth with constant dualizing sheaf. This section also contains the cohomological results described in Section \ref{Sec:CohomSummary}. 

In Section \ref{Sec:CohCompact}, we study the cohomology of compact Shimura varieties of abelian type and establish its compatibility with Fargues--Scholze parameters for certain classical groups of type $A$, $B$, $D$, and $C_2$. These compatibilities are then used in the proof of Theorem \ref{Thm:IntroCompatibility}, which occupies most of Section \ref{Sec:Compatibility}. We prove Corollary \ref{Cor:IntroActualEvenOrthogonalLL} in Section \ref{Sec:Compatibility} as well. 

Finally, Appendix \ref{sub:SixFunctor} is devoted to extending results of \cite[Section 8.2]{DvHKZIgusaStacks} on canonical Frobenius descent data, using six-functor formalisms.

\subsection{Notation and conventions} \label{Sec:Conventions}

\begin{itemize}[leftmargin=*]
  \item A map $f \colon G \to H$ of reductive groups over a field $F$ is an \textit{ad-isomorphism} if $f(Z(G)) \subset Z(H)$ and the induced map $f^\mathrm{ad} \colon G^\mathrm{ad} \to H^\mathrm{ad}$ is an isomorphism.
  \item Let $\perf$ be the category of perfectoid spaces in characteristic $p$.
  \item If $T$ is a topological space, we denote by $\ul{T}$ the sheaf on
    $\perf$ defined such that $\ul{T}(X)$ is the set of continuous maps from the
    underlying topological space $\lvert X \rvert$ of $X$ to $T$.
  \item For every $X \in \perf$ we denote by $\phi_X \colon X \to X$ the
    absolute Frobenius automorphism. This naturally extends to an absolute
    Frobenius $\phi \colon X \to X$ for an arbitrary v-stack $X \in
    \operatorname{Stk}(\perf_\mathrm{v})$, and induces $\phi \colon
    \mathcal{Y}_{[0,\infty)}(R, R^+) \to \mathcal{Y}_{[0,\infty)}(R, R^+)$ as
    well as $\phi \colon \qpbr \to \qpbr$ a lift of the absolute Frobenius.
  \item For $G$ an algebraic group or a profinite group, by a $G$-torsor we mean
    a left $G$-torsor unless specified otherwise.
  \item Let $G/\qp$ be a connected reductive group. By a $G$-isocrystal on a
    perfect $\fp$-algebra $R$ we mean a $G$-torsor $\mathscr{P}$ on $\spec
    W(R)[1/p]$ together with a morphism $\phi_\mathscr{P} \colon \phi^\ast
    \mathscr{P} \dashrightarrow \mathscr{P}$. Given an element $b \in G(\qpbr)$,
    we denote by $\mathscr{P}_b$ the $G$-isocrystal
    \[
      (G_{\qpbr}, \phi^\ast G_{\qpbr} \cong G_{\qpbr} \xrightarrow{r_{b^{-1}}}
      G_{\qpbr})
    \]
    where $r_{b^{-1}}$ is right multiplication by $b^{-1}$. We also denote by
    $\mathscr{P}_b$ the induced map $\spd \fpbar \to \bung$, see
    \cite[Theorem~7.14]{GleasonIvanovZillinger}. Its automorphism group is given
    by
    \[
      G_b(\qp) = \lbrace g \in G(\qpbr) : g b \phi(g)^{-1} = b \rbrace,
    \]
    with the right translation action. 
  \item For $E/F$ a Galois extension, we let $\gal(E/F)$ act from the left on
    $E$ and act from the right on $\spec E$.
  \item Given $E$ a local field, we use the geometric normalization for the
    local Artin map $\mathrm{Art}_E \colon E^\times \hookrightarrow
    \gal(E^\mathrm{ab}/E)$ so that the uniformizer maps to a lift of
    $\phi^{-[k_E:\fp]} \in \gal(\fpbar/k_E)$. For the global Artin map, we use
    the normalization that is compatible with the local Artin map, via inclusion
    in the source and restriction in the target. For example,
    $\mathrm{Art}_\mathbb{Q} \colon \widehat{\mathbb{Z}}^\times \cong
    \mathbb{Q}^\times \backslash \mathbb{A}_\mathbb{Q}^\times / \mathbb{R}_{>0}
    \to \gal(\mathbb{Q}(\zeta_\infty)/\mathbb{Q})$ satisfies
    $\mathrm{Art}_\mathbb{Q}(\gamma)(\zeta_n) = \zeta_n^{\gamma \bmod{n}}$.

    \item We will denote Shimura data and their reflex fields by sans serif fonts, e.g., $\gx$ for a Shimura datum and $\mathsf{E}$ for its reflex field. For a prime $p$ and a place $v$ of $\mathsf{E}$ above $p$, we will write $E$ for the completion of $\mathsf{E}$ at $v$, which has ring of integers $\mathcal{O}_E$ and residue field $k_E$. We will then write $G$ for the base change of $\mathsf{G}$ to $\qp$.

    \item Our (conjugacy class of) Hodge cocharacter(s), normalized as in \cite[Section 1.3.1]{KisinPoints}, is denoted by $\mu$ and is defined over $\mathsf{E}$. We consider this as a $G(\qpbar)$-conjugacy class using the place $v$ of $\mathsf{E}$. We use $\mathrm{Gr}_G$ to denote the $\mathbf{B}_\mathrm{dR}^+$-affine Grassmannian of Scholze--Weinstein, with its stratification by Schubert cells as defined in \cite[Definition 19.2.2]{ScholzeWeinsteinBerkeley} (this
    agrees with \cite{PappasRapoportShtukas} but is the opposite of
    \cite{CaraianiScholzeCompact}). The Shimura variety over $E$ with infinite level at $p$ has a Hodge--Tate period map with target in $\mathrm{Gr}_{G,{\mu^{-1}}}$. Its Newton stratification is moreover indexed by the $\mu^{-1}$-admissible set $\bgmu \subset B(G)$.
\end{itemize}

\subsection{Acknowledgements} We are very grateful to David Hansen for suggesting to us that Theorem \ref{Thm:IntroIgusaGeneric} could be used to prove Theorem \ref{Thm:IntroCompatibility}, and for sketching the strategy of the proof. We thank Linus Hamann sincerely for many helpful conversations and for patiently answering our questions regarding \cite{HamannGSp4}, \cite{BMHN}. We would like to thank Sandeep Varma for answering our questions about \cite{Varma}, and Lucas Mann for suggesting the argument in the Appendix. Thanks also to Alexander Bertoloni Meli, Rui Chen, Yuanyang Jiang, Tasho Kaletha, Masao Oi, Peter Scholze, Jack Sempliner, Sug Woo Shin, Matteo Tamiozzo, Richard Taylor, Alex Youcis, and Xinwen Zhu for helpful conversations and correspondences.
}

{\section{Preliminaries} \label{Sec:Prelims}

\NewDocumentCommand\shgx{o t0 t? t\d}{\mathbf{Sh}\IfValueT{#1}{_{#1}}\gx\IfBooleanTF{#4}{^{\IfBooleanT{#2}{\circ,}\IfBooleanT{#3}{?,}\diamondsuit}}{\IfBooleanT{#2}{^\circ}\IfBooleanT{#3}{^?}}}

We work
with the category of sheaves (or stacks) on $\perf$ with the v-topology. For a
reductive group $G/\qp$, we may consider the moduli space $\bung$ of $G$-torsors
over the Fargues--Fontaine curve, and given a geometric cocharacter $\mu \colon
\mathbb{G}_{m,\qpbar} \to G_{\qpbar}$ the corresponding subset $B(G, \mu^{-1})$
defines an open v-subsheaf $\bungmu \subseteq \bung$. For each scheme or adic
space $X/\qp$, there is a natural v-sheaf $X^\diamondsuit \to \spa \qp$. We refer the reader to \cite[Section~2.1]{DvHKZIgusaStacks} for
the details.

\subsection{Igusa stacks}

\subsubsection{}
Let $\gx$ be a Shimura datum with reflex field $\mathsf{E}$. We choose a place
$v \mid p$ of $\mathsf{E}$ and write $E = \mathsf{E}_ v$ and $G =
\mathsf{G}_{\qp}$. There is an \'{e}tale scheme $X_\ast(\g) \sslash \g$ over
$\mathbb{Q}$ that parametrizes cocharacters of $\g$ up to conjugacy, where the Hodge
cocharacter defines an open and closed embedding $\mu \colon \spec \mathsf{E}
\hookrightarrow X_\ast(\g) \sslash \g$. Over $\qp$, we then get a conjugacy class
$\lbrace \mu \colon \mathbb{G}_{m,\qpbar} \to G_{\qpbar} \rbrace$ upon fixing an
embedding $E \hookrightarrow \qpbar$. 

\subsubsection{}
For each neat open compact subgroup $K \subseteq \gaf$ there is a Shimura
variety $\shgx[K]$ over $\mathsf{E}$. We may base change to $E$ and consider the
associated v-sheaf $\shgx[K]\d_E$. We also consider the infinite level Shimura
variety
\[
  \shgx\d_E = \varprojlim_{K \to \{1\}} \shgx[K]\d_E,
\]
where the limit is taken in the category of v-sheaves. Because the locally
profinite group $\gaf$ acts continuously on the tower, the limit $\shgx\d_E$ carries
an action of $\ul{\gaf}$.

\subsubsection{}
When $\gx$ is of abelian type, \cite{ImaiMieda} constructs for every neat open
compact subgroup $K \subseteq \gaf$ a quasi-compact open subset
\[
  \shgx[K]0_E \subseteq \shgx[K]^\mathrm{an}_E
\]
of the rigid analytification of $\shgx[K]_E$ called the \textit{potentially
crystalline locus}. It is uniquely characterized by quasi-compactness together
with the property that its classical points are precisely those whose associated
$p$-adic Galois representation is potentially crystalline. The potentially
crystalline locus is compatible with pulling back along $K^\prime
\hookrightarrow K$, and hence we obtain a quasi-compact open v-subsheaf
\[
  \shgx0\d_E \subseteq \shgx\d_E.
\]
When $\mathsf{G}$ is $\mathbb{Q}$-anisotropic modulo center, this inclusion is
an equality, see \cite[Example 5.20]{ImaiMieda}.

\subsubsection{}
As discussed in \cite[Section~5.4]{HansenOberwolfach},
\cite[Section~2.6]{PappasRapoportShtukas}, \cite[Section~4]{KimFunctorial},
there is a $\ul{G(\qp)}$-equivariant morphism
\[
  \pi_\mathrm{HT} \colon \shgx\d_E \to \grgmu,
\]
called the Hodge--Tate period map. We recall also the Beauville--Laszlo map
\[
  \mathrm{BL} \colon \grgmu \to \bungmu,
\]
which is surjective in the pro-\'{e}tale topology, see the proof of
\cite[Proposition~6.4.1]{DvHKZIgusaStacks}.

\begin{Def}[{\cite[Definition~5.5]{KimFunctorial}}] \label{Def:IgusaStack}
  Let $\gx$ be a Shimura datum, choose a place $v \mid p$ of $\mathsf{E}$, and
  consider $? \in \{\circ, \emptyset\}$. An \textit{Igusa stack} (for $?$) is a v-stack
  $\igs^?\gx$ carrying a $\ul{\g(\afp)}$-action, together with a
  $\ul{\g(\afp)}$-equivariant map $\bar{\pi}_\mathrm{HT} \colon \igs^?\gx \to
  \bungmu$, and a $2$-Cartesian diagram
  \[ \begin{tikzcd}
    \shgx?\d_E \arrow{r}{\pi_\mathrm{HT}} \arrow{d} & \grgmu
    \arrow{d}{\mathrm{BL}} \\ \igs^?\gx \arrow{r}{\bar{\pi}_\mathrm{HT}} &
    \bungmu
  \end{tikzcd} \]
  such that the induced isomorphism
  \[
    \shgx?\d_E \cong \igs^?\gx \times_{\bungmu} \grgmu
  \]
  is equivariant for the natural $\ul{\g(\af)}$-actions on both sides, and $\phi
  \times \id$ acts trivially on the right hand side.
\end{Def}

Oftentimes, it is more convenient to encode the information of the Igusa stack
as a single map between v-sheaves.

\begin{Def}[{\cite[Definition~5.11]{KimFunctorial}}] \label{Def:UniversalUniformization}
  In the setting of \Cref{Def:IgusaStack}, a \textit{global uniformization} is a
  map of v-sheaves
  \[
    \Theta \colon \shgx?\d_E \times_{\bungmu} \grgmu \to \shgx?\d_E
  \]
  such that
  \begin{enumerate}
    \item the map $\Theta$ is $\ul{\g(\af) \times G(\qp)}$-equivariant, where
      the action on $\shgx?\d_E$ is given via $\g(\af) \times G(\qp) = \g(\afp)
      \times G(\qp) \times G(\qp) \xrightarrow{\mathrm{pr}_{13}} \g(\afp) \times
      G(\qp) = \g(\af)$,
    \item the map $\Theta$ intertwines the $\phi \times \id$-action on the
      source and the trivial action on the target,
    \item the diagrams
      \[ \begin{tikzcd}
        \shgx?\d_E \arrow[equals]{r} \arrow{d}{(\id,\pi_\mathrm{HT})} &
        \shgx?\d_E \arrow{d}{\pi_\mathrm{HT}} \\ \shgx?\d_E \times_{\bungmu}
        \grgmu \arrow{r}{\mathrm{pr}_2} \arrow{ru}{\Theta} & \grgmu
      \end{tikzcd} \]
      \[ \begin{tikzcd}
        \shgx?\d_E \times_{\bungmu} \grgmu \times_{\bungmu} \grgmu
        \arrow{r}{\Theta \times \id} \arrow{d}{\mathrm{pr}_{13}} & \shgx?\d_E
        \times_{\bungmu} \grgmu \arrow{d}{\Theta} \\ \shgx?\d_E
        \times_{\bungmu} \grgmu \arrow{r}{\Theta} & \shgx?\d_E
      \end{tikzcd} \]
      commute.
  \end{enumerate}
\end{Def}

We may now summarize the results of \cite{KimFunctorial} as follows.

\begin{Thm} \label{Thm:IgusaStackKim}
  \def\shgxpqd{\mathbf{Sh}\gxp^{?,\diamondsuit}_{E^\prime}}
  \def\bungpmu{\bun_{G^\prime,\mu^{\prime-1}}}
  \def\flgpmud{\Fl_{G^\prime,\mu^{\prime-1}}^\diamondsuit}
  Let $\gx$ be a Shimura datum with reflex field $\mathsf{E}$.
  \begin{enumerate}[$($i\,$)$]
    \item The existence of an Igusa stack $\igs^?\gx$
      is equivalent to the existence of a global uniformization $\Theta$.
    \item If there exists an Igusa stack $\igs^?\gx$, then it is unique up to
      unique isomorphism. Equivalently, if there exists a global uniformization
      $\Theta$, then it is unique. 
    \item Let $\gx \to \gxp$ be a morphism of Shimura data inducing a morphism
      $\shgx?\d_E \to \shgxpqd$. If there exist Igusa stacks $\igs^?\gx$ and
      $\igs^?\gxp$, then there exists an
      induced map $\igs^?\gx \to \igs^?\gxp$ which is unique up to unique isomorphism. Equivalently, the global
      uniformization maps fit in a commutative diagram
      \[ \begin{tikzcd}[row sep=small]
        \shgx?\d_E \times_{\bungmu} \grgmu \arrow{r}{\Theta} \arrow{d} &
        \shgx?\d_E \arrow{d} \\ \shgxpqd \times_{\bungpmu} \operatorname{Gr}_{G', \mu^{\prime-1}}
        \arrow{r}{\Theta^\prime} & \shgxpqd.
      \end{tikzcd} \]
  \end{enumerate}
\end{Thm}

\subsection{Locally profinite groups}

\subsubsection{}
We recall that a short exact sequence of locally profinite groups is an
(algebraic) short exact sequence
\[
  1 \to H \to G \to K \to 1
\]
where $H \subseteq G$ is closed with respect to the subspace topology, and $G \to K$ has the
quotient topology. Note that every closed subgroup of a locally profinite group
$G$ is locally profinite, and given such a closed subgroup $H \subseteq G$, the
quotient $G/H$ with its quotient topology is a locally profinite set. 

\begin{Prop} \label{Prop:LocProfExact}
  For a short exact sequence $1 \to H \to G \to K \to 1$ of locally profinite
  groups, the induced sequence
  \[
    1 \to \underline{H} \to \underline{G} \to \underline{K} \to 1
  \]
  of presheaves on $\perf ($ hence of pro-\'{e}tale sheaves or v-sheaves$)$ is short exact.
\end{Prop}

\begin{proof}
  It is clear that the sequence is left exact, and it remains to verify surjectivity of $\underline{G} \to \underline{K}$. This follows
  from the fact that $G \to K$ admits a continuous section, see
  \cite[Proposition~2.2.2]{RZProfinite}.
\end{proof}

\subsubsection{} \label{Sec:UnimodularOuter}
Let $G$ be a locally profinite group. There is a corresponding v-sheaf of
automorphisms $\Aut(\ul{G})$, and its action on the $1$-dimensional
$\mathbb{Q}$-vector space $\operatorname{Haar}(G, \mathbb{Q})$ of right Haar
measures defines a character 
\[
  \delta_G \colon \Aut(\ul{G}) \to \ul{\mathbb{Q}_{>0}^\times}, \quad
  \delta_G(\sigma) \mu(\sigma(U)) = \mu(U),
\]
see Lemma \ref{Lem:ModularCharacter}. If $G$ is moreover unimodular, then the inner
automorphisms act trivially on $\operatorname{Haar}(G, \mathbb{Q})$, and hence
we obtain
\[
  \delta_G \colon \operatorname{Out}(\ul{G}) \coloneqq \Aut(\ul{G}) / \ul{G} \to
  \ul{\mathbb{Q}_{>0}^\times}.
\]

\begin{Lem} \label{Lem:ModularCharacter}
  Let $X$ be a topological space and $G$ be a locally profinite group. Let $f
  \colon X \times G \to X \times G$ be a homeomorphism over $X$ with the
  property that $f_x \colon G \to G$ is a group automorphism for every $x \in
  X$. Then the function
  \[
    \delta_G(f) \colon X \to \mathbb{Q}_{>0}^\times; \quad x \mapsto
    \delta_G(f_x)
  \]
  is locally constant. 
\end{Lem}

\begin{proof}
  Fix $x \in X$, fix a compact open $K \subseteq G$, and write $L = f_x(K)$.
  It suffices to find a neighborhood $U \ni x$ such that $f(U \times K) = U
  \times L$. Because $f(X \times K)$ is open and contains $\{x\} \times L$, it
  contains $U_1 \times L$ for some open neighborhood $U_1 \ni x$ since $L$ is
  compact. Similarly, $f^{-1}(X \times L)$ contains $U_2 \times K$ for some
  neighborhood $U_2 \ni x$. We now observe that $U = U_1 \cap U_2$ does the
  trick.
\end{proof}

\begin{Example} \label{Ex:ModularCharQp}
  When $G$ is profinite, every automorphism of $G$ preserves every Haar measure
  since the Haar measure $\mu$ can be normalized to satisfy $\mu(G) = 1$. It
  follows that $\delta_G = 1$. There are unimodular groups $G$ with non-trivial
  $\delta_G$. For example, $G = \qp$ has
  \[
    \delta_G \colon \operatorname{Out}(\ul{G}) \cong \ul{\qp^\times} \to
    p^\mathbb{Z}, \quad \delta_G(x) = p^{v_p(x)}.
  \]
\end{Example}

\subsection{An integral conjecture}
Fix $\gx$ and $v \mid p$ as before. Let $\mathcal{G}$ be a parahoric model of
$G$. Recall the cuspidal quotient $\g \to \g^c$ as in \cite[1.5.8]{KisinShinZhu}, which induces a parahoric model
$\mathcal{G}^c$ of $G^c=(\g^c)_{\qp}$ (see \cite[Section~4.1.1]{DanielsYoucis}).
There is a (conjectural) canonical integral model $\scrs_{K_p}\gx$ of
$\mathbf{Sh}_{K_p}\gx$ over $\mathcal{O}_{E}$, characterized by a morphism
\[
  \scrs_{K_p}\gx^{\diamondsuit/} \to \operatorname{Sht}_{\mathcal{G}^c, \mu^c},
\]
see \cite[Conjecture~4.5]{Daniels} and
\cite[Conjecture~4.2.2]{PappasRapoportShtukas}. The existence is known when $p \ge 3$ and $\gx$ is of abelian type, see \cite{DanielsYoucis}, or when $\gx$ is of Hodge type, see \cite{PappasRapoportShtukas} and \cite{Companion}.

\begin{Conj} \label{Conj:IgusaMainInt}
  Assume that $\gx$ admits an Igusa stack $\igs^\circ\gx$ at $v$. Given a
  parahoric model $\mathcal{G}$ of $G$ with $K_p=\mathcal{G}(\zp)$, we consider
  the fiber product
  \[ \begin{tikzcd} \label{Eq:ConjectureCartesianDiagramInt}
    S \arrow{r}{\pi_{\mathrm{crys}}} \arrow{d} & \shtgmu
    \arrow{d}{\mathrm{BL}^{\circ}} \\ \igs^\circ\gx
    \arrow{r}{\overline{\pi}_\mathrm{HT}} & \bungmu.
  \end{tikzcd} \]
  Then the composition $S \to \shtgmu \to \mathrm{Sht}_{\mathcal{G}^c, \mu^c}$ is
  $\ul{\gafp}$-equivariantly isomorphic to $\scrs_{K_p}\gx^{\diamond} \to
  \mathrm{Sht}_{\mathcal{G}^c, \mu^c}$.
\end{Conj}

If $\gx$ is of Hodge type, then $\g=\g^c$ and this is
\cite[Theorem~VII]{DvHKZIgusaStacks}.

\begin{Rem}
One could also ask for an analogue of Conjecture \ref{Conj:IgusaMainInt} for $\igs\gx$ instead of $\igs^\circ\gx$. In this case, we expect the fiber product to be $\scrs_{K_p}\gx^{\diamondsuit}$. The existence of the map $\scrs_{K_p}\gx^{\diamondsuit} \to \mathrm{Sht}_{\mathcal{G}^c, \mu^c}$ is being investigated in ongoing work \cite{MaoWu} of Mao--Wu.
\end{Rem}

\begin{Rem} \label{Rem:IntegralNonSV5}
  Recall that the generic fiber $\shtgmu \times_{\spd \mathcal{O}_E} \spd E$ is
  isomorphic to the quotient stack (see \cite[Proposition 11.16]{ZhangThesis})
  \[
    \left[ \grgmu / \ul{K_p} \right].
  \]
  Since we have a $G(\qp)$-equivariant Hodge--Tate period map $\shgx\d \to
  \grgmu$, the existence of the map $\shgx[K_p]\d \to \shtgmu$ on the generic
  fiber comes down to $\shgx \to \shgx[K_p]$ being a $K_p$-torsor. To see this,
  it suffices to verify that the closure $Z_\mathsf{G}(\mathbb{Q})^- \subseteq
  \gaf$ intersects $G(\qp)$ trivially, see \cite[Section 1.5.8]{KisinShinZhu}. This assertion is verified in Lemma \ref{Lem:ChevalleyCongruence} below. Note
  that at finite level, the map $\mathbf{Sh}_{K^p}\gx \to \mathbf{Sh}_{K}\gx$ is
  not necessarily a $K_p$-torsor, and hence we cannot generally expect a
  $\mathcal{G}$-shtuka over $\scrs_{K}\gx^{\diamond}$.
\end{Rem}

\begin{Lem} \label{Lem:ChevalleyCongruence}
  Let $\mathsf{H}/\mathbb{Q}$ be a multiplicative (not necessarily connected)
  group. Then the closure $\mathsf{H}(\mathbb{Q})^-$ of $\mathsf{H}(\mathbb{Q})
  \subseteq \mathsf{H}(\af)$ intersects $\prod_{p \in S} \mathsf{H}(\qp)$
  trivially for every finite set of primes $S$.
\end{Lem}

\begin{proof}
  This easily follows from the congruence subgroup property for tori: Upon
  choosing a set of generators of $X^\ast(\mathsf{H})$, we obtain an embedding
  of $\mathsf{H} \hookrightarrow \prod_{i=1}^N
  \operatorname{Res}_{F_i/\mathbb{Q}} \mathbb{G}_m$ for finite extensions
  $F_i/\mathbb{Q}$. Then the problem reduces to the case when $\mathsf{H} =
  \operatorname{Res}_{F/\mathbb{Q}} \mathbb{G}_m$. At this point, we are
  considering the closure of $F^\times \hookrightarrow \mathbb{A}_{F,f}^\times$.
  Once we intersect with the open subset $U = \prod_{p \notin S}
  \mathcal{O}_{F,p}^\times \times \prod_{p \in S} F_p^\times \subseteq
  \mathbb{A}_{F,f}^\times$, we get
  \[
    (F^\times)^- \cap U = (\text{closure of } \mathcal{O}_{F,S}^\times
    \hookrightarrow U),
  \]
  where the group of $S$-units $\mathcal{O}_{F,S}^\times$ is a finitely
  generated abelian group. Applying \cite[Theorem~1]{ChevalleyDeuxTheoremes} to
  $\mathcal{O}_{F,S}^\times$, we see that for every integer $m \ge 1$ there
  exists an open subgroup $U^\prime = U^{S\prime} \times \prod_{p \in S}
  F_p^\times \subseteq U$ for which $U^\prime \cap \mathcal{O}_{F,S}^\times
  \subseteq (\mathcal{O}_{F,S}^\times)^m$. Taking the intersection over all $m$,
  we obtain the desired result.
\end{proof}

}

{\section{Igusa stacks for tori} \label{Sec:Tori}

\def\eab{E^\mathrm{ab}}
\def\eabhat{\widehat{E}^\mathrm{ab}}
\def\eabhatflat{\widehat{E}^{\mathrm{ab},\flat}}
\def\shigeompts{\mathsf{T}(\mathbb{Q})^- \backslash \mathsf{T}(\af)}
\def\XFF{\mathcal{X}_\mathrm{FF}}
\def\YFF{\mathcal{Y}_{(0,\infty)}}
\def\XFFeabhat{\XFF(\eabhatflat, \mathcal{O}_{\eabhatflat})_E}

In this section we prove Theorem \ref{Thm:IntroIgusaGeneric} for Shimura data $\gx$ with $\g$ a torus. To this end, we give a definition of the Igusa stack in Section \ref{Sub:ToralConstruction}. We prove that our definition satisfies Theorem \ref{Thm:IntroIgusaGeneric} in Section \ref{Sub:ProofTori}, which relies on an explicit description of the Beauville--Laszlo map for tori in Section \ref{Sub:BLTori}.

\subsection{Construction of the Igusa stack} \label{Sub:ToralConstruction}

\subsubsection{} \label{Sec:AGb}
Let $G/\qp$ be a connected reductive group and let $b \in G(\qpbr)$ be an
element. We consider the group
\[
  A_{G,b} = \{ (x, \sigma) \in G(\qpbr) \rtimes \gal(\qpbr/\qp) : x^{-1} b
  \phi(x) = \sigma(b) \},
\]
which is the centralizer of the element $(b, \phi)$. This is a closed subgroup
of $G(\qpbr) \rtimes \gal(\qpbr/\qp)$, and fits in a short exact sequence
\[
  1 \to G_b(\qp) \to A_{G,b} \to \gal(\qpbr/\qp) \cong \gal(\fpbar/\fp) \to 1
\]
of locally profinite groups, where $G_b(\qp) = \{(x, \sigma) : x^{-1} b \phi(x)
= b\}$. We also consider the subgroup $A_{G,b}^0 \subset A_{G,b}$ generated by
$G_b(\qp)$ and $(b, \phi)$, which can also be described as the preimage of
$\phi^\mathbb{Z} \subseteq \gal(\fpbar/\fp)$. The element $(b, \phi)$ induces a
splitting
\[
  A_{G,b}^0 \cong G_b(\qp) \times \mathbb{Z}.
\]

\begin{Lem} \label{Lem:ExtensionCentral}
  There exists an open subgroup $U \subseteq \gal(\fpbar/\fp)$ over which there
  exists a section $U \to A_{G,b}$. Moreover, the locally profinite group
  $A_{G,b}$ is the closure of the subgroup $A_{G,b}^0$ in $G(\qpbr) \rtimes
  \gal(\fpbar/\fp)$. In particular, if $G = T$ is a torus, then $A_{T,b}$ is
  abelian.
\end{Lem}

\begin{proof}
  For the first claim, we use \cite[Section~4.3]{Kottwitz1} to find an element
  $g_0 \in G(\qpbr)$ for which $g_0^{-1} b \phi(g_0) \in G(\mathbb{Q}_{p^n})$
  for some positive integer $n$. Then we see that
  \[
    U = \gal(\qpbr/\mathbb{Q}_{p^n}) \ni \sigma \mapsto (g_0 \sigma(g_0)^{-1},
    \sigma) \in A_{G,b}
  \]
  is a section, where we compute that $\sigma(g_0^{-1} b \phi(g_0)) = g_0^{-1} b
  \phi(g_0)$ implies
  \[
    (g_0 \sigma(g_0)^{-1})^{-1} b \phi(g_0 \sigma(g_0)^{-1}) = \sigma(b).
  \]
  Such a section induces a homeomorphism between $A_{G,b} \vert_U$ and $U \times
  G_b(\qp)$, under which the subset $A_{G,b}^0 \vert_U$ corresponds to $(U \cap
  \phi^\mathbb{Z}) \times G_b(\qp)$. This is dense, and hence the closure of
  $A_{G,b}^0 \subseteq G(\qpbr) \rtimes \gal(\fpbar/\fp)$ is $A_{G,b}$.

  Finally, if $G = T$ is a torus, then $G_b(\qp) = T(\qp)$. As $T(\qp) \subseteq
  T(\qpbr) \rtimes \gal(\qpbr/\qp)$ is central, the subgroup $A_{T,b}^0$
  generated by $T(\qp)$ and $(b, \phi)$ is abelian. It follows that its
  closure $A_{T,b}$ is abelian as well.
\end{proof}

\subsubsection{} \label{Sec:AGbFunctorial}
Given a $\phi$-conjugate $g b \phi(g)^{-1}$ of $b$, where $g \in G(\qpbr)$ is an
arbitrary element, there is an isomorphism of locally profinite groups
\[
  c_g \colon A_{G,b} \xrightarrow{\cong} A_{G,gb\phi(g)^{-1}}; \quad (x, \sigma)
  \mapsto (g x \sigma(g)^{-1}, \sigma).
\]
Under this isomorphism, we observe that $(b, \phi) \in A_{G,b}$ is mapped to
$(gb\phi(g)^{-1}, \phi)$. It is also clear that for a homomorphism $f \colon G
\to H$ of connected reductive groups and $b \in G(\qpbr)$, there is a natural
continuous group homomorphism
\[
  f \colon A_{G,b} \to A_{H,f(b)}
\]
that respects the projection to $\gal(\fpbar/\fp)$.

\subsubsection{} \label{Sec:AutOfPb}
We may also interpret the locally profinite groups $A_{G,b}$ as follows. By \cite[Theorem 5.3]{AnschutzIsocrystal}, every element $b \in G(\qpbr)$ defines a map $b \colon \spd \fpbar \to
\bung$. We may consider the automorphism group of this map, regarded as a
functor on perfectoid $\fp$-algebras. More precisely, let $(R, R^+)$ be a
perfectoid $\fp$-algebra that admits an $\fpbar$-algebra structure $f \colon
\fpbar \to R^+$. Any other $\fpbar$-structure may be obtained by conjugating $f$
by an element $\sigma \in \underline{\gal(\fpbar/\fp)}(R, R^+)$, i.e., as $f
\circ \sigma \colon \fpbar \to (R, R^+)$. Both of these $\fpbar$-algebra
structures give rise to $G$-torsors $f^\ast \mathscr{P}_b$ and $(f \circ
\sigma)^\ast \mathscr{P}_b$ on $\XFF(R, R^+)$, see \Cref{Sec:Conventions}.
Given a lift of $\sigma$ to a continuous map $x \colon \lvert \spa(R, R^+)
\rvert \to A_{G,b}$ we obtain an isomorphism
\[
  f^\ast \mathscr{P}_b \xrightarrow{\cong} (f \circ \sigma)^\ast \mathscr{P}_b
\]
given by $\phi$-descending right multiplication $r_{f(x)}$ by $f(x) \in
G(\YFF(R, R^+))$. (The map $f$ defines a structure map $\YFF(R, R^+) \to \spd
\qpbr$, and we denote by $f(x) \in G(\YFF(R, R^+))$ the pullback of $x \in
G(\qpbr)$ along this structure map.) This is because the diagram
\[ \begin{tikzcd}
  \phi^\ast G_{\YFF(R, R^+)} \arrow{r}{r_{f(\phi(x))}}
  \arrow{d}[']{r_{f(b)^{-1}}} & \phi^\ast G_{\YFF(R, R^+)}
  \arrow{d}{r_{f(\sigma(b))^{-1}}} \\ G_{\YFF(R, R^+)} \arrow{r}{r_{f(x)}} &
  G_{\YFF(R, R^+)}
\end{tikzcd} \]
commutes, as we have $x^{-1} b \phi(x) = \sigma(b)$ from \Cref{Sec:AGb}. This
construction defines a morphism of v-stacks
\[
  [\spd \fpbar / \underline{A_{G,b}}] \to \bung^{[b]} \subseteq \bung,
\]
where the right $A_{G,b}$-action on $\spd \fpbar$ is through $A_{G,b}
\twoheadrightarrow \gal(\fpbar/\fp)$.

\begin{Prop}
  Assume that $[b] \in B(G)$ is basic. Then the map
  \[
    [\spd \fpbar / \underline{A_{G,b}}] \xrightarrow{\cong} \bung^{[b]}
  \]
  constructed in \Cref{Sec:AutOfPb} is an isomorphism of v-stacks.
\end{Prop}

\begin{proof}
  As $\spd \fpbar \to \spd \fp$ is a v-cover, we may check that the map is an
  isomorphism after base changing to $\fpbar$. Upon identifying $\spd \fpbar
  \times \spd \fpbar \cong \spd \fpbar \times \underline{\gal(\fpbar/\fp)}$,
  this boils down to checking that
  \[
    [\spd \fpbar / \underline{G_b(\qp)}] \to \mathrm{Bun}_{G,\fpbar}^{[b]}
  \]
  is an isomorphism. This is done in \cite[Theorem~III.0.2]{FarguesScholze}.
\end{proof}

\begin{Rem}
  Even when $b$ is not basic, a similar description of $\bung^{[b]}$ as a
  quotient of $\spd \fpbar$ can be obtained by using an extension of
  $\gal(\fpbar/\fp)$ by $\tilde{G}_b$.
\end{Rem}

\subsubsection{} \label{Sec:BunGPresentation}
Given elements $b, g, b^\prime = gb\phi(g)^{-1} \in G(\qpbr)$, the element $g$
induces an isomorphism between $b, b^\prime \colon \spd \fpbar \to \bung$, and 
there is an isomorphism $c_g \colon A_{G,b} \cong A_{G,b^\prime}$ from
\Cref{Sec:AGbFunctorial}. Assume from now on that $[b] = [b^\prime] \in B(G)$ is basic.
It follows from the construction of \Cref{Sec:AutOfPb} that the isomorphism
\[
  [\spd\fpbar / \underline{A_{G,b}}] \cong \bung^{[b]} = \bung^{[b^\prime]}
  \cong [\spd\fpbar / \underline{A_{G,b^\prime}}]
\]
agrees with $[\mathrm{id} / c_g]$. We also recall from
\cite[Proposition~2.1.10]{DvHKZIgusaStacks} that there exists a canonical
isomorphism between $\phi, \id
\colon \bung \to \bung$, and hence between $\phi, \id \colon \bung^{[b]} \to
\bung^{[b]}$. Using the presentation $\bung^{[b]} \cong [\spd\fpbar /
\underline{A_{G,b}}]$, we see that the absolute Frobenius $\phi$ can be
presented as
\[
  \phi = [\phi_{\fpbar} / \id] \colon [\spd\fpbar / \underline{A_{G,b}}]
  \xrightarrow{\cong} [\spd\fpbar / \underline{A_{G,b}}].
\]
Unraveling the construction, we verify that the isomorphism from $\id$ to $\phi$
is given as right multiplication by the element $(b, \phi)$, where we note that
$(b, \phi) \in A_{G,b}$ is central.

\subsubsection{}
Let $\mathsf{T}/\mathbb{Q}$ be a torus, and write $T = \mathsf{T}_{\qp}$. For
each $b \in T(\qpbr)$, we have an extension of locally profinite abelian groups
\[
  1 \to T(\qp) \to A_{T,b} \to \gal(\fpbar/\fp) \to 1.
\]
We also have a profinite group $\shigeompts$ and a continuous group homomorphism
\[
  \iota \colon T(\qp) \hookrightarrow \mathsf{T}(\af) \twoheadrightarrow
  \shigeompts.
\]

\begin{Lem} \label{Lem:LocalIota}
  There exists a unique continuous group homomorphism
  \[
    \iota_{T,b} \colon A_{T,b} \to \shigeompts
  \]
  with the property that $\iota_{T,b}(b, \phi) = 1$ and the restriction of
  $\iota_{T,b}$ to $T(\qp) \subset A_{T,b}$ recovers the inclusion $\iota$.
\end{Lem}

\begin{proof}
  Because $A_{T,b}^0$ is the direct sum of $(b, \phi)^\mathbb{Z}$ and $T(\qp)$,
  we have a group homomorphism
  \[
    \iota_{T,b}^0 \colon A_{T,b}^0 \to \shigeompts; \quad (b, \phi) \mapsto 1,
    \quad (t, 1) \mapsto \iota(t)
  \]
  that we wish to extend to $A_{T,b}$. Uniqueness follows from the density of
  $A_{T,b}^0 \subset A_{T,b}$, see \Cref{Lem:ExtensionCentral}. For existence,
  we use \Cref{Lem:ExtensionCentral} to find an open subgroup $U = n
  \gal(\fpbar/\fp) \subseteq \gal(\fpbar/\fp)$ over which there exists an
  isomorphism $A_{T,b} \vert_U \cong U \oplus T(\qp)$ of topological abelian
  groups. This restricts to $A_{T,b}^0 \vert_U \cong \phi^{n\mathbb{Z}} \oplus
  T(\qp)$, and because $\shigeompts$ is profinite, any group homomorphism
  $\phi^{n\mathbb{Z}} \to \shigeompts$ uniquely extends to a continuous group
  homomorphism $U \to \shigeompts$. This gives a continuous extension
  \[
    A_{T,b} \vert_U \to \shigeompts
  \]
  of $\iota_{T,b}^0 \vert_U$. We now combine the two group homomorphisms, one on
  $A_{T,b} \vert_U$ and one on $A_{T,b}^0$, to a group homomorphism on $A_{T,b}$
  using
  \[
    A_{T,b} \vert_U + A_{T,b}^0 = A_{T,b}, \quad A_{T,b} \vert_U \cap A_{T,b}^0
    = A_{T,b}^0 \vert_U.
  \]
  The resulting group homomorphism $\iota_{T,b} \colon A_{T,b} \to \shigeompts$
  is indeed continuous because its restriction to the open subgroup $A_{T,b}
  \vert_U$ is continuous.
  
\end{proof}

\subsubsection{} \label{Sec:ExtensionActionCompatibility}
Recall that given $b \in T(\qpbr)$ and $b^\prime = gb\phi(g)^{-1}$ there is a
canonical isomorphism $c_g \colon A_{T,b} \cong A_{T,b^\prime}$. Because $T$ is
commutative, the restriction of $c_g$ to $T(\qp)$ is the identity map. It
follows from uniqueness above that we have a commutative diagram
\[ \begin{tikzcd}
  A_{T,b} \arrow{d}{\cong}[']{c_g} \arrow{r}{\iota_{T,b}} & \shigeompts
  \arrow[equals]{d} \\ A_{T,b^\prime} \arrow{r}{\iota_{T,b^\prime}} &
  \shigeompts.
\end{tikzcd} \]

\begin{Def} \label{Def:IgusaTori}
  For each $b \in T(\qpbr)$ corresponding to a map $\spd \fpbar \to \bun_T$,
  hence inducing an identification $\bun_T^{[b]} \cong [\spd \fpbar /
  \underline{A_{T,b}}]$, we define the v-stack
  \[
    \mathrm{Igs}(\mathsf{T}, b) = [(\underline{\shigeompts} \times \spd \fpbar)
    / \underline{A_{T,b}}] \to [\spd \fpbar / \underline{A_{T,b}}] \cong
    \bun_T^{[b]},
  \]
  where $(x, \sigma) \in A_{T,b}$ acts on $\shigeompts$ via translation by
  $\iota_{T,b}(x, \sigma)$ and on $\spd \fpbar$ via $\sigma$.
\end{Def}

\subsubsection{}
Write $b^\prime = gb\phi(g)^{-1}$ for some $g \in T(\qpbr)$. We see from
\Cref{Sec:ExtensionActionCompatibility} that the action of $\underline{A_{T,b}}$
on $\underline{\shigeompts} \times \spd \fpbar$ is compatible with the action of
$A_{T,b^\prime}$ under the identification $c_g \colon A_{T,b} \cong
A_{T,b^\prime}$. It follows that there is a natural commutation relation
\[ \begin{tikzcd}[row sep=0em]
  \mathrm{Igs}(\mathsf{T}, b) \arrow{dd}[']{\lbrack \id\times\id/c_g
  \rbrack}{\cong} \arrow{r} & \lbrack\spd \fpbar / \underline{A_{T,b}}\rbrack
  \arrow{dd}[']{\lbrack \id/c_g \rbrack}{\cong} \arrow{dr}{\cong} \\ & &
  \bun_T^{[b]} \\ \mathrm{Igs}(\mathsf{T}, b^\prime) \arrow{r} & \lbrack\spd
  \fpbar / \underline{A_{T,b^\prime}}\rbrack \arrow{ru}{\cong}
\end{tikzcd} \]
that is moreover associative in the sense that it is compatible with a further
$\phi$-conjugation $b^{\prime\prime} = g^\prime b^\prime \phi(g^\prime)^{-1}$.
Using these identifications, we may regard $\mathrm{Igs}(\mathsf{T}, b)$ as
depending only on the $\phi$-conjugacy class $[b] \in B(T)$.

\begin{Def}
  We define the \textit{Igusa stack} for $\mathsf{T}/\mathbb{Q}$ as the stack
  \[
    \mathrm{Igs}(\mathsf{T}) = \coprod_{[b] \in B(T)} \mathrm{Igs}(\mathsf{T},
    b) \to \coprod_{[b] \in B(T)} \bun_T^{[b]} = \bun_T.
  \]
\end{Def}

\begin{Rem}
  We remark that this does not depend on the choice of the cocharacter $\mu$ nor
  the choice of a place $v \mid p$ of the reflex field. See also
  \Cref{Rem:NotCanonical}.
\end{Rem}

\subsubsection{} \label{Sec:IgsTTameHecke}
For every element $t \in \mathsf{T}(\afp)$, there is an induced map
\[
  t \colon [(\underline{\shigeompts} \times \spd \fpbar) / \underline{A_{T,b}}]
  \to [(\underline{\shigeompts} \times \spd \fpbar) / \underline{A_{T,b}}]
\]
given by translation by $t^{-1}$ on the abelian group $\shigeompts$. This
defines a left action of $\underline{\mathsf{T}(\afp)}$ on
$\mathrm{Igs}(\mathsf{T}, b)$, where this action does not change when we replace
$b$ with a $\phi$-conjugate. Moreover, the natural map
\[
  \mathrm{Igs}(\mathsf{T}, b) \to \mathrm{Bun}_T^{[b]}
\]
is naturally $\underline{\mathsf{T}(\afp)}$-equivariant, where we give the
trivial action on $\mathrm{Bun}_T^{[b]}$.

\subsubsection{} \label{Sec:IgsTFrobenius}
As in \Cref{Sec:BunGPresentation} we may define an isomorphism
\[
  \id_{\mathrm{Igs}(\mathsf{T}, b)} \xrightarrow{\cong}
  \phi_{\mathrm{Igs}(\mathsf{T}, b)}
\]
given by right multiplication by the element $(b, \phi) \in A_{T,b}$. This is
compatible with the isomorphism $\id_{\bun_T} \cong \phi_{\bun_T}$. More
precisely, the diagram
\[ \begin{tikzcd}
  \mathrm{Igs}(\mathsf{T}) \arrow{r} \arrow[bend right]{d}[']{\phi} \arrow[bend
  left]{d}{\id} & \bun_T \arrow[bend right]{d}[']{\phi} \arrow[bend
  left]{d}{\id} \\ \mathrm{Igs}(\mathsf{T}) \arrow{r} & \bun_T
\end{tikzcd} \]
is 2-commutative, i.e., the two isomorphisms between $\mathrm{Igs}(\mathsf{T})
\xrightarrow{\phi} \mathrm{Igs}(\mathsf{T}) \to \bun_T$ and
$\mathrm{Igs}(\mathsf{T}) \to \bun_T \xrightarrow{\id} \bun_T$ agree.

\subsection{The Beauville--Laszlo map} \label{Sub:BLTori}

\subsubsection{} \label{Sec:BLToriDef}
Let $T$ be a torus over $\qp$, and let $\mu \colon \mathbb{G}_{m,\qpbar} \to
T_{\qpbar}$ be a geometric cocharacter whose field of definition is $E \subseteq
\qpbar$. Writing $[b] \in B(T, \mu^{-1})$ for the unique element, we have a map
\[
  \mathrm{BL}_{\mu^{-1}} \colon \mathrm{Gr}_{T,\mu^{-1}} \cong \spd E \to
  \bun_T^{[b]} \cong [\spd \fpbar / \underline{A_{T,b}}]
\]
that sends a perfectoid $E$-algebra $(R^\sharp, R^{\sharp+})$ to the
$T$-torsor on the Fargues--Fontaine curve $\XFF(R, R^+)$
obtained by modifying the trivial $T$-torsor along the specified untilt by
$\mu^{-1}$. Our goal is to understand the map $\mathrm{BL}_{\mu^{-1}}$ in terms
of group theory.

\subsubsection{}
Let us write $S = \operatorname{Res}_{E/\qp} \mathbb{G}_m$. There is a
decomposition $E \otimes_{\qp} E \cong E \times \prod_i {L_i}$ where the
projection $E \otimes_{\qp} E \twoheadrightarrow E$ is given by multiplication.
It follows that there is a canonical group homomorphism $E^\times \to (E
\otimes_{\qp} E)^\times = E^\times \times \prod_i L_i^\times$ given by $x
\mapsto (x, 1, \dotsc, 1)$, which can be promoted to a morphism
\[
  \mu_S \colon \mathbb{G}_{m,E} \to (\operatorname{Res}_{E/\qp} \mathbb{G}_m)_E
  = S_E
\]
of $E$-tori. As $\mu$ has field of definition $E$, there is an induced map
\[
  r_\mu \colon S = \operatorname{Res}_{E/\qp} \mathbb{G}_m \to T
\]
of $\qp$-tori, where $(r_\mu)_E \circ \mu_S$ recovers $\mu$.

\subsubsection{} \label{Sec:Choosingb}
Let $E_0 \subseteq E$ be the maximal subfield unramified over $\qp$, and let
$\ebreve = E \qpbr$. Choose a uniformizer $\pi_E$ of $E$. The homomorphism
$r_\mu$ induces a map
\[
  B(S, \mu_S^{-1}) \xrightarrow{r_\mu} B(T, \mu^{-1}),
\]
where both sides are singletons. As in \cite[Section~2.5]{Kottwitz1}, we may and
will choose
\[
  b_S = (\pi_E^{-1}, 1, \dotsc, 1) \in \prod_{E_0 \hookrightarrow \qpbr}
  \ebreve^\times = S(\qpbr)
\]
so that $[b_S] \in B(S, \mu_S^{-1})$, and use its image $b = r_\mu(b_S) \in
T(\qpbr)$ as a representative of $[b] \in B(T, \mu^{-1})$. We note that $b_S$
may also be understood as the norm $\operatorname{Nm}_{E/E_0}(\mu_S(\pi_E^{-1}))
\in S(E_0) \subseteq S(\qpbr)$, and hence $b$ may be described as $b =
\operatorname{Nm}_{E/E_0}(\mu(\pi_E^{-1})) \in T(E_0) \subseteq T(\qpbr)$. We
also compute
\[
  b_S \phi(b_S) \dotsm \phi^{f-1}(b_S) = \pi_E^{-1} \in S(\qp) \subseteq
  S(\qpbr).
\]
where $f = [E_0 : \qp]$ is the inertial degree and $\phi \colon S(\qpbr) \to
S(\qpbr)$ is the Frobenius on $S(\qpbr)$.

\subsubsection{} \label{Sec:TorsorVsVector}
Since $S = \operatorname{Res}_{E/\qp} \mathbb{G}_m$, a left $S$-torsor over
$\XFF(R, R^+)$ is equivalently a line bundle on the adic space
\[
  \XFF(R, R^+)_E = \XFF(R, R^+) \otimes_{\spa \qp} \spa E.
\]
More precisely, given a vector bundle $\mathscr{V}$ of rank $n$ on $\XFF(R,
R^+)_E$, its pushforward $\pi_\ast \mathscr{V}$ along $\pi \colon \XFF(R, R^+)_E
\to \XFF(R, R^+)$ is a vector bundle with an $E$-action, and the sheaf of
$E$-linear isomorphisms
\[
  \mathscr{P}(\mathscr{V}) = \underline{\operatorname{Isom}}_E(\pi_\ast
  \mathscr{V}, \mathscr{O}_{\XFF(R, R^+)}^{\oplus n} \otimes_{\qp} E)
\]
is a left $\operatorname{Res}_{E/\qp} \mathrm{GL}_n$-torsor.

\subsubsection{}
We now construct two left $S$-torsors on the Fargues--Fontaine curve. Let $k_E$
be the residue field of $E$. For every perfectoid $k_E$-algebra $(R, R^+)$, we
have a morphism $\YFF(R, R^+)_E \to \spa(E_0 \otimes_{\qp} E)$. Hence we may
define the line bundle
\[
  \mathscr{O}_{\XFF(R, R^+)_E}(b_S) = \mathscr{O}_{\YFF(R, R^+)_E} / (\phi^\ast
  \mathscr{O}_{\YFF(R, R^+)_E} \cong \mathscr{O}_{\YFF(R, R^+)_E}
  \xrightarrow{b_S} \mathscr{O}_{\YFF(R, R^+)_E})
\]
where we note $b_S \in (E_0 \times_{\qp} E)^\times = S(E_0)$. Under the
construction of \Cref{Sec:TorsorVsVector}, this corresponds to the $S$-torsor on
$\XFF(R, R^+)$ given by Frobenius-descending the isomorphism
\[
  \phi^\ast(S \times_{\qp} \YFF(R, R^+)) \cong S \times_{\qp} \YFF(R, R^+)
  \xrightarrow{r_{b_S^{-1}}} S \times_{\qp} \YFF(R, R^+),
\]
which is the torsor $\mathscr{P}_{b_S}$ considered in \Cref{Sec:AutOfPb}.

\subsubsection{}
On the other hand, for every perfectoid $E$-algebra $(R^\sharp, R^{\sharp+})$
with tilt $(R, R^+)$, there is a distinguished point
\[
  \infty \colon \spa(R^\sharp, R^{\sharp+}) \hookrightarrow \XFF(R,
  R^+)_E
\]
that is moreover a Cartier divisor, see
\cite[Proposition~11.3.1]{ScholzeWeinsteinBerkeley}. We now define the line
bundle
\[
  \mathscr{O}_{\XFF(R, R^+)_E}(\infty) \supseteq \mathscr{O}
\]
that is the sheaf of rational functions on $\XFF(R, R^+)_E$ with
possibly a pole at $\infty$ of order at most $1$. Again, this defines a map
\[
  \mathrm{BL}_{\mu_S^{-1}} \colon \spd E \to \mathrm{Bun}_S,
\]
where it is clear from the construction that this agrees with map described in
\Cref{Sec:BLToriDef} for the torus $S$ and the cocharacter $\mu_S$.

\subsubsection{}
By \cite{LubinTate} there exists, up to isomorphism, a unique $1$-dimensional
formal group $\mathcal{LT}_{\pi_E}$ over $\mathcal{O}_E$ with an
$\mathcal{O}_E$-module structure with the property that $[\pi_E] \equiv
\mathrm{Frob}_q \pmod{\pi_E}$. As in \cite[Section~II.2.1]{FarguesScholze}, one
may choose a coordinate $\mathcal{LT}_{\pi_E} \cong \spf \mathcal{O}_E[[X]]$ for
which the map
\[
  \log \colon \mathcal{LT}_{\pi_E,E} \to \widehat{\mathbb{G}}_{a,E}; \quad X
  \mapsto \sum_{n=0}^\infty \pi_E^{-n} X^{q^n}
\]
of formal schemes over $E$ is $\mathcal{O}_E$-linear. We consider the universal
cover
\[
  \widetilde{\mathcal{LT}}_{\pi_E} = \varprojlim_{\pi_E} \mathcal{LT}_{\pi_E}
  \cong \spf \mathcal{O}_E[[\tilde{X}^{1/p^\infty}]], \quad \tilde{X} = \lim_{n
  \to \infty} X_n^{q^n}
\]
where $X_n$ is the coordinate on the $n$th copy of $\mathcal{LT}_{\pi_E} \cong
\spf \mathcal{O}_E[[X_n]]$ and $q$ is the order of the residue field $k_E$ of $E$.

\begin{Prop}[{\cite[Proposition~II.2.2]{FarguesScholze}}] \label{Prop:BanachColmez}
  Let $(R^\sharp, R^{\sharp+})$ be a perfectoid $E$-algebra with tilt $(R,
  R^+)$, which is naturally a $k_E$-algebra. Write
  \[
    c = (1, 0, \dotsc, 0) \in \prod_{E_0 \hookrightarrow E} E \cong E_0
    \otimes_{\qp} E,
  \]
  where $1$ is at the coordinate corresponding to the natural embedding. Then
  the map
  \begin{align*}
    \widetilde{\mathcal{LT}}_{\pi_E}(R^{\sharp+}) \cong R^{\circ\circ} &\to
    H^0(Y_{(0,\infty)}(R, R^+)_E, \mathscr{O}_{Y_{(0,\infty)}(R,
    R^+)_E})^{\phi=b_S^{-1}} \\ &\qquad\cong H^0(\XFF(R, R^+)_E,
    \mathscr{O}_{\XFF(R, R^+)_E}(b_S)) \\
    X &\mapsto \sum_{n=-\infty}^\infty b_S \phi(b_S) \dotsm
    \phi^{n-1}(b_S) \phi^n(c) ([X^{p^n}] \otimes 1)
  \end{align*}
  is an $E$-linear isomorphism. Under this isomorphism, the evaluation map
  \[
    H^0(\XFF(R, R^+)_E, \mathscr{O}_{\XFF(R, R^+)_E}(b_S)) \to
    R^\sharp
  \]
  at $\infty$ is identified with
  \[
    \widetilde{\mathcal{LT}}_{\pi_E}(R^{\sharp+}) \to
    \mathcal{LT}_{\pi_E}(R^{\sharp+}) \xrightarrow{\log} R^{\sharp}.
  \]
\end{Prop}

\subsubsection{}
We choose a nonzero $\pi_E$-torsion point $\lambda_1 \in
\mathcal{LT}(\mathcal{O}_{\cp}) \cong \mathfrak{m}_{\cp}$, and also choose
points $\lambda_2, \lambda_3, \dotsc \in \mathcal{LT}(\mathcal{O}_{\cp}) =
\mathfrak{m}_{\cp}$ so that $[\pi_E](\lambda_i) = \lambda_{i-1}$. Let
\[
  \lambda = (0, \lambda_1, \lambda_2, \dotsc) \in
  \widetilde{\mathcal{LT}}(\mathcal{O}_{\cp}).
\]
For each element $u \in \mathcal{O}_E^\times$, we may define
\[
  [u](\lambda) = (0, [u](\lambda_1), [u](\lambda_2), \dotsc) \in
  \widetilde{\mathcal{LT}}(\mathcal{O}_{\cp}),
\]
noting that each $[u](\lambda_n)$ is well-defined as it only depends on $u
\bmod{\pi_E^n \mathcal{O}_E}$.

\begin{Thm}[{\cite[Theorem~3]{LubinTate}}] \label{Thm:LubinTate}
  We have $\lambda_1, \lambda_2, \dotsc \in \eab$, where
  $\eab$ is the maximal abelian extension of $E$. Moreover, for every
  $n \in \mathbb{Z}$ and $u \in \mathcal{O}_E^\times$, we have
  \[
    \mathrm{Art}_E(u \pi_E^n)(\lambda) = [u](\lambda)
  \]
  where $\mathrm{Art}_E \colon E^\times \to \gal(\eab/E)$ is the Artin
  reciprocity map.
\end{Thm}

\subsubsection{}
We consider the completed maximal abelian extension $\eabhat/E$, so that
$\ebreve \subseteq \eabhat$. Because $\eabhat$ is a (completed) pro-\'{e}tale
extension of $\qp^\mathrm{cyc} = \qp(\zeta_{p^\infty})^\wedge$, which is
perfectoid, we see that $\eabhat$ is a perfectoid field as well.
Using \Cref{Prop:BanachColmez}, we obtain a nonzero section
\[
  s \in H^0(\XFFeabhat, \mathscr{O}_{\XFFeabhat}(b_S))
\]
corresponding to $\lambda \in
\widetilde{\mathcal{LT}}_{\pi_E}(\mathcal{O}_{\eabhat})$. On the other hand, it
is clear that $\log(\lambda) = 0$, and hence $s \vert_{\spa \eabhat} = 0$. That
is, the section $s$ defines a morphism
\[
  s \colon \mathscr{O}_{\XFFeabhat}(\infty)
  \to \mathscr{O}_{\XFFeabhat}(b_S),
\]
which is an isomorphism by \cite[Proposition~2.17]{FarguesLCFT}. Equivalently,
$s$ defines an homotopy filling in
\[ \begin{tikzcd}[column sep=large]
  \spd \eabhat \arrow{r}{\rho} \arrow{d} & \spd \fpbar \arrow{d}{b_S} \\ \spd E
  = [\spd \eabhat / \underline{\gal(\eab/E)}]
  \arrow{r}{\mathrm{BL}_{\mu_S^{-1}}} & \bun_S^{[b_S]} = [\spd \fpbar /
  \underline{A_{S,b_S}}],
\end{tikzcd} \]
where $\rho \colon \spd \eabhat \to \spd \fpbar$ is projection to its perfect
residue field.

\subsubsection{} \label{Sec:EquivRelation}
Taking the \v{C}ech nerve of both vertical maps, we obtain a map 
\begin{align*}
  \underline{\gal(\eab/E)} \times \spd \eabhat &\xrightarrow{(\mathrm{pr}_2,
  \mathrm{act}), \cong} \spd \eabhat \times_{\spd E} \spd \eabhat \\ &\to \spd
  \fpbar \times_{\bun_S} \spd \fpbar \xleftarrow{(\mathrm{pr}_2, \mathrm{act}),
  \cong} \underline{A_{S,b_S}} \times \spd \fpbar.
\end{align*}
This may be described as follows. Given any $\sigma \in \gal(\eab/E)$, we obtain
an isomorphism
\begin{align*}
  \gamma_\sigma \colon \mathscr{O}_{\XFFeabhat}(b_S) &\xrightarrow{s^{-1}}
  \mathscr{O}_{\XFFeabhat}(\infty) \\ &\qquad\cong \sigma^\ast
  \mathscr{O}_{\XFFeabhat}(\infty) \xrightarrow{\sigma^\ast s} \sigma^\ast
  \mathscr{O}_{\XFFeabhat}(b_S).
\end{align*}
A $\sigma$-conjugated automorphism of $\mathscr{O}(b_S)$ corresponds to an
element of $A_{S,b_S}$ contained in the fiber of $\sigma \vert_{\fpbar} \in
\gal(\fpbar/\fp)$, see \Cref{Sec:AutOfPb}. Thus the construction $\sigma \mapsto
\gamma_\sigma$ defines a continuous homomorphism of locally profinite groups\footnote{Note that the functor from locally profinite groups to v-sheaves is fully faithful, see e.g. \cite[Example 11.12]{EtCohDiam}.}
\[
  \alpha \colon \gal(\eab/E) \to A_{S,b_S}
\]
lifting the restriction map $\gal(\eab/E) \to \gal(\fpbar/\fp)$. It follows from
the construction that the map on morphisms $\spd \eabhat \times
\underline{\gal(\eab/E)} \to \spd \fpbar \times \underline{A_{S,b_S}}$ is simply
$\rho \times \alpha$, and hence the Beauville--Laszlo map is identified as
\[
  \mathrm{BL}_{\mu_S^{-1}} \colon \spd E = [\spd \eabhat /
  \underline{\gal(\eab/E)}] \xrightarrow{[\rho/\alpha]} [\spd \fpbar /
  \underline{A_{S,b_S}}] = \bun_S^{[b_S]}.
\]

\subsubsection{} \label{Sec:SplittingOfA}
Our goal now is to compute the continuous group homomorphism $\alpha$. Recall
that for our choice of $b_S$, we have $b_S = \operatorname{Nm}_{E/E_0}
(\mu_S(\pi_E^{-1})) \in S(E_0)$. It follows from the definition in
\Cref{Sec:AGb} that the preimage of $\gal(\fpbar/k_E) \subseteq
\gal(\fpbar/\fp)$ under $A_{S,b_S} \to \gal(\fpbar/\fp)$ is
\[
  S(\qp) \times \gal(\fpbar/k_E) \subseteq A_{S,b_S}.
\]
That is, $A_{S,b_S} \to \gal(\fpbar/\fp)$ naturally splits over
$\gal(\fpbar/k_E)$. The image of $\alpha$ is contained in this open subgroup,
since the image of $\gal(\eab/E) \to \gal(\fpbar/\fp)$ is $\gal(\fpbar/k_E)$.

\begin{Prop} \label{Prop:AlphaUsingLCFT}
  Consider the composition
  \[
    \beta \colon \gal(E^\mathrm{ab}/E) \xrightarrow{\mathrm{Art}_E^{-1}}
    \widehat{E^\times} \cong \mathcal{O}_E^\times \times
    \widehat{\pi_{E}^\mathbb{Z}} \xrightarrow{\pi_{E} \mapsto
    \phi^{[k_E:\fp]}} \mathcal{O}_E^\times \times \gal(\fpbar/k_E).
  \]
  Then we have $\alpha = -\beta$ where $\mathcal{O}_E^\times \times
  \gal(\fpbar/k_E) \subseteq A_{S,b_S}$ as in \Cref{Sec:SplittingOfA} upon
  identifying $E^\times \cong S(\qp)$.
\end{Prop}

\begin{proof}
  It follows from basic properties of $\mathrm{Art}_E$ that $\mathrm{pr}_2
  \circ \beta$ agrees with the negation of the induced map on residue fields. On
  the other hand, $\mathrm{pr}_2 \circ \alpha$ is by construction the induced
  map on residue fields, see \Cref{Sec:EquivRelation}. Hence it remains to verify
  that $\mathrm{pr}_1 \circ \alpha = -\mathrm{pr}_1 \circ \beta$.

  Let $\sigma \in \gal(\eab/E)$ be any automorphism. It follows from the
  discussion of \Cref{Sec:EquivRelation} and \Cref{Sec:SplittingOfA} that the
  element $\mathrm{pr}_1(\alpha(\sigma)) \in E^\times$ corresponds to the
  automorphism
  \begin{align*}
    \gamma_\sigma \colon \mathscr{O}_{\XFFeabhat}(b_S) & \xrightarrow{s^{-1}}
    \mathscr{O}_{\XFFeabhat}(\infty) \cong \sigma^\ast
    \mathscr{O}_{\XFFeabhat}(\infty) \\ &\xrightarrow{\sigma^\ast s}
    \sigma^\ast \mathscr{O}_{\XFFeabhat}(b_S) \cong
    \mathscr{O}_{\XFFeabhat}(b_S)
  \end{align*}
  where in the last isomorphism we are using that $b_S$ is defined over $k_E$.
  The isomorphism
  \[
    \widetilde{\mathcal{LT}}_{\pi_E}(\mathcal{O}_{\eabhat}) \cong
    H^0(\XFFeabhat, \mathscr{O}_{\XFFeabhat}(b_S))
  \]
  from \Cref{Prop:BanachColmez} is both $\gal(\eabhat/E)$-equivariant and
  $E$-linear. Hence $\gamma_\sigma$ is multiplication by
  \[
    \sigma(s)/s = \sigma(\lambda)/\lambda = \mathrm{pr}_1(\beta(\sigma))
    \in \mathcal{O}_E^\times
  \]
  by \Cref{Thm:LubinTate}, where by $\sigma(\lambda)/\lambda$ we mean the
  unique element $x \in \mathcal{O}_E^\times$ for which $x \lambda =
  \sigma(\lambda)$.
  
  Following the constructions from \Cref{Sec:TorsorVsVector} and
  \Cref{Sec:AutOfPb}, we see that the corresponding map $\mathscr{P}_{b_S} \to
  \sigma^\ast \mathscr{P}_{b_S}$ is right multiplication by
  $\mathrm{pr}_1(\beta(\sigma))^{-1}$, where the inverse comes from the fact
  that the construction involves dualizing the vector bundle. Therefore
  $\mathrm{pr}_1(\alpha(\sigma)) = \mathrm{pr}_1(\beta(\sigma))^{-1}$ by the
  construction in \Cref{Sec:AutOfPb}.
\end{proof}

\subsubsection{} \label{Sec:BLGeneral}
Recall we have a homomorphism $r_\mu \colon S = \operatorname{Res}_{E/\qp}
\mathbb{G}_m \to T$. This induces a map $r_\mu \colon \bun_S \to \bun_T$ as well
as $r_\mu \colon A_{S,b_S} \to A_{T,b}$. From the previous discussion, it
follows that the Beauville--Laszlo map
\[
  \mathrm{BL}_{\mu^{-1}} \colon \spd E \to \bun_T
\]
may be identified with
\[
  [\rho / (r_\mu \circ \alpha)] \colon [\spd \eabhat / \underline{\gal(\eab/E)}]
  \to [\spd \fpbar / \underline{A_{T,b}}].
\]

\subsection{The fiber product diagram} \label{Sub:ProofTori}

\subsubsection{}
Let $\mathsf{T}$ be a torus over $\mathbb{Q}$, and let $\mu \colon \mathbb{G}_
{m,\mathsf{E}} \to \mathsf{T}_\mathsf{E}$ be a cocharacter whose field of
definition is $\mathsf{E}$. Then $(\mathsf{T}, \mu)$ is a Shimura datum with
reflex field $\mathsf{E}$. There is a corresponding infinite level Shimura
variety 
\[
  \mathbf{Sh}(\mathsf{T}, \mu)_\mathsf{E} = \varprojlim_{K \to \{1\}}
  \mathbf{Sh}_K(\mathsf{T}, \mu)_\mathsf{E}.
\]
By \cite[Theorem~5.28]{Milne}, for example, its $\bar{\mathsf{E}}$-points are
given by
\[
  \mathsf{Sh}(\mathsf{T}, \mu)_\mathsf{E}(\bar{\mathsf{E}}) = \shigeompts,
\]
where $\mathsf{T}(\mathbb{Q})^-$ denotes the closure of $\mathsf{T}(\mathbb{Q})$
inside $\mathsf{T}(\af)$.

\subsubsection{} \label{Sec:ShimuraGalois}
The Galois action is defined as follows. There is an Artin isomorphism
\[
  \mathrm{Art}_\mathsf{E} \colon (\mathsf{E}^\times)^- \backslash \Bigl(
  \mathbb{A}_\mathsf{E}^{\infty,\times} \times \prod_{\mathsf{E}_v
  \xrightarrow{\sim} \mathbb{R}} (\mathbb{Z}/2\mathbb{Z}) \Bigr) \cong
  \gal(\mathsf{E}^\mathrm{ab}/\mathsf{E})
\]
with the conventions from \Cref{Sec:Conventions}.
On the other hand, the cocharacter $\mu$ induces a morphism $r_\mu =
\sum_{\mathsf{E} \to \bar{\mathbb{Q}}} \mu \colon
\operatorname{Res}_{\mathsf{E}/\mathbb{Q}} \mathbb{G}_m \to \mathsf{T}$ of tori.
Combining the two, we obtain a map
\[
  \gal(\mathsf{E}^\mathrm{ab}/\mathsf{E})
  \xrightarrow{\mathrm{Art}_\mathsf{E}^{-1}} (\mathsf{E}^\times)^- \backslash
  \mathbb{A}_\mathsf{E}^{\infty,\times} \xrightarrow{r_\mu} \shigeompts.
\]
The Galois descent datum on the infinite level Shimura variety
$\mathbf{Sh}(\mathsf{T}, \mu)_\mathsf{E}$ is defined so that $\sigma \in
\gal(\mathsf{E}^\mathrm{ab}/\mathsf{E})$ acts on $\mathbf{Sh}(\mathsf{T},
\mu)_\mathsf{E}(\bar{\mathsf{E}}) \cong \shigeompts$ as translation by
$r_\mu(\mathrm{Art}_\mathsf{E}^{-1}(\sigma))$. This means that we may identify
\[
  \mathbf{Sh}(\mathsf{T}, \mu)_\mathsf{E} \cong [(\underline{\shigeompts} \times
  \spec \bar{\mathsf{E}}) / \gal(\bar{\mathsf{E}}/\mathsf{E})]
\]
where the right $\gal(\bar{\mathsf{E}}/\mathsf{E})$-action on $\spec
\bar{\mathsf{E}}$ is the usual one and the right
$\gal(\bar{\mathsf{E}}/\mathsf{E})$-action on $\shigeompts$ is translation by
$-r_\mu(\mathrm{Art}_\mathsf{E}^{-1})$.

\subsubsection{} \label{Sec:LocalGalActionShi}
We now fix a place $v \mid p$ of $\mathsf{E}$ and base change to the local field
$E = \mathsf{E}_v$. We also write $T = \mathsf{T}_{\qp}$, so that using the
embedding $\mathsf{E} \hookrightarrow E$ we may regard $\mu$ as a geometric
cocharacter of $T$ with field of definition $E$. There is a corresponding map of
$\qp$-tori
\[
  r_{\mu,v} \colon \operatorname{Res}_{E/\qp} \mathbb{G}_m \to T.
\]
It is related to $r_\mu$ via the inclusion
\[ \begin{tikzcd}
  \operatorname{Res}_{E/\qp} \mathbb{G}_m \arrow[hook]{r} \arrow{d}{r_{\mu,v}} &
  (\operatorname{Res}_{\mathsf{E}/\mathbb{Q}} \mathbb{G}_m)_{\qp}
  \arrow{d}{r_\mu} \\ T \arrow[equals]{r} & \mathsf{T}_{\qp},
\end{tikzcd} \]
where on $\qp$-points the top horizontal arrow is given by $E^\times \hookrightarrow
(\mathsf{E} \otimes_{\mathbb{Q}} \qp)^\times$ corresponding to the place $v \mid
p$.

\subsubsection{}
Consider the Shimura variety base-changed to the local field $E$. It has the
same $\bar{E}$-points
\[
  \mathbf{Sh}(\mathsf{T}, \mu)_E(\bar{E}) = \shigeompts,
\]
and the Galois action is given as the restriction along $\gal(\eab/E) \to
\gal(\mathsf{E}^\mathrm{ab}/\mathsf{E})$. By compatibility of the local and
global Artin maps, we have a commutative diagram
\[ \begin{tikzcd}
  \gal(\mathsf{E}^\mathrm{ab}/\mathsf{E})
  \arrow{r}{\mathrm{Art}_\mathsf{E}^{-1}} & (\mathsf{E}^\times)^- \backslash
  \mathbb{A}_\mathsf{E}^{\infty,\times} \arrow{r}{r_\mu} & \shigeompts \\
  \gal(\eab/E) \arrow{r}{\mathrm{Art}_E^{-1}} \arrow{u}{\mathrm{res}} &
  \widehat{E^\times} \arrow{u}{\tilde{\iota}} \arrow{ru}
\end{tikzcd} \]
where $\widehat{E^\times}$ is the profinite completion of $E^\times$,
and $\tilde{\iota}$ is the unique extension of $\iota \colon E^\times
\hookrightarrow \mathbb{A}_\mathsf{E}^{\infty,\times} \twoheadrightarrow
(\mathsf{E}^\times)^- \backslash \mathbb{A}_\mathsf{E}^{\infty,\times}$ to a
continuous map. It follows that the local Galois action on the Shimura variety
can be described as translation by $r_\mu \circ \tilde{\iota}$, which is the
unique extension of
\[
  E^\times \xrightarrow{r_{\mu,v}} T(\qp) \to \shigeompts
\]
to a continuous map.

\subsubsection{} \label{Sec:FiberProduct}
As in \Cref{Sec:Choosingb}, we choose a uniformizer $\pi_E \in E$ and let
$b = \operatorname{Nm}_{E/E_0}(\mu(\pi_E^{-1})) \in T(E_0) \subseteq T(\qpbr)$.
We have constructed in \Cref{Def:IgusaTori} the Igusa stack
$\mathrm{Igs}(\mathsf{T}, b) \to \bun_T^{[b]}$, and we have given in
\Cref{Sec:BLGeneral} an explicit description of the Beauville--Laszlo map
$\mathrm{BL}_{\mu^{-1}} \colon \spd E \to \bun_T^{[b]}$. We can now compute the
fiber product $\mathrm{Igs}(\mathsf{T},b) \times_{\bun_T} \spd E$ as
\[ \begin{tikzcd}
  \lbrack (\underline{\shigeompts} \times \spd \eabhat) /
  \underline{\gal(\eab/E)} \rbrack \arrow{r} \arrow{d} & \lbrack \spd \eabhat /
  \underline{\gal(\eab/E)} \rbrack \arrow{d}{\mathrm{BL}_{\mu^{-1}} = \lbrack
  \rho / (r_{\mu,v} \circ \alpha) \rbrack} \\ \mathrm{Igs}(\mathsf{T}, b) =
  \lbrack (\underline{\shigeompts} \times \spd \fpbar) / \underline{A_{T,b}}
  \rbrack \arrow{r}{\lbrack \mathrm{pr}_2 / \id \rbrack} & \lbrack \spd \fpbar /
  \underline{A_{T,b}} \rbrack = \bun_T^{[b]},
\end{tikzcd} \]
where the action of $\gal(\eab/E)$ on $\shigeompts$ is via translation along the
homomorphism
\[
  \gal(\eab/E) \xrightarrow{\alpha} A_{S, b_S} \xrightarrow{r_{\mu,v}} A_{T,b}
  \xrightarrow{\iota_{T,b}} \shigeompts,
\]
see \Cref{Def:IgusaTori}.

\begin{Lem} \label{Lem:ActionComparison}
  We have
  \[
    -r_\mu \circ \tilde{\iota} \circ \mathrm{Art}_E^{-1} = \iota_{T,b} \circ
    r_{\mu,v} \circ \alpha
  \]
  as continuous group homomorphisms $\gal(\eab/E) \to \shigeompts$.
\end{Lem}

\begin{proof}
  Let us write $f = [k_E : \fp]$. By \Cref{Prop:AlphaUsingLCFT}, it is enough to
  show that the diagram
  \[ \begin{tikzcd}[row sep=small]
    \widehat{E^\times} \arrow{r}{\tilde{\iota}} \arrow[equals]{d} &[3.5em]
    (\mathsf{E}^\times)^- \backslash \mathbb{A}_\mathsf{E}^{\infty,\times}
    \arrow{r}{r_\mu} & \shigeompts \\ \mathcal{O}_E^\times \times
    \widehat{\pi_E^\mathbb{Z}} \arrow{r}{r_{\mu,v} \times (\pi_E \mapsto
    \phi^f)} & T(\qp) \times \gal(\fpbar/k_E) \arrow[hook]{r} & A_{T,b}
    \arrow{u}{\iota_{T,b}}
  \end{tikzcd} \]
  commutes. Since $E^\times$ is dense in its profinite completion, it suffices
  to check that the two compositions agree upon restricting to $E^\times$.
  Using the characterization in \Cref{Sec:LocalGalActionShi}, we may instead
  show that the diagram
  \[ \begin{tikzcd}[row sep=small]
    E^\times \arrow{r}{r_{\mu,v}} \arrow[equals]{d} &[3.5em] T(\qp)
    \arrow[hook]{r}{\iota} & \shigeompts \\ \mathcal{O}_E^\times \times
    \pi_E^\mathbb{Z} \arrow{r}{r_{\mu,v} \times (\pi_E \mapsto \phi^f)} & T(\qp)
    \times \gal(\fpbar/k_E) \arrow[hook]{r} & A_{T,b} \arrow{u}{\iota_{T,b}}
  \end{tikzcd} \]
  commutes.

  We now evaluate both maps on $u \pi_E^n \in E^\times$, where $n \in
  \mathbb{Z}$ and $u \in \mathcal{O}_E^\times$. The commutativity of the diagram
  is equivalent to the identity
  \[
    \iota(r_{\mu,v}(u \pi_E^n)) = \iota_{T,b}(r_{\mu,v}(u), \phi^{fn}).
  \]
  Recall from \Cref{Lem:LocalIota} that $\iota_{T,b}(b, \phi) = 1$ by
  construction. We compute
  \[
    1 = \iota_{T,b}(b, \phi)^f = \iota_{T,b}(b \phi(b) \dotsm \phi^{f-1}(b),
    \phi^f) = \iota_{T,b}(r_{\mu,v}(\pi_E^{-1}), \phi^f)
  \]
  using the discussion in \Cref{Sec:Choosingb}. It follows that
  \[
    \iota_{T,b}(r_{\mu,v}(u), \phi^{fn}) = \iota_{T,b}(r_{\mu,v}(u \pi_E^n),
    1)
  \]
  On the other hand, we also have $\iota_{T,b}(x, 1) = \iota(x)$ for all $x \in
  T(\qp)$ by construction, see \Cref{Lem:LocalIota}. Hence
  \begin{align}
      \iota_{T,b}(r_{\mu,v}(u \pi_E^n),1)= \iota(r_{\mu,v}(u \pi_E^n)),
  \end{align}
  showing the desired identity.
\end{proof}

\begin{Thm} \label{Thm:IgusaStackTorus}
  The Igusa stack $\mathrm{Igs}(\mathsf{T}, b) \to \bun_T^{[b]}$ from
  \Cref{Def:IgusaTori}, together with prime-to-$p$ Hecke action from
  \Cref{Sec:IgsTTameHecke}, is an Igusa stack for the Shimura datum
  $(\mathsf{T}, \mu)$.
\end{Thm}

\begin{proof}
  It follows from \Cref{Sec:ShimuraGalois}, \Cref{Sec:FiberProduct} and
  \Cref{Lem:ActionComparison} that there exists an isomorphism of v-sheaves
  \[
    \mathbf{Sh}(\mathsf{T}, \mu)_E^\diamondsuit \cong \mathrm{Igs}(\mathsf{T},
    b) \times_{\mathrm{Bun}_T} \spd E.
  \]
  It is clear that the isomorphism is
  $\underline{\mathsf{T}(\afp)}$-equivariant. Finally, it follows from
  \Cref{Sec:IgsTFrobenius} that $\phi \times_{\phi \cong \id} \id$ on the right
  hand side agrees with $\id \times_{\id} \id = \id$, see \Cref{Def:IgusaStack}.
\end{proof}

\begin{Rem} \label{Rem:NotCanonical}
  We emphasize that there are many $\mathsf{T}(\af)$-equivariant isomorphisms
  $\mathbf{Sh}(\mathsf{T}, \mu)_E^\diamondsuit \cong \mathrm{Igs}(\mathsf{T}, b)
  \times_{\mathrm{Bun}_T} \spd E$. In other words, the argument shows that there
  exists a $\mathsf{T}(\afp)$-equivariant isomorphism of Igusa stacks
  $\mathrm{Igs}(\mathsf{T}, \mu) \cong \mathrm{Igs}(\mathsf{T}, \mu^\prime)$ for
  certain pairs of cocharacters $\mu, \mu^\prime$. For Hodge type Shimura varieties, such isomorphisms are expected to be closely related to mod $p$ isogenies between the corresponding CM abelian varieties, but this relation is not clear from the above description. 
\end{Rem}

}

{\section{Igusa stacks for Shimura varieties of abelian type} \label{Sec:IgusaStacksAbelianType}

The goal of this section is to show that Shimura varieties of abelian type admit Igusa stacks, which we do in Theorem \ref{Thm:MainThmIgusa} and Corollary \ref{Cor:FiniteLevelIgusa1}. Recall that a Shimura datum is of \emph{abelian type} if there exists a pair $(\gx,\Xi)$, where $\gx$ is a Shimura datum of Hodge type and $\Xi:\g^{\mathrm{der}} \to \g_2^\mathrm{der}$ is a central isogeny inducing an isomorphism $\gxad \to (\g_2^{\mathrm{ad}}, \mathsf{X}_2^{\mathrm{ad}})$. We will refer to the pair $(\gx,\Xi)$ as an \emph{auxiliary Hodge-type Shimura datum}. Shimura data of abelian type are classified in \cite[Appendix B]{Milne}; in particular all Shimura data $\gx$ with $\g$ of type $A,B,C$ are of abelian type, and Shimura data of type $D$ are often of abelian type.

As explained in the introduction, to construct an Igusa stack for a Shimura datum of abelian type $\gxtwo$, we find a diagram of Shimura data
\begin{equation*}
    \begin{tikzcd}
        & \gxthree \arrow{dr} \arrow{dl} \\
        \gxtwo && \gx,
    \end{tikzcd}
\end{equation*}
where $\gx$ is of Hodge type, where the left arrow is an ad-isomorphism and where the right arrow induces an isomorphism of derived groups. In Section \ref{Sec:Ascending}, we show that the existence of Igusa stacks ``ascends'' from $\gx$ to $\gxthree$. In Section \ref{Sec:Pushout}, we show that the existence of Igusa stacks ``descends'' from $\gxthree \to \gxtwo$, under an assumption on the reflex field. To do this, we recall in Section \ref{Sub:PushoutShimura} how the Shimura varieties for $\gxtwo$ can be constructed out of those for $\gxthree$. In Section \ref{sub:ConclusionConstruction}, we put everything together and prove Theorem \ref{Thm:MainThmIgusa}. We construct Igusa stacks at finite level $K^p \subset \gafp$ in Section \ref{sub:FiniteLevelIgusa}.

\subsection{Ascending Igusa stacks} \label{Sec:Ascending}
The aim of this section is to prove that the construction of Igusa stacks ascends along certain pullback diagrams of Shimura data; see Theorem \ref{Thm:AscendingIgusaStacks} below. We begin with two lemmas about the potentially good reduction locus of a Shimura variety. 

\subsubsection{} \label{subsub:ProductNotation} Suppose that $\gx$ and $\gxp$ are Shimura data of abelian type,
with reflex fields $\mathsf{E}$ and $\mathsf{E}'$. Let $\mathsf{E}''$ denote the
compositum of $\mathsf{E}$ and $\mathsf{E}'$. Fix a prime $v''$ of
$\mathsf{E}''$ above $p$, and let $v$ and $v'$ be the induced primes of
$\mathsf{E}$ and $\mathsf{E}'$. Denote by $E$, $E'$, and $E''$ the completions
of $\mathsf{E}$, $\mathsf{E}'$, and $\mathsf{E}''$ at $v, v'$, and $v''$,
respectively.

\begin{Lem}\label{Lem:GoodReductionProduct}
    For any neat compact open subgroups $K$ and $K'$ of $\g(\af)$ and $\g'(\af)$, respectively, there is a canonical isomorphism
    \begin{equation}\label{Eq:GoodReductionProduct}
        \mathbf{Sh}_{K\times K'}(\g \times \g', \mathsf{X} \times \mathsf{X}')^\circ_{E''} \xrightarrow{\sim} \mathbf{Sh}_K\gx^\circ_{E''} \times_{\spa {E''}} \mathbf{Sh}_{K'}\gxp^\circ_{E''}. 
    \end{equation}
\end{Lem}

\begin{proof}
    Both sides of \eqref{Eq:GoodReductionProduct} are quasicompact open subsets of $\mathbf{Sh}_{K\times K'}(\g\times\g', \mathsf{X} \times \mathsf{X}')^\mathrm{an}_{E''}$, so by \cite[Lemma 3.6 (i)]{ImaiMieda}, it is enough to check that the two have the same classical points. To a classical point $x''$ of $\mathbf{Sh}_{K\times K'}(\g\times\g', \mathsf{X} \times \mathsf{X}')^\mathrm{an}_{E''}$ we can attach a $(\g^\mathrm{ad}\times \g'^{,\mathrm{ad}})(\af)$-valued Galois representation $\phi_{x''}^\mathrm{ad}$ (see \cite[Definition 5.1]{ImaiMieda}), and by \cite[Corollary 5.14]{ImaiMieda}, the point $x''$ lies in $\mathbf{Sh}_{K\times K'}(\g\times\g', \mathsf{X} \times \mathsf{X}')^\circ$ if and only if the projection $\phi_{x'',p}^\mathrm{ad}$ of $\phi_{x''}^\mathrm{ad}$ onto $(\g^\mathrm{ad}\times \g'^{,\mathrm{ad}})(\qp)$ is potentially crystalline. 
    
    Fix faithful algebraic representations $\xi$ and $\xi'$ of $\g^\mathrm{ad}$ and $\g'^{,\mathrm{ad}}$ on finite free $\qp$-vector spaces $V$ and $V'$, respectively. Then $\xi'' = \xi \oplus \xi'$ is a faithful representation of $\g^\mathrm{ad} \times \g'^{,\mathrm{ad}}$, and by \cite[Lemma 2.9 (ii) and (iii)]{ImaiMieda}, $\phi_{x'',p}^\mathrm{ad}$ is potentially crystalline if and only if $\xi'' \circ \phi_{x'',p}^\mathrm{ad}$ is potentially crystalline.
    
    Let $\kappa_{x''}$ denote the residue field of $x''$, and let $\overline{\kappa}_{x''}$ denote an algebraic closure of $\kappa_{x''}$. Let $x$ and $x'$ be the classical points of $\mathbf{Sh}_K\gx^\mathrm{an}_{E''}$ and $\mathbf{Sh}_{K'}\gxp^\mathrm{an}_{E''}$ mapped to by $x''$ under the projection morphisms. By \cite[Proposition 5.3]{ImaiMieda}, the representation $\xi'' \circ \phi_{x'',p}^\mathrm{ad}$ decomposes as the direct sum of the restrictions of $\xi\circ \phi_{x,p}^\mathrm{ad}$ and $\xi'\circ \phi_{x',p}^\mathrm{ad}$ to $\operatorname{Gal}(\overline{\kappa}_{x''}/\kappa_{x''})$, up to conjugation by $\xi(K^\mathrm{ad})$ and $\xi'(K'^{,\mathrm{ad}})$. Therefore $\xi''\circ \phi_{x'',p}^\mathrm{ad}$ is potentially crystalline if and only if both $\xi\circ\phi_{x,p}^\mathrm{ad}$ and $\xi'\circ\phi_{x',p}^\mathrm{ad}$ are potentially crystalline. In turn, $\phi_{x'',p}^\mathrm{ad}$ is potentially crystalline if and only if both $\phi_{x,p}^\mathrm{ad}$ and $\phi_{x',p}^\mathrm{ad}$ are potentially crystalline, and the lemma follows. 
\end{proof}

\begin{Lem} \label{Lem:GoodReductionAdIsom}
    Suppose $f\colon\gx \to \hy$ is an ad-isomorphism of Shimura data. Fix a prime $v$ of $\mathsf{E}$, and let $w$ be the induced place of $\mathsf{F}$ under $\mathsf{F} \subset \mathsf{E}$, and let $E$ and $F$ denote the completions of $\mathsf{E}$ and $\mathsf{F}$ at $w$ and $v$, respectively. Let $\mathsf{K}$ and $\mathsf{L}$ be neat compact open subgroups of $\g(\af)$ and $\h(\af)$, respectively, such that $f(\mathsf{K}) \subset \mathsf{L}$. Then $f$ induces an isomorphism
	\begin{equation}\label{Eq:GoodReductionAdIsomII}
		\mathbf{Sh}_{\mathsf{K}}\gx^\circ \xrightarrow{\sim} \mathbf{Sh}_{\mathsf{L}} \hy_{E}^\circ \times_{\mathbf{Sh}_{\mathsf{L}}\hy^{\mathrm{an}}_{E}} \mathbf{Sh}_{\mathsf{K}}\gx^\mathrm{an}.
	\end{equation}
\end{Lem}
\begin{proof}
	This follows by assembling various results in \cite{ImaiMieda} as in the proof of Lemma \ref{Lem:GoodReductionProduct}. Both the source and the target of \eqref{Eq:GoodReductionAdIsomII} are locally closed constructible subsets of $\mathbf{Sh}_{\mathsf{K}}\gx^\mathrm{an}$, so by \cite[Lemma 3.6 (i)]{ImaiMieda}, it is enough to check that the two have the same classical points. To a classical point $x$ we can attach a $\mathsf{G}^\mathrm{ad}(\af)$-valued Galois representation $\phi_x^\mathrm{ad}$ (see \cite[Definition 5.1]{ImaiMieda}), and by \cite[Corollary 5.14]{ImaiMieda}, $x$ is in the potentially good reduction locus of $\mathbf{Sh}_{\mathsf{K}}\gx^\mathrm{an}$ if and only if the projection $\phi_{x,p}^\mathrm{ad}$ of $\phi_{x}^\mathrm{ad}$ onto $\mathsf{G}^\mathrm{ad}(\qp)$ is potentially crystalline. It is clear that this can be checked after projecting to the Shimura variety for the adjoint group, and the lemma follows. 
\end{proof}

\begin{Prop}\label{Prop:ProductIgusaStack}
Let the notation be as in Section \ref{subsub:ProductNotation}. Fix $? \in \{\emptyset, \circ\}$. Suppose $\gx$ admits an Igusa stack $\igs^{?}\gx$ for $?$ at the place $v$, and $\gxp$ admits an Igusa stack $\igs^?\gxp$ for $?$ at the place $v'$. Then the product $\igs^{?} \gx \times \igs^?\gxp$ is an Igusa stack for the product Shimura datum $(\g \times \g', \mathsf{X} \times \mathsf{X}')$ for $?$ at the place $v''$.
\end{Prop}

\begin{proof}
    Let $G = \g_{\qp}$ and $G'=\g'_{\qp}$. Then we have a natural isomorphism $\bun_{G\times G'} \equiv \bun_G \times \bun_{G'}$, and thus the proposition for the entire Shimura variety (i.e., $? = \emptyset$) will follow formally once we verify the decomposition
    \begin{equation*}
        \operatorname{Gr}_{G \times G', (\mu^{-1} \times {\mu'}^{-1})} \cong (\operatorname{Gr}_{G,\mu^{-1}})_{E''} \times_{\spd {E''}} (\operatorname{Gr}_{G',{\mu'}^{-1}})_{E''}.
    \end{equation*}
    There is a natural map \begin{equation}\label{Eq:GrProduct}
    	\operatorname{Gr}_{G \times G', (\mu^{-1} \times {\mu'}^{-1})} \to (\operatorname{Gr}_{G,\mu^{-1}})_{E''} \times_{\spd {E''}} (\operatorname{Gr}_{G',{\mu'}^{-1}})_{E''}
    \end{equation}
    induced by the projections $G \times G' \to G$ and $G\times G' \to G'$. By \cite[Proposition 20.2.3]{ScholzeWeinsteinBerkeley} and \cite[Corollary 11.29]{EtCohDiam}, both $\operatorname{Gr}_{G \times G', (\mu^{-1} \times {\mu'}^{-1})}$ and $(\operatorname{Gr}_{G,\mu^{-1}})_{E''} \times_{\spd {E''}} (\operatorname{Gr}_{G',{\mu'}^{-1}})_{E''}$ are spatial diamonds, and therefore the map \eqref{Eq:GrProduct} is qcqs. By \cite[Lemma 12.5]{EtCohDiam} it is then enough to check that \eqref{Eq:GrProduct} is a bijection on $\spa(C,C^+)$-points for all algebraically closed perfectoid fields $(C,C^+)$ of characteristic $p$, and this follows from the definition of Schubert cells in the affine Grassmannian \cite[Definition 19.2.2]{ScholzeWeinsteinBerkeley}.

    For the potentially crystalline locus (i.e., $? = \circ$), we observe that by \Cref{Lem:GoodReductionProduct}, for any pair of neat compact open subgroups $K$ and $K'$ of $\g(\af)$ and $\g'(\af)$ we have the decomposition \eqref{Eq:GoodReductionProduct}. Since any compact open subgroup of $(\g \times \g')(\af)$ is contained in one of the form $K \times K'$ for such a $K$ and $K'$, by taking limits of both sides of \eqref{Eq:GoodReductionProduct} we obtain
    \begin{equation*}
        \mathbf{Sh}(\g \times \g', \mathsf{X} \times \mathsf{X}')^\circ \cong \mathbf{Sh}\gx^\circ_{E''} \times_{\spa {E''}} \mathbf{Sh}\gxp^\circ_{E''},
    \end{equation*}
    and once again the result follows formally.
\end{proof}

Now we can prove the main result of this section. 

\begin{Thm} \label{Thm:AscendingIgusaStacks}
Fix $? \in \{\emptyset, \circ\}$. Let $\gx$ and $\hy$ be Shimura data of abelian type with reflex fields $\mathsf{E}$ and $\mathsf{F}$, and let 
  \[
    f \colon \gx \to \hy
  \]
  be a morphism of Shimura data which induces an isomorphism on derived groups.
  Fix a prime $v$ of $\mathsf{E}$ above a rational prime $p$, let $w$ be the
  induced prime of $\mathsf{F}$ under $\mathsf{F} \subset \mathsf{E}$, and let
  $E$ and $F$ be the corresponding completions. If $\hy$ admits an Igusa stack
  $\igs^?\hy$ for $?$ at the place $w$, then $\gx$ admits an Igusa stack
  $\igs^?\gx$ at the place $v$.
\end{Thm}

\begin{proof}
    Let $\mathsf{E}'\subset \mathsf{E}$ denote the reflex field of $\gxab$, and let $v'$ be the prime of $\mathsf{E}'$ induced by $v$. The product $(\h\times \g^\mathrm{ab}, \y \times \x^\mathrm{ab})$ is a Shimura datum with reflex field given by the compositum $\mathsf{E}''\subset \mathsf{E}$ of $\mathsf{F}$ and $\mathsf{E}'$. By Theorem \ref{Thm:IgusaStackTorus}, $\gxab$ admits an Igusa stack for $?$ at $v'$, so by Proposition \ref{Prop:ProductIgusaStack}, $(\h \times \g^\mathrm{ab}, \y\times \x^\mathrm{ab})$ admits an Igusa stack $\igs^?(\h \times \g^\mathrm{ab}, \y \times \x^\mathrm{ab})$ for $?$ at the prime $v''$ of $\mathsf{E}''$ induced by $v$.
    Since $f \colon \g \to \h$ induces an isomorphism on derived groups, the diagram
    \begin{equation}\label{Eq:AscendingDiagram}
	\begin{tikzcd}
		\gx 
		\arrow[r] \arrow[d]
		& \gxab 
		\arrow[d]
		\\ \hy  
		\arrow[r]
		& \hyab
	\end{tikzcd} 
    \end{equation}
    is Cartesian. Thus we have a closed embedding of Shimura data
    \begin{equation*}
        \gx \hookrightarrow (\h\times \g^\mathrm{ab}, \y\times \x^\mathrm{ab}),
    \end{equation*}
    and we conclude using \cite[Theorem 11.4]{KimFunctorial}, which shows that the existence of Igusa stacks passes to sub Shimura data, together with Lemma \ref{Lem:GoodReductionAdIsom}.
\end{proof}

\subsection{Descending Shimura varieties} \label{Sub:PushoutShimura}
The goal of this section is to prove Proposition \ref{Prop:PushoutShimura}, which relates two Shimura varieties along an ad-isomorphism of Shimura data.

\subsubsection{} Let $\gx$ be a Shimura datum and let $\zg$ be the center of $\g$. Let $\Gad(\mathbb{R})^1$ be the image of $\g(\mathbb{R}) \to \Gad(\mathbb{R})$ and let $\Gad(\mathbb{Q})^1$ be its intersection with $\Gad(\mathbb{Q})$. Note that $\Gad(\mathbb{Q})$ acts on $\g$ by conjugation and that $\Gad(\mathbb{Q})^1$ acts on $\gx$ by morphisms of Shimura data.

\subsubsection{} For $K \subset \gaf$ a neat compact open subgroup we will write $\mathbf{Sh}_K\gx$ for the Shimura variety over $\mathsf{E}$ of level $K$, and we will write $\mathbf{Sh}\gx=\varprojlim_K \mathbf{Sh}_K\gx$ for the inverse limit. This inverse limit comes equipped with a continuous left action of the locally profinite group $\gaf$, or equivalently, a left action of the locally profinite group scheme $\ul{\gaf}$ over $\mathsf{E}$. We will write $\zgqbar \subset \gaf$ for the closure of $\zgq$ inside $\gaf$. The following lemma is well known.

\begin{Lem} \label{Lem:NonSV5}
    The action of $\ul{\zgqbar} \subset \ul{\gaf}$ on $\mathbf{Sh}\gx$ is trivial.
\end{Lem}

\begin{proof}
  Since the inverse limit $\mathbf{Sh}\gx$ is separated scheme, the locus
  \[
    Z = \lbrace (g, x) \in \ul{\zgqbar} \times \mathbf{Sh}\gx : gx = x \rbrace
  \]
  is a closed subscheme. By \cite[Corollary~2.1.12]{DeligneVarietes}, the locus
  $Z$ contains all $\mathbb{C}$-points, and thus $Z = \ul{\zgqbar} \times
  \mathbf{Sh}\gx$.
\end{proof}

\subsubsection{} By considering the long exact sequences in cohomology with $\mathbb{Q}$ and $\mathbb{R}$ coefficients associated with $1 \to \zg \to \g \to \Gad \to 1$, we see that the subgroup $\g(\mathbb{Q})/\zgq \subset \gadqone$ is normal with abelian cokernel given by $\ker (H^1(\mathbb{Q}, \zg) \to H^1(\mathbb{R}, \zg) \times H^1(\Q, \g))$.\footnote{In particular, this cokernel is finite if $\zg$ is connected.} Together with Lemma \ref{Lem:NonSV5}, this gives rise to an action of the abstract group\footnote{The multiplication in the semidirect product is given by $(x, g) (y, h) = (x \cdot {}^g y, gh)$ and the map $\g(\mathbb{Q}) \to (\gaf/\zgqbar) \rtimes \Gad(\mathbb{Q})^1$ is given by $g \mapsto (g^{-1}, g)$, cf. \cite[Section 3.3.1]{KisinModels}.}
\begin{align*}
    \mathcal{A}(\g):=\left(\frac{\gaf}{\zgqbar} \rtimes \gadqone\right) \Big/ \frac{\g(\mathbb{Q})}{\zgq}
\end{align*}
on $\mathbf{Sh}\gx$. We now equip ${\gaf}/{\zgqbar}$ with the quotient topology, the groups $\gadqone$ and $\g(\mathbb{Q})/\zgq$ with the discrete topology, the product $({\gaf}/{\zgqbar}) \rtimes \gadqone$ with the product topology, and $\mathcal{A}(\g)$ with the quotient topology; note that this is Hausdorff since $\g(\mathbb{Q})/\zgq$ has closed image. It follows that $\ul{\mathcal{A}(\g)}$ acts on $\mathbf{Sh}\gx$ since $\ul{\gaf}$ and $\gadqone$ both act on $\mathbf{Sh}\gx$. 

\subsubsection{} If $f:\gx \to \hy$ is a morphism of Shimura data such that $f(\zg) \subset Z_{\h}$, then there is an induced morphism $\ag \to \ah$ and the corresponding morphism of Shimura varieties $\mathbf{Sh}\gx \to \mathbf{Sh} \hy_{\mathsf{E}}$ is $\ul{\ag}$-equivariant in the obvious way. The following result is well known (see e.g. \cite[Lemma 4.56]{YangZhuGeneric}, \cite[Lemma 2.7.11.(b)]{DeligneVarietes}). We give a proof for the sake of completeness, following \cite[Section 2]{DeligneVarietes}. 

\begin{Prop} \label{Prop:PushoutShimura}
If $f:\gx \to \hy$ is an ad-isomorphism, then the natural map $\mathbf{Sh}\gx \to \mathbf{Sh}\hy_{\mathsf{E}}$ induces an isomorphism
  \[
    \ul{\ah} \times^{\ul{\ag}} \mathbf{Sh}\gx \to \mathbf{Sh}\hy_{\mathsf{E}},
  \]
  where the quotient is taken in the pro-\'etale topology.
\end{Prop}
\begin{proof}
  \def\shgx{\mathbf{Sh}\gx}
  \def\shhy{\mathbf{Sh}\hy_\mathsf{E}}
  \def\shhyfin{\mathbf{Sh}_K\hy_\mathsf{E}}
  \def\pizhy{\mathsf{H}(\af)/\mathsf{H}(\Q)_+^-}
  There is indeed a natural map of $\mathsf{E}$-schemes
  \[
    \pi \colon \ul{\ah} \times \shgx \to \shhy
  \]
  that is invariant for the $\ul{\ag}$-action on the source. It remains to show
  that $\pi$ is a pro-\'{e}tale $\ul{\ag}$-torsor.

  It suffices to show that $\pi$ is weakly \'{e}tale, an fpqc cover, and a
  quasi-torsor for $\ul{\ag}$ in the sense that the action map
  \[
    \alpha \colon (\ul{\ah} \times \shgx) \times \ul{\ag} \to (\ul{\ah} \times
    \shgx) \times_{\shhy} (\ul{\ah} \times \shgx)
  \]
  is an isomorphism. Note that once we fix a neat level $K \subseteq
  \mathsf{H}(\af)$, both sides of $\pi$ are disjoint unions of pro-finite
  \'{e}tale covers of $\shhyfin$, and hence weakly \'{e}tale. It then follows
  from \cite[Proposition~2.2.3.(4)]{BhattScholzeProEtale} that $\pi$ is weakly
  \'{e}tale as well. To show that $\pi$ is an fpqc cover, it suffices to find a
  compact open subset $U \subseteq \ah$ for which $\ul{U} \times \shgx \to
  \shhy$ is surjective. This follows from the fact that a connected component of
  $\shgx$ surjects onto a connected component of $\shhy$, and that
  $\pi_0(\mathbf{Sh}\hy_\mathbb{C}) = \pizhy$ as
  profinite sets, see \cite[Proposition~2.1.14]{DeligneVarietes}.\footnote{Here we write $\g(\R)_{+}$ for the inverse image of the identity component (in the real topology) of $\Gad(\R)$ under the natural map $\g(\mathbb{R}) \to \g^{\mathrm{ad}}(\mathbb{R})$, and $\g(\Q)_{+}=\g(\Q) \cap \g(\R)_{+}$.}Indeed,
  because $\pizhy$ is profinite, there exists a compact open $\tilde{U}
  \subseteq \mathsf{H}(\af)$ that surjects onto $\pizhy$ as in
  \Cref{Prop:LocProfExact}, and then an open neighborhood $U$ of $\im(\tilde{U}
  \to \ah)$ has the property that $\ul{U} \times \shgx \to \shhy$ is surjective.

  We now prove that $\alpha$ is an isomorphism. Recall that both sides are
  disjoint unions of pro-finite \'{e}tale covers over $\shhyfin$. If we write
  $\shhyfin = \coprod_i V_i$ where $V_i$ are connected, then the category of
  schemes that are a disjoint union of pro-finite \'{e}tale covers of $V_i$ is
  equivalent to the category of locally profinite sets with a continuous
  $\pi_1^\mathrm{\acute{e}t}(V_i, \bar{v}_i)$-action. Thus it is enough to show
  that $\alpha$ induces isomorphisms on fibers over $\mathbb{C}$-points of
  $\shhyfin$. Since these fibers correspond to locally profinite sets, we may
  further reduce to showing that the continuous map
  \[
    \ah \times \shgx(\mathbb{C})^\mathrm{an} \to
    \mathbf{Sh}\hy(\mathbb{C})^\mathrm{an}
  \]
  of topological spaces is an $\ag$-torsor (i.e., a principal bundle), where we
  give $\shgx(\mathbb{C})$ and $\mathbf{Sh}\hy(\mathbb{C})$ the inverse limit
  topologies of the analytic topologies at finite level.

  In view of \cite[Section~2.1.9]{DeligneVarietes}, we have
  \[
    \shgx(\mathbb{C})^\mathrm{an} = \biggl( \frac{\mathsf{G}(\af)}{\zgqbar}
    \times \mathsf{X} \biggr) \bigg/ \frac{\mathsf{G}(\Q)}{\zgq}
  \]
  as topological spaces, where the quotient is properly discontinuous. We need
  to check
  \[
    S_1 = \ah \times \frac{(\mathsf{G}(\af)/Z_\mathsf{G}(\Q)^-) \times
    \mathsf{X}}{\mathsf{G}(\Q)/Z_\mathsf{G}(\Q)} \to
    \frac{(\mathsf{H}(\af)/Z_\mathsf{H}(\Q)^-) \times
    \mathsf{Y}}{\mathsf{H}(\Q)/Z_\mathsf{H}(\Q)}
  \]
  is a torsor for $\ag$. It suffices to check that
  \[
    S_2 = \ah \times ((\mathsf{G}(\af)/Z_\mathsf{G}(\Q)^-) \times \mathsf{X})
    \to \frac{(\mathsf{H}(\af)/Z_\mathsf{H}(\Q)^-) \times
    \mathsf{Y}}{\mathsf{H}(\Q)/Z_\mathsf{H}(\Q)}
  \]
  is a torsor for $\Gamma = (\mathsf{G}(\af)/Z_\mathsf{G}(\Q)^-) \rtimes
  \mathsf{G}^\mathrm{ad}(\Q)^1$ with the source having the diagonal action,
  because $S_2 \to S_1$ is a Galois covering space for the induced
  $\mathsf{G}(\Q)/Z_\mathsf{G}(\Q)$-action. Because being a torsor can be
  checked after base changing by a surjective covering map, it remains to show
  that
  \[
    ((\mathsf{H}(\af)/Z_\mathsf{H}(\Q)^-) \rtimes \mathsf{H}^\mathrm{ad}(\Q)^1)
    \times ((\mathsf{G}(\af)/Z_\mathsf{G}(\Q)^-) \times \mathsf{X}) \to
    (\mathsf{H}(\af)/Z_\mathsf{H}(\Q)^-) \times \mathsf{Y}
  \]
  is a $\Gamma$-torsor. On the other hand, this map is
  $(\mathsf{H}(\af)/Z_\mathsf{H}(\Q)^-)$-equivariant, hence is a base change of
  \[
    \mathsf{H}^\mathrm{ad}(\Q)^1 \times ((\mathsf{G}(\af)/Z_\mathsf{G}(\Q)^-)
    \times \mathsf{X}) \to \mathsf{Y},
  \]
  and so it suffices to show that this is a $\Gamma$-torsor. As the map is
  evidently $\Gamma$-invariant, this is equivalent to showing that
  \[
    \mathsf{H}^\mathrm{ad}(\Q)^1 \times \mathsf{X} \to \mathsf{Y}
  \]
  is a Galois covering space for $\mathsf{G}^\mathrm{ad}(\Q)^1$. This is now
  clear since $\mathsf{X} \hookrightarrow \mathsf{Y}$ is a union of connected
  components, $\mathsf{G}^\mathrm{ad}(\Q)^1 \hookrightarrow
  \mathsf{H}^\mathrm{ad}(\Q)^1$ is a subgroup, and
  \[
    \mathsf{H}^\mathrm{ad}(\Q)^1 \times^{\mathsf{G}^\mathrm{ad}(\Q)^1}
    \pi_0(\mathsf{X}) = 
    \mathsf{H}^\mathrm{ad}(\mathbb{R})^1 \times^{\mathsf{G}^\mathrm{ad}(\mathbb{R})^1} \pi_0(\mathsf{X}) = \pi_0(\mathsf{Y})
  \]
  because $\mathsf{X}, \mathsf{Y}$ are orbits for $\mathsf{G}(\mathbb{R}),
  \mathsf{H}(\mathbb{R})$. 
\end{proof}

\subsubsection{} Let $\gx \to \hy$ and $\mathsf{F} \subset \mathsf{E}$ be as above. For a finite place $v$ of $\mathsf{E}$ with induced place $w$ of $\mathsf{F}$, we consider the completions $F \subset E$.

\begin{Cor} \label{Cor:PushoutShimuraII}
If $\gx \to \hy$ is an ad-isomorphism, then the natural map of Proposition \ref{Prop:PushoutShimura} induces an isomorphism $($where the quotient is taken in the pro-\'etale topology$)$
    \begin{align*}
        \ul{\ah} \times^{\ul{\ag}} \mathbf{Sh}\gx^{\circ,\diamondsuit} \to \mathbf{Sh}\hy^{\circ,\diamondsuit}_{E}
    \end{align*}
\end{Cor}
\begin{proof}
This is a direct consequence of Lemma \ref{Lem:GoodReductionAdIsom} and Proposition \ref{Prop:PushoutShimura}. 
\end{proof}

\subsection{Descending Igusa stacks} \label{Sec:Pushout}

The goal of this section is to prove the following result. Let $\gx \to \hy$ be a morphism of abelian-type Shimura data which induces an inclusion $\mathsf{F} \subset \mathsf{E}$ of reflex fields. Fix a prime $v$ of $\mathsf{E}$ above a rational prime $p$, let $w$ be the induced prime of $\mathsf{F}$, and let $F$ and $E$ be the corresponding completions. Write $G=\g \otimes \qp$ and $H=\h \otimes \qp$.

\begin{Thm} \label{Thm:DescendingIgusaStacks}
Assume that $F=E$. If $\gx$ admits an Igusa stack for $? \in \{\circ, \emptyset\}$ at the place $v$ and $\g \to \h$ is an ad-isomorphism, then $\hy$ admits an Igusa stack for $?$ at the place $w$.
\end{Thm}

The idea of the proof is relatively simple: We note that $\ul{\ag}$ acts on
$\igs^{?} \gx$ via its quotient $\apg:=\ul{\ag}/\ul{G(\qp)}$, and the Igusa
stacks for $\hy$ should be the quotient
\begin{align*}
    \igs^{?}\hy := \aph \times^{\apg} \igs^{?} \gx.
\end{align*}
From this perspective, it is however not easy to see that $\igs^{?}\hy$ admits a natural map to $\operatorname{Bun}_{H}$. Moreover, it is perhaps not clear how to make sense of the quotient if $\igs^{?} \gx$ is not $0$-truncated. Therefore, in the actual proof we will work from the perspective of the uniformization maps introduced by Kim in \cite[Section 5]{KimFunctorial}, see Definition \ref{Def:UniversalUniformization}.

\subsubsection{Local preparations} We will need the following result in our proof of Theorem \ref{Thm:DescendingIgusaStacks}. Let $f:G \to H$ be an ad-isomorphism of connected reductive groups over $\qp$, and let $\mu$ be a minuscule $G(\qpbar)$-conjugacy class of cocharacters with $\mu_{H}$ the induced $H(\qpbar)$-conjugacy class. There is an induced map $B(G,\mu) \to B(H, \mu_{H})$ which is a bijection by the discussion in \cite[Section 6.5]{Kottwitz2}. We have the following slight generalization of \cite[Lemma 6.4]{KimFunctorial}.

\begin{Lem} \label{Lem:CartesianBL}
Suppose that $G \to H$ is an ad-isomorphism. If $\mu$ and $\mu_{H}$ have the same reflex field, then the $2$-commutative diagram
\begin{equation}
    \begin{tikzcd}
        \lbrack \ul{G(\qp)} \backslash \operatorname{Gr}_{G, \mu} \rbrack \arrow{r} \arrow{d} & \lbrack \ul{H(\qp)} \backslash \operatorname{Gr}_{H, \mu_{H}} \rbrack \arrow{d} \\
        \operatorname{Bun}_{G,\mu} \arrow{r} & \operatorname{Bun}_{H,\mu_{H}}
    \end{tikzcd}
\end{equation}
is $2$-Cartesian.
\end{Lem}
\begin{proof}
The proof of \cite[Lemma 6.4]{KimFunctorial} works almost verbatim. To show that $[b]$ is trivial as in the last paragraph of the proof, we can instead argue that the Newton cocharacter $\nu_{b}$ has trivial image in $H$ and thus in $\gad$ since $G \to H$ is an ad-isomorphism. The rest of the argument works to show that $[b]$ is trivial. 
\end{proof}

\subsubsection{Local preparations II}
There is a left action of $\gad$ on the algebraic group $G$ given by
conjugation, and hence also a left action of $\gad(\qp)$ on $G(\qp)$. For
$i=2,3$ we consider the locally profinite group and the closed normal subgroup
\[
  G(\qp)^i \rtimes \gad(\qp) \supseteq \lbrace (\Delta^i(g), g^{-1}) : g \in
  G(\qp) \rbrace,
\]
where $\gad(\qp)$ acts on $G(\qp)^i$ diagonally and where $\Delta^i \colon
G(\qp) \to G(\qp)^i$ is the diagonal inclusion. We see that the quotient fits in
a short exact sequence
\[
  1 \to G(\qp)^{i-1} \xrightarrow{\underline{y} \mapsto ((1, \underline{y}), 1)}
  \frac{G(\qp)^i \rtimes \gad(\qp)}{\lbrace (\Delta^i(g), g^{-1}) : g \in G(\qp)
  \rbrace} \xrightarrow{((x, \underline{y}), z) \mapsto xz} \gad(\qp) \to 1
\]
of locally profinite groups. We will denote the locally profinite group in the
middle of the above short exact sequence by $\Sigma^i(G)$.

\begin{Lem}
  Consider the natural $\underline{G(\qp)}^2$-action on $\grg \times_{\bung}
  \grg$ as well as the $\underline{\gad(\qp)}$-action on $\grg \times_{\bung}
  \grg$ induced by the conjugation action on $G$. The two actions combine into an
  action of $\Sigma^2(G)$ on $\grg \times_{\bung} \grg$. Similarly, consider the natural $\underline{G(\qp)}^3$-action on $\grg \times_{\bung}
  \grg \times_{\bung} \grg$ as well as the $\underline{\gad(\qp)}$-action on $\grg \times_{\bung}
  \grg \times_{\bung} \grg$ induced by conjugation action on $G$. The two actions combine into an
  action of $\Sigma^3(G)$ on $\grg \times_{\bung} \grg \times_{\bung} \grg$.
\end{Lem}

\begin{proof}
  We first combine the two actions into an action of $\underline{G(\qp)^2
  \rtimes \gad(\qp)}$. To show this, we need to check that first
  acting by $(g, h) \in \underline{G(\qp)}^2$ and then acting by $\gamma \in
  \underline{\gad(\qp)}$ is the same as first acting by $\gamma$ and then by
  $({}^\gamma g, {}^\gamma h) \in \underline{G(\qp)}^2$. Given a triple $(x, y,
  \alpha)$ where $x \in \grg\pair[1]$, $y \in \grg\pair[2]$ and $\alpha \colon
  \mathscr{P}_x \cong \mathscr{P}_y$ is an isomorphism of $G$-torsors on
  $\ffcurve\upair$, acting by $(g, h)$ yields the triple $(gx, hy, r_{h^{-1}}
  \alpha r_g)$. Then acting by $\gamma$ yields $({}^\gamma g {}^\gamma x,
  {}^\gamma h {}^\gamma y, r_{{}^\gamma h^{-1}} {}^\gamma \alpha r_{{}^\gamma
  g})$. This is the same as first acting by $\gamma$ and then acting by
  $({}^\gamma g, {}^\gamma h)$.

  We next show that the subgroup $\lbrace ((g^{-1}, g^{-1}), g) \rbrace$ acts
  trivially. We first note that for $g \in \underline{G(\qp)}\upair$ and
  $\mathscr{P}$ a $G$-torsor on $\ffcurve\upair$, there is a canonical
  isomorphism
  \[
    \tau_g \colon {}^{g} \mathscr{P} = G \times^{\mathrm{conj}_g,G}
    \mathscr{P} \xrightarrow{\cong} \mathscr{P}; \quad (x, p) \mapsto xgp
  \]
  of $G$-torsors. We now consider a point $(x, y, \alpha) \in (\grg
  \times_{\bung} \grg)\upair$ as above. By naturality of $\tau$, there is a
  commutative diagram
  \[ \begin{tikzcd}
    G \vert_{\ffcurve\upair} \arrow[dashed]{r}{x} & \mathscr{P}_ x
    \arrow{r}{\alpha} & \mathscr{P}_ y & G \vert_{\ffcurve\upair}
    \arrow[dashed]{l}[']{y} \\ G \vert_{\ffcurve\upair} \arrow[dashed]{r}{{}^g
    x} \arrow{u}{\tau_g} & {}^g \mathscr{P}_ x \arrow{r}{{}^g \alpha}
    \arrow{u}{\tau_g} & {}^g \mathscr{P}_ y \arrow{u}{\tau_g} & G
    \vert_{\ffcurve\upair}. \arrow[dashed]{l}[']{{}^g y} \arrow{u}{\tau_g}
  \end{tikzcd} \]
  On the other hand, we see from the formula that $\tau_g$ on the trivial
  $G$-torsor $G \vert_{\ffcurve\upair}$ agrees with right multiplication $r_g$
  by $g$. Because the composition
  \[ \begin{tikzcd}
    G \vert_{\ffcurve\upair} \arrow{r}{r_{g^{-1}}} & G \vert_{\ffcurve\upair}
    \arrow[dashed]{r}{x} & \mathscr{P}_ x
  \end{tikzcd} \]
  recovers $g^{-1}x \in \grg\pair[1]$, the action of $((g^{-1}, g^{-1}), g)$ on
  $(x, y, \alpha)$ recovers $(x, y, \alpha)$.

  The statement for $\Sigma^3(G)$ follows from a similar argument.
\end{proof}

\subsubsection{}
Recall that there is an action of $\underline{\ag}$ on the v-sheaf $\shd$. There
is a homomorphism of locally profinite groups
\[
  \ag = \biggl( \frac{\g(\af)}{Z_\mathsf{G}(\mathbb{Q})^-} \rtimes
  \Gad(\mathbb{Q})^1 \biggr) \bigg/ \frac{\g(\mathbb{Q})}{Z_\g(\mathbb{Q})} \to
  \gad(\qp); \quad (x, y) \mapsto xy.
\]
We observe that the Hodge--Tate period map
\[
  \pi_\mathrm{HT} \colon \shd \to \grgmu
\]
is equivariant with respect to this group homomorphism. It follows that the v-sheaf
\begin{align*}
  \shd_{(2)}&:=\shd \times_{\bungmu} \grgmu \\
  &= \shd \times_{\grgmu} (\grgmu \times_{\bungmu} \grgmu)
\end{align*}
has an action of the locally profinite group
\[
  \Xi^2(\g) = \ag \times_{\gad(\qp)} \Sigma^2(G).
\]
This group fits in a short exact sequence
\[
  1 \to G(\qp) \xrightarrow{\mathrm{incl}_2} \Xi^2(\g)
  \xrightarrow{\mathrm{pr}_1} \ag \to 1
\]
of locally profinite groups. Similarly, the v-sheaf
\[
  \shd_{(3)}:=\shd \times_{\bungmu} \grgmu \times_{\bungmu} \grgmu
\]
has an action of the locally profinite group $\Xi^3(\g)$ defined as
\[
  1 \to G(\qp)^2 \xrightarrow{\mathrm{incl}_2} \Xi^3(\g) = \ag
  \times_{\gad(\qp)} \Sigma^3(G) \xrightarrow{\mathrm{pr}_1} \ag \to 1.
\]

\begin{Lem} \label{Lem:PushoutRelation}
  For $i = 2$ or $i = 3$, consider the natural map
  \[
    \shd_{(i)} \to \shtwod_{(i)}
  \]
  which is equivariant with respect to the natural group homomorphism $\Xi^i(\g)
  \to \Xi^i(\h)$ by construction. The induced map
  \[
    [\underline{\Xi^i(\g)} \backslash \shd_{(i)}] \to
    [\underline{\Xi^i(\h)} \backslash \shtwod_{(i)}]
  \]
  is an isomorphism of v-stacks. In other words, there exists an
  $\underline{\Xi^i(\h)}$-equivariant isomorphism
  \[
    \underline{\Xi^i(\h)} \times^{\underline{\Xi^i(\g)}} \shd_{(i)}
    \xrightarrow{\cong} \shtwod_{(i)}
  \]
\end{Lem}

\begin{proof}
  Write $\grgmug=[\underline{G(\qp)} \backslash \grgmu]$ and similarly for $H$.
  Because $G(\qp) \subseteq \Xi^2(\g)$ is the subgroup consisting of elements of
  the form $((1, 1), ((1, y), 1))$, we naturally identify
  \[
    \lbrack \underline{G(\qp)} \backslash (\shd \times_{\bung} \grgmu) \rbrack =
    \shd \times_{\bung} \grgmug,
  \]
  and moreover this v-sheaf has an induced $\underline{\ag}$-action. By
  \Cref{Lem:CartesianBL} the natural diagram
  \[ \begin{tikzcd}[row sep=small]
    \shd \times_{\bung} \grgmug \arrow{r}{\mathrm{pr}_1} \arrow{d} & \shd
    \arrow{d} \\ \shtwod \times_{\bunh} \grgtwomug \arrow{r}{\mathrm{pr}_1} &
    \shtwod
  \end{tikzcd} \]
  is Cartesian. It is also clear from the construction that this is equivariant
  with respect to the actions of $\ag \to \ah$. It then formally follows that
  \[ \begin{tikzcd}[row sep=small]
    \lbrack \underline{\ag} \backslash (\shd \times_{\bung} \grgmug) \rbrack
    \arrow{r}{\mathrm{pr}_1} \arrow{d} & \lbrack \underline{\ag} \backslash \shd
    \rbrack \arrow{d} \\ \lbrack \underline{\ah} \backslash (\shtwod
    \times_{\bunh} \grgtwomug) \rbrack \arrow{r}{\mathrm{pr}_1} & \lbrack
    \underline{\ah} \backslash \shtwod \rbrack
  \end{tikzcd} \]
  is also Cartesian. Because the right vertical map is an isomorphism by
  \Cref{Prop:PushoutShimura} and \Cref{Cor:PushoutShimuraII}, so is the left
  vertical map. On the other hand, it follows from the short exact sequence
  \[
    1 \to G(\qp) \to \Xi^2(\mathsf{G}) \to \ag \to 1
  \]
  that the left vertical map can be identified with
  \[
    [\Xi^2(\mathsf{G}) \backslash \shd_{(2)}] \to [\Xi^2(\mathsf{H}) \backslash
    \shtwod_{(2)}].
  \]
  The argument for $i = 3$ is similar.
\end{proof}

\subsubsection{}
We define a second continuous group homomorphism $\xi_\g \colon \Xi^2(\g) \to
\ag$. Consider the continuous map
\begin{align*}
  (\g(\af) \rtimes \Gad(\mathbb{Q})^1) \times_{\gad(\qp)} (G(\qp)^2 \rtimes
  \gad(\qp)) &\to (\g(\af) \rtimes \Gad(\mathbb{Q})^1), \\
  (((a^p, a_p), b), ((x, y), z)) &\mapsto ((a^p, y x^{-1} a_p), b),
\end{align*}
which is a group homomorphism because
\begin{align*}
  y x^{-1} a_p\, {}^b (y^\prime x^{\prime-1} a_p^\prime) &= yx^{-1}\, {}^{a_p
  b}(y^\prime x^{\prime-1})\, a_p {}^b a_p^\prime = yx^{-1}\, {}^{xz}(y^\prime
  x^{\prime-1})\, a_p {}^b a_p^\prime \\ &= y\, {}^z(y^\prime x^{\prime-1})
  x^{-1}\, a_p {}^b a_p^\prime = (y\, {}^z y^\prime) (x\, {}^z x^\prime)^{-1}
  (a_p {}^b a_p^\prime)
\end{align*}
as $a_p b = x z \in \gad(\qp)$. This sends the closed normal subgroup
\[
  ((Z_\g(\mathbb{Q})^- \rtimes \{1\}) \cdot \lbrace (g, g^{-1}) : g \in
  \g(\mathbb{Q}) \rbrace) \times \lbrace ((g,g), g^{-1}) : g \in G(\qp) \rbrace
\]
on the left hand side to the closed normal subgroup
\[
  (Z_\g(\mathbb{Q})^- \rtimes \{1\}) \cdot \lbrace (g, g^{-1}) : g \in
  \g(\mathbb{Q}) \rbrace
\]
on the right hand side. Taking quotients on both sides, we obtain a continuous
group homomorphism
\[
  \xi_\g \colon \Xi^2(\g) \to \ag.
\]

\begin{Lem} \label{Lem:EquivarianceLemma}
  The global uniformization map
  \[
    \Theta \colon \shd \times_{\bungmu} \grgmu \to \shd
  \]
  is equivariant with respect to the group homomorphism $\xi_\g \colon \Xi^2(\g)
  \to \ag$.
\end{Lem}

\begin{proof}
  We need to check equivariance with respect to an arbitrary
  \[
    (((a^p, a_p), b), ((x, y), z)) \in \underline{(\g(\af) \rtimes
    \Gad(\mathbb{Q})^1) \times_{\gad(\qp)} (G(\qp)^2 \rtimes \gad(\qp))}\pair.
  \]
  Because the subgroup $\lbrace 1 \rbrace \times \lbrace ((g, g), g^{-1}) : g
  \in G(\qp) \rbrace$ acts trivially on both sides, we may reduce to the case
  when $x = a_p$. Note that this implies $b \mapsto z \in \gad(\qp)$.

  We may now write
  \[
    (((a^p, a_p), b), ((x, y), z)) = (((a^p, a_p), 1), ((a_p, y), 1)) \cdot
    (((1,1), b), ((1,1), z)),
  \]
  and check equivariance for both elements separately. For $b \in
  \underline{\Gad(\mathbb{Q})^1}\pair$, equivariance with respect to $((1,b),
  (1,b)) \mapsto (1, b)$ follows from functoriality of $\Theta$ for the morphism
  of Shimura data
  \[
    \operatorname{ad}(b) \colon \gx \to \gx,
  \]
  see \cite[Proposition~10.12]{KimFunctorial}. On the other hand, equivariance
  with respect to $(((a^p, a_p), 1), ((a_p, y), 1)) \mapsto ((a^p, y), 1)$ is
  the axiom for the Hecke action.
\end{proof}

\subsubsection{} \label{Sec:XiGCocycle}
The map $\xi_\mathsf{G} \colon \Xi^2(\mathsf{G}) \to \ag$ fits in the
commutative diagram
\[ \begin{tikzcd}
  \ag \arrow[equals]{rr} \arrow{d}[']{(a,b) \mapsto ((a,b),((1,1),a_pb))} &
  &[2.5em]
  \ag \arrow{d}{(x,y) \mapsto xy} \\ \Xi^2(\mathsf{G})
  \arrow{rru}{\xi_\mathsf{G}} \arrow{r} & \Sigma^2(G) \arrow{r}{((x,y),z)
  \mapsto yz} & \gad(\qp)
\end{tikzcd} \]
of locally profinite groups. We also see that the diagram
\[ \begin{tikzcd}
  \Xi^3(\mathsf{G}) \arrow{d}{((a,b), ((x,y,z),w)) \mapsto ((a,b), ((x,z),w))}
  \arrow{r}{((a,b), ((x,y,z),w)) \mapsto ((a^p, yx^{-1}a_p, b), ((y,z), w))}
  &[12em] \Xi^2(\mathsf{G}) \arrow{d}{\xi_\mathsf{G}} \\ \Xi^2(\mathsf{G})
  \arrow{r}{\xi_\mathsf{G}} & \ag
\end{tikzcd} \]
commutes.

We are now ready to prove \Cref{Thm:DescendingIgusaStacks}. 

\begin{proof}[Proof of \Cref{Thm:DescendingIgusaStacks}]
  Using \Cref{Lem:EquivarianceLemma} and \Cref{Lem:PushoutRelation}, we can push
  both sides of $\Theta$ out along the vertical maps of the commutative diagram
  \[ \begin{tikzcd}[row sep=small]
    \Xi^2(\g) \arrow{r}{\xi_\g} \arrow{d} & \ag \arrow{d} \\ \Xi^2(\h)
    \arrow{r}{\xi_{\h}} & \ah
  \end{tikzcd} \]
  to obtain
  \[
    \Theta^\prime \colon \shtwod \times_{\bunh} \grgtwomu \to \shtwod.
  \]
  By \cite[Proposition~5.12]{KimFunctorial}, to prove the theorem, it suffices
  to check that $\Theta^\prime$ is indeed a global uniformization map for $\hy$
  in the sense of \cite[Definition~5.10]{KimFunctorial}. By construction, the
  map $\Theta^\prime$ is equivariant with respect to the group homomorphism
  $\xi_{\h} \colon \Xi^2(\h) \to \ah$. Because the diagram
  \[ \begin{tikzcd}[column sep=8em]
    \h(\af) \times H(\qp) \arrow{r}{\mathrm{pr}_{13} \colon (a, y) \mapsto (a^p,
    y)} \arrow{d}{(a, y) \mapsto ((a, 1), ((a_p, y), 1))} & \h(\af) \arrow{d}{a
    \mapsto (a, 1)} \\ \Xi^2(\h) \arrow{r}{\xi_{\h}} & \ah
  \end{tikzcd} \]
  commutes, we see that $\Theta^\prime$ is also equivariant with respect to
  $\mathrm{pr}_{13} \colon \h(\af) \times H(\qp) \to \h(\af)$.

  We now check that the diagrams
  \begin{equation}
    \begin{tikzcd}
      \shtwod \arrow[r, equals] \arrow{d}{(\mathrm{id}, \pi_\mathrm{HT})} &
      \shtwod \arrow{d}{\pi_\mathrm{HT}} \\ \shtwod_{(2)}
      \arrow{r}{\mathrm{pr}_2} \arrow{ur}[']{\Theta^\prime} & \grgtwomu
    \end{tikzcd}
    \begin{tikzcd}
      \shtwod_{(3)} \arrow{r}{\Theta^\prime \times \mathrm{id}}
      \arrow{d}{\mathrm{pr}_{13}} & \shtwod_{(2)} \arrow{d}{\Theta^\prime} \\
      \shtwod_{(2)} \arrow{r}{\Theta^\prime} & \shtwod
    \end{tikzcd}
  \end{equation}
  commute. These two diagrams are compatible with the corresponding diagrams for
  $\gx$, which commute by \cite[Proposition~5.12]{KimFunctorial}. The
  commutativity of the diagrams for $\hy$ follows formally once we observe that
  the natural maps
  \[
    \ul{\ah} \times^{\ul{\ag}} \shtwod \to \shtwod, \quad \ul{\Xi^i(\h)}
    \times^{\ul{\Xi^i(\g)}} \shtwod_{(i)} \to \shtwod_{(i)}
  \]
  are isomorphisms. Indeed, they can be obtained by considering corresponding
  diagrams for $\gx$, and then pushing out along the diagrams of
  \Cref{Sec:XiGCocycle} mapping to the corresponding diagrams for $\mathsf{H}$. 

  Finally, for the absolute Frobenius, we note that the isomorphism
  $\phi_{\bung} \cong \id_{\bung}$ is functorial in $G$, and hence the action of
  $\phi \times \id$ on $\shd_{(2)}$ pushes out to the action of $\phi \times
  \id$ on $\shtwod_{(2)}$. Because $\Theta$ is assumed to be $\phi \times
  \id$-invariant, we conclude that $\Theta^\prime$ is also $\phi \times
  \id$-invariant.
\end{proof}

\begin{Rem}
  The two maps $\mathrm{pr}_1, \xi_\g \colon \Xi^2(\g) \to \ag$ identify the
  sheaf $\ul{\Xi^2(\g)}$ with the fiber product $\ul{\ag}
  \times_{\apg} \ul{\ag}$ where we define $\apg =
  \ul{\ag}/\ul{G(\qp)}$.\footnote{Note that $G(\qp)$ is not typically a closed subgroup of $\ag$, so it is better to define the quotient in the world of v-sheaves (or condensed groups).} Then when we push out the \v{C}ech nerve of $\shd \to
  \igs\gx$ with the action of the \v{C}ech nerve of $\ul{\ag} \to \apg$ to
  $\ul{\ah} \to \aph$ we see that the induced map
  \[
    \aph \times^{\apg} \igs^?\gx \to \igs^?\hy
  \]
  is an isomorphism. This justifies the heuristic description given after the
  statement of \Cref{Thm:DescendingIgusaStacks}.
\end{Rem}

{\subsection{Finishing the construction} \label{sub:ConclusionConstruction} In this section we prove the following theorem.
\begin{Thm} \label{Thm:MainThmIgusa}
Fix $? \in \{\circ, \emptyset\}$. Let $\gxtwo$ be a Shimura datum of abelian type and let $v_2$ be a place of the reflex field $\mathsf{E}_2$ above a rational prime $p$. Then $\gxtwo$ admits an Igusa stack for $?$ at the place $v_2$. 
\end{Thm}

\subsubsection{} 
Let $\gxtwo$ be a Shimura datum of abelian type with reflex field $\mathsf{E}_2$ and fix a $p$-adic place $v_2$ of $\mathsf{E}_2$. Recall from the beginning of this section that an \emph{auxiliary Hodge-type Shimura datum} for $\gxtwo$ is a pair $(\gx,\Xi)$, where $\gx$ is a Shimura datum of Hodge type and $\Xi:\g^{\mathrm{der}} \to \g_2^\mathrm{der}$ is a central isogeny inducing an isomorphism $\gxad \to (\g_2^{\mathrm{ad}}, \mathsf{X}_2^{\mathrm{ad}})$. The following lemma should be compared to \cite[Lemma~4.53]{YangZhuGeneric}.

\begin{Lem} \label{Lem:ExistenceGoodRoof}
  Let $(\gx,\Xi)$ be an auxiliary Hodge type Shimura datum with reflex field
  $\mathsf{E}$, and write $\mathsf{E}'=\mathsf{E} \cdot \mathsf{E}_2$. Then
  there exists a Shimura datum $\gxthree$ with reflex field $\mathsf{E}'$
  together with maps of Shimura data
  \[
    \gxthree \to \gx, \quad \gxthree \to \gxtwo,
  \]
  where the first map induces an isomorphism on derived groups and the second
  map is an ad-isomorphism.
\end{Lem}

\begin{proof}
This is a direct consequence of \cite[Proposition 3.4.2]{lovering2017filtered}\footnote{We are citing an unpublished preprint here, however the proposition we need is simply a convenient summary of the results in the (published) \cite[Section 4.6]{LoveringAutomorphic}.} as we will now explain. Recall that there are natural inclusions $\x \subset \x^{\mathrm{ad}}$ and $\x_2 \subset \x^{\mathrm{ad}}$. By real approximation for $\Gad$, there is an element $g \in \Gad(\Q)$ such that $\left(\operatorname{Ad} g \x \right) \cap \x_2$ is nonempty. Therefore, after replacing $(\gx,\Xi)$ by $(\gx,\Xi \circ \operatorname{Ad} g)$, we may assume that $\x_2 \cap \x$ is nonempty and thus contains a connected component $\x^+$ of $\x^{\mathrm{ad}}$. \smallskip

We now apply the construction of \cite[Proposition 3.4.2]{lovering2017filtered} for the triples $$(\g^{\mathrm{der}}, \mathsf{X}^+, \mathsf{E}),(\g^{\mathrm{der}}, \mathsf{X}^+, \mathsf{E}'), (\g_2^{\mathrm{der}}, \mathsf{X}^+, \mathsf{E}_2), (\g_2^{\mathrm{der}}, \mathsf{X}^+, \mathsf{E}'),$$ which gives us Shimura data
\begin{align*}
    (\mathsf{B}, \mathsf{X}), (\mathsf{B}', \mathsf{X}'), (\mathsf{B}_2, \mathsf{X}_2), (\mathsf{B}_2', \mathsf{X}_2'),
\end{align*}
respectively, with reflex fields $ \mathsf{E}, \mathsf{E}', \mathsf{E}_2, \mathsf{E}'$, respectively. By construction, these come equipped with morphisms of Shimura data
  \[
    (\mathsf{B}, \mathsf{X}) \to \gx, \quad
    (\mathsf{B}_2, \mathsf{X}_2) \to \gxtwo
  \]
that induce isomorphisms of derived groups. By \cite[Proposition 3.4.2.(2)]{lovering2017filtered}, there are morphisms of Shimura data
  \[
     (\mathsf{B}', \mathsf{X}') \to (\mathsf{B}, \mathsf{X}), \quad
     (\mathsf{B}_2', \mathsf{X}_2') \to (\mathsf{B}_2, \mathsf{X}_2)
   \]
that are isomorphisms on derived groups. By \cite[Proposition 3.4.2.(1)]{lovering2017filtered} there is a morphism of Shimura data (induced by $\Xi$)
\begin{align*}
   (\mathsf{B}', \mathsf{X}') \to (\mathsf{B}_2', \mathsf{X}_2')
\end{align*}
which induces a central isogeny on derived groups, it is thus an ad-isomorphism. The lemma is proved by taking $\gxthree=(\mathsf{B}', \mathsf{X}')$ with its corresponding maps to $\gx$ and $\gxtwo$. For the convenience of the reader, we record the following diagram summarizing the argument:
\begin{equation*}
    \begin{tikzcd}
            (\mathsf{B}', \mathsf{X}') \arrow{r} \arrow{d} & (\mathsf{B}_2', \mathsf{X}_2') \arrow{r} & (\mathsf{B}_2, \mathsf{X}_2) \arrow{r} & \gxtwo \\
           (\mathsf{B}, \mathsf{X}) \arrow{rrr} & & &\gx.
    \end{tikzcd}
\end{equation*}
\end{proof}

\begin{proof}[Proof of Theorem \ref{Thm:MainThmIgusa}]
  By the proof of \cite[Lemma~4.6.22]{KisinPappas}, we can find an auxiliary
  Hodge type Shimura datum $(\gx,\Xi)$  with reflex field $\mathsf{E}$ such that
  every prime above $p$ in $\mathsf{E}_2$ splits completely in the extension
  $\mathsf{E}'=\mathsf{E} \cdot \mathsf{E}_2$. Choose $\gxtwo \leftarrow
  \gxthree \rightarrow \gx$ as in \Cref{Lem:ExistenceGoodRoof}. Choose a place
  $w'$ of $\mathsf{E}'$ lifting $v_2$, which induces a place $w$ of $\mathsf{E}$, and
  note that we have an equality of local reflex fields
  $\mathsf{E}'_{w'}=\mathsf{E}_{2,v_2}$.

  The existence of the Igusa stack for $\gx$ is
  \cite[Theorem~I]{DvHKZIgusaStacks} and \cite[Theorem~D]{KimFunctorial} for $?=\circ$, $? = \emptyset,$ respectively. The existence of the Igusa stack for $\gxthree$ now
  follows from \Cref{Thm:AscendingIgusaStacks} and the existence of the Igusa
  stack for $\gxtwo$ is then a consequence of \Cref{Thm:DescendingIgusaStacks}.
\end{proof}

\subsection{A finite level Igusa stack} \label{sub:FiniteLevelIgusa} 

Let $\gx$ be a Shimura datum of abelian type with reflex field $\mathsf{E}$ and let $v$ be a place of $\mathsf{E}$ above a prime number $p$. Let $G=\g \otimes \qp, E=\mathsf{E}_v$ and $\mu$ be as before. Then by \Cref{Thm:MainThmIgusa}, $\gx$ admits an Igusa stack $\igsgx \to \bung$ for $?$ in $\{\circ, \emptyset\}$ at the place $v$. The goal of this section is to construct a finite-level Igusa stack $\igsfingx$ which fits in a Cartesian diagram with $\shdpinf$. This is more difficult in the case of abelian type Shimura data than in the case of Hodge type Shimura data, because $K^p$ does not generally act freely on $\shdinf$ and so we cannot simply take the stack quotient. 

\begin{Thm} \label{Thm:FiniteLevelIgusa}
  Let $K^p \subseteq \gafp$ be a neat compact open subgroup. For $? \in \{\circ,
  \emptyset\}$, there is a v-stack $\igsfingx$ fitting into the $2$-commutative
  diagram
  \[ \begin{tikzcd}[row sep=small] \label{Eq:CartesianSquareFiniteLevel}
    \shdinf \arrow{r} \arrow{d} & \shdpinf \arrow{r} \arrow{d} & \operatorname{Gr}_{G, \mu^{-1}} \arrow{d} \\
    \igsgx \arrow{r} & \igsfingx \arrow{r} & \operatorname{Bun}_{G},
  \end{tikzcd} \]
  where both squares are $2$-Cartesian. Moreover, the absolute Frobenius acts trivially
  on $\igsfingx$ with the trivialization being compatible with that of $\bung$.
\end{Thm}

\subsubsection{} To construct the Igusa stack at finite level, we will use the
following lemma. Recall from Theorem \ref{Thm:IgusaStackKim} that giving
an Igusa stack $\igs^{?}\gx$ is equivalent to giving a uniformization map
\[
  \Theta \colon \shdinf \times_{\bung} \grgmu \to \shdinf
\]
in the sense of Definition \ref{Def:UniversalUniformization}.

\begin{Lem} \label{Lem:GroupoidFiniteLevel}
  For $? \in \{\circ, \emptyset\}$, there is a unique map $\Theta_{K^p}$ filling
  in the diagram
  \[ \begin{tikzcd}[row sep=small]
    \shdinf \times_{\bung} \grgmu \arrow{r}{\Theta} \arrow{d} & \shdinf \arrow{d} \\
    \shdpinf  \times_{\bung} \grgmu  \arrow[r, dashed, "\Theta_{K^p}"] & \shdpinf.
  \end{tikzcd} \]
  Moreover, the map $\Theta_{K^p}$ intertwines the action of $\phi \times \id$
  on the source with the action of $\id$. Furthermore, the following diagrams
  commute.
  \[ \begin{tikzcd}
    \shdpinf \arrow[r, equals] \arrow{d}[']{(\id, \pi_{\mathrm{HT}, K^p})} & \shdpinf \arrow{d}{\pi_{\mathrm{HT}, K^p}} \\
    \shdpinf  \times_{\bung} \grgmu \arrow{r}{\operatorname{pr}_{2}} \arrow{ur}{\Theta_{K^p}} & \grgmu
  \end{tikzcd} \]
  \[ \begin{tikzcd}
    \shdpinf  \times_{\bung} \grgmu \times_{\bung} \grgmu \arrow{r}{\Theta_{K^p} \times \id} \arrow{d}{\operatorname{pr}_{1,3}} & \shdpinf  \times_{\bung} \grgmu \arrow{d}{\Theta_{K^p}} \\
    \shdpinf  \times_{\bung} \grgmu \arrow{r}{\Theta_{K^p}} & \shdpinf.
  \end{tikzcd} \]
\end{Lem}

\begin{proof}
  The group $\ul{\gaf} \times \ul{G(\qp)}$ acts on $\shdinf \times_{\bung}
  \grgmu$ and the map $\Theta$ is equivariant for the action of $\ul{\gaf}
  \times \ul{G(\qp)}$ along $\operatorname{pr}_{1,3} \colon \ul{\gaf} \times
  \ul{G(\qp)} = \ul{\gafp} \times \ul{G(\qp)} \times \ul{G(\qp)} \to \ul{\gafp}
  \times \ul{G(\qp)}$. By \cite[Lemma~6.2]{KimFunctorial}, the closed subgroup
  \[
    \ul{(\id \times \Delta)(\zgqbar)} \subseteq \ul{\gafp} \times \ul{G(\qp)}
    \times \ul{G(\qp)}
  \]
  acts trivially on $\shdinf \times_{\bung} \grgmu$. Since $\shdinf \to
  \shdpinf$ is a torsor for the profinite group (see \cite[Section 1.5.8]{KisinShinZhu})
  \[
    \frac{K^p \cdot \zgqbar}{\zgqbar} \subseteq \frac{\gaf}{\zgqbar},
  \]
  it follows that 
  \[
    \shdinf \times_{\bung} \grgmu \to \shdpinf \times_{\bung} \grgmu
  \]
  is a torsor for the profinite group 
  \[
    \frac{(K^p \times \{1\} \times \{1\}) \cdot (\id \times \Delta)(\zgqbar)}{
      (\id \times \Delta)(\zgqbar)} \subseteq \frac{\gafp \times G(\qp) \times
    G(\qp)}{(\id \times \Delta)(\zgqbar)}.
  \]
  This maps isomorphically to $(K^p \cdot \zgqbar) / \zgqbar$ under
  $\operatorname{pr}_{1,3}$, showing that $\Theta$ descends to a map
  $\Theta_{K^p}$. The statement about the $\phi$-action follows formally from
  the statement about the $\phi$-action for $\Theta$, which holds by
  \cite[Proposition~5.12]{KimFunctorial}.

  The two diagrams that we have to show are commutative, sit in commutative
  diagrams with the same diagrams without the subscripts $K^p$ as in
  \cite[Definition~5.10.(3)]{KimFunctorial}; these latter diagrams moreover
  commute by \cite[Proposition~5.12]{KimFunctorial}. Since the maps from the
  objects without subscript $K^p$ to the objects with subscript $K^p$ are
  v-covers, it follows that the diagrams at level $K^p$ must also commute.
\end{proof}

As a corollary, we have the following result (which, together with Theorem \ref{Thm:MainThmIgusa} and Proposition \ref{Prop:DualizingSheaf}, proves Theorem \ref{Thm:IntroIgusaGeneric} from the introduction).

\begin{Cor} \label{Cor:FiniteLevelIgusa1}
  For $? \in \{\circ, \emptyset\}$, there is a v-stack $\igsfingx$ fitting into
  the following $2$-commutative diagram
  \[ \begin{tikzcd}[row sep=small]
    \shdinf \arrow{r} \arrow{d} & \shdpinf \arrow{r}{\pi_{\mathrm{HT}}^?} \arrow{d} & \operatorname{Gr}_{G, \mu^{-1}} \arrow{d} \\
    \igsgx \arrow{r} & \igsfingx \arrow{r}{\overline{\pi}_{\mathrm{HT}}^{?}} & \operatorname{Bun}_{G},
  \end{tikzcd} \]
  where both squares are $2$-Cartesian. Moreover, the absolute Frobenius acts
  trivially on $\igsfingx$ where there is a canonical trivialization that is
  compatible with that of $\bung$.
\end{Cor}

\begin{proof}
This follows from \Cref{Lem:GroupoidFiniteLevel} together with arguments in \cite[Proposition~5.12]{KimFunctorial}.
\end{proof}
}

}

{\section{Cohomology of Shimura varieties} \label{Sec:Cohomology}

In this section we generalize many of the results of \cite[Section 8, Section 9]{DvHKZIgusaStacks} to abelian type Shimura varieties. We compute the dualizing sheaf of the Igusa stack in Section \ref{sub:DualizingSheaf}, define the Igusa sheaf and study its (conjectural) properties in Section \ref{Sub:SheafConjecture}. In Section \ref{Sub:WeilCohShim}, we compute the (compactly supported) cohomology of the Shimura variety in terms of a Hecke operator acting on the Igusa sheaf. In Section \ref{sub:ProductFormula} we prove a version of the Mantovan product formula, in Section \ref{sub:WeakCompatibility} we prove a compatibility statement of the cohomology of Shimura varieties with the Fargues--Scholze local Langlands correspondence, in Section \ref{Sub:ES} we prove the Eichler--Shimura relation and in Section \ref{sub:Plectic} we prove the split local plectic conjecture.

\subsection{The dualizing sheaf of the Igusa stack} \label{sub:DualizingSheaf} Let $\gx$ be a Shimura datum of abelian type with reflex field $\mathsf{E}$ and let $v$ be a place of $\mathsf{E}$ above a prime number $p$. Let $G=\g \otimes \qp, E=\mathsf{E}_v$ and $\mu$ be as before and let $K^p \subset \gafp$ be a neat compact open subgroup. For $?$ in $\{\circ, \emptyset\}$, we study the Igusa stack $\igsfingx$ of \Cref{Cor:FiniteLevelIgusa1}. The goal of this section is to prove that $\igsfingx$ is an $\ell$-cohomologically smooth v-stack of $\ell$-cohomological dimension zero with trivial dualizing sheaf, generalizing \cite[Theorem~8.3.1]{DvHKZIgusaStacks}. This is more difficult in the case of abelian type Shimura data than in the case of Hodge type Shimura data, because $K^p$ does not generally act freely on $\shdinf$. 

\begin{Lem} \label{Lem:PiHTBarfdcs}
For $? \in \{\emptyset, \circ\}$, the morphism $\overline{\pi}_{\mathrm{HT}}^? \colon \igs^?_{K^p}\gx \to \bungmu$
  is compactifiable, representable in locally spatial
  diamonds, and has $\dimtrg \overline{\pi}_{\mathrm{HT}} < \infty$. If $?=\emptyset$, then it is moreover partially proper. 
\end{Lem}

\begin{proof}
  The morphism $\pi_\mathrm{HT}^?:\shdpinf \to \operatorname{Gr}_{G, \mu^{-1}}$ is separated, compactifiable, representable
  in locally spatial diamonds, and has $\dimtrg \pi_\mathrm{HT} < \infty$ since
  it is a map between pro-(finite-\'{e}tale) covers of rigid analytic varieties. We also
  note that $\grgmu \to \bungmu$ is surjective in the pro-\'{e}tale topology, see
  \cite[Corollary~6.4.2]{DvHKZIgusaStacks}. Using \cite[Proposition~10.11.(ii),
  Proposition~13.4.(iv)]{EtCohDiam} we see that $\overline{\pi}_\mathrm{HT}^{?}$ is
  separated and representable in locally spatial diamonds. To check the finiteness of $\dimtrg
  \overline{\pi}_{\mathrm{HT}}^?$ it is enough to check on geometric points on
  $\bungmu$ in the image of $\igs^?_{K^p}\gx$. By surjectivity of
  $\grgmu \to \bungmu$ in the pro-\'{e}tale topology, such a point lifts to a point of $\grgmu$, and therefore
  it follows from $\dimtrg \pi_{\mathrm{HT}}^? < \infty$. 

  Now take $?=\emptyset$. The canonical
  compactification of $\overline{\pi}_\mathrm{HT}$ base changes to the canonical
  compactification of $\pi_\mathrm{HT}$ by \cite[Proposition~18.6]{EtCohDiam},
  hence partial properness of $\pi_\mathrm{HT}$ implies partial properness of
  $\overline{\pi}_\mathrm{HT}$. 
\end{proof}

\subsubsection{} Next, we prove that the dualizing sheaf of the Igusa stack is constant. To perform this computation, we work with the nuclear sheaf theory developed in \cite{MannNuclear}.  In what follows, we let $\Lambda$ be a $\zl$-algebra in which $\ell$ is nilpotent. In this case, \cite[Proposition~3.20]{MannNuclear} states that the nuclear sheaf theory of
\cite{MannNuclear} agrees with the \'{e}tale sheaf theory of \cite{EtCohDiam},
where the six functors also agree by construction. Note that the dualizing sheaf indeed makes sense because both $\overline{\pi}_\mathrm{HT}^{?}$ and $\bungmu \to \spd \fp$ admit $!$-functors (they are $\ell$-fine in the sense of \cite[Definition 5.8]{MannNuclear}) see
\Cref{Lem:DualizingBunG}.

\begin{Def}
  Let $G$ be a locally profinite group that is unimodular, and let $X$ be a
  v-stack. Write $\Gamma_G = \im(\delta_G \colon \operatorname{Out}(G) \to
  \mathbb{Q}_{>0}^\times)$, see \Cref{Sec:UnimodularOuter}. We say that a
  $\ul{G}$-gerbe $Y \to X$ for the v-topology is \textit{unimodular} when the
  corresponding $\operatorname{Out}(\ul{G})$-torsor (see for example
  \cite[Section~IV.2.2]{GiraudGerbes} or \cite[Section~2.10]{BreenGerbes})
  \[
    \operatorname{Isom}(Y, X \times [\ast/\ul{G}]) \to X
  \]
  has the property that its pushout along $\delta_G \colon
  \operatorname{Out}(\ul{G}) \twoheadrightarrow \ul{\Gamma_G}$ is a trivial
  $\ul{\Gamma_G}$-torsor over $X$.
\end{Def}

\begin{Example} \label{Exa:UnimodularGerbes}
  Here are examples of unimodular and non-unimodular gerbes.
  \begin{enumerate}
    \item If $G$ is compact, then $\Gamma_G = 1$ and hence all $\ul{G}$-gerbes
      are unimodular.
    \item If $k$ is a discrete perfect field of characteristic $p$, and a
      $\ul{G}$-gerbe $Y \to \spd k$ splits over $\spd \bar{k}$, it is
      unimodular. This is because the $\ul{\Gamma_G}$-torsor over $k$ that
      splits over $\bar{k}$ corresponds to an element of
      \[
        H^1(\bar{k}/k, \Gamma_G) = \Hom_\mathrm{cts}(\gal(\bar{k}/k), \Gamma_G)
        = 0,
      \]
      where the last equality holds because $\Gamma_G \subseteq
      \mathbb{Q}_{>0}^\times$ is torsion free and has the discrete topology.
    \item The $\ul{\qp}$-gerbe over $[\spd \fpbar/\phi^\mathbb{Z}]$
      constructed by descending $\spd \fpbar \times [\ast / \ul{\qp}]$
      along the multiplication map $p \colon [\ast / \ul{\qp}] \to [\ast /
      \ul{\qp}]$ is not unimodular, see \Cref{Ex:ModularCharQp}.
    \item Let $1 \to K \to G \to H \to 1$ be a short exact sequence of locally
      profinite groups, where the conjugation action of $G$ on
      $\operatorname{Haar}(K, \mathbb{Q})$ is trivial (i.e., $\delta_G =
      \delta_H$ as characters of $G$). Then the $\ul{K}$-gerbe $[\ast/\ul{G}]
      \to [\ast/\ul{H}]$ is unimodular. This is because the
      $\operatorname{Out}(\ul{K})$-torsor is obtained by pushing out the
      $\ul{H}$-torsor $\ast \to [\ast/\ul{H}]$ along $\ul{H} \to
      \operatorname{Out}(\ul{K})$, where the composition
      \[
        H \to \operatorname{Out}(K) \xrightarrow{\delta_K}
        \mathbb{Q}_{>0}^\times
      \]
      is trivial.
  \end{enumerate}
\end{Example}

\begin{Lem} \label{Lem:GerbeDualizing}
  Let $G$ be a locally profinite group that is locally pro-prime-to-$\ell$ and
  unimodular. Let $\pi \colon Y \to X$ be a unimodular $\ul{G}$-gerbe for the
  v-topology, and assume that $\pi$ is $\ell$-fine. Then $\pi$ is
  $\ell$-cohomologically smooth, and moreover there exists a (noncanonical)
  isomorphism $\pi^! \Lambda \cong \Lambda[0]$.
\end{Lem} 

\begin{Rem}
  In general, even if $Y$ is not unimodular, $\pi^! \Lambda$ is canonically
  identified with the constant sheaf $\operatorname{Haar}(G, \Lambda)^\ast$
  twisted by the $\ul{\Gamma_G}$-torsor corresponding to the gerbe.
\end{Rem}

\begin{proof}[Proof of Lemma \ref{Lem:GerbeDualizing}]
  \def\dualhaar{\operatorname{Haar}(G, \Lambda)^\ast}
  We choose a trivialization of the $\ul{\Gamma_G}$-torsor
  \[
    P = \operatorname{Isom}(Y, X \times [\ast/\ul{G}])
    \times^{\operatorname{Out}(\ul{G})} \ul{\Gamma_G}.
  \]
  Let $\tilde{X} \to X$ be a v-cover by a perfectoid space $\tilde{X}$
  that splits $\pi$, so that there is an isomorphism
  \[
    f \colon \tilde{X} \times_X Y \cong \tilde{X} \times [\ast/\ul{G}]
  \]
  over $\tilde{X}$. Here, we note that this defines a trivialization of
  $P_{\tilde{X}}$, which may differ with our chosen trivialization of $P$ by an
  element of $\operatorname{Fun}_\mathrm{cts}(\lvert \tilde{X} \rvert,
  \Gamma_G)$. Upon choosing a set-theoretic section of $\Aut(G)
  \twoheadrightarrow \Gamma_G$, we may modify $f$ so that the two
  trivializations of $P$ and $P_{\tilde{X}}$ are compatible. By \cite[Theorem~10.13, Example~10.11.(b)]{MannNuclear}, the map $[\ast/\ul{G}]
  \to \ast$ is $\ell$-cohomologically smooth, and hence so is $\tilde{X} \times_X Y
  \to \tilde{X}$. Because $\pi$ is assumed to be $\ell$-fine, it is also
  $\ell$-cohomologically smooth by \cite[Lemma~8.7.(ii)]{MannNuclear}.

  Once $\pi$ is $\ell$-cohomologically smooth, we may compute the dualizing
  sheaf $\pi^! \Lambda$ v-locally using
  \cite[Proposition~8.5.(ii).(a)]{MannNuclear}. As discussed in
  \cite[Example~4.2.4]{HansenKalethaWeinstein}, the dualizing sheaf of
  $[\ast/\ul{G}] \to \ast$ is canonically identified with $\dualhaar$. Here, we
  note that $\dualhaar \in \mathcal{D}([\ast/\ul{G}], \Lambda)$ is a constant
  sheaf because $G$ is unimodular. Write $\tilde{X}^{(2)} = \tilde{X} \times_X
  \tilde{X}$ and
  \[
    \alpha \colon \tilde{X}^{(2)} \times [\ast/\ul{G}] \cong \tilde{X} \times_X
    (\tilde{X} \times_X Y) \xrightarrow{\mathrm{swap}_{23}} (\tilde{X} \times_X
    Y) \times_X \tilde{X} \cong \tilde{X}^{(2)} \times [\ast/\ul{G}]
  \]
  for the descent datum so that
  \[ \begin{tikzcd}[row sep=0em]
    \tilde{X}^{(2)} \times \lbrack \ast/\ul{G} \rbrack \arrow{rr}{\alpha}
    \arrow{dr}[']{\mathrm{pr}_1} & & \tilde{X}^{(2)} \times \lbrack \ast/\ul{G}
    \rbrack \arrow{ld}{\mathrm{pr}_1} \\ & \tilde{X}^{(2)}
  \end{tikzcd} \]
  naturally commutes. It remains to verify that the induced map
  \[
    \dualhaar \cong \mathrm{pr}_1^! \Lambda = \alpha^! \mathrm{pr}_1^! \Lambda
    \cong \alpha^! \dualhaar \cong \dualhaar
  \]
  is the identity map in $\mathcal{D}(\tilde{X} \times [\ast/\ul{G}], \Lambda)$.

  Because $\dualhaar$ is a free $\Lambda$-module of rank $1$, such an
  automorphism corresponds to an element of
  $\operatorname{Fun}_\mathrm{cts}(\lvert \tilde{X}^{(2)} \times [\ast/\ul{G}]
  \rvert, \Lambda^\times) = \operatorname{Fun}_\mathrm{cts}(\lvert
  \tilde{X}^{(2)} \rvert, \Lambda^\times)$. Therefore we may compute this
  element after pulling back along each geometric point $x = \spa(C, C^+) \to
  \tilde{X}^{(2)}$. Write $q \colon x \to x \times [\ast/\ul{G}]$ for the
  natural map. Because every $\ul{G}$-torsor over $x$ is trivial, there exists a
  2-morphism filling in
  \[ \begin{tikzcd}[row sep=0em]
    & x \arrow{ld}[']{q} \arrow{rd}{q} \\ x \times \lbrack \ast/\ul{G} \rbrack
    \arrow{rr}[']{\alpha_x} & & x \times \lbrack \ast/\ul{G} \rbrack.
  \end{tikzcd} \]
  By considering the automorphism groups, we obtain an automorphism $\beta_x
  \colon G \to G$ of topological groups together with an isomorphism $\alpha_x
  \cong [\ast/\beta_x]$. Then the automorphism
  \[
    \dualhaar \cong \mathrm{pr}_1^! \Lambda = \beta_x^! \mathrm{pr}_1^! \Lambda
    \cong \beta_x^! \dualhaar \cong \dualhaar
  \]
  is the one induced by $\beta_x$, and hence is multiplication by
  $\delta_G(\beta_x)$. Because we have chosen the isomorphism $f \colon
  \tilde{X} \times_X Y \cong \tilde{X} \times [\ast/\ul{G}]$ to respect the
  trivialization of $P$, the automorphism $\beta_x$ is contained in the kernel
  of
  \[
    \Aut(G) \twoheadrightarrow \operatorname{Out}(G) \xrightarrow{\delta_G}
    \Gamma_G,
  \]
  and therefore $\delta_G(\beta_x) = 1$.
\end{proof}

\begin{Lem} \label{Lem:DualizingShimura}
  The map $\pi \colon [\shdpinf / \ul{G(\qp)}] \to \spd E$ is
  $\ell$-cohomologically smooth, and moreover the dualizing sheaf is isomorphic
  to $\Lambda(d)[2d]$, where $d = \langle 2\rho, \mu\rangle = \dim \mathsf{X}$.
\end{Lem}

\begin{proof}
  \def\shdgqp{[\shdinf/\ul{B}]}
  We first verify that $\pi \colon [\shdpinf / \ul{G(\qp)}] \to \spd E$ is
  $\ell$-fine. By \cite[Lemma~10.8]{MannNuclear}, the map $\ast \to
  [\ast/\ul{G(\qp)}]$ is fdcs and admits universal
  $\ell$-codescent.\footnote{While this is stated only for profinite groups, we
  can simply choose a pro-$p$ compact open subgroup $K_p \subseteq \ul{G(\qp)}$
  and observe that $\ast \to [\ast/\ul{K_p}]$ and $[\ast/\ul{K_p}] \to
  [\ast/\ul{G(\qp)}]$ are both fdcs and admit universal $\ell$-codescent,
  because the second map is \'{e}tale. Then \cite[Lemma~3.3.2.(4)]{LiuZheng}
  shows that their composition also admits universal $\ell$-codescent.} Hence
  its base change $\shdpinf \to [\shdpinf / \ul{G(\qp)}]$ is also fdcs and
  admits universal $\ell$-codescent. On the other hand, the map $\shdpinf \to
  \spd E$ is a limit of smooth rigid analytic varieties along finite \'{e}tale
  maps, hence is fdcs. It now follows from \cite[Definition~5.8]{MannNuclear}
  that $\pi$ is $\ell$-fine.

  We now compute the dualizing sheaf. Choose a sufficiently small open compact
  subgroup $K_p \subseteq G(\qp)$ and consider the locally profinite groups
  \[
    A = G(\qp) \times \frac{K^p \zgqbar}{\zgqbar} \twoheadrightarrow B =
    \frac{(K^p G(\qp)) \zgqbar}{\zgqbar} \supseteq \frac{(K^p K_p)
    \zgqbar}{\zgqbar},
  \]
  where the first map is multiplication (noting that the image of the two
  factors commute with each other), and the second map is an inclusion of an
  open subgroup. These groups naturally act on $\shdinf$, and taking the
  respective quotients recovers
  \[
    [\shdpinf / \ul{G(\qp)}] \xrightarrow{f} \shdgqp \leftarrow \shdfin.
  \]
  Note that $\shdfin \to \shdgqp$ is an \'{e}tale quotient because it
  corresponds to an open subgroup, hence $\rho \colon \shdgqp \to \spd E$ is
  $\ell$-fine. Applying \cite[Lemma~5.10.(iv)]{MannNuclear}, we see that $f$ is
  $\ell$-fine as well.

  We note that $f$ is the base change of
  \[
    \bar{f} \colon [\ast/\ul{A}] \to [\ast/\ul{B}].
  \]
  The group $A$ is unimodular because it is the product of two unimodular
  groups. The group $B$ is unimodular because it is an open subgroup of 
  $\gaf/\zgqbar$, which is unimodular as it is a central quotient of the
  unimodular group $\gaf$. It follows from \Cref{Exa:UnimodularGerbes}.(4) that
  $\bar{f}$ is a unimodular gerbe for $\ker(A \to B)$, and hence its base change
  $f$ is also a unimodular gerbe. The kernel of $A \to B$ is a closed subgroup
  of $G(\qp)$, hence locally pro-$p$, and therefore by \Cref{Lem:GerbeDualizing}
  we obtain an isomorphism
  \[
    \pi^! \Lambda = f^! \rho^! \Lambda \cong f^\ast \rho^! \Lambda
    \otimes_\Lambda f^! \Lambda \cong f^\ast \rho^! \Lambda.
  \]
  On the other hand, by \cite[Remark~8.3.5]{DvHKZIgusaStacks} the dualizing
  sheaf $\rho^! \Lambda$ is isomorphic to $\Lambda(d)[2d]$ because $\shdgqp$ is
  an \'{e}tale quotient of a smooth rigid analytic variety.

  Finally, to show that $\pi$ is $\ell$-cohomologically smooth, we note that
  $\shdfin$ is $\ell$-cohomologically smooth over $\spd E$. Because $\shdfin \to
  \shdgqp$ is \'{e}tale and $f$ is $\ell$-cohomologically smooth by
  \Cref{Lem:GerbeDualizing}, we may use \cite[Lemma~8.7.(ii)]{MannNuclear} to
  conclude that $[\shdpinf/\ul{G(\qp)}]$ is $\ell$-cohomologically smooth over
  $\spd E$.
\end{proof}

\begin{Lem} \label{Lem:DualizingBunG}
  The map $\pi \colon \bung \to \spd \fp$ is $\ell$-cohomologically smooth with
  dualizing sheaf $\pi^! \Lambda \cong \Lambda[0]$.
\end{Lem}

\begin{proof}
  The $\ell$-fineness appears in \cite[Theorem~IV.1.19]{FarguesScholze}; while
  the theorem is stated over $\fpbar$, the proof also works over $\fp$. For
  $\ell$-cohomological smoothness, we use \cite[Lemma~8.7.(ii)]{MannNuclear} to
  reduce to the statement over $\fpbar$. The dualizing complex over $\fpbar$ was
  computed in \cite[Proposition~3.18]{HamannImai}, and it remains to descend it
  to $\fp$. Because the union of $\bung^{[b]} \subseteq \bung$ over all basic $b
  \in B(G)$ intersects all connected components of $\bung$, it suffices to show
  that the dualizing complex of $\bung^{[b]} \to \spd \fp$ is trivial. Because
  this map is a $\ul{G_b(\qp)}$-gerbe, it now follows from
  \Cref{Exa:UnimodularGerbes}.(2) and \Cref{Lem:GerbeDualizing}.
\end{proof}

\begin{Prop} \label{Prop:DualizingSheaf}
  The map $\igsfingx \to \spd \fp$ is $\ell$-cohomologically smooth with
  dualizing sheaf isomorphic to $\Lambda[0]$.
\end{Prop}

\begin{proof}
  We first compute the dualizing sheaf. The second half of the proof of
  \cite[Theorem~8.3.1]{DvHKZIgusaStacks} works almost verbatim. Consider the
  Cartesian diagram
  \[ \begin{tikzcd}[row sep=small]
    \lbrack \shdpinf / (\ul{G(\qp)} \times \phi^\mathbb{Z}) \rbrack
    \arrow{r}{\tilde{f}} \arrow{d}{\tilde{g}} & \lbrack \grgmu / (\ul{G(\qp)}
    \times \phi^\mathbb{Z}) \rbrack \arrow{d}{g} \\ \igsfingx \arrow{r}{f} &
    \bungmu.
  \end{tikzcd} \]
  By \Cref{Lem:DualizingBunG} it suffices to show that $f^! \Lambda \cong
  \Lambda[0]$. The map $g$ is $\ell$-cohomologically smooth by combining
  \cite[Theorem~IV.1.19]{FarguesScholze} with \cite[Lemma~8.7.(i)]{MannNuclear}.
  Hence by \cite[Proposition~8.5.(ii)]{MannNuclear} we have
  \[
    \tilde{g}^\ast f^! \Lambda \cong \tilde{f}^! g^\ast \Lambda =
    \tilde{f}^! \Lambda \cong \Lambda[0]
  \]
  where the last isomorphism follows from \cite[Lemma~8.3.4]{DvHKZIgusaStacks},
  \Cref{Lem:DualizingShimura}. This shows that
  $f^! \Lambda$ is an invertible sheaf in degree zero, and by adjunction we
  produce a map $f^! \Lambda = \tau_{\le 0} f^! \Lambda \to \tau_{\le 0}
  R\tilde{g}_\ast \Lambda \cong \Lambda[0]$ that we check is an isomorphism upon
  pulling back along $\tilde{g}$. Here, the isomorphism $\tau_{\le 0}
  R\tilde{g}_\ast \Lambda \cong \Lambda[0]$ comes from the fact that the fiber
  of $\tilde{g}$ over a $(C, C^+)$-point, which is a fiber of $g$, is a relative
  admissible locus over $\ffcurve(C, C^+)_E$ and hence connected by
  \cite[Theorem~1.1]{GleasonLourenco}.

  To show that $\pi \colon \igsfingx \to \spd \fp$ is $\ell$-cohomologically
  smooth, we simply note that $[\shdfin/\ul{G(\qp)}]$ is $\ell$-cohomologically
  smooth over $\spd \fp$ by \Cref{Lem:DualizingShimura} and
  \cite[Proposition~24.5]{EtCohDiam}, and also $\ell$-cohomologically smooth
  over $\igsfingx$. Because $\pi$ is already $\ell$-fine by
  \Cref{Lem:PiHTBarfdcs}, we may use \cite[Lemma~8.7.(ii)]{MannNuclear} to
  conclude that $\pi$ is $\ell$-cohomologically smooth.
\end{proof}

\subsection{Conjectures about the Igusa sheaf} \label{Sub:SheafConjecture} In this section, we define the Igusa sheaf and state some conjectures.

\subsubsection{} Let $\Lambda$ be a $\zl$ algebra in which $\ell$ is nilpotent. Let $W$ be a continuous representation of $K^p$ on a finite free $\Lambda$-module that factors through the image of $K^p$ under $\gaf \to \g^c(\af)$. As in \cite[Section 1.5.8]{KisinShinZhu}, we can use $W$ to define a local system $\mathbb{W}$ on $\mathbf{Sh}_{K_pK^p}\gx$ for each compact open subgroup $K_p$ of $G(\qp)$; these are compatible with one another as the level changes. We will refer to local systems defined in this way from such a $W$ as \emph{standard $\Lambda$-local systems}. 

\subsubsection{} By Corollary \ref{Cor:FiniteLevelIgusa1}, the covering group of $\igs\gx \to \igs_{K^p}\gx$ is the same as the covering group of $\mathbf{Sh}\gx^{\diamondsuit} \to \mathbf{Sh}_{K^p}\gx^{\diamondsuit}$. Thus we can descend a standard $\Lambda$-local system $\mathbb{W}$ to $\igs_{K^p}$, and we will refer to both the local system and its descent as $\mathbb{W}$. For such a $\mathbb{W}$ we define (using Lemma \ref{Lem:PiHTBarfdcs})
\begin{alignat*}{2}
    \mathcal{F}_{c, \mathbb{W}}&:=\overline{\pi}_{\mathrm{HT}, !} \mathbb{W} \quad &\mathcal{F}_{\mathbb{W}}:=R \overline{\pi}_{\mathrm{HT}, \ast} \mathbb{W} \\
    \mathcal{F}_{c,\mathbb{W}}^{\circ}&:=\overline{\pi}^{\circ}_{\mathrm{HT}, !} \mathbb{W} \quad &\mathcal{F}^{\circ}_{\mathbb{W}}:=R \overline{\pi}^{\circ}_{\mathrm{HT}, \ast} \mathbb{W}.
\end{alignat*}

\begin{Rem} \label{Rem:Proper}
The sheaf denoted by $\mathcal{F}$ in \cite[Section 8.4.9]{DvHKZIgusaStacks} is the sheaf $\mathcal{F}^{\circ}_{\Lambda}$ defined above. Note that if $\mathbf{Sh}_K\gx$ is proper, then $\igs^{\circ}\gx=\igs\gx$ and moreover $\pi_{\mathrm{HT}}$, hence $\overline{\pi}_{\mathrm{HT}}$, is proper. This shows that $\mathcal{F}_{\mathbb{W}}^{\circ}=\mathcal{F}_{\mathbb{W}}=\mathcal{F}_{c,\mathbb{W}}=\mathcal{F}_{c,\mathbb{W}}^{\circ}$.
\end{Rem}

We have the following conjecture. Let $k$ be an algebraic closure of $\fp$. 

\begin{Conj} \label{Conj:ULAPerverse}
Let $\mathbb{W}$ be a standard $\Lambda$-local system.
    \begin{enumerate}[$($i\,$)$]
        \item For $? \in \{\emptyset, \circ\}$, the sheaves $\mathcal{F}_{c,\mathbb{W}}^{?}, \mathcal{F}_{\mathbb{W}}^?$ are universally locally acyclic with respect to $\bungmuk \to \spd k$.
        \item The sheaf $\mathcal{F}_{c, \mathbb{W}}$ lies in ${}^pD^{\le 0}(\bungmuk,\Lambda)$ for the perverse t-structure on $\mathcal{D}(\bungmuk, \Lambda)$ of \cite[Proposition~8.1.5]{DvHKZIgusaStacks}.
    \end{enumerate}
\end{Conj}

\begin{Rem}
Assume that $\mathbb{W}=\Lambda$. For many PEL type Shimura varieties, both parts of \Cref{Conj:ULAPerverse} follow from \cite[Proposition~3.7]{Hamann-Lee} using Lemma \ref{Lem:CohomIgusaVar}. If $\gx$ is a Shimura variety of Hodge type, then \Cref{Conj:ULAPerverse}(i) for $?=\circ$ is \cite[Corollary~8.5.4]{DvHKZIgusaStacks}. If $\mathbf{Sh}_K\gx$ is moreover compact, then \Cref{Conj:ULAPerverse}(ii) is proved under mild assumptions in \cite[Theorem~8.6.3]{DvHKZIgusaStacks}. 
\end{Rem}

\begin{Prop} \label{Prop:ULA}
  If $\mathbf{Sh}_K\gx$ is proper, then \Cref{Conj:ULAPerverse}$($i$)$ holds.
\end{Prop}

\begin{proof}
  We note that universal locally acyclicity is equivalent to dualizability in
  the sense of \cite[Definition~7.5]{MannNuclear}, compare
  \cite[Proposition~IV.2.19, Definition~IV.2.22,
  Definition~IV.2.31]{FarguesScholze} with \cite[Proposition~7.7, Corollary~7.8,
  Proposition~8.6]{MannNuclear}. By \cite[Proposition~9.10]{MannNuclear} applied to
  \[
    \igsk \xrightarrow{\bar{\pi}_\mathrm{HT}} \bungmuk \to \spd k,
  \]
  it suffices to check that $\bar{\pi}_\mathrm{HT}$ is cohomologically proper and that $\mathbb{W}$ is universally locally acyclic with
  respect to $\igsk \to \spd k$. The properness follows from the properness of the Shimura variety, see
  \Cref{Rem:Proper}. The sheaf $\mathbb{W}$ is universally locally acyclic with
  respect to $\igsk \to \spd k$ because the map is $\ell$-cohomologically smooth
  by \Cref{Prop:DualizingSheaf}, and $\mathbb{W}$ is locally constant with
  finite projective stalks; for example, in
  \cite[Proposition~7.7.(ii)]{MannNuclear} both sides are identified with
  $\pi_1^\ast \mathbb{W}^\vee \otimes \pi_2^\ast \mathbb{W}$ as the dualizing
  sheaf is trivial.
\end{proof}

\subsubsection{Igusa varieties and the Igusa sheaves} For any element $b \in G(\qpbr)$ with corresponding class $[b] \in B(G,\mu^{-1})$, we denote by $i_b$ the locally closed immersion $i_b \colon \bungk^{[b]} \hookrightarrow \bungk$. Consider also the Newton stratum corresponding to $[b]$ in the Igusa stack $\igs_{K^p}\gx$ with finite level away from $p$ 
\begin{equation*}
    \igs_{K^p}^{[b]}\gx := \igs_{K^p} \times_{\bung} \bung^{[b]},
\end{equation*}
with $\overline{\pi}_b:\igs_{K^p}^{[b]}\gx \to \bung^{[b]}$ the structure map. It follows from Lemma \ref{Lem:PiHTBarfdcs} that $\overline{\pi}_b$ is compactifiable, representable in locally spatial diamonds, and $\dimtrg \overline{\pi}_b < \infty$. For $b:\spd \fpbar \to \bung^{[b]}$ we define the \emph{v-sheaf Igusa variety}\footnote{If $\mathbf{Sh}_K\gx$ is not compact, then this conflicts with the notation in \cite[Section 4.4]{DvHKZIgusaStacks}. The v-sheaf Igusa varieties there are (the inverse limit over $K^p$) of $\operatorname{Ig}^{b, \circ,\mathrm{v}}_{K^p}\gx:=\igs^{\circ}_{K^p} \times_{\bung} \spd \fpbar$} $\operatorname{Ig}^{b,\mathrm{v}}_{K^p}\gx$ as the fiber product \[ \igs_{K^p}^{[b]}\gx \times_{\bung^{[b]}} \spd \fpbar.\] If $x:\spa C \to \spd \fpbar \to \bung^{[b]}$ is a morphism, then we write $\operatorname{Ig}^{b,\mathrm{v}}_{K^p,C}$ for the base-change of the v-sheaf Igusa variety to $\spa C$. By pro-\'etale surjectivity of $\grgmu \to \bungmu$, the map $x$ lifts to a $\spa C$-point of $\grgmu$ and identifies $\operatorname{Ig}^{b,\mathrm{v}}_{K^p,C}$ with the fiber of $\mathbf{Sh}_{K^p}\gx \to \grgmu$. Let $\mathbb{W}$ be a standard $\Lambda$-local system. 
\begin{Lem} \label{Lem:CohomIgusaVar}
In $\mathcal{D}(G_b(\qp), \Lambda)$, we have isomorphisms
	\begin{equation*}
		i_b^\ast \mathcal{F}_{c,\mathbb{W}} \simeq R\overline{\pi}_{b,!} \mathbb{W} \simeq R\Gamma_c(\operatorname{Ig}^{b,\mathrm{v}}_{K^p,C}, \mathbb{W}),
	\end{equation*}
where $\mathbb{W}$ denotes both the restriction of $\mathbb{W}$ to $\igs_{K^p}^{[b]}\gx$ and its pullback to $\operatorname{Ig}^{b,\mathrm{v}}_{K^p}$.
\end{Lem} 

\begin{proof}
	The first isomorphism follows from proper base change \cite[Proposition~22.19]{EtCohDiam}. The second follows because $\operatorname{Ig}^{b,\mathrm{v}}_{K^p,C}$ is naturally identified with the fiber $\overline{\pi}_{\mathrm{HT}}^{-1}(x)$.
\end{proof}

\begin{Rem}
  If $\mathbf{Sh}_K\gx$ is proper and Milne's axiom SV5 holds, then \Cref{Conj:ULAPerverse}(1) would also follow from a variant of \Cref{Lem:CohomIgusaVar}; let us explain. Let $\mathcal{G}$ be a parahoric $\zp$-model for $G$ and let $K_p=\mathcal{G}(\zp)$. When $p>2$, by \cite[Theorem~A]{DanielsYoucis} there is a morphism
  \begin{equation} \label{eq:DanielsYoucis}
      \mathscr{S}_{K}\gx^\diamond \to \operatorname{Sht}_{\mathcal{G},\mu},
  \end{equation}
  where $\mathscr{S}_{K}\gx$ is the $\mathcal{O}_E$-integral model for $\mathbf{Sh}_{K}\gx$ defined in \cite{KisinPappasZhou}, and $\shtgmu$ is the v-stack of $\mathcal{G}$-shtukas bounded by $\mu$. Using \eqref{eq:DanielsYoucis}, one can apply the construction of \cite[Section~2.14]{HamacherKimPointCounting} to obtain a \textit{perfect Igusa variety} $\operatorname{Ig}_{K^p}^b$. We conjecture that, as in \cite[Lemma~8.5.3]{DvHKZIgusaStacks}, one can identify $\operatorname{Ig}^{b,\mathrm{v}}_{K^p,C}$ with $(\operatorname{Ig}_{K^p}^b)^{\diamond}_{C}$; this would follow from Conjecture \ref{Conj:IgusaMainInt}. If so, the proof of \cite[Corollary~8.5.4]{DvHKZIgusaStacks} would imply that $\mathcal{F}$ is universally locally acyclic with respect to $\bungmuk \to \spd k$.
\end{Rem}

\begin{Rem}
  Additionally in the compact case, part (2) of \Cref{Conj:ULAPerverse} would follow from the arguments of \cite[Section~8.6]{DvHKZIgusaStacks} using the conjectural identification $\operatorname{Ig}^{b,\mathrm{v}}_{K^p,C} \simeq (\operatorname{Ig}_{K^p}^b)_{C}$, but it would require additional inputs as well. Namely, the arguments of loc.\ cit.\ make use of a scheme-theoretic local model diagram in the sense of \cite[Conjecture~4.1.5]{DvHKZIgusaStacks} and the smoothness and affineness of central leaves in Newton strata. Neither of these results are known in the abelian-type case, but we expect (at least under mild assumptions) that they should be accessible.
\end{Rem}

\subsection{Cohomology of Shimura varieties} \label{Sub:WeilCohShim} The goal of this section is to give a formula for the (compactly supported) cohomology of $\mathbf{Sh}_{K^p}\gx_C$ in terms of the action of a Hecke operator on the relative (compactly supported) cohomology of the Igusa stack over $\bun_{G,k}$. This generalizes the main result of \cite[Section~8.4]{DvHKZIgusaStacks}.

\subsubsection{} Let $\gx$, $\mathsf{E}$, $v$, $E$, $G$, and $\mu$ be as in \Cref{sub:DualizingSheaf}. We also continue to fix a neat compact open subgroup $K^p \subset \g(\afp)$. Let $d = \langle 2\rho, \mu\rangle=\dim \mathsf{X}$. Let $k$ be an algebraic closure of $k_{E}$, and let $\Lambda$ be a $\zl$-algebra in which $\ell$ is nilpotent, that contains a square root of $p$. We choose such a square root and denote it by $\sqrt{p} \in \Lambda$. Let $C$ be the completion of an algebraic closure of $E$, and denote by $\mathcal{O}_C$ the ring of integers of $C$. Fix an identification between $k$ and the residue field of $\mathcal{O}_C$. \begin{Lem} \label{Lem:NonSV5TorsorGroup}
The morphism
\begin{align*}
    \mathbf{Sh}_{K^p}\gx \to \mathbf{Sh}_{K}\gx
\end{align*}
is a pro-\'etale torsor for the profinite group $U_p=K_p/N_p$, where $N_p=N_p(K) \subset K_p$ is the closed subgroup that is the image of the projection
\begin{align*}
    \left(K \cap \zgqbar\right) \to K_p.
\end{align*}
\end{Lem}
\begin{proof}
This follows from the discussion in \cite[Section 1.5.8]{KisinShinZhu}, cf. \cite[Section 4.1.2]{DanielsYoucis}. 
\end{proof}
\begin{Def}\label{Def:acceptable}
    We say that $K^p$ is \emph{acceptable with respect to} $K_p$ if the group $N_p(K_pK^p)$ is pro-$p$.
\end{Def}
Note that if $K^p$ is acceptable with respect to $K_p$, then it is also acceptable with respect to $K_p'$ for any $K_p' \subset K_p$. The following lemma shows that there are many compact open subgroups $K^p$ that are acceptable with respect to $K_p$.

\begin{Lem} \label{Lem:CofinalProP}
For any compact open subgroup $K_p \subset G(\qp)$, there is a cofinal collection of compact open subgroups $U^p \subset \gaf$ such that the group $N_p(K_pU^p)$ is pro-$p$.
\end{Lem}
\begin{proof}
Fix $K_p$ as in the statement of the lemma and let $K^p \subset \gafp$ be compact open. It follows from the discussion in \cite[Section 1.5.8]{KisinShinZhu} that $K$ naturally surjects onto $\gal(\mathbf{Sh}/\mathbf{Sh}_{K})$ and similarly that
\begin{align*}
    f_{K^p}: \{1 \} \times K_p \subset K \to \gal(\mathbf{Sh}/\mathbf{Sh}_{K}) \to \gal(\mathbf{Sh}_{K^p} /\mathbf{Sh}_{K})
\end{align*}
is surjective. We are trying to show that $f_{K^p}$ has pro-$p$ kernel for a cofinal collection of $K^p$. Now, it is a direct consequence of the definitions that
\begin{align*}
    \gal(\mathbf{Sh}/\mathbf{Sh}_{K}) = \varprojlim_{U^p \subset K^p} \gal(\mathbf{Sh}_{U^p} /\mathbf{Sh}_{K}). 
\end{align*}
Since $\{1 \} \times K_p \subset K \to \gal(\mathbf{Sh}/\mathbf{Sh}_{K})$ is injective, see Lemma \ref{Lem:ChevalleyCongruence}, and since kernels commute with inverse limits, we deduce that
\begin{align*}
    \{1\} = \varprojlim_{U^p} \ker f_{U^p}.
\end{align*}
In other words $\cap_{U^p}  \ker f_{U^p} =\{1\}$, and thus there is a cofinal collection of $U^p$ such that $\ker f_{U^p}$ is pro-$p$. Indeed, the group $N_p$ contains a compact open pro-$p$ group, since $K_p$ does. 
\end{proof}
\subsubsection{} \label{subsub:BigDiagramI} Consider the commutative diagram
\begin{equation*}
    \begin{tikzcd}
         \left[\mathbf{Sh}_{K^p}\gx/(\ul{G(\qp)} \times \phi^{\mathbb{Z}})\right]  \arrow{d}{y} & \left[\mathbf{Sh}_{K^p}\gx / (\underline{K_p} \times \phi^{\mathbb{Z}})\right] \arrow{l}{x} \arrow{r}{z} \arrow{d}{h}&\left[\mathbf{Sh}_{K}\gx^\diamondsuit /\phi^{\mathbb{Z}}\right] \arrow{d}{f}  \\
         \left[\spd E / (\ul{G(\qp)} \times \phi^{\mathbb{Z}})\right] & \left[\spd E/ (\underline{K_p} \times \phi^{\mathbb{Z}})\right] \arrow{l}{b} \arrow{r}{a} & \left[\spd E / \phi^{\mathbb{Z}}\right],
    \end{tikzcd}
\end{equation*}
where the left square is Cartesian. Let $\mathbb{W}$ be a standard $\Lambda$-local system, and let us also write $\mathbb{W}$ for the canonical descent (in the sense of Section \ref{subsub:CanonicalDescent}) of $\mathbb{W}$ to the small v-stacks in the top row of the diagram. An important ingredient of our proof of Theorem \ref{Thm:WeilCohShimVar} is an identification of $y_{k,!} \mathbb{W}$ and $R y_{k,\ast} \mathbb{W}$, where $k$ denotes the base change of the morphism along $\spd k \to \spd \fp$. Since these objects are smooth representations of $G(\qp)$, it suffices to identify their $K_p$ fixed vectors for a cofinal collection of compact open subgroups $K_p$. This is achieved by the following proposition.
\begin{Prop} \label{Prop:AcceptableInvariants}
If $K^p$ is acceptable with respect to $K_p$, then there are natural isomorphisms 
\begin{align*}
    a_{k,!} b_k^{\ast} y_{k,!} \mathbb{W} &\to f_{k,!} \mathbb{W} \\
    a_{k,!} b_k^{\ast} Ry_{k,\ast} \mathbb{W} &\to R f_{k,\ast} \mathbb{W}.
\end{align*}
\end{Prop}
\begin{Lem} \label{Lem:CohomologicallyProper}
Let $c: X \to Y$ be an $\ell$-fine morphism of small v-stacks that is a gerbe for $\ul{H}$, with $H$ a $p$-adic Lie group. The morphism $c$ is cohomologically proper, i.e., $R c_{\ast}=c_{!}$. If $H$ is pro-$p$, then the natural transformation $1 \to R c_{\ast} c^{\ast}$ is an isomorphism.
\end{Lem}
\begin{proof}
Because $c$ is $\ell$-fine, we can check that it is cohomologically proper after a v-cover. Thus we may assume that $X \simeq Y \times_{\spd \fpbar} [\spd \fpbar / \ul{H}]$, and then the cohomological properness is \cite[Proposition 10.9, Lemma 9.8.(ii)]{MannNuclear}. Showing that $1 \to c_{\ast} c^{\ast} = c_{!} c^{\ast}$ is an isomorphism can also be done $v$-locally (by proper base change). Thus we may assume that $X \simeq Y \times_{\spd \fpbar} [\spd \fpbar / \ul{H}]$, and using proper base change again, we may further reduce to $c':\spd \fpbar \to \spd \fpbar / \ul{H}$. Using \cite[Lemma 10.3]{MannNuclear} to identify $c'_{\ast}=c'_{!}$ with continuous group cohomology, this follows from the fact that a pro-$p$ group has trivial continuous group cohomology on $\Lambda$-modules.
\end{proof}
\begin{proof}[Proof of Proposition \ref{Prop:AcceptableInvariants}]
We first apply smooth (resp. proper) base-change to the left hand square of Section \ref{subsub:BigDiagramI} to identify
\begin{align*}
    b_{k}^{\ast} R y_{k,\ast} \mathbb{W} &= R h_{k,\ast} \mathbb{W} \\
    b_{k}^{\ast} y_{k,!} \mathbb{W} &= h_{k,!} \mathbb{W}. 
\end{align*}
Next, recall that $K_p / N_p(K) = U_p(K)$, and consider the commutative diagram (the arrow is labeled $c$ by abuse of notation because we will apply Lemma \ref{Lem:CohomologicallyProper} to it) 
\begin{equation*}
    \begin{tikzcd}
         \left[\mathbf{Sh}_{K^p}\gx^\diamondsuit / (\underline{K_p} \times \phi^{\mathbb{Z}})\right] \arrow{r}{z} \arrow{d}{h}& \left[\mathbf{Sh}_{K}\gx^\diamondsuit /\phi^{\mathbb{Z}}\right] \arrow{d}{g} \arrow{dr}{f}  \\
          \left[\spd E/ (\underline{K_p} \times \phi^{\mathbb{Z}})\right]  \arrow{r}{d} & \left[\spd E/ (\underline{U_p(K)} \times \phi^{\mathbb{Z}})\right]  \arrow{r}{c} & \left[\spd E / \phi^{\mathbb{Z}}\right],
    \end{tikzcd}
\end{equation*}
where the square is Cartesian by Lemma \ref{Lem:NonSV5TorsorGroup}. Noting that $a=c \circ d$, we identify (using the first part of Lemma \ref{Lem:CohomologicallyProper})
\begin{align*}
    a_{k,!} R h_{k,\ast} \mathbb{W} &= a_{k,\ast} R h_{k,\ast} \mathbb{W} = R c_{k,\ast} R g_{k,\ast} R z_{k,\ast} z_{k}^{\ast} \mathbb{W} \\
    a_{k,!} h_{k,!} \mathbb{W} &= c_{k,!} g_{k,!} z_{k,!} z_{k}^{\ast} \mathbb{W}.
\end{align*}
By the second part of Lemma \ref{Lem:CohomologicallyProper} and the fact that $z$ is a gerbe for the pro-$p$ (by assumption) group $N_p(K)$, we deduce that the natural map $\mathbb{W} \to R z_{k,\ast} z_{k}^{\ast} \mathbb{W}= z_{k!} z_{k}^{\ast} \mathbb{W}$ is an isomorphism; since $f=c \circ g$, this allows us to conclude.
\end{proof}
\begin{Rem} \label{Rem:HeckeAction}
It follows from Proposition \ref{Prop:AcceptableInvariants} that if $K^p$ is acceptable with respect to $K_p$, then the (derived) Hecke algebra of level $K_p$ acts on $R \Gamma(\mathbf{Sh}_{K}\gx_{C}^{\diamondsuit},\mathbb{W})$ and $R \Gamma_c(\mathbf{Sh}_{K}\gx_{C}^{\diamondsuit},\mathbb{W})$.
\end{Rem}

\subsubsection{Hecke operators} Let $\Div = [\spd E / \phi^\mathbb{Z}]$, where $\phi=\phi_{\spd E}$ denotes the absolute Frobenius of $\spd E$ (see \cite[Section 2.1.9]{DvHKZIgusaStacks}). By \cite[Section III.3]{FarguesScholze}, the morphism $[\grgmu/\phi^\mathbb{Z}] \to \Div$ represents the functor sending $S \to \Div$ to the set of isomorphism classes of pairs consisting of a $G$-bundle $\mathscr{E}_1$ on $X_S$ and a modification $\mathscr{E}^0 \mid_{(X_S)_E} \dashrightarrow \mathscr{E}_1 \mid_{(X_S)_E}$ from the trivial bundle $\mathscr{E}^0$ which is bounded by $\mu^{-1}$. The functor sending such a modification to $\mathscr{E}^0$ defines a morphism
\begin{equation*}
	\tilde{h}_2 \colon [\grgmu / (\phi^\mathbb{Z} \times \underline{G(\qp)})] \to \bun_G^{[1]} \times \Div,
\end{equation*}
see \cite[Section~8.4]{DvHKZIgusaStacks} for justification of the naming convention here. Additionally, the Beauville--Laszlo morphism 
\begin{equation*}
	\mathrm{BL} \colon  [\grgmu / (\phi^\mathbb{Z} \times \underline{G(\qp)})] \to \bun_G
\end{equation*}
factors through $\bungmu$, and we denote the factorization $[\grgmu / (\phi^\mathbb{Z} \times \underline{G(\qp)})] \to \bungmu$ by $g$. As in \cite[Definition~8.4.7]{DvHKZIgusaStacks}, we define the functor
\begin{align*}
	T_\mu^{[1]} \colon \mathcal{D}(\bungmuk, \Lambda) &\to \mathcal{D}(\bungk^{[1]}\times_k \Divk,\Lambda) \\
	A &\mapsto R\tilde{h}_{2,k,\ast}(g_k^\ast A[d](d/2)).
\end{align*}
The operator $T_\mu^{[1]}$ is a version of a Hecke operator as defined in \cite{FarguesScholze}; see \cite[Section 8.4.6]{DvHKZIgusaStacks} for the precise relation. We furthermore recall from \cite[Proposition IV.7.1]{FarguesScholze} that there is a fully faithful functor
\begin{align*}
    \mathcal{D}(G(\qp), \Lambda)^{BW_E} \to \mathcal{D}(\bungk^{[1]}\times_k \Divk,\Lambda),
\end{align*}
and that $T_\mu^{[1]}$ lands in its essential image, see \cite[Corollary IX.2.3]{FarguesScholze}.

\subsubsection{} Let $\Lambda$ be a $\zl$-algebra in which $\ell$ is nilpotent, containing $\sqrt{p} \in \Lambda$ and let $\mathbb{W}$ be a standard local system. Recall that there is a natural $\gal_{E}$-equivariant isomorphism
\begin{align*}
    R\Gamma(\mathbf{Sh}_{K^p}\gx_{C}, \mathbb{W}) \simeq \varinjlim_{K_p} R\Gamma(\mathbf{Sh}_{K_pK^p}\gx_{C}, \mathbb{W}),
\end{align*}
and we define (usually one defines compactly supported cohomology only for finite type separated schemes)
\begin{align*}
    R\Gamma_{c}(\mathbf{Sh}_{K^p}\gx_{C}, \mathbb{W}):=\varinjlim_{K_p} R\Gamma_{c}(\mathbf{Sh}_{K_pK^p}\gx_{C}, \mathbb{W}).
\end{align*}
Both of these are naturally (complexes of) continuous representations of $G(\qp) \times \operatorname{Gal}_{E}$, such that the underlying (complex) of $G(\qp)$-representations consists of smooth representations. In other words, they can be considered as objects of $\mathcal{D}(G(\qp), \Lambda)^{BW_E}$. The following is the main result of this section, cf. \cite[Theorem 8.4.10]{DvHKZIgusaStacks}. 
\begin{Thm} \label{Thm:WeilCohShimVar}
For a standard $\Lambda$-local system $\mathbb{W}$, there are isomorphisms
\begin{align*}
&R\Gamma_c(\mathbf{Sh}_{K^p}\gx_C, \mathbb{W}) \simeq T_\mu^{[1]}(\mathcal{F}_{\mathbb{W},c}[-d])(-d/2), \\
&R\Gamma(\mathbf{Sh}_{K^p}\gx_C, \mathbb{W}) \simeq T_\mu^{[1]}(\mathcal{F}_{\mathbb{W}}[-d])(-d/2)
\end{align*}
in $\mathcal{D}(G(\qp), \Lambda)^{BW_E}$. 
\end{Thm}
\begin{Rem}
Our proof of Theorem \ref{Thm:WeilCohShimVar} below (or the proof of \cite[Theorem 8.4.10]{DvHKZIgusaStacks}) can also be used to show a version of Theorem \ref{Thm:WeilCohShimVar} with $\mathcal{F}^{\circ}_{\mathbb{W}}$ instead of $\mathcal{F}_{\mathbb{W}}$. The main difficulty is to generalize \cite[Proposition 8.4.12]{DvHKZIgusaStacks}, comparing the cohomology of the good reduction locus with the cohomology of the whole Shimura variety, to the case of abelian type Shimura varieties. This is now possible thanks to the recent work of Wu, see \cite[Proposition 5.21]{wu2025arithmeticcompactificationsintegralmodels}, which replaces the use of \cite[Corollary 5.20]{LanStrohII} in the proof of \cite[Proposition 8.4.12]{DvHKZIgusaStacks}. 
\end{Rem}

\begin{proof}[Proof of \Cref{Thm:WeilCohShimVar}]
From the Cartesian diagram \eqref{Eq:CartesianSquareFiniteLevel} we obtain a 2-commutative diagram, where we relabel the arrows in the interest of brevity:
	\begin{equation*}
		\begin{tikzcd}
			\left[\mathbf{Sh}_{K^p}\gx^\diamondsuit / (\underline{G(\qp)} \times \phi^{\mathbb{Z}})\right]
							\arrow[r, "\pi"] \arrow[d, "\tilde{g}"]
						& \left[\grgmu / (\underline{G(\qp)} \times \phi^{\mathbb{Z}})\right]
							\arrow[d, "g"]
						\\ \igs_{K^p}\gx 
							\arrow[r, "\overline{\pi}"]
						& \bungmu.
		\end{tikzcd}
	\end{equation*}
As in the proof of \cite[Theorem 8.4.10]{DvHKZIgusaStacks}, we see that $\tilde{h}_{2} \circ \pi$ is the structure morphism
\begin{equation*}
    y:[\mathbf{Sh}_{K^p}\gx^\diamondsuit / (\phi^{\mathbb{Z}} \times \ul{G(\qp)})] \to
    [\spd E/( \phi^{\mathbb{Z}} \times \ul{G(\qp)})].
\end{equation*}
By definition we have
	\begin{align*}
		T_\mu^{[1]}(\mathcal{F}_{\mathbb{W},c}[-d])(-d/2) &\simeq R\tilde{h}_{2,k,\ast}g_k^\ast \overline{\pi}_{k,!}\mathbb{W} \\
        T_\mu^{[1]}(\mathcal{F}_{\mathbb{W}}[-d])(-d/2) &\simeq R\tilde{h}_{2,k,\ast}g_k^\ast R\overline{\pi}_{k,\ast}\mathbb{W}
	\end{align*}
	By \cite[Proposition 20.2.3]{ScholzeWeinsteinBerkeley}, the map $\grgmu \to \spd E$ is proper, and therefore $\tilde{h}_2$ is proper as well. Thus by applying proper (resp. smooth) base change, we obtain 
	\begin{align*}
		T_\mu^{[1]}(\mathcal{F}_{\mathbb{W},c}[-d])(-d/2) &\simeq \tilde{h}_{2,k,!} \pi_{k,!} \mathbb{W} = y_{k,!} \mathbb{W} \\
        T_\mu^{[1]}(\mathcal{F}_{\mathbb{W}}[-d])(-d/2) &\simeq R\tilde{h}_{2,k,\ast} R\pi_{k,\ast} \mathbb{W} = R y_{k,\ast} \mathbb{W}.
	\end{align*}
Since these are (complexes of) smooth representations of $G(\qp)$, we can recover them from the colimit of their $K_p$-invariants. Applying Proposition \ref{Prop:AcceptableInvariants} then gives us $G(\qp)$-equivariant isomorphisms in $\mathcal{D}(\Lambda)^{B W_E}$
	\begin{align*}
		T_\mu^{[1]}(\mathcal{F}_{\mathbb{W},c}[-d])(-d/2) &\simeq \varinjlim_{K_p} f_{K_p,k,!} \mathbb{W}, \\
        T_\mu^{[1]}(\mathcal{F}_{\mathbb{W}}[-d])(-d/2) &\simeq \varinjlim_{K_p} R f_{K_p,k,\ast} \mathbb{W},
	\end{align*}
where $f_{K_p}$ is the structure map $$\left[\mathbf{Sh}_{K_pK^p}\gx^{\diamondsuit}/(\ul{K_p} \times \phi^{\mathbb{Z}})\right] \to \left[\spd E / \phi^{\mathbb{Z}}\right].$$
By Proposition \ref{Prop:GaloisRestriction} and comparison theorems between the \'etale cohomology of an algebraic variety and its analytification, see Lemma \ref{Lem:HuberComparison} and \cite[Lemma 15.6, Proposition 27.5]{EtCohDiam}, for each $K_p$, there are canonical $W_E$-equivariant identifications 
\begin{align*}
    R f_{K_p,k,\ast} \mathbb{W} &\simeq R\Gamma(\mathbf{Sh}_{K}\gx_{C}, \mathbb{W}), \\
    f_{K_p,k,!} \mathbb{W} &\simeq R\Gamma_c(\mathbf{Sh}_{K}\gx_{C}, \mathbb{W}),
\end{align*}
finishing the proof.
\end{proof}

\subsection{A product formula} \label{sub:ProductFormula} In this section we prove a Mantovan-style product formula, following \cite[Section 8.5]{DvHKZIgusaStacks}. Let $\Lambda$ be a $\zl$-algebra in which $\ell$ is nilpotent containing a fixed square root of $p$. Let $k$ be an algebraic closure of the residue field of $E$ and base-change everything to $k$ as before. 

\subsubsection{} Before we can state the Mantovan product formula, we need to establish a bit more notation. First fix a left-invariant Haar measure on $G_b(\qp)$ and write $\mathcal{H}(G_b)$ for the algebra $C_c(G_b(\qp),\Lambda)$ of compactly supported locally constant $\Lambda$-valued functions on $G_b(\qp)$. Additionally, define the local Shimura variety at infinite level attached to $(G,b,\mu)$ as the fiber product
\begin{equation}\label{eq:LSVinfinity}
	\begin{tikzcd}
		\mathcal{M}_{G,b,\mu,\infty}
			\arrow[r] \arrow[d]
		& \grgmu^{[b]}
			\arrow[d, "\mathrm{BL}"]
		\\ \spd k
			\arrow[r, "b"]
		& \bung^{[b]}.
	\end{tikzcd}
\end{equation}
We note that the top horizontal arrow in this diagram is a $\tilde{G}_b$-torsor, and the action of $\underline{G(\qp)}$ on $\grgmu^{[b]}$ lifts to an action on $\mathcal{M}_{G,b,\mu,\infty}$ commuting with the $\tilde{G}_b$-action. 

By \cite[Proposition~3.15]{HamannImai}, there is a character $\delta_b$ of $G_b(\qp)$ such that the dualizing complex of $\bungk^{[b]}$ is given by $\delta_b^{-1}[-2d_b]$, where $d_b = \langle 2\rho, \nu_b\rangle$ (here, as usual, $\delta_b$ is viewed as a sheaf on $\bungk^{[b]}$ using the equivalence $\mathcal{D}(\bungk^{[b]},\Lambda) \simeq \mathcal{D}(G_b(\qp),\Lambda)$ as in \cite[Theorem~I.5.1.(ii)]{FarguesScholze}). We define $R\Gamma_c(\mathcal{M}_{G,b,\mu,\infty},\delta_b)$ to be the exceptional pushforward of $\delta_b$ along the structure map $\mathcal{M}_{G,b,\mu,\infty} \to \spd C$, where we view $\delta_b$ as a sheaf on $\mathcal{M}_{G,b,\mu,\infty}$ by pulling back along either composition in the diagram \eqref{eq:LSVinfinity}. As in \cite[Section~8.5]{DvHKZIgusaStacks}, $\mathcal{M}_{G,b,\mu,\infty}$ admits a descent to $\mathrm{Div}_{E,k}^1$, and it follows that $R\Gamma_c(\mathcal{M}_{G,b,\mu,\infty},\delta_b)$ is equipped with an action of $G_b(\qp)\times G(\qp) \times W_E$. 

\begin{Thm} \label{Thm:MantovanProduct} There exists a filtration on $R\Gamma_c(\mathbf{Sh}_{K^p}\gx_{\overline{E}}, \mathbb{W})$ of complexes of smooth representations of $G(\qp)\times W_E$, whose graded pieces are given by
	\begin{equation*}
			R\Gamma_c(\operatorname{Ig}^{b,\mathrm{v}}_{K^p,C},\mathbb{W})^{\mathrm{op}} \otimes_{\mathcal{H}(G_b)}^\mathbb{L} R\Gamma_c(\mathcal{M}_{G,b,\mu,\infty},\delta_b)[2d_b].
	\end{equation*}
\end{Thm}
\begin{proof}
	The arguments in the proof of \cite[Theorem 8.5.7]{DvHKZIgusaStacks} show that $R\Gamma_c(\mathbf{Sh}_{K^p}\gx_{\overline{E}}, \mathbb{W})$ admits a filtration whose graded pieces are given by
	\begin{equation*}
		R\overline{\pi}_{b,!}\mathbb{W} \otimes_{\mathcal{H}(G_b)}^\mathbb{L} R\Gamma_c(\mathcal{M}_{G,b,\mu,\infty},\delta_b)[2d_b].
	\end{equation*}
	We then conclude using Lemma \ref{Lem:CohomIgusaVar}.
\end{proof}

\subsection{Compatibility with the Fargues--Scholze correspondence} \label{sub:WeakCompatibility}
We continue to let $\gx$, $\mathsf{E}, v, E, G$, $k$, and $\mu$ be as in Section \ref{sub:DualizingSheaf}. Let us assume that $\Lambda$ is a $\zl$-algebra containing a square root of $p$ in which $\ell$ is nilpotent. We fix a compact open subgroup $K^p \subset \g(\afp)$, and for any compact open subgroup $K_p \subset G(\qp)$, we write $K = K_pK^p$. Moreover, we let $d=\langle2\rho,\mu\rangle$ be the dimension of the Shimura variety. In this section we study the Igusa sheaves using the spectral action of \cite{FarguesScholze}. Our exposition here largely follows \cite[Section 9]{DvHKZIgusaStacks}, and we refer the reader there for detailed explanations of the arguments which are here only briefly summarized.

\subsubsection{} To state the theorem in this section we need to recall a bit of notation. Let $\widehat{G}$ be a dual group of $G$ over $\zl(\sqrt{p})$, equipped with an action of $W_{\qp}$ restricted from an action of $\gal_{\qp}$. As in \cite[Section 2.1]{BuzzardGee}, define the $L$-group $^LG$ to be $\widehat{G} \rtimes W_{\qp}$. Associated with $\mu$ is a flat perverse sheaf $\mathcal{S}_\mu$, which lies in the Satake category for the local Hecke stack $\mathcal{H}\mathrm{ck}_{G,\mathrm{Div}^1}$, see \cite[Section 8.4.1]{DvHKZIgusaStacks} for the precise definition. Applying the fiber functor $F$ defined in \cite[Definition / Proposition IV.7.10]{FarguesScholze} to $\mathcal{S}_\mu$, we obtain a $W_{\qp}$-equivariant algebraic representation of the dual group $\widehat{G}$, which formally corresponds to a  representation $r_\mu$ of $^LG = \widehat{G} \rtimes W_{\qp}$. Note here we are applying \cite[Theorem VI.11.1]{FarguesScholze} to identify the dual group $\widehat{G}$ with the Tannaka group $\check{G}$ of the Satake category $\operatorname{Sat}(\mathcal{H}\mathrm{ck}_{G,\mathrm{Div}^1},\Lambda)$; this uses our choice of $\sqrt{p} \in \zl(\sqrt{p})$.

Denote by $\parg$ the stack of $L$-parameters over $\zl(\sqrt{p})$ as in \cite{DHKMModuli}, \cite{ZhuCoherent}, and \cite{FarguesScholze}, and let $\perf(\parg)$ be the $\infty$-category of perfect complexes on $\parg$. By \cite[Theorem X.0.2, Theorem V.4.1]{FarguesScholze}, when $\ell$ is coprime to the order of $\pi_0(Z(G))$, there is a natural $\Lambda$-linear action (the spectral action) of $\perf(\parg)$ on $\mathcal{D}(\bun_{G,k},\Lambda)$, which preserves the compact objects. If $\mathcal{V}$ is a perfect complex on $\parg$, we denote the spectral action of $\mathcal{V}$ on $\mathcal{D}(\bun_{G,k},\Lambda)$ by $\mathcal{V} \ast (-)$. 

\subsubsection{} When $\mathcal{V}$ is defined from an algebraic representation $V$ of $\widehat{G}$, the spectral action is naturally identified with the Hecke operator $T_V$ defined in \cite[Chapter IX]{FarguesScholze}, see \cite[Example 9.1.2]{DvHKZIgusaStacks}. In particular, the representation $r_\mu$ determines a vector bundle $\mathcal{V}_\mu$ on $\parg$, and the spectral action of $\mathcal{V}_\mu$ recovers the action of the Hecke operator $T_\mu$. We write $\pi_0 \operatorname{End}_{\parg}(\mathcal{V}_\mu)$ for the ring of endomorphisms of $\mathcal{V}_\mu$, considered as a vector bundle on $\parg$ (as opposed to the endomorphisms in the category of perfect complexes). 

For an algebraic representation $V$ and corresponding vector bundle $\mathcal{V}$ on $\parg$, via the action of $\operatorname{End}_{\parg}(\mathcal{V})$ the objects $\mathcal{V} \ast A$ for $A$ in $\mathcal{D}(\bun_{G,k}, \Lambda)$ inherit an action of $W_E$, which we call the \textit{spectral $W_E$-action}. This action is compatible with the natural action of $W_E$ on the Hecke operators $T_V$ under the identification $\mathcal{V}\ast(-) = T_V$, see \cite[Remark 9.1.3]{DvHKZIgusaStacks}.

\subsubsection{} In the rest of this section, we will implicitly base-change the above objects along $\zl(\sqrt{p}) \to \Lambda$, using our fixed $\sqrt{p} \in \Lambda$. 

\subsubsection{} If $\ell$ is coprime to the order of $\pi_0(Z(G))$, then by \cite[Corollary IX.0.3]{FarguesScholze}, there is a morphism
\begin{equation}\label{eq:BernsteinCenters}
	\Psi_G \colon \mathcal{Z}^\mathrm{spec}(G,\Lambda) \to \mathcal{Z}(G(\qp),\Lambda),
\end{equation}
where $\mathcal{Z}^\mathrm{spec}(G,\Lambda) = \Gamma(\parg,\mathcal{O})$ is the spectral Bernstein center \cite[Definition IX.0.2]{FarguesScholze} and $\mathcal{Z}(G(\qp),\Lambda)$ is the (usual) Bernstein center of the abelian category of smooth $G(\qp)$-representations on $\Lambda$-modules. 

Finally, let us denote by $j_\mu\colon \bungmu \hookrightarrow \bung$ the open immersion of the $\mu^{-1}$-admissible locus and by $i_1 \colon \bung^{[1]} \hookrightarrow \bung$ the open immersion of the Newton stratum corresponding to $1 \in G(\qpbr)$. 

\begin{Prop}\label{Prop:SpectralActionSV}
Assume that $\ell$ is coprime to the order of $\pi_0(Z(G))$, and let $\mathbb{W}$ be a standard $\Lambda$-local system. Then there is an isomorphism
	\begin{equation*}
		R\Gamma(\mathbf{Sh}_{K^p}\gx_{\overline{E}}, \mathbb{W}) \simeq i_{1,k}^\ast \left( \mathcal{V}_\mu \ast (j_{\mu, k, !} \mathcal{F}_{\mathbb{W}}[-d])(-\frac{d}{2})\right)
	\end{equation*}
	in $\mathcal{D}(G(\qp), \Lambda)^{BW_E} \simeq \mathcal{D}(\bun^{[1]}_{G,k}, \Lambda)^{BW_E}$. Moreover, if $K_p$ is a compact open subgroup of $\g(\qp)$ such that $K^p$ is acceptable with respect to $K_p$, then there is a map of $\Lambda$-algebras
	\begin{equation*}
		\pi_0 \operatorname{End}_{\parg}(\mathcal{V}_\mu)  \to \operatorname{End}_{\mathcal{D}(G(\qp),\Lambda)}(R\Gamma(\mathbf{Sh}_{K_pK^p}\gx_{\overline{E}},\mathbb{W})).
	\end{equation*}
\end{Prop}
\begin{proof}
The isomorphism follows from Theorem \ref{Thm:WeilCohShimVar}. Moreover, by Proposition \ref{Prop:AcceptableInvariants}, when $K^p$ is acceptable with respect to $K_p$, the complex $R\Gamma(\mathbf{Sh}_{K_pK^p}\gx_{\overline{E}},\mathbb{W})$ is identified with the derived $K_p$-invariants of $R\Gamma(\mathbf{Sh}_{K^p}\gx_{\overline{E}},\mathbb{W})$. In that case the existence of the map of $\Lambda$-algebras is proven exactly as in \cite[Theorem 9.1.4]{DvHKZIgusaStacks}. 
\end{proof}

\subsubsection{} Suppose $K^p$ is acceptable with respect to $K_p$, and let $K = K_pK^p$. Then from Proposition \ref{Prop:SpectralActionSV}, we obtain an action of $W_E$ on $R\Gamma(\mathbf{Sh}_{K}\gx_{\overline{E}}, \mathbb{W})$ for any $W$ and $\Lambda$ as in the statement of the theorem. We refer to this as the spectral $W_E$-action on $R\Gamma(\mathbf{Sh}_{K}\gx_{\overline{E}},\mathbb{W})$. If $\Lambda$ is instead the ring of integers in a finite extension of $\ql$, then we can consider local systems $\mathbb{W}$ constructed from a continuous action of $K^p$ on a finite free $\Lambda$-module $W$, which factors through the image of $K^p$ in $\g^c(\afp)$; we will call such local systems standard $\Lambda$-local systems. For such a standard $\Lambda$-local system, if we fix $\sqrt{p} \in \Lambda$, then we similarly obtain a $W_E$-action on $R\Gamma(\mathbf{Sh}_{K}\gx_{\overline{E}}, \mathbb{W})$ (again called the spectral $W_E$-action) by taking a limit of the actions on $R\Gamma(\mathbf{Sh}_{K}\gx_{\overline{E}}, \mathbb{W}/\ell^n\mathbb{W})$ for each $n$. 

\begin{Thm}\label{Thm:SpectralAction}
	Suppose $\Lambda$ is a $\zl$-algebra containing a fixed square root of $p$ which is either $\ell$-torsion or the ring of integers in a finite extension of $\ql$, and let $\mathbb{W}$ be a standard $\Lambda$-local system. Assume that $\ell$ is coprime to the order of $\pi_0(Z(G))$. If $K^p$ is acceptable with respect to $K_p$, then $\pi_0\operatorname{End}_{\parg}(\mathcal{V}_\mu)$ acts on $R\Gamma(\mathbf{Sh}_{K_pK^p}\gx_{\overline{E}},\mathbb{W})$. Moreover, the spectral and usual $W_E$-action on $R\Gamma(\mathbf{Sh}_{K_pK^p}\gx_{\overline{E}}, \mathbb{W})(\tfrac{d}{2})$ agree\footnote{The Tate twist by $(d/2)$ was erroneously omitted in \cite[Theorem 9.1.4 and Corollary 9.1.5 of v1]{DvHKZIgusaStacks}.}, and the action of $\mathcal{Z}^\mathrm{spec}(G,\Lambda) \subset \pi_0 \operatorname{End}_{\parg}(\mathcal{V}_{\mu})$ on $R\Gamma(\mathbf{Sh}_{K_pK^p}\gx_{\overline{E}},\mathbb{W})$ factors via $\Psi_G$ through the natural action of $\mathcal{Z}(G(\qp), \Lambda)$.
\end{Thm}
\begin{proof}
	The theorem follows from Proposition \ref{Prop:SpectralActionSV} by the arguments in the proofs of \cite[Theorem 9.1.4, Corollary 9.1.5]{DvHKZIgusaStacks}.
\end{proof}

\begin{Rem}
    Let us denote the (possibly) non-commutative $\Lambda$-algebra $\pi_0\operatorname{End}_{\parg}(\mathcal{V}_\mu)$ by $E_{v,\mu}$. For each place $v$ of $\mathsf{E}$ not above $\ell$, the above result gives an algebra homomorphism $E_{v,\mu}\to \operatorname{End}_\Lambda(R\Gamma(\mathbf{Sh}_K(\mathsf{G},\mathsf{X})_{\overline{\mathsf{E}}},\mathbb{W}))$. Then by varying $v$, we get an algebra homomorphism from the free product $\ast_{v \nmid \ell}E_{v,\mu}\to \operatorname{End}_\Lambda(R\Gamma(\mathbf{Sh}_K(\mathsf{G},\mathsf{X})_{\overline{\mathsf{E}}},\mathbb{W}))$. The image of this map contains the (global) unramified Hecke algebra, and should encode certain (global) Galois theoretic information, since the vector bundles $\mathcal{V}_\mu$ carry Weil group actions. Though it seems hard to describe this image explicitly.
\end{Rem}

\subsubsection{} In the remainder of this subsection we assume $\Lambda$ is the ring of integers in a finite extension of $\ql$ and that $\Lambda$ contains a fixed square root of $p$. We fix $K=K_pK^p$ as before with $K^p$ acceptable with respect to $K_p$. We write $\mathcal{H}_{K_p}$ for the Hecke algebra $\Lambda[G(\qp) \sslash K_p]$ of level $K_p$ with coefficients in $\Lambda$, and fix a standard $\Lambda$-local system $\mathbb{W}$. Since $K^p$ is acceptable with respect to $K_p$, the complex $R\Gamma(\mathbf{Sh}_{K}\gx_{\overline{E}},\mathbb{W})$ carries a natural right action by $\mathcal{H}_{K_p}$, see Remark \ref{Rem:HeckeAction}. Let $\mathcal{Z}_{K_p}$ denote the center of $\mathcal{H}_{K_p}$, and suppose $\chi \colon \mathcal{Z}_{K_p} \to L$ is a character, where $L$ is either $\flbar$ or $\qlbar$. For any $i$, define\footnote{The Tate twist by $(d/2)$ was erroneously omitted in \cite[Section 1.2, Section 9.2 of v1]{DvHKZIgusaStacks}.}
\begin{equation*}
    W^i(\chi) := H^i(\mathbf{Sh}_K\gx_{\overline{E}},\mathbb{W})(d/2) \otimes_{\mathcal{Z}_{K_p},\chi} L.
\end{equation*}

The following theorem describes a compatibility between the Fargues--Scholze local Langlands correspondence and the cohomology of Shimura varieties. Recall that using \eqref{eq:BernsteinCenters} and \cite[Proposition VIII.3.2]{FarguesScholze}, we can associate to $\chi$ a semisimple $L$-parameter $\phi_{\chi}^{\mathrm{FS}}:W_{\qp} \to {}^LG(L)$.

\begin{Thm} \label{Thm:SV-FS-Compatibility}
    Suppose $\ell$ is coprime to the order of $\pi_0(Z(G))$. Then each irreducible $L$-linear $W_E$-representation occurring as a subquotient of $W^i(\chi)$ also occurs as a subquotient of $r_\mu \circ \restr{\phi_\chi^{\mathrm{FS}}}{{W_E}}$. 
\end{Thm}
\begin{proof}
    This follows from the same proof as that of \cite[Theorem 9.2.1]{DvHKZIgusaStacks}, using Theorem \ref{Thm:SpectralAction} to construct the commutative diagrams therein.
\end{proof}

\subsection{Hecke polynomials and the Eichler--Shimura relation} \label{Sub:ES}

With the results of the previous section in hand, we can now derive an Eichler--Shimura congruence relation for Shimura varieties of abelian type. For this we follow \cite[Section 9.4]{DvHKZIgusaStacks} and \cite[Section 5]{vdH}. 

Let $\Lambda$ be the ring of integers in a finite extension of $\ql$ containing a fixed square root of $p$. Fix an algebraic representation $V$ of $\widehat{G}$, and assume $V$ is of rank $n$ over $\Lambda$. Suppose that it extends to a representation $r_V:\widehat{G} \rtimes W_{E} \to \GL(V)$ for some finite extension $E$ over $\qp$, and let $\phi_E^\mathrm{univ}$ be the universal 1-cocycle on $Z^1(W_{\qp}, \widehat{G})$ restricted to $W_E$. Then for any $\gamma \in W_E$ we obtain an element $(r_V \circ \phi_E^{\mathrm{univ}})(\gamma)$ of $\pi_0 \operatorname{End}_{\parg}(\mathcal{V})$, where as before $\mathcal{V}$ is the vector bundle on $\parg$ corresponding to $V$. 

\subsubsection{} Following \cite[Definition 9.3.2]{DvHKZIgusaStacks}, for $\gamma \in W_E$, we define the spectral Hecke polynomial as
\begin{equation*}
	H^\mathrm{spec}_{G,V,\gamma}(X) = \det(X-(r_V \circ \phi_E^\mathrm{univ})(\gamma)) = \sum_{i=0}^n (-1)^i S_{\wedge^{n-i} V, \gamma} X^i \in \mathcal{Z}^\mathrm{spec}(G,\Lambda).
\end{equation*}
Here $S_{\wedge^i V,\gamma} \in \mathcal{Z}^\mathrm{spec}(G,\Lambda)$ is an excursion operator whose attached function sends an $L$-parameter $\phi$ to the trace of $\wedge^i(r_V \circ \phi)(\gamma)$, see \cite[Section 9.3]{DvHKZIgusaStacks} for details. By \cite[Corollary 9.3.3]{DvHKZIgusaStacks}, there is a $\mathcal{Z}^\mathrm{spec}(G,\Lambda)$-algebra homomorphism
\begin{equation}\label{eq:SpectralHeckePolyMap}
	\mathcal{Z}^\mathrm{spec}(G,\Lambda)[X] / (H^\mathrm{spec}_{G,V,\gamma}(X)) \to \pi_0 \operatorname{End}_{\parg}(\mathcal{V})
\end{equation}
given by mapping $X$ to $(r_V \circ \phi_E^\mathrm{univ})(\gamma)$.

\subsubsection{} Let $H_{G,V,\gamma}(X)$ denote the image of $H^\mathrm{spec}_{G,V,\gamma}$ under the map $\Psi_G$ (see \eqref{eq:BernsteinCenters}). We have the following proposition, which follows exactly as \cite[Corollary 9.3.5]{DvHKZIgusaStacks}.

\begin{Prop} \label{Prop:SpectralHecke}
	Let $\ell$ be coprime to the order of $\pi_0(Z(G))$. Let $K=K_pK^p$ be a compact open subgroup such that $K^p$ is acceptable with respect to $K_p$. Then for every element $\gamma \in W_E$ and every standard $\Lambda$-local system $\mathbb{W}$ the endomorphism $H_{G,V_\mu, \gamma}(\gamma)$ acts as zero on $R\Gamma(\mathbf{Sh}_K\gx_{\overline{E}},\mathbb{W})(\frac{d}{2})$. 
\end{Prop}
\begin{proof}
	Given the homomorphism \eqref{eq:SpectralHeckePolyMap}, this is immediate from \Cref{Thm:SpectralAction} and \Cref{Prop:AcceptableInvariants}.
\end{proof}

\begin{Rem}
If $\gamma$ is a lift of Frobenius, then the Hecke polynomial $H_{G,V_\mu, \gamma}$ is independent of $\ell$. More precisely, it follows from \cite[Theorem 6.2, proof of Corollary 6.3]{ScholzeMotivicGeometrization} that it lies inside $\mathcal{Z}(G(\qp), \mathbb{Z}[\sqrt{p},\tfrac{1}{p |\pi_0(\zg)|}]) \subset \mathcal{Z}(G(\qp),\Lambda)$. Here the inclusion of Bernstein centers is described in \cite[Lemma 3.2]{DHKMFamilies}. 
\end{Rem}

\subsubsection{} \label{subsub:EichlerShimura} Now let us fix a special parahoric subgroup $K_p \subset G(\qp)$ with $\zp$-model $\mathcal{G}$, and let $K_p' \subset K_p$ be an Iwahori subgroup with $\zp$-model $\mathcal{I}$. Write $K' = K_p'K^p$. Fix also a quasi-split inner form $G^\ast$ of $G$, along with a very special parahoric integral model $\mathcal{G}^\ast$. Let $\mathcal{H}_\mathcal{G}$, $\mathcal{H}_\mathcal{I}$, and $\mathcal{H}_{\mathcal{G}^\ast}$ denote the $\Lambda$-linear Hecke algebras for $\mathcal{G}$, $\mathcal{I}$, and $\mathcal{G}^\ast$, respectively. 

Let $q$ be the cardinality of the residue field of $E$, and let $I=I_{\mathbb{Q}_p}$ denote the inertia subgroup of $W_{\qp}$. Fix (arithmetic) Frobenius lifts $\sigma\in W_{\qp}$ and $\sigma_E \in W_E$. Because $G^\ast$ is quasi-split, \cite[Proposition 2.2]{vdH} provides an isomorphism between $\mathcal{H}_{\mathcal{G}^\ast}$ and the algebra of $\widehat{G}^I$-invariant functions on $\widehat{G}^Iq^{-1}\sigma$. Following \cite[Section 5]{vdH}, using this isomorphism we define an auxiliary renormalized Hecke polynomial in $\mathcal{H}_{\mathcal{G}^\ast}[X]$ by
\begin{equation*}
	H_\mu^\ast = \det(X-q^{\frac{d}{2}}r_\mu(g,\sigma_E)).
\end{equation*}

Following \cite[Section 3]{vdH}, we write
\begin{equation*}
	\tilde{t} \colon \mathcal{H}_{\mathcal{G}^\ast} \to \mathcal{H}_{\mathcal{G}}
\end{equation*} 
for the normalized transfer homomorphism between $\mathcal{H}_{\mathcal{G}^\ast}$ and $\mathcal{H}_{\mathcal{G}}$. Then the renormalized Hecke polynomial $H_\mu \in \mathcal{H}_{\mathcal{G}}[X]$ is defined to be the image of $H_\mu^\ast$ under $\tilde{t}$, i.e.,
\begin{equation*}
	H_\mu = \tilde{t}(H_\mu^\ast).
\end{equation*}

We can now state our main theorem of this section. Let $K_p'$ be as above.

\begin{Thm} \label{Thm:EichlerShimura}
Assume that $\ell$ be coprime to the order of $\pi_0(Z(G))$. Let $K^p \subset \gafp$ be a compact open subgroup such that $K^p$ is acceptable with respect to $K_p'$, set $K'=K_p'K^p$, and let $\mathbb{W}$ be a standard $\Lambda$-local system. Then the inertia subgroup $I_E$ of $W_E$ acts unipotently on $H^i(\mathbf{Sh}_{K'}\gx_{\overline{E}}, \mathbb{W})$ for all $i$. Moreover, the action of any Frobenius lift $\sigma_E \in W_E$ on $R\Gamma(\mathbf{Sh}_{K'}\gx_{\overline{E}}, \mathbb{W})$ satisfies $H_\mu(\sigma_E) = 0$. 
\end{Thm}
\begin{proof}
	This follows from Proposition \ref{Prop:SpectralHecke} by the arguments in the proof of \cite[Theorem 5.4]{vdH}, which comprise a generalization of the arguments in the proof of \cite[Theorem 9.4.11]{DvHKZIgusaStacks}.
\end{proof}

\subsection{Partial Frobenii and their Eichler--Shimura relations} \label{sub:Plectic}
\newcommand{\Par}{\operatorname{Par}}
In this subsection, we show that when an abelian type Shimura variety comes from restriction of scalars from a totally real field, then at split primes its cohomology satisfies the local plectic conjecture of Nekov\'ar--Scholl.\footnote{The local plectic conjecture is most interesting at primes that are not split. For such primes, it will be proved in upcoming work of Feng--Tamiozzo--Zhang.} If we take the level at $p$ to be hyperspecial, this proves the existence of partial Frobenius operators each satisfying an Eichler--Shimura relation; this generalizes a result of Lee \cite[Theorem 1.0.4]{Lee}. Let $\Lambda$ be a $\zl$-algebra in which $\ell$ is nilpotent, containing a fixed $\sqrt{p} \in \Lambda$.

\subsubsection{} More precisely, we let $\gx$ be an abelian type Shimura datum, such that $\mathsf{G}=\operatorname{Res}^\mathsf{F}_\mathbb{Q} \mathsf{H}$ for some connected reductive group $\mathsf{H}$ over a totally real number field $\mathsf{F}$ of degree $r$ over $\mathbb{Q}$. We assume that $p$ splits completely in $\mathsf{F}$, and let $\mathfrak{p}_i\mid p$, $i=1,\dots,r$ be primes of $\mathsf{F}$ above $p$. Denote by $F_i\simeq \qp$ the completion of $\mathsf{F}$ at $\mathfrak{p}_i$. Then $G=\mathsf{G}_{\qp}$ is isomorphic to $\prod_{i=1}^r H_i$, where $H_i =\mathsf{H}\times_{\mathsf{F}} F_{i}$. We can similarly identify
\begin{align*}
    \operatorname{Aut}_{\mathsf{F}\otimes_{\mathbb{Q}}\qp}(\mathsf{F} \otimes_{\mathbb{Q}} \qpbar) \simeq \prod_{i=1}^r \gal_{\qp},
\end{align*}
which contains a ``diagonal'' copy of $\gal_{\qp} \to \operatorname{Aut}_{\mathsf{F} \otimes_{\mathbb{Q}} \qp}(\mathsf{F} \otimes_{\mathbb{Q}} \qpbar)$.

\subsubsection{} \label{subsub:LocalPlecticReflex} Let $\mathsf{E}$ be the reflex field of $\gx$ as before. Choose an isomorphism $\iota_p:\mathbb{C}\simeq \qpbar$, which induces a $p$-adic place $v$ of $\mathsf{E}$. Let $E=\mathsf{E}_v$ be the completion. The Hodge cocharacter $\mu$ can be viewed as a cocharacter of $G$ via $\iota_p$ and it decomposes as a product $\mu=\prod_{i=1}^{r} \mu_i$, where $\mu_i$ is a cocharacter of $H_i$. For each $i$, we let $E_i$ be the reflex field of $\mu_i$. The group $\prod_{i=1}^r \gal_{\qp}$ naturally acts on the space of conjugacy classes of cocharacters of $G$, and its stabilizer is given by $\prod_{i=1}^r \gal_{E_i}$. Intersecting with the diagonal $\gal_{\qp}$ gives the stabilizer $\gal_{E}$, and in particular there is a natural injective map $\gal_{E} \to \prod_{i=1}^r \gal_{E_i}$. We consider the induced embedding $W_{E}\to \prod_{i=1}^r W_{E_i}$, restricting along which defines a functor between the corresponding derived categories of continuous representations $\mathcal{D}(\prod_{i=1}^r W_{E_i},\Lambda)\to \mathcal{D}(W_E,\Lambda)$. We have the following result, which is a special case of the local version of Nekov\'ar--Scholl's plectic conjecture. 

\begin{Thm}\label{Thm:Plectic}
    Let $K=K_pK^p \subset \gaf$ be a neat compact open subgroup such that $K^p$ is acceptable with respect to $K_p$ and let $\mathbb{W}$ be a standard $\Lambda$-local system. If $\ell$ is coprime to the order of $\pi_0(Z(G))$, then the cohomology complex
    \[R\Gamma(\mathbf{Sh}_{K}(\mathsf{G},\mathsf{X})_{\overline{E}}, \mathbb{W})[d](d/2)\in \mathcal{D}(W_{E},\Lambda)\]
    lifts canonically to an object in $\mathcal{D}(\prod_i W_{E_i},\Lambda)$, compatibly with the Hecke actions. 
\end{Thm}
\begin{proof}
    In our situation, the stack of $L$-parameters for $G$ decomposes as a product 
    \[\operatorname{Par}_{G_{\qp}}\simeq \prod^r_{i=1}\operatorname{Par}_{H_{i,F_i}}.\] Under this isomorphism, we have, accordingly, a decomposition $\mathcal{V}_\mu\simeq\boxtimes_{i=1}^r\mathcal{V}_{\mu_i}$, where $\mathcal{V}_{\mu_i}$ is a vector bundle on $\Par_{H_i,F_i}$ corresponding to the tilting representations of $\widehat{H}_i$ labeled by $\mu_i$. Each $\mathcal{V}_{\mu_i}$ carries an action of $W_{E_i}$ and the decomposition $\mathcal{V}_\mu\simeq\boxtimes_{i=1}^r\mathcal{V}_{\mu_i}$ is equivariant for the $W_{E}$-action on the left hand side and the $\prod_iW _{E_i}$-action on the right hand side with respect to the diagonal embedding $W_{E} \to \prod_i W_{E_i}$. The desired result now follows from Theorem~\ref{Thm:WeilCohShimVar} and Proposition \ref{Prop:AcceptableInvariants}.
\end{proof}
\begin{Rem}
For $p>2$ and $G_{\qp}$ unramified, this is proved with $\zlbar$ coefficients for the basic part of the cohomology of the Shimura variety by Li-Huerta \cite[Theorem A]{LiHuertaPlectic}.
\end{Rem}

\subsubsection{} Now let $K_p$ be a hyperspecial parahoric subgroup of $G(\qp)$ with integral model $\mathcal{G}$, and let $\mathcal{H}_{\mathcal{G}}$ be the Hecke algebra for $K_p$ with $\Lambda$-coefficients as before. Note that we can write $\mathcal{G}=\prod_i \mathcal{G}_i$ and $\mathcal{H}_{\mathcal{G}}=\otimes_i \mathcal{H}_{\mathcal{G}_i}$. Let $k_E$ be the residue field of $E$, with  cardinality $q$, and denote by $\operatorname{Frob}_q\in \operatorname{Gal}(\overline{k}_E/k_E)$ the arithmetic Frobenius. Fix a neat compact open subgroup $K^p \subset \gafp$ which is acceptable with respect to $K_p$. Recall that $R\Gamma(\mathbf{Sh}_{K}(\mathsf{G},\mathsf{X})_{\overline{E}}, \mathbb{W})$ is an unramified representation of $W_E$.\footnote{This follows from the existence of smooth integral models due to Kisin \cite[Main theorem]{KisinModels} and Kim--Madapusi \cite[Theorem 1]{KimMadapusi}, and the comparison between the cohomology of the special fiber and the generic fiber of Lan--Stroh \cite[Corollary 4.6]{LanStrohII}.} We obtain the following corollary.
\begin{Cor} \label{Cor:PartialEichlerShimura}
If $\ell$ is coprime to the order of $\pi_0(Z(G))$, then there exist mutually commuting partial Frobenius operators $\sigma_{i}$ acting on $R\Gamma(\mathbf{Sh}_{K}(\mathsf{G},\mathsf{X})_{\overline{E}}, \mathbb{W})$, such that for any lift $\sigma$ of $\operatorname{Frob}_q$, we have
    \[\sigma=\prod_{i=1}^r \sigma_i\]
as endomorphisms of $R\Gamma(\mathbf{Sh}_{K}(\mathsf{G},\mathsf{X})_{\overline{E}}, \mathbb{W})$. Moreover, each $\sigma_i$ is annihilated by the renormalized Hecke polynomial $H_{\mu_i}(X)\in \mathcal{H}_{\mathcal{G}_i}[X]$ constructed in Section \ref{subsub:EichlerShimura}.
\end{Cor}

\begin{proof}
Let $\mathcal{V}_\mu\simeq\boxtimes_{i=1}^r\mathcal{V}_{\mu_i}$ be the decomposition as in the proof of Theorem~\ref{Thm:Plectic}. For each $\mathcal{V}_{\mu_i}$, we can apply the discussion in Section~\ref{Sub:ES}. Namely, for each element $\gamma\in W_{E_i}$, its action on $\mathcal{V}_{\mu_i}$ is annihilated by a (renormalized) minimal polynomial $H_{\mu_i,\gamma}(X)$, with coefficients in the spectral Bernstein center for $H_i$. Choose an arithmetic Frobenius element $\sigma\in W_{E}$ and denote its image under the map $W_{E}\to \prod_i W_{E_i}\to W_{E_i}$ by $\sigma_{i}$. We thus have the relation
\[\sigma=\boxtimes_{i=1}^r\sigma_{i}\]
as endomorphisms on $\mathcal{V}_\mu$. It is annihilated by the polynomial
\[H_{\mu,\sigma}(X)=\prod_{i=1}^r H_{\mu_i,\sigma_{i}}(X).\]
When restricted to the unramified locus on $\Par_G$, i.e. where the universal $L$-parameter (up to $\widehat{G}$-conjugation) is trivial on the inertia subgroup of $W_{\qp}$, this polynomial, as well as each $H_{\mu_i,\sigma_{i}}$, becomes independent of the choice of $\sigma$. The resulting polynomial we get is the Hecke polynomial and we relabel it as $H_\mu(X)$. It has coefficients in $\mathcal{H}_{\mathcal{G}}$; similarly we have the $H_{\mu_i}(X)$'s, which have coefficients in the $\mathcal{H}_{\mathcal{G}_i}$'s. We can now invoke the argument in Theorem~\ref{Thm:EichlerShimura} to conclude the desired result about the cohomology of Shimura varieties.   
\end{proof}
\begin{Rem}
For $E=\qp$, Corollary \ref{Cor:PartialEichlerShimura} is also proved by Lee \cite[Theorem 1.0.4]{Lee} by a completely different method.
\end{Rem}

}

{\section{Cohomology of some compact Shimura varieties} \label{Sec:CohCompact} In this section we study the cohomology of compact Shimura varieties of abelian type. We focus on particular examples in types $A, B, C, D^{\mathbb{R}}$, and $ D^{\mathbb{H}}$ in which the existence of global Galois representations or $L$-parameters is known, and where these are known to be attached to Hecke-eigensystems appearing in the cohomology of Shimura varieties. We then use Theorem \ref{Thm:SV-FS-Compatibility} to prove a weak local-global compatibility for these Galois representations or $L$-parameters, see Corollary \ref{Cor:WeakCompatibility}. For certain classical groups of type $A,B,D$, and $C_2$, we prove a strong local-global compatibility, see Corollary \ref{Cor:StrongCompatibility}. This will be used in Section \ref{Sec:Compatibility} to prove Theorem \ref{Thm:IntroCompatibility}. 

\subsection{Abstract setup}
Let $\gx$ be a Shimura datum of abelian type with reflex field $\mathsf{E}$, and let $v$ be a finite place of $\mathsf{E}$ above a rational prime $p$. Assume that $\g^{\mathrm{ad}}$ is $\Q$-anisotropic, so that the Shimura varieties for $\gx$ are projective. We let $W$ be an algebraic representation of $\g^c$ over $\qlbar$, for some $\ell \neq p$ with $\ell$ coprime to the order of $\pi_0(Z(G))$. We denote by $\mathbb{W}$ the associated $\ell$-adic local system on $\mathbf{Sh}_K\gx$ for $K \subset \gaf$ neat, and write
\begin{align*} 
     H^i_{\mathrm{sm}}(\mathbf{Sh}\gx_{\ebar}, \mathbb{W}):=\varinjlim_K H^i(\mathbf{Sh}_K\gx_{\ebar}, \mathbb{W}),
\end{align*}
equipped with its continuous and commuting actions of $\gal_{\mathsf{E}}$ and $\gaf$. 

\subsubsection{} Fix an isomorphism $\iota_{\ell}:\qlbar \to \mathbb{C}$. Recall that Matsushima's formula \cite[Theorem 3.2]{BorelWallach} implies that we have a direct sum decomposition
\begin{align*}
    H_{\mathrm{sm}}^\ast(\mathbf{Sh}\gx_{\ebar}, \mathbb{W})(d/2) = \bigoplus_{\Pi^{\infty}} \Pi^{\infty} \otimes_{\qlbar} \sigma^\ast(\Pi^{\infty}),
\end{align*}
of smooth $\gaf \times \gal_{\mathsf{E}}$-representations, where $\gaf$ acts on $\Pi^{\infty}$, and $\gal_\mathsf{E}$ acts on $\sigma^\ast(\Pi^\infty)\coloneqq \operatorname{Hom}(\Pi^\infty, H^\ast(\mathbf{Sh}\gx_{\ebar}, \mathbb{W})(d/2))$, which is the isotypic part for $\Pi^\infty$ as a Hecke module. The direct sum runs over smooth representations $\Pi^{\infty}$ of $\gaf$ over $\qlbar \xrightarrow{\sim} \mathbb{C}$ such that there exists a $(\mathfrak{g},K)$-module $\Pi_{\infty}$ over $\mathbb{C}$ such that $\Pi=\Pi^{\infty} \otimes \Pi_{\infty}$ is a discrete automorphic representation of $\g(\mathbb{A})$, and such that $\Pi_{\infty}$ is cohomological with infinitesimal character equal to that of $W$.

\begin{Cor} \label{Cor:WeakCompatibility}
Assume that $\ell$ is coprime to the order of $\pi_0(Z(G))$. Fix $\Pi^{\infty}$ as above, and let $\sigma \subset \sigma^i(\Pi^{\infty})$ be a $W_E$-subrepresentation of dimension equal to $\dim V_{\mu}$. If $\sigma$ is $W_E$-multiplicity free\footnote{That is, its Jordan--H\"older factors are pairwise non-isomorphic.}, then $\sigma^{\mathrm{ss}} \simeq r_{\mu} \circ \restr{\phi^{\mathrm{FS}}_{\Pi_{p}}}{W_E}$ as representations of $W_E$.
\end{Cor}

\begin{proof}
Let $\chi$ be the character of $\mathcal{Z}(G(\qp), \qlbar)$ corresponding to $\Pi_p$ and consider the localization $H^i_\mathrm{sm}(\mathbf{Sh}\gx_{\ebar} \mathbb{W})_{\chi}$ of $H^i_\mathrm{sm}(\mathbf{Sh}\gx_{\ebar} \mathbb{W})$ at $\chi$. Then it follows from the semisimplicity of the Hecke action discussed above that we have a direct sum decomposition 
\begin{align*}
  H_{\mathrm{sm}}^i(\mathbf{Sh}\gx_{\ebar} \mathbb{W})_{\chi} \simeq \bigoplus_{\chi_{\Pi_{1,p}}=\chi} \Pi_1^{\infty} \otimes_{\qlbar} \sigma^i(\Pi_1^{\infty}),
\end{align*}
where $\Pi_{1}$ runs over cuspidal automorphic representations of $\g(\mathbb{A})$ such that $\Pi_{1,\infty}$ is cohomological for $W$. A choice of nonzero vector $v$ in $\Pi^{\infty}$ thus exhibits
\begin{align*}
    v \otimes \sigma \subset \Pi^{\infty} \otimes \sigma \subset H^i_\mathrm{sm}(\mathbf{Sh}\gx_{\ebar} \mathbb{W})_{\chi}
\end{align*}
as a subrepresentation. It then follows from Theorem \ref{Thm:SV-FS-Compatibility} that all the irreducible $W_E$-subquotients of $\sigma$ also occur as irreducible $W_E$-subquotients of $r_{\mu} \circ \restr{\phi^{\mathrm{FS}}_{\Pi_{p}}}{W_E}$. Since $\sigma$ has the same dimension as $r_{\mu} \circ \restr{\phi^{\mathrm{FS}}_{\Pi_{p}}}{W_E}$ and is multiplicity free, it follows that $r_{\mu} \circ \restr{\phi^{\mathrm{FS}}_{\Pi_{p}}}{W_E}$ is also multiplicity free. We deduce that $\sigma$ and $r_{\mu} \circ \restr{\phi^{\mathrm{FS}}_{\Pi_{p}}}{W_E}$ have the same set of irreducible $W_E$-subquotients, and thus have isomorphic semisimplifications. To finish the proof, we note that $\phi_{\Pi_{p}}^{\mathrm{FS}}$ is semisimple and hence $r_{\mu} \circ \restr{\phi^{\mathrm{FS}}_{\Pi_{p}}}{W_E}$ is semisimple. Indeed, since we are in characteristic zero $\phi_{\Pi_{p}}^{\mathrm{FS}}$ being semisimple is equivalent to the statement that the identity component of the Zariski closure of the image of $\phi_{\Pi_{p}}^{\mathrm{FS}}$ is reductive, see e.g. \cite[Section 2.2]{CompleteReducibility}. This is clearly preserved under restriction to $W_E$ and composition with $r_{\mu}$, and thus $r_{\mu} \circ \restr{\phi^{\mathrm{FS}}_{\Pi_{p}}}{W_E}$ is semisimple. 
\end{proof}

\begin{Rem}
We think of Corollary \ref{Cor:WeakCompatibility} as a way to turn knowledge of the cohomology of Shimura varieties into knowledge about the Fargues--Scholze local Langlands correspondence; this is how it will be applied in the rest of this article. 
\end{Rem}

\subsection{Examples} \label{sub:examples} We will now specialize the setting of the previous section to one of the following cases. In all cases $\mlp \neq\mathbb{Q}$ will denote a totally real field. 
\begin{enumerate}
\item[$A$] Let $\mathsf{L}$ be a quadratic CM extension of $\mlp$ and let $\h$ be an inner form of the quasi-split unitary group (defined by an $\mathsf{L}/\mathsf{L}^+$-hermitian space) over $\mlp$ such that for precisely one embedding $\tau_0:\mlp \to \mathbb{R}$ the group $\h \otimes_{\mlp, \tau_0} \mathbb{R}$ is isomorphic to $\operatorname{U}(1,n-1)$ for $n>2$, and for all other real embeddings $\tau$ the group has compact modulo center $\R$-points. We let $N(\h)=n$.

\item[$B$] Let $V$ be a quadratic space over $\mlp$ with special orthogonal group $\h$, such that for precisely one embedding $\tau_0:\mlp \to \mathbb{R}$ the group $\h \otimes_{\mlp, \tau_0} \mathbb{R}$ is isomorphic to $\operatorname{SO}(2,2n-1)$, and for all other real embeddings $\tau$ the group is isomorphic to $\operatorname{SO}(0,2n+1)$. In this case we set $\ml=\mlp$ and $N(\h)=2n$.

\item[$C$] Let $\h$ be an inner form of a split general symplectic group $\operatorname{GSp}_{2g,\mlp}$ such that for precisely one embedding $\tau_0:\mlp \to \mathbb{R}$ the group $\h \otimes_{\mlp, \tau_0} \mathbb{R}$ is split, and for all other embeddings $\tau:\mlp \to \mathbb{R}$ the group $\h \otimes_{\mlp, \tau} \mathbb{R}$ has compact modulo center $\R$-points. In this case we set $\ml=\mlp$ and $N(\h)=1+\tfrac{g(g+1)}{2}$.

\item[$D^{\mathbb{R}}$] Let $V$ be a quadratic space over $\mlp$ with special orthogonal group $\h$, such that for precisely one embedding $\tau_0:\mlp \to \mathbb{R}$ the group $\h \otimes_{\mlp, \tau_0} \mathbb{R}$ is isomorphic to $\operatorname{SO}(2,2n-2)$, and for all other real embeddings $\tau$ the group is isomorphic to $\operatorname{SO}(0,2n)$. In this case we set $\ml=\mlp$ and $N(\h)=2n$.

\item[$D^{\mathbb{H}}$] Let $n$ be an integer and let $\ml$ be an extension of $\mlp$ which is totally imaginary quadratic if $n$ is odd and equal to $\mlp$ if $n$ is even. Let $\h^{\ast}$ be the (quasi-split) connected reductive group denoted by $\operatorname{GSO}^{\ml/\mlp}_{2n}$ in \cite[beginning of Section 6]{KretShinOrthogonal}. Let $\h$ be an inner form of $\h^{\ast}$ such that for precisely one embedding $\tau_0:\mlp \to \mathbb{R}$ the group $\h \otimes_{\mlp, \tau_0} \mathbb{R}$ is isomorphic to the group $\operatorname{GSO}^J_{2n}$ of \cite[Section 8, discussion before Lemma 8.1]{KretShinOrthogonal}, and for all other embeddings $\tau:\mlp \to \mathbb{R}$ the group $\h \otimes_{\mlp, \tau} \mathbb{R}$ has compact modulo center $\R$-points. We let $N(\h)=1+\tfrac{n(n-1)}{2}$.

\end{enumerate}
In all of these cases we set $\g=\operatorname{Res}_{\mlp/\mathbb{Q}} \h$ and equip it with the Shimura datum $\x$ which is trivial at all real embeddings except $\tau_0$, and which at $\tau_0$ is given by:
\begin{itemize}
  \item[$D^{\mathbb{H}}$] For $\epsilon \in \{-,+\}$, the Shimura datum $\x^{\epsilon}$ described in \cite[Equation (9.1)]{KretShinOrthogonal}.

  \item[other cases] The conjugacy class of the morphism $\mathbb{S} \to \h \otimes_{\mlp,\tau_0} \mathbb{R}$ described in \cite[Appendix B]{Milne}. To be precise, in type $B, D^{\mathbb{R}}$, the morphism described in loc. cit. lands in $\operatorname{GSpin}(V)$ and we consider its composition with $\operatorname{GSpin}(V) \to \operatorname{SO}(V)$. 
\end{itemize}
It follows from the classification of abelian type Shimura data, see \cite[Appendix B]{Milne}, that $\gx$ is of abelian type. It follows from a standard argument, see e.g. \cite[Remark 2.3.12]{DeligneVarietes} that the reflex field is given by $\ml \subset \mathbb{C}$, where the embedding $\ml \to \mathbb{C}$ is one that extends $\tau_0:\mlp \to \mathbb{R}$. Note that since $\mlp \neq\mathbb{Q}$, the group $\g^{\mathrm{ad}}$ is $\mathbb{Q}$-anisotropic. 

\subsubsection{} Fix a rational prime $p$ and let $G=\g \otimes \qp$. Then we have a direct product decomposition 
\begin{align*}
    G = \prod_{\mathfrak{p} | p} \operatorname{Res}_{\mlp_{\mathfrak{p}}/\qp} \h_{\mlp_{\mathfrak{p}}}
\end{align*}
over primes $\mathfrak{p}$ above $p$ in $\mlp$, and 
\begin{align} \label{Eq:GeometricProductDecompositionGroup}
    G_{\qpbar} = \prod_{\tau} \h \otimes_{\mlp,\tau} \qpbar
\end{align}
over embeddings $\tau:\mlp \to \qpbar$. We have the following description of the dual group of $H_{\mathfrak{p}}=\h_{\mlp_{\mathfrak{p}}}$:
\begin{itemize}
    \item[$A$] If $\mathfrak{p}$ splits in $\ml$, then $H_{\mathfrak{p}}=\operatorname{GL}_{n}$ and its dual group is $\operatorname{GL}_{n}$ with trivial $W_{\mlp_{\mathfrak{p}}}$-action. If $\mathfrak{p}$ does not split, then the dual group of $H_{\mathfrak{p}}$ is given by $\operatorname{GL}_{n}$ with $W_{\mlp_{\mathfrak{p}}}$-action factoring through $\gal(\ml_{\mathfrak{p}}/\mlp_{\mathfrak{p}})$ with the nontrivial element acting by transpose inverse. 

    \item[$B$] The dual group of $H_{\mathfrak{p}}$ is given by $\operatorname{Sp}_{2n}$ with trivial $W_{\mlp_{\mathfrak{p}}}$-action. 
    
    \item[$C$] The dual group of $H_{\mathfrak{p}}$ is given by the general spinor group $\operatorname{GSpin}_{2n+1}$ with trivial $W_{\mlp_{\mathfrak{p}}}$-action, see \cite[Section 1]{KretShinSymplectic}.
    \item[$D^{\mathbb{R}}$] The dual group of $H_{\mathfrak{p}}$ is given by $\operatorname{SO}_{2n}$ with the action of $ W_{\mlp_{\mathfrak{p}}}$ determined by the discriminant $\operatorname{Disc}(V \otimes_{\mlp} \mlp_{\mathfrak{p}}) \in (\mlp_{\mathfrak{p}})^{\times}/((\mlp_{\mathfrak{p}})^{\times})^2$ of the quadratic space $V \otimes_{\mlp} \mlp_{\mathfrak{p}}$: If the discriminant is $1$, then the action of $ W_{\mlp_{\mathfrak{p}}}$ is trivial. Otherwise, the discriminant determines a quadratic extension $K/\mlp_{\mathfrak{p}}$ and the action of $ W_{\mlp_{\mathfrak{p}}}$ factors through $\operatorname{Gal}(K/\mlp_{\mathfrak{p}})$ via the unique outer automorphism of $\operatorname{SO}_{2n}$.

    \item[$D^{\mathbb{H}}$] The dual group of $H_{\mathfrak{p}}$ is given by $\operatorname{GSpin}_{2n}$ with trivial $W_{\mlp_{\mathfrak{p}}}$-action if $n$ is even, and if $n$ is odd with $W_{\mlp_{\mathfrak{p}}}$-action factoring through $\gal(\ml_{\mathfrak{p}}/\mlp_{\mathfrak{p}})$ with the nontrivial element acting by the involution $\theta$ of \cite[discussion before Lemma 2.3]{KretShinOrthogonal}. 
\end{itemize}

\subsubsection{} \label{subsub:DualGroupRepresentations} Fix an isomorphism $\mathbb{C} \to \qpbar$, which induces a morphism $\tau_0:\ml \to \qpbar$. The induced conjugacy class of cocharacters $\mu$ of $G(\qpbar)$ induced by $v$ and $\x$ is trivial (in the decomposition of \eqref{Eq:GeometricProductDecompositionGroup}) in the factors corresponding to $\tau \neq\tau_0$. The cocharacter $\mu_{\tau_{0}}$ corresponds to the following (highest weight) representation of the dual group of $H$:
\begin{itemize}
    \item[$A$] The identity $\operatorname{GL}_n \to \operatorname{GL}_n$.

    \item[$B$] The standard representation $\operatorname{Sp}_{2n} \to \operatorname{GL}_{2n}$.

    \item[$C$] The spin representation $\operatorname{GSpin}_{2n+1} \to \operatorname{GL}_{2^n}$.

    \item[$D^{\mathbb{R}}$] The standard representation $\operatorname{SO}_{2n} \to \operatorname{O}_{2n} \to \operatorname{GL}_{2n}$.  

    \item[$D^{\mathbb{H}}$] The half-spin representation $\operatorname{spin}^{\epsilon}: \operatorname{GSpin}_{2n} \to \operatorname{GL}_{2^{n-1}}$ of \cite[Definition 2.6]{KretShinOrthogonal}, see \cite[Section 4, Equation (4.2)]{KretShinOrthogonal}.
\end{itemize}
We denote by $\sim^0$ the equivalence relation on $\widehat{\h}$-valued
$L$-parameters given by $\widehat{\h}$-conjugacy in types $A,B, C,$ and $
D^{\mathbb{H}}$, and given by $\operatorname{O}_{2n}$-conjugacy in type
$D^{\mathbb{R}}$.

\subsection{The existence of global Galois representations }\label{sec: GlobalParameters}
In this section we state the existence of global Galois representations or $L$-parameters for the groups in Section \ref{sub:examples}, under some assumptions. 

\subsubsection{} Recall the global $L$-group $\LH = \widehat{\mathsf{H}} \rtimes \operatorname{Gal}_{\mlp}$, considered as a group scheme over $\mathbb{Q}$, where the Galois group is considered as a pro-finite algebraic group. 

\subsubsection{} \label{subsub:Satake} For a finite place $v$ of $\mlp$ we also consider the local $L$-group
\begin{align*}
    \lH &= \widehat{\mathsf{H}} \rtimes W_{\mlp_v}. 
\end{align*}
If $\mathsf{H}$ is unramified at $v$, then we recall from
\cite[Section 2.2]{BuzzardGee} that to an irreducible smooth unramified representation $\pi$ of $\mathsf{H}(\mlp_v)$ with $\mathbb{C}$-coefficients, we may associate an $\widehat{\mathsf{H}}(\mathbb{C})$-conjugacy class of morphisms $\phi_{\pi}:W_{\mlp_v} \to \lH$; this is the Satake parameter of $\pi$.

\subsubsection{} \label{subsub:NonSplitEvenOrthogonal} Recall that for a finite place $v'$ of $\ml$ above $v$, the cocharacter $\mu$ (depending on $v'$) induces a representation
\begin{align}
    r_{\mu}: \widehat{\mathsf{H}} \rtimes W_{\ml_{v'}} \to \operatorname{GL}(V_{\mu}).
\end{align}
Note that the action of $W_{\ml_{v'}}$ on $\widehat{\mathsf{H}}$ is trivial in all cases, except in type $D^{\mathbb{R}}$ where it factors through the natural map $W_{\mlp_{v}}=W_{\ml_{v'}} \to \gal(K/\mlp_{v})$. If $K \neq \mlp_v$, then the semidirect product $\widehat{\mathsf{H}} \rtimes \gal(K/\mlp_{v}) = \operatorname{SO}_{2n} \rtimes \gal(K/\mlp_{v})$ is isomorphic to $\operatorname{O}_{2n}$, see \cite[first Table in Section 7]{GGP}. Moreover, the representation $r_{\mu}$ is the standard representation $\operatorname{O}_{2n} \to \operatorname{GL}_{2n}$; this is well known.\footnote{We recall that $r_{\mu}$ is the unique extension of the standard representation of $\operatorname{SO}_{2n}$, such that $1 \rtimes \gal(K/\mlp_{v})$ acts trivially on the highest weight vector, see \cite[Lemma 2.1.2]{KottwitzOrbital}. Consider the description of $\operatorname{SO}_{2n} \subset \operatorname{GL}_{2n}$ of \cite[Section 18.1]{FultonHarris} with $T$ the torus $(t_1, \dotsc, t_n, t_1^{-1}, \dotsc, t_n^{-1})$, and $B$ the Borel given by the stabilizer of the isotropic flag $\langle e_1 \rangle \subset \dotsb \subset \langle e_1, \dotsc , e_n \rangle$. Then the element $g \in \operatorname{GL}_{2n}$ swapping $e_{n}$ and $e_{2n}$ lies in $\operatorname{O}_{2n}$, and conjugation by $g$ induces the correct outer automorphism of $\operatorname{SO}_{2n}$. It also preserves the highest weight vector, which is $e_1$.}

\subsubsection{} To state the existence of global Galois representations, we need to formulate some assumptions. Let $\Pi$ be a cuspidal automorphic representation of $\h(\mathbb{A}_{\mlp})$, with $\h$ as in Section \ref{sub:examples}, and consider the following assumptions on $\Pi$ and $\h$. Let $S_{\mathrm{sc}}$ be the set of finite places $w$ of $\mlp$ where $\Pi_w$ is supercuspidal.
\begin{enumerate}
    \item[\textbf{coh}] \customlabel{\textup{\textbf{coh}}}{ass:coh} For all real places $\infty$ of $\mlp$, the representation $\Pi_{\infty}$ is cohomological in the sense of \cite[Definition 1.12]{KretShinSymplectic}. 

    \item[\textbf{coh-reg}] \customlabel{\textup{\textbf{coh-reg}}}{ass:coh-reg} Assumption \ref{ass:coh} holds. Moreover, the representation $\Pi$ is std-regular in the sense of \cite[Definition 3.2.1]{ShinWeakTransfer}.
    
    \item[\textbf{ssc}] \customlabel{\textup{\textbf{ssc}}}{ass:ssc} In cases $A$, $B$, and $D^{\mathbb{R}}$, there is a place $w_{\mathrm{s}}$ of $\mlp$ such that the classical $L$-parameter\footnote{By the \emph{classical $L$-parameter} of $\Pi_{w}$ we mean the $L$-parameter associated to $\Pi_{w}$ by the local Langlands correspondences listed in Section \ref{subsub:ClassicalGroups}, see \cite[Theorem 2.5.1, 3.2.1]{Mok} and \cite[Theorem 1.6.1]{KMSW} in type $A$, \cite[Theorem 1.5.1]{Arthur} and \cite[Theorem 1.2]{Ishimoto} in type $B$, \cite[Theorem 1.5.1]{Arthur} and \cite[Theorem A.2]{ChenZou} in type $D^{\mathbb{R}}$, and \cite[Main Theorem]{GanTakeda} and \cite[Main Theorem]{GanTantono} in type $C_{2}$.} of $\Pi_{w_{\mathrm{s}}}$ has the property that $r_{\mu} \circ \phi_{\Pi_{w_{\mathrm{s}}}}$ is irreducible, cf. \cite[Section 2.2]{Peng}.

    \item[\textbf{simgen}] \customlabel{\textup{\textbf{simgen}}}{ass:simgen} In cases $A$, $B$, and $D^{\mathbb{R}}$, the formal (Arthur) parameter $\Psi$ of $\Pi$ consists of a single cuspidal representation, see \cite[Theorem 4.2.1, Definition 4.3.3 of v3]{Peng}.\footnote{The notation \ref{ass:simgen} stands for \emph{simple generic}. We note that $\Pi$ satisfies \ref{ass:simgen} precisely when its Arthur parameter is simple and generic in the sense of Arthur.} 
    
    \item[\textbf{St}] \customlabel{\textup{\textbf{St}}}{ass:St} There is a finite place $v_{\mathrm{st}}$ of $\mlp$ such that $\Pi_{v_{\mathrm{st}}}$ is isomorphic to an unramified character twist of the Steinberg representation of $\h(\mlp_{v_{\mathrm{st}}})$. Furthermore, the group $\h$ is quasi-split at all finite places except possibly at $w \in S_{\mathrm{sc}}$ and $v_{\mathrm{st}}$. In type $C$ and $D^{\mathbb{H}}$, we further assume that $\h$ is quasi-split at $v \in S_{\mathrm{sc}}$. 

    \item[\textbf{St}$'$] \customlabel{\textup{\textbf{St}}$'$}{ass:Stprime} Assumption \ref{ass:St} holds. Moreover, in cases $A$, $B$, and $D^{\mathbb{R}}$, the place of $\mathbb{Q}$ underlying $v_{\mathrm{st}}$ is inert in $\mlp$ and is not $2$. 
    \end{enumerate}
Then we have the following theorem due to many mathematicians. 
\begin{Thm} \label{Thm:ExistenceLParameters}
Assume \cite[Assumption H1]{ShinWeakTransfer}. Fix any rational prime $\ell$ and an isomorphism $\iota_{\ell}:\mathbb{C}\xrightarrow{\sim}\qlbar$. Let $\Pi$ be a cuspidal automorphic representation of $\h(\mathbb{A}_{\mlp})$.
\begin{enumerate}[$($i$\,)$]
    \item If $\mathsf{H}$ is of type $C$ or $D^{\mathbb{H}}$ and $\Pi$ satisfies \ref{ass:coh-reg} and \ref{ass:St}, then there is a unique (up to $\sim^0$) continuous semisimple Galois representation 
    \[\rho_{\Pi}:\operatorname{Gal}_{\mlp} \to \LH(\qlbar)\]
    satisfying: a) Its composition with the projection $\LH(\qlbar) \to \operatorname{Gal}_{\mlp}$ is the identity map. b) For all finite places $w\neq v_\mathrm{st}$, $w\nmid \ell$ of $\mlp$ where $\h$ and $\Pi$ are unramified, we have 
    \[\restr{\rho_{\Pi}}{W_{\mlp_w}} \sim \iota_{\ell}(\phi_{\Pi_{w}}) \otimes \chi_{\ell}^{\tfrac{-1+N(\h)}{2}}\]
    where $\phi_{\Pi_{w}}$ is as in Section \ref{subsub:Satake} and $\chi_{\ell}$ is the $\ell$-adic cyclotomic character.
    \item If $\mathsf{H}$ is of type $A$, $B$, or $D^{\mathbb{R}}$ and $\Pi$ satisfies \ref{ass:coh-reg}, \ref{ass:ssc}, \ref{ass:St}, \ref{ass:simgen}, then there is a unique continuous semisimple Galois representation 
    \begin{align*}
    \rho_{\Pi}:\operatorname{Gal}_{\ml} \to \operatorname{GL}(V_{\mu})(\qlbar)
    \end{align*}
    such that for all finite places $w\nmid \ell$ of $\ml$ with restriction $w^+$ to $\mlp$, we have 
    \begin{align*}
    \restr{\rho_{\Pi}}{W_{\ml_w}} \simeq \iota_{\ell}(r_{\mu} \circ \restr{\phi_{\Pi_{w^+}}}{W_{\ml_{w}}}) \otimes \chi_{\ell}^{\tfrac{-1+N(\h)}{2}},
\end{align*}
where $\phi_{\Pi_{w^+}}$ is the classical $L$-parameter of $\Pi_{w^+}$.
\end{enumerate}
\end{Thm}

\begin{proof}
For part (i), we first transfer $\Pi$ to a cuspidal automorphic representation $\Pi^{\sharp}$ of $\h^{\ast}(\mathbb{A}_{\mlp})$, where $\h^{\ast}$ is the quasi-split inner form of $\h$, such that $\Pi_w$ and $\Pi^{\sharp}_w$ are isomorphic at the unramified places $w$ described in the theorem; this is possible by \cite[Proposition 6.3]{KretShinSymplectic} and assumption \ref{ass:St}. Note that if $\Pi$ is standard regular, then it is regular by \cite[Lemma 3.2.2]{ShinWeakTransfer}. Hence in particular, if it is $\xi$-cohomological, then $\xi$ is regular. It now suffices to prove the analogous theorem for $\Pi^{\sharp}$. In type $C$, the result is \cite[Theorem A]{KretShinSymplectic}, and in type $D^{\mathbb{H}}$ it is \cite[Theorem A]{KretShinOrthogonal}. 

Part (ii) follows from work of many authors, see \cite[Theorem 4.3.6]{Peng} for the statement and attributions, noting that our $\chi_{\ell}$ is the inverse of his norm character $|-|$. Namely, first by \cite{Arthur}, \cite{Mok}, \cite{KMSW}, \cite{Ishimoto}, \cite{ChenZouMultiplicity} we can reduce to the construction of Galois representations for cuspidal automorphic representations of $\operatorname{GL}_N$ and its local-global compatibility, which is due to \cite{Clozel}, \cite{HarrisTaylor}, \cite{TaylorYoshida}, \cite{ShinCohomology}, \cite{CaraianiI}, \cite{ChenevierHarris}.
\end{proof}

\subsection{Some local-global compatibility} We now state some local-global compatibility for the Galois representations from the conclusion of Theorem \ref{Thm:ExistenceLParameters}. 

\subsubsection{} In type $C$ and $D^{\mathbb{H}}$ we define $R_{\Pi} = r_{\mu} \circ \restr{\rho_{\Pi}}{W_{\ml}}$. We have the following corollary of Theorem \ref{Thm:ExistenceLParameters} and Corollary \ref{Cor:WeakCompatibility}. \textbf{We will assume in the rest of Section \ref{Sec:CohCompact} that \cite[Assumption H1]{ShinWeakTransfer} holds}.\footnote{The results of \cite{KretShinOrthogonal} and \cite{KretShinSymplectic} used in the proof of Corollary \ref{Cor:StrongCompatibility} below, also rely on the work of Arthur \cite{Arthur}. By the work of \cite{AGIKMS}, they are thus (only) conditional on \cite[Assumption H1]{ShinWeakTransfer}. }
\begin{Cor} \label{Cor:StrongCompatibility}
Let $\Pi$ be a cuspidal automorphic representation of $\h(\mathbb{A}_{\mlp})$ satisfying 
\ref{ass:coh-reg}, \ref{ass:ssc}, \ref{ass:simgen} and \ref{ass:Stprime}. Let $\ell$ be a prime together with an isomorphism $\iota_{\ell}:\mathbb{C} \xrightarrow{\sim} \qlbar$, and assume that $\ell \neq2$ in type $D^{\mathbb{R}}$. If $v$ is a prime of $\ml$ not dividing $\ell$, with induced prime $v^+$ of $\mlp$ such that $v^+ \neq v_{\mathrm{st}}$ and such that $\restr{R_{\Pi}}{W_{\ml_{v}}}$ is a multiplicity free representation of $W_{\ml_{v}}$, then
\begin{align*}
    \left(\restr{R_{\Pi}}{W_{\ml_{v}}}\right)^{\operatorname{ss}}(d/2) &\simeq r_{\mu} \circ \restr{\phi_{\Pi_{v^{+}}}^{\mathrm{FS}}}{W_{\ml_{v}}}.
\end{align*}
\end{Cor}
\begin{proof}
We consider $\sigma^i(\Pi^{\infty})^{\operatorname{ss}}$ as in Corollary \ref{Cor:WeakCompatibility}, with $i$ the dimension of the Shimura variety and with $W$ the automorphic local system corresponding to the algebraic representation with respect to which $\Pi_{\infty}$ is cohomological. In type $C$, it follows from \cite[Corollary 8.7, discussion before Theorem 9.1, Corollary 12.4]{KretShinSymplectic} that under \ref{ass:coh} and \ref{ass:St} we have an isomorphism\footnote{Note that our definition of $\sigma^i(\Pi^{\infty})$ includes a Tate twist by $d/2$ (in type $C$ we have $d=\tfrac{g(g+1)}{2}$). The $L$-parameter $\rho_{\pi}$ (denoted by $\rho_{\pi}^C$ in \cite[Theorem 9.1]{KretShinSymplectic}) is normalized so that $r_{\mu} \circ \rho_{\Pi}$ occurs in the middle degree cohomology of the Shimura variety, see \cite[Theorem 10.3 and the discussion in the first paragraph of its proof]{KretShinSymplectic}, which means that $(r_{\mu} \circ \rho_{\Pi})(d/2)$ occurs in $\sigma^d(\Pi^{\infty})$. A similar comment applies to our citation of \cite{KretShinOrthogonal} in case $D^{\mathbb{H}}$. }
\begin{align*}
    \sigma^i(\Pi^{\infty})^{\operatorname{ss}} \simeq R_{\Pi}(d/2).
\end{align*}
The same result holds in type $D^{\mathbb{H}}$ by \cite[Theorem 9.6, Proposition 14.2]{KretShinOrthogonal} under \ref{ass:coh-reg} and \ref{ass:St}. Here we used Chebotarev density and Brauer--Nesbitt theorems. The result now follows from Corollary \ref{Cor:WeakCompatibility}. \smallskip 

In types $A$, $B$, and $D^{\mathbb{R}}$, the representation $R_{\Pi}$ is irreducible because $\restr{R_{\Pi}}{\gal_{\ml_{w_{\mathrm{s}}}}}$ is irreducible by assumption \ref{ass:ssc}. By \cite[Corollary 4.5.7]{Peng}, under assumptions \ref{ass:ssc}, \ref{ass:coh-reg}, \ref{ass:simgen} and \ref{ass:Stprime}, the $\gal_{\ml}$-representation $\sigma^i(\Pi^{\infty})$ contains a subrepresentation $\sigma$ isomorphic to $R_{\Pi}$.\footnote{Note that in the notation of \cite{Peng}, we have $N(\mathsf{G})-1=\operatorname{Dim}_{\mathbb{C}} \x=d$; this follows from inspection in all three cases, see \cite[Appendix B]{Milne}. Unwinding the definitions in loc. cit. and using \cite[Corollary 4.5.7]{Peng}, we see that $R_{\Pi}(d/2)$ occurs in $\sigma^i(\Pi^{\infty})$.} The result now follows from Corollary \ref{Cor:WeakCompatibility} and Theorem \ref{Thm:ExistenceLParameters}.(ii).
\end{proof}

\begin{Rem}
Assume that $\h$ is of type $C$ or $D^{\mathbb{H}}$. A particular consequence of Corollary \ref{Cor:StrongCompatibility} is that (under the assumptions of the corollary), the representation $$\left(r_{\mu} \circ \restr{\rho_{\Pi}}{W_{\ml_{v}}}\right)^{\mathrm{ss}}$$ only depends on $\Pi_{v^+}$ and not on $\Pi$. This result seems to be new in general.  
\end{Rem}

\subsubsection{} We have the following consequence of Corollary \ref{Cor:StrongCompatibility}. 
\begin{Cor} \label{Cor:StrongCompatibilityII}
Assume that $\h$ is of type $A$, $B$, or $ D^{\mathbb{R}}$, or that $\h$ is of type $C$ with $g=2$. Let $\Pi$ be a cuspidal automorphic representation of $\h(\mathbb{A}_{\mlp})$ satisfying \ref{ass:coh-reg}, \ref{ass:ssc}, \ref{ass:simgen} and \ref{ass:Stprime}. Let $\ell$ be a prime together with an isomorphism $\iota_{\ell}:\mathbb{C} \xrightarrow{\sim} \qlbar$, and assume that $\ell \neq2$ in type $D^{\mathbb{R}}$. If $v\nmid \ell$ is a prime of $\ml$ with induced prime $v^+$ of $\mlp$ such that $v^+ \not= v_{\mathrm{st}}$ and such that $r_{\mu} \circ \restr{\phi_{\Pi_{v^+}}}{W_{\ml_{v}}}$ is a multiplicity-free representation of $W_{\ml_{v}}$, then
\begin{align*}
    \left(\phi_{\Pi_{v^+}} \right)^{\operatorname{ss}} \sim^0 \phi_{\Pi_{v^{+}}}^{\mathrm{FS}},
\end{align*}
where $\phi_{\Pi_{v^{+}}}$ denotes the classical $L$-parameter of $\Pi_{v^+}$.
\end{Cor}
\begin{proof}
By Corollary \ref{Cor:StrongCompatibility}, we have an isomorphism
    \begin{align} \label{Eq:IdentityLParameterComposed}
        (\restr{R_{\Pi}}{W_{\ml_v}})^{\operatorname{ss}}(d/2) \simeq r_{\mu} \circ \restr{\phi_{\Pi_{v^{+}}}^{\mathrm{FS}}}{W_{\ml_v}}.
    \end{align}
Now we observe that $d=N(\h)-1$, and we deduce from Theorem \ref{Thm:ExistenceLParameters} that (in type $C_{2}$, we use \cite[Main Theorem (a)]{SorensenGSp4} instead\footnote{The Galois representation constructed by Sorensen in loc. cit. agrees with the one constructed by Kret--Shin, because of the uniqueness proved in \cite[Theorem A]{KretShinSymplectic}.})
\begin{align*}
     (\restr{R_{\Pi}}{W_{\ml_v}})^{\operatorname{ss}}(d/2) &= (r_{\mu} \circ \restr{\phi_{\Pi_{v^+}}}{W_{\ml_v}})^{\operatorname{ss}} \\
     &\simeq r_{\mu} \circ \restr{\phi^{\mathrm{ss}}_{\Pi_{v^+}}}{W_{\ml_v}},
\end{align*}
and recall our descriptions of $r_{\mu}$ in all the different types of Section \ref{subsub:DualGroupRepresentations}. By \cite[Theorem 8.1]{GGP}, it follows that in types $A$, $B$, and $D^{\mathbb{R}}$, equation \eqref{Eq:IdentityLParameterComposed} implies
\begin{align*}
    \phi_{\Pi_{v^+}}^{\mathrm{ss}} \sim^0 \phi^{\mathrm{FS}}_{\Pi_{v^+}}. 
\end{align*}
In type $C$ (where $\ml=\mlp$), we know that $\widehat{\h}=\operatorname{GSpin}_{5} \simeq \operatorname{GSp}_4$ and under this isomorphism the representation $r_{\mu}$ corresponds to the standard embedding $\operatorname{GSp}_4 \subset \operatorname{GL}_4$. For the similitude character $\operatorname{Sim}:\operatorname{GSp}_{4} \to \operatorname{GL}_{1}$, we know that 
\begin{align*}
    \left(\operatorname{Sim} \circ \restr{\rho_{\Pi}}{W_{\mlp_{v^+}}}\right)^{\operatorname{ss}}
\end{align*}
corresponds to the central character of $\Pi_{v^{+}}$ by \cite[Theorem A]{KretShin}. The same is true for 
\begin{align*}
    \operatorname{Sim} \circ \phi_{\Pi_{v^+}}^{\mathrm{FS}}
\end{align*}
by \cite[Theorem I.9.6.(iii)]{FarguesScholze}. By \cite[Lemma 6.1]{GanTakeda}, it follows from this and \eqref{Eq:IdentityLParameterComposed} that
\begin{align*}
    \phi_{\Pi_{v^+}} \sim \phi_{\Pi_{v^+}}^{\mathrm{FS}}.
\end{align*}
\end{proof}
\begin{Rem}
The proof of Corollary \ref{Cor:StrongCompatibilityII} does not work for $\operatorname{GSp}_{2g}$ for $g > 2$, because $L$-parameters valued in $\operatorname{GSpin}_{2n+1}$ are not determined by their composition with the spin representation; the same should hold in type $D^{\mathbb{H}}$ for $n>4$. Perhaps in the latter case for $n=4$, one might expect to prove that the parameters agree up to the outer $S_3$ action on $L$-parameters. 
\end{Rem}

}

{
\section{Compatibility of local Langlands correspondences} \label{Sec:Compatibility}

In this section we use the Igusa stacks for abelian type Shimura varieties to show compatibility between the construction of local Langlands correspondence for classical groups by \cite{Arthur, Ishimoto, Mok, KMSW}, and the semisimple local Langlands correspondence constructed by Fargues--Scholze. Recently, Peng \cite{Peng} has established similar results for groups over unramified extensions of $\qp$ with $p>2$, based on the method of \cite{HamannGSp4} and \cite{BMHN}. Building on his work, we generalize this compatibility to the ramified case. Compared to the method in the literature, the major difference lies in that our main geometric result allows us to directly relate Fargues--Scholze parameters to automorphic representations, whereas the previous literature relies on $p$-adic uniformization of the basic locus of Shimura varieties, which is available only in the unramified case. This argument was suggested to us by David Hansen and Linus Hamann. We thank them heartily.

\subsection{Statement of the compatibility}\label{Section:CompatibilityStatement}
Fix a rational prime $p$. Let $F$ be a finite extension of $\qp$ and let $G$ over $F$ be a connected reductive group. We denote by $W_F$ the Weil group of $F$ and write $\phantom{}^LG$ for the $L$-group of $G$. This is a semi-direct product $\widehat{G}\rtimes W_F$, where the Langlands dual group $\widehat{G}$ is viewed as a split reductive group over $\mathbb{Z}$ equipped with an algebraic action by $W_F$ that factors through a finite quotient. We choose another prime $\ell\neq p$ and fix an isomorphism $\iota: \qlbar \xrightarrow{\sim} \mathbb{C}$. Note that the positive square root $\sqrt{p} \in \mathbb{R}$ defines a square root $\sqrt{p}\in \qlbar$ using $\iota$. 

\subsubsection{} \label{subsub:LParameters} Recall that an $L$-parameter for $G$ is a continuous homomorphism
\[W_F\times \mathrm{SL}_2(\mathbb{C})\to {}^LG(\mathbb{C}),\]
that is semisimple on the $W_F$ factor, algebraic on the $\operatorname{SL}_2$ factor and commutes with the projections to $W_F$ on the source and target. We let $\Phi(G)$ be the set of $\widehat{G}(\mathbb{C})$-conjugacy classes of $L$-parameters for $G$. An $L$-parameter is \emph{tempered} if its 
restriction to $W_F$ has image that projects to a bounded subset of $\widehat{G}(\mathbb{C})$. We let $\Phi_\mathrm{temp}(G) \subset \Phi(G)$ be the subset of equivalence classes of tempered $L$-parameters. 

\subsubsection{} Let $\Pi(G)$ be the set of isomorphism classes of $\mathbb{C}$-valued irreducible representations of $G(F)$, and $\Pi(G)_\mathrm{temp} \subset \Pi(G)$ be the subset of isomorphism classes of tempered representations. We will consider a collection of maps with finite fibers, labeled by Levi subgroups $M\subseteq G$ (including $G$ itself)
\[\operatorname{LL}_M: \Pi_\mathrm{temp}(M)\to \Phi_\mathrm{temp}(M),\, \pi\mapsto \phi_\pi.\]
We will call such a collection of maps a \textit{candidate local Langlands correspondence for $G$}. For a candidate local Langlands correspondence $\{\LL_M\}_{M\subseteq G}$, we write $\Pi_\phi(G)$ for the fiber $\LL_G^{-1}(\phi)$ and call it the (candidate) $L$-packet of a Langlands parameter $\phi$.

\subsubsection{}\label{subsub: ParabolicInd} 
Using the unique existence of the Langlands quotient and a similar classification of $L$-parameters, a candidate Langlands correspondence extends to a construction of $L$-parameters for general (not necessarily tempered) irreducible representations. To state this, recall from \cite[Section 4.3]{Taibi} that for a parabolic subgroup $P\subset G$, with Levi quotient $M$, we have an $L$-embedding $\iota_M:{}^LM\to {}^LG$, which is well-defined up to conjugation by $\widehat{G}$. 
\begin{Lem} \label{Lem:TemperedToFull}
    Let $\{\LL_M\}_{M\subseteq G}$ be a candidate local Langlands correspondence for $G$. Then it extends to a collection of well-defined maps
    \[\LL_{M}: \Pi(M)\to \Phi(M)\]
    from the set of all irreducible smooth representations of $M(F)$ to the set of $\widehat{M}$-conjugacy classes of $L$-parameters for $M$.
\end{Lem}
\begin{proof}
This is \cite[Proposition 7.1]{SilbergerZink}. We briefly recall the constructions for the readers' convenience. Following the notation in \textit{loc. cit.}, the Langlands classification theorem, see \cite[Theorem 4.1]{Silberger} or \cite[Theorem 1.4]{SilbergerZink}, gives a bijection between $\Pi(M)$ and the set of triples $(P,\sigma, \nu)$, where $P\subseteq M$ is a standard parabolic subgroup, $\sigma$ is a tempered representation of its Levi quotient $L_P$ and $\nu\in X_F^\ast(L_P)_\mathbb{R}$ is a regular positive (with respect to $P$) character. The bijection takes $(P,\sigma,\nu)$ to the unique irreducible quotient $J(P,\sigma,\nu)$ of the normalized parabolic induction $i_P^M(\sigma\otimes\chi_\nu)$, for an unramified character $\chi_\nu$ determined by $\nu$.

On the other hand, \cite[Section 4.6]{SilbergerZink} gives a similar classification of $L$-parameters: There is a bijection between $\Phi(M)$ and the set of triples $(P, {}^t\phi, \nu)$, where $P\subseteq M$ is a standard parabolic, ${}^t\phi$ is a tempered $L$-parameter for $L_P$ and $\nu\in X_F^\ast(L_P)_\mathbb{R}$ is a regular positive character. The bijection takes $(P, {}^t\phi, \nu)$ to the composition of a twist ${}^t\phi_{z(\nu)}$ of ${}^t\phi$ by $\nu$, with the $L$-embedding $\iota_{L_P}:{}^LL_P\to {}^LM$. 

Hence for any $\pi \in \Pi(M)$ corresponding to $(P,\sigma,\nu)$, we can define $\LL_{M}(\pi)$ to be the $L$-parameter corresponding to $(P,\LL_{L_P}(\sigma),\nu)$.
\end{proof}

\subsubsection{} Below given a candidate local Langlands correspondence $\{\LL_M\}_{M\subseteq G}$ (on tempered representations), if we write $\LL_M(\pi)$ for a non-tempered representation $\pi$, we mean $\LL_{M}(\pi)$, using the above construction of $\LL_{M}$.

\subsubsection{} \label{subsub:Semisimplification} We say that an $L$-parameter is \emph{semisimple} if its restriction to $\mathrm{SL}_2(\mathbb{C})$ is trivial. In this case, we also identify a semisimple $L$-parameter as a group homomorphism $W_F \to {}^LG(\mathbb{C})$ that recovers the original $L$-parameter as the composition
\[
  W_F \times \mathrm{SL}_2(\mathbb{C}) \xrightarrow{\mathrm{pr}_1} W_F \to {}^LG(\mathbb{C}).
\]
For each $L$-parameter $\phi$, we attach to it a semisimple $L$-parameter $\phi^\mathrm{ss}$ by precomposing $\phi$ with the twisted embedding
\[ W_F\to W_F\times \operatorname{SL}_2(\mathbb{C}), \quad \gamma \mapsto (\gamma, \operatorname{diag}(|\gamma|^{\tfrac{1}{2}}, |\gamma|^{-\tfrac{1}{2}})),\]
with $|\cdot|$ being the composition $W_F\to W_F^\mathrm{ab}\xrightarrow{\mathrm{Art}_F^{-1}}F^\times\xrightarrow{|\cdot|_F}\mathbb{R}_+$.
\begin{Rem} \label{Rem:T}
Through the isomorphism $\iota: \mathbb{C} \xrightarrow{\sim} \qlbar$, we can turn an $L$-parameter $\phi$ as defined above into an $\ell$-adic version of the $L$-parameter $\psi:W_F \to {}^LG(\qlbar)$ by first passing to Weil--Deligne representations via the Jacobson--Morozov Theorem, and then using \cite[Section 8]{DeligneLFunction}, see \cite[Theorem 4.2.1]{Tate}. By construction, the semisimplification of $\psi$ as a representation agrees with $\phi^{\mathrm{ss}}$ under $\iota$.  
\end{Rem}

\subsubsection{} An $L$-parameter $\phi:W_{F} \times \operatorname{SL}_2(\mathbb{C}) \to {}^{L} G(\mathbb{C})$ is called \emph{discrete} if its centralizer $S_{\phi}$ contains $Z(\widehat{G})^{\gal_{F}}$ as a finite index subgroup. An $L$-parameter $W_F\times \mathrm{SL}_2(\mathbb{C})\to {}^LG(\mathbb{C})$ is called \emph{supercuspidal} if it is discrete and semisimple. The following lemma is well known.  
\begin{Lem}\label{Lem:SCTestedbySS}
Let $\phi$ be an $L$-parameter, if $\phi^\mathrm{ss}$ is supercuspidal, then $\phi$ is already supercuspidal.    
\end{Lem}
\begin{proof}
    It suffices to show $\phi|_{\mathrm{SL}_2(\mathbb{C})}$ is trivial, so that $\phi=\phi^\mathrm{ss}$ is supercuspidal. Let $T$ be the diagonal torus of $\mathrm{SL}_2$. Note that the image of $T(\mathbb{C})$ under $\phi$ lies in the centralizer of $\phi^\mathrm{ss}$, since 
    \[\phi(\gamma, \operatorname{diag}(|\gamma|^{\tfrac{1}{2}}, |\gamma|^{-\tfrac{1}{2}}))\cdot \phi(e,t)=\phi(e,t)\cdot \phi(\gamma, \operatorname{diag}(|\gamma|^{\tfrac{1}{2}}, |\gamma|^{-\tfrac{1}{2}}))\]
    for all $\gamma\in W_F$ and $t\in T(\mathbb{C})$. But by assumption, $\phi^\mathrm{ss}$ is discrete, so up to a finite group, the projection of $\phi(T)$ to $\widehat{G}$ is contained in the center. In particular, up to a finite group, $\phi(T)$ is central in $\phi(\mathrm{SL}_2)$. This is possible only if $\phi(\mathrm{SL}_2)$ is itself finite. But $\mathrm{SL}_2$ is connected, so $\phi(\mathrm{SL}_2)$ has to be trivial as desired.
\end{proof}

\subsubsection{} We denote by $\Phi^\mathrm{ss}(G)$ the set of semisimple $L$-parameters. Moreover, for a collection of maps $\{\LL_M\}_{M\subseteq G}$ labeled by Levi subgroups of $G$, we denote by $\{\LL_M^\mathrm{ss}\}_{M\subseteq G}$ the collection of maps 
\[\LL_M^\mathrm{ss}\coloneqq(-)^\mathrm{ss}\circ \LL_M: \Pi_\mathrm{temp}(M)\to \Phi_\mathrm{temp}(M)\to \Phi^\mathrm{ss}(M),\]
and call it the semi-simplification of $\{\LL_M\}_{M\subseteq G}$. 

\subsubsection{} Recall that we have fixed an isomorphism $\iota:\qlbar \to \mathbb{C}$, which induces a $\sqrt{p} \in \qlbar$. For each irreducible smooth representation, and in particular for any $\pi\in \Pi_\mathrm{temp}(G)$, Fargues--Scholze constructed a semisimple\footnote{In the sense of \cite[VIII.3.1]{FarguesScholze}.} $\ell$-adic $L$-parameter $\phi_{\iota\pi}^\FS: W_F\to \phantom{}^L{G}(\qlbar)$, using the fixed $\sqrt{p} \in \qlbar$. Their construction works uniformly for all connected reductive groups, and is independent of $\ell$ and $\iota$ by \cite[Theorem 1.1]{ScholzeMotivicGeometrization}. We consider the collection of maps 
\[\{\LL^\FS_M:\Pi_\mathrm{temp}(M)\to \Phi^\mathrm{ss}(M), \pi\mapsto \iota^{-1}\phi^\FS_{\iota\pi}\}_{M\subseteq G},\]
where $M$ runs through Levi subgroups of $G$, and call it the \textit{Fargues--Scholze (semisimple) local Langlands correspondence for $G$}. For each semisimple $L$-parameter $\phi$ for $M$, we denote the fiber $\LL_M^{\FS,-1}(\phi)$ by $\Pi^\FS_\phi(M)$, and call it the Fargues--Scholze $L$-packet.

For a candidate local Langlands correspondence $\{\LL_M\}_{M\subseteq G}$, one can ask about the compatibility between $\{\operatorname{LL}^\mathrm{ss}_M\}_M$ and $\{\operatorname{LL}^{\FS}_M\}_M$. We say that the candidate local Langlands correspondence $\{\operatorname{LL}_M\}_M$ for $G$ is \textit{compatible with the Fargues--Scholze local Langlands}, if $\iota\phi_\pi^\mathrm{ss}=\phi_{\iota\pi}^\FS$ for all $\pi\in \Pi_\mathrm{temp}(M)$ and all Levi subgroups $M\subseteq G$.

\subsection{Characterization of the semisimple local Langlands correspondence} \label{sub:CharacterizationLL}
If $\{\LL_M\}_{M\subseteq G}$ is a candidate local Langlands correspondence for $G$, then we call the collection $\{\LL^\mathrm{ss}_M\}_{M\subseteq G}$ a (candidate) semisimple local Langlands correspondence. We give in this subsection a characterization of the semisimple local Langlands correspondence in terms of Fargues--Scholze's construction. We let $F/\qp$ be a finite extension and $G/F$ be a connected reductive group as before. We start with the case of quasi-split groups.

\subsubsection{Quasi-split groups}\label{subsub:SemisimpleLLQuasiSplit}

Let $G/F$ be quasi-split and let $\{\LL_M\}_{M\subseteq G}$ be a candidate local Langlands correspondence for $G$. Consider the following assumptions:
\begin{enumerate}[label=(\arabic*)]
        \item \label{ass:one} For each $M$ and each parabolic subgroup $P \subset M$ (with Levi quotient denoted by $L_P$), if $\pi \in \Pi(M)$ is a subquotient of the normalized parabolic induction $i_P^M\rho$ for some $\rho \in \Pi(L_P)$, then 
        \[\LL^\mathrm{ss}_M(\pi)=\iota_{L_P}\circ\LL_{L_P}^\mathrm{ss}(\rho),\]
        where $\iota_{L_P}$ is the $L$-embedding ${}^LL_P\to {}^LM$ as explained in Section~\ref{subsub: ParabolicInd}.
      
        \item \label{ass:two} For each $M$ and each $\phi \in \Phi_{\mathrm{temp}}(M)$ the following holds: The parameter $\phi$ is a supercuspidal $L$-parameter if and only if all elements of $\Pi_{\phi}(M)$ are supercuspidal representations. 
                
        \item \label{ass:three} For each $M$, and each supercuspidal $L$-parameter $\phi\in \Phi_\temp(M)$, the $L$-packet $\Pi_\phi(M)$ is contained in $\Pi^\FS_\phi(M)$.
        \item[(4a)] \customlabel{(4a)}{ass:foura} \customlabel{(4)}{ass:four} For each $M$, and each $\phi\in \Phi_\mathrm{temp}(M)$, there is a stable virtual character
        \[\Theta^1_{\phi}=\sum_{\pi\in \Pi_{\phi}(M)}a_\pi\Theta_\pi,\]
        with $a_\pi \neq 0$ for all $\pi \in \Pi_\phi(M)$, where $\Theta_\pi$ is the Harish-Chandra character of $\pi$.

        \item[(4b)] \customlabel{(4b)}{ass:fourb} For each $M$, the stable virtual characters $\Theta_\phi^1$ are
        atomically stable, i.e., if a linear combination $\Theta = \sum_{\pi\in
    \Pi_{\phi}(M)} b_\pi \Theta_\pi$ is stable then $\Theta$ is a multiple of
    $\Theta_\phi^1$.
\end{enumerate}
We will often refer to Assumption \ref{ass:four} which is the conjunction of Assumptions \ref{ass:foura} and \ref{ass:fourb}.

\begin{Rem} \label{Rem:FSConstant}
  Assumption~\ref{ass:four} implies Assumption
  \begin{itemize}
    \item[(4$'$)] \customlabel{(4$'$)}{ass:fourprime} For each $M$, each $\phi\in \Phi_\temp(M)$, and $\pi,\pi' \in \Pi_\phi(M),$ we have $\phi_{\pi}^{\FS}=\phi_{\pi'}^{\FS}$.
  \end{itemize}
  This follows from work of Hansen, see \cite[Corollary~1.2]{HansenStable}.
\end{Rem}

\begin{Rem}[Warning!]
In general the $L$-parameter $\LL_M(\pi)$ does not have to factor through a Levi subgroup even if its semi-simplification does, like in the case of the Steinberg representation for $\GL_2/\qp$.
\end{Rem}

\begin{Rem}
    Assumption \ref{ass:one} is an expected property of a local Langlands correspondence, see for example \cite[Conjecture 6.1(10)]{Taibi}. It implies the ``only if'' direction of Assumption \ref{ass:two}. The ``if'' direction of Assumption \ref{ass:two} is not expected in general if $G$ is not quasi-split. For example, when $G=D^\times$ is the non-quasi-split inner form of $\GL_2/\qp$, then the unique existence of local Langlands correspondence is known. The trivial representation has a discrete but non-supercuspidal $L$-parameter in this case. Since the $L$-packets are singletons, this gives an example of a supercuspidal $L$-packet with non-supercuspidal $L$-parameter. But when $G$ is quasi-split, this is conjectured to always hold, see \cite[Section 3.5(iv)]{DebackerReeder}. Note also that when $G$ is quasi-split, so are its Levi subgroups. Hence the Assumption \ref{ass:two} is expected to always hold.
\end{Rem}

\begin{Rem} \label{Rem:VarmaHypo}
  Varma proves in \cite[Proposition~2.5.2, Lemma~2.5.7.(ii)]{Varma} that
  Assumption \ref{ass:four} is a consequence of \cite[Hypothesis~2.5.1]{Varma} (by setting
  $\mathcal{O}_M = \{\mathrm{id}_M\}$ for all $M$), where atomicity follows from the 
  $\Theta_\phi^1$ forming a basis of the space $\mathrm{SD}^\mathrm{temp}(G)$ of
  stable tempered virtual characters. In particular, it holds for many
  quasi-split classical groups as listed in \cite[Proposition~5.1.2]{Varma},
  including special orthogonal and unitary groups. 
\end{Rem}

\begin{Rem} \label{Rem:GLn}
These assumptions are known for the usual local Langlands correspondence for $G=\operatorname{GL}_n$ of \cite{HarrisTaylor}. Indeed, \ref{ass:one} and \ref{ass:two} are well known, \ref{ass:three} is \cite[Theorem IX.7.4]{FarguesScholze}, and \ref{ass:foura} and \ref{ass:fourb} are essentially vacuous because the $L$-packets are singletons and all virtual characters are stable.
\end{Rem}

\begin{Prop}\label{prop: SemisimpleLLQuasiSplit}
    Let $G/F$ be quasi-split. If a candidate local Langlands correspondence $\{\LL_M\}_{M\subseteq G}$ for $G$ satisfies Assumptions \ref{ass:one}, \ref{ass:two}, \ref{ass:three}, \ref{ass:fourprime}, then it is compatible with the Fargues--Scholze local Langlands. In other words, the semisimple local Langlands correspondence $\{\LL^\mathrm{ss}_M\}_{M\subseteq G}$ is uniquely determined by these assumptions.
\end{Prop}

\begin{proof}
    Let $M\subseteq G$ be a Levi subgroup of $G$ and $\pi\in\Pi_\temp(M)$ be a tempered irreducible representation. If $\phi_\pi:=\LL_M(\pi)$ is a supercuspidal $L$-parameter, then it agrees with $\LL^\FS_M(\pi)$ by Assumption \ref{ass:three}.

    Therefore, we may assume $\phi_\pi$ is not supercuspidal. Then by Assumption \ref{ass:two}, there is some representation $\pi'\in \Pi_{\phi_\pi}(M)$ that is a subquotient of $i_P^M\rho$ for some parabolic subgroup $P\subset M$ with Levi $L_P$ and some $\rho\in \Pi(L_P)$. By Assumption \ref{ass:fourprime}, we may replace $\pi$ by $\pi'$ and hence assume $\pi$ itself is of this form.

    Moreover, we have $\phi^\mathrm{ss}_\pi=\iota_{L_P}\circ \LL^\mathrm{ss}_{L_P}(\rho)$ by Assumption \ref{ass:one}. Since the Fargues--Scholze local Langlands is compatible with parabolic inductions, see \cite[Corollary IX.7.3]{FarguesScholze}\footnote{Note that there the parabolic induction is unnormalized, while the $L$-embedding ${}^LL_P\hookrightarrow {}^LM$ is twisted by a cyclotomic cocycle. Using normalized parabolic induction cancels out this cyclotomic twist, see \cite[Beginning of Section IX.7.1]{FarguesScholze}, also \cite[p.3 (1), cf. Section 3.1]{DHKM}. Note, however, to get the description of the twisted embedding as stated in \cite[Section IX.7.1]{FarguesScholze}, one needs to identify the Borel of the Satake group differently, see \cite[Remark 6.23]{CvdHS}.}, we also have 
    \[\phi^\FS_\pi=\iota_{L_P}\circ \LL^\FS_{L_P}(\rho).\]
    This reduces us to showing $\LL_{L_P}^\mathrm{ss}(\rho)=\LL_{L_P}^\FS(\rho)$. But by picking $L_P$ to be the minimal Levi through which $\phi_\pi^\mathrm{ss}$ factors, we can assume that $\LL_{L_P}^\mathrm{ss}(\rho)$ is supercuspidal. Since $L_P$ is itself a Levi subgroup of $G$, we may use Lemma~\ref{Lem:SCTestedbySS} and Assumption \ref{ass:three} to conclude. 
\end{proof}

\subsubsection{Non-quasi-split groups}
We now consider the case of extended pure inner forms of quasi-split groups via the theory of endoscopic transfer. This does not encompass all connected reductive groups over $F$, but the most important cases of them, see the discussion in \cite[Section 4.5]{Kaletha}. 

Let $G/F$ be a reductive group that is realized as an extended pure inner form of a quasi-split group $G^*$ via $(\xi,b)$, where $\xi$ is an isomorphism $G^\ast_{\overline{F}}\xrightarrow{\sim} G_{\overline{F}}$ and $b\in B(G^\ast)$ is an element such that $\xi^{-1}\sigma(\xi)=\mathrm{Ad}(b(\sigma))$ for all $\sigma \in \operatorname{Gal}_F$. In this case, the triple $\mathbf{1}:=(G^\ast,1,\mathrm{id})$ is an extended endoscopic triple for $G$, see \cite[Definition 2]{Kaletha} for the definition. We fix a Whittaker datum $\mathfrak{w}$ for $G^\ast$. Following \cite{HansenStable}, see Footnote 5 in \textit{loc. cit.}, cf. \cite[Remark 3.2.2(i)]{Varma}, we may and do normalize the transfer factor $\Delta_{[\mathfrak{w},\mathbf{1},b]}$ to be the Kottwitz sign $e(G)\in \{\pm 1\}$, which is defined in \cite{KottwitzSign}. Recall also the notion of matching functions (depending on the transfer factor) from \cite[Definition 3, Theorem 4]{Kaletha}.

Assume now that we have a candidate local Langlands correspondence $\{\LL_M\}_{M\subseteq G}$ for $G$, as well as one $\{\LL_{M^\ast}\}_{M^\ast\subseteq G^\ast}$ for $G^\ast$. Consider the following assumption:
\begin{itemize}
    \item[(5)] \customlabel{(5)}{ass:five}Suppose that $\{\LL_{M^\ast}\}_{M^\ast\subseteq G^\ast}$ satisfies Assumption \ref{ass:foura}. For each $\phi\in \Phi_\temp(G)=\Phi_\temp(G^\ast)$, and for all matching functions $f^\ast\in \mathcal{C}^\infty_c(G^\ast(F))$, $f\in \mathcal{C}_c^\infty(G(F))$, there is an equality
    \[\Theta^\mathbf{1}_{\phi}(f^\ast)=\sum_{\pi\in \Pi_\phi(G)}c_\pi\Theta_\pi(f),\]
    for some coefficients $c_\pi \in \mathbb{C}^\times$ (depending on $\mathfrak{w}$, but not on $f, f^\ast$), where $\Theta_\phi^\mathbf{1}$ is the stable virtual character supported on the $L$-packet $\Pi_\phi(G^\ast)$ as in Assumption \ref{ass:foura}. The same is true when $G$ is replaced by any Levi subgroup $M\subseteq G$.
\end{itemize}

We have the following analogous statement to Proposition~\ref{prop: SemisimpleLLQuasiSplit}.

\begin{Prop}\label{prop: SemisimpleLLNonQuasiSplit}
    If $\{\LL_{M^\ast}\}_{M^\ast\subseteq G^\ast}$ satisfies Assumptions \ref{ass:one}, \ref{ass:two}, \ref{ass:three}, \ref{ass:foura} and \ref{ass:fourprime}, and if Assumption \ref{ass:five} holds, then $\{\LL_{M}\}_{M\subseteq G}$ is compatible with the Fargues--Scholze local Langlands.
\end{Prop}

Before we give the proof of this, let us briefly recall the construction of Fargues--Scholze $L$-parameters in terms of distributions, as well as results from \cite{HansenStable}. We start by recalling the action of the Bernstein center on the space of distributions, following \cite[Section 3.1, 3.2]{HainesStable}, \cite[Section 2.1]{Varma}, \cite[Section 2.1]{HansenStable}. 

\subsubsection{} Let $H/F$ be a connected reductive group and let $C^{\infty}_{c}(H(F))$ be the convolution algebra of locally constant functions $H(F) \to \mathbb{C}$ that are compactly supported. The $\mathbb{C}$-linear dual of $C^{\infty}_{c}(H(F))$ is called the space of distributions. It contains a subspace $\operatorname{D}(H)$ spanned by the Harish-Chandra characters $\Theta_\pi: f\mapsto \operatorname{tr}(\pi|f)$ attached to irreducible representations $\pi$. Distributions also have a convolution product which is defined if one of the distributions is \emph{essentially compact} (which means that the convolution by it preserves compactly supported smooth functions), see \cite[Lemma 3.1.1]{HainesStable}. 

\subsubsection{} By a result of Bernstein--Deligne, the Bernstein center $\mathcal{Z}(H)$ can be identified with the algebra of essentially compact $H(F)$-invariant distributions, \cite[Section 3.2]{HainesStable}. In particular, $\mathcal{Z}(H)$ acts on both $\operatorname{D}(H)$ and $C^{\infty}_{c}(H(F))$ by convolution. This action on $\operatorname{D}(H)$ preserves the subspace $\operatorname{D}^\temp(H)$ spanned by Harish-Chandra characters of tempered irreducible representations. The most important property for us is that if $\pi$ is a smooth irreducible representation with associated distribution $\Theta_{\pi}$, then for each $z \in \mathcal{Z}(H)$ we have
    \begin{align*}
        z \ast \Theta_{\pi} = z_{\pi} \Theta_{\pi}
    \end{align*}
    for some scalar $z_{\pi}$. The association $z \mapsto z_{\pi}$ defines a character (ring homomorphism) $\chi_{\pi}:\mathcal{Z}(H) \to \mathbb{C}$, and $\mathcal{Z}(H)$ acts on $\pi$ through $\chi_\pi$.

\subsubsection{}\label{subsub: FSviaDistributions} In \cite[IX.0.3]{FarguesScholze}, Fargues--Scholze constructed a morphism
\[\Psi_H: \mathcal{Z}^\mathrm{spec}(H)\to \mathcal{Z}(H),\]
from the spectral Bernstein center to the usual Bernstein center (base-changed to $\qlbar$ via the chosen isomorphism $\iota:\mathbb{C}\simeq \qlbar$). The Fargues--Scholze $L$-parameter $\phi^\FS_\pi$ attached to an irreducible representation $\pi$ is constructed via the composition $\chi_\pi\circ \Psi_H$. Indeed, this composition corresponds to a point on the coarse moduli of the stack of $L$-parameters, and consequently a semisimple $L$-parameter by \cite[Proposition VIII.3.2]{FarguesScholze}.

One can alternatively phrase the construction of Fargues--Scholze $L$-parameters in terms of distributions. Namely, through the map $\Psi_H$, there is an action of the spectral Bernstein center on the space of distributions $D(H)$. Then for any smooth irreducible representation $\pi$, and any $z\in \mathcal{Z}^\mathrm{spec}(H)$, $\Psi_H(z)\ast \Theta_\pi =z_\pi \Theta_\pi$, for some scalar $z_\pi\in\mathbb{C}$. In this way, we get a character $z\mapsto z_\pi$ of $\mathcal{Z}^\mathrm{spec}(H)$, and hence a semisimple $L$-parameter $\phi_\pi^\FS$.

\subsubsection{}\label{subsub: HansenStable}
According to \cite[Theorem 1.1, 1.4]{HansenStable}, the image $\mathcal{Z}^\FS(H)$ of Fargues--Scholze's map $\Psi_H$ lies in the subalgebra of very stable\footnote{That is, a stable distribution whose convolution with unstable functions gives unstable functions, see \cite[Section 2.1, Lemma 2.2]{HansenStable}.} central distributions. Moreover, if $H$ is an extended pure inner form of a quasi-split group $H^\ast$, then the map $\Psi_H$ factors through a surjective map $\tau_H: \mathcal{Z}^\FS(H^\ast)\to \mathcal{Z}^\FS(H)$. This is compatible with the usual transfer of tempered stable distributions in the following sense: Suppose $\Theta^\ast$, resp.\ $\Theta$ are stable virtual characters of $H^\ast$, resp.\ $H$, supported on tempered irreducible representations. If for all matching functions $f^\ast\in \mathcal{C}^\infty_c(H^\ast(F))$, $f\in \mathcal{C}_c^\infty(H(F))$, there is an equality
\[\Theta^\ast(f^\ast)=\Theta(f),\]
then for any $z\in \mathcal{Z}^\mathrm{spec}(H)\simeq \mathcal{Z}^\mathrm{spec}(H^\ast)$,  one has
\[(\Psi_{H^\ast}(z)\ast\Theta^\ast)(f^\ast)=(\Psi_H(z)\ast\Theta) (f).\]

\smallskip
We can now give the proof of Proposition~\ref{prop: SemisimpleLLNonQuasiSplit}.
\begin{proof}[Proof of Proposition~\ref{prop: SemisimpleLLNonQuasiSplit}]
    Let $M\subseteq G$ be a Levi subgroup. Let $\pi\in\Pi_\temp(M)$ be an irreducible tempered smooth representation and $\phi:=\LL_M(\pi)$ be its attached $L$-parameter under $\LL_M$. 
    
    By Assumption \ref{ass:foura}, there exists a stable virtual character $\Theta^1_\phi$ supported on $\Pi_{\phi}(M^\ast)$. By Assumption \ref{ass:fourprime}, for any $z\in \mathcal{Z}^\mathrm{spec}(M^\ast)$, the actions of $\Psi_{M^\ast}(z)$ on all members of $\Pi_\phi(M^\ast)$ will be through the same scalar. Hence for any such $z$, we have
    \[(\Psi_{M^\ast}(z)\ast\Theta_\phi^\mathbf{1})=z_\phi\Theta_\phi^\mathbf{1},\]
    for some scalar $z_\phi\in \mathbb{C}$ that only depends on $\phi$, but not the individual members of the $L$-packet. The association $z \mapsto z_{\phi}$ defines a nonzero character of $\mathcal{Z}^\mathrm{spec}(M^\ast)$, which corresponds to the Fargues--Scholze parameter of any member $\rho \in \Pi_\phi(M^\ast)$. We first want to compare this character with the one corresponding to $\phi_\pi^\FS$.
    
    For this, we use Assumption \ref{ass:five} to transfer $\Theta_\phi^1$ to a virtual character \[\sum_{\pi\in\Pi_\phi(M)}c_\pi\Theta_\pi,\] which is supported on the candidate $L$-packet $\Pi_\phi(M)\coloneqq\LL_M^{-1}(\phi)$. This is a stable (and in fact tempered) virtual character according to \cite[Remark 3.2.2(ii)]{Varma}, which says that under transfer, stable distributions go to stable distributions.

    But by \cite[Theorem 1.4.ii]{HansenStable} (see the explanation in Section~\ref{subsub: HansenStable}), for any $z\in \mathcal{Z}^\mathrm{spec}(M)\simeq \mathcal{Z}^\mathrm{spec}(M^\ast)$, and all matching functions $f^\ast\in \mathcal{C}^\infty_c(M^\ast(F))$, $f\in \mathcal{C}_c^\infty(M(F))$, we have
    \begin{align*}
    z_\phi\left(\sum_{\pi\in\Pi_\phi(M)}c_\pi\Theta_\pi\right)(f) &= z_\phi \Theta_\phi^\mathbf{1}(f^\ast)
        = (\Psi_{M^\ast}(z)\ast\Theta_\phi^\mathbf{1})(f^\ast)\\
        &=\sum_{\pi\in\Pi_\phi(M)}c_\pi (\Psi_{M}(z)\ast\Theta_\pi)(f)
        =\sum_{\pi\in\Pi_\phi(M)}c_\pi \cdot z_\pi\Theta_\pi(f),
    \end{align*}
    where $z_\pi\in \mathbb{C}$ is the scalar, through which $\mathcal{Z}^\mathrm{spec}(M)$ acts on $\pi$. If we compare the coefficients on both sides, then we see, by linear independence of the Harish-Chandra characters, that $z_\pi=z_\phi$ for all $\pi\in \Pi_\phi(M)$. Now as explained in Section~\ref{subsub: FSviaDistributions}, this means the Fargues--Scholze parameters for all members of the $L$-packet $\Pi_\phi(M)$ are the same, and they equal $\phi^\FS_\rho$, for any $\rho\in\Pi_\phi(M^\ast)$.

    Now under our assumptions, we know the compatibility of $\{\LL_{M^\ast}\}_{M^\ast\subset G^\ast}$ with the Fargues--Scholze local Langlands correspondence by Proposition~\ref{prop: SemisimpleLLQuasiSplit}. Therefore, we have $\phi^\FS_\pi=\phi^\FS_\rho=\phi^\mathrm{ss}$ as desired.
\end{proof}

\begin{Rem}
    Assumption \ref{ass:five} is a form of an endoscopic character identity, which is expected as part of the refined local Langlands conjecture for non-quasi-split groups, see \cite[Conjecture F]{Kaletha}. Note that here we only require the property that the $L$-packet for the parameter $\phi$ is the support of the endoscopic transfer of $\Theta_\phi^1$, whereas the refined conjecture in \textit{loc. cit.} says that if one renormalizes the transfer factor $\Delta_{[\mathfrak{w},\mathbf{1},b]}$ suitably (which will only change it by a scalar and hence does not change the support of the transfer of $\Theta_\phi^1$), then one can precisely predict the coefficients $c_\pi$. These coefficients capture the inner structure of the $L$-packet.
\end{Rem}

\subsection{Compatibility for unitary and odd special orthogonal groups.} 

In this and the next subsection, we prove the compatibility between Fargues--Scholze local Langlands correspondence with previous constructions of local Langlands correspondences for extended pure inner forms of quasi-split unitary groups, odd special orthogonal groups, and $\operatorname{GSp}_{4}$. Because the previous correspondences are constructed using the work of Arthur \cite{Arthur}, \textbf{we will assume for the rest of Section \ref{Sec:Compatibility} that \cite[Hypothesis H1]{ShinWeakTransfer} holds}. Therefore, we may appeal to \cite{AGIKMS} and unconditionally use the results of Arthur \cite{Arthur} and Mok \cite{Mok}.

\subsubsection{} \label{Sec:ABC-llc-list} Let $F/\qp$ be a finite extension and $G/F$ be a reductive group as before. We assume $G$ is an extended pure inner form of a quasi-split group $G^\ast$, which is either the quasi-split unitary group attached to an Hermitian space with respect to a quadratic extension $F'/F$ (type $A$), the split special orthogonal group $\operatorname{SO}_{2n+1,F}$ for some integer $n$ (type $B$), or $\operatorname{GSp}_{4,F}$ (type $C_2$). We set $F'=F$ in cases $B$ and $C_{2}$. We let $\{\LL_M\}_{M\subseteq G}$ be a (candidate) local Langlands correspondence for $G$ in the list below.\footnote{Here we are using the fact that Levi subgroups of unitary groups (resp. special orthogonal groups of odd rank) are isomorphic to products of general linear groups with unitary groups (resp. special orthogonal groups of odd rank), and in type $C_{2}$ that the Levi subgroups are products of inner forms of general linear groups. This allows us to define a local Langlands correspondence for these Levi subgroups $M$. Proving properties \ref{ass:one}, \ref{ass:two}, \ref{ass:three} and \ref{ass:four} for all Levi subgroups $M$ follows then, by known properties of the local Langlands correspondence for $\operatorname{GL}_n$, see Remark \ref{Rem:GLn}, by proving them for all quasi-split unitary (resp. odd orthogonal) groups; in type $C_{2}$ it is enough to deal with $M=G$.}
\begin{itemize}
    \item Type $A$: 
    \begin{enumerate}
        \item For $G=G^\ast$, the construction by Mok \cite[Theorem 2.5.1, 3.2.1]{Mok},
        \item For general $G$, by Kaletha--Minguez--Shin--White \cite[Theorem 1.6.1]{KMSW}.\footnote{We note that \cite[Theorem 1.6.1]{KMSW} is stated for all representations (or $L$-parameters), but is only proved for tempered representations (or $L$-parameters) in \cite[Theorem 1.6.1]{KMSW}. This does not affect us, since we only use their results for tempered representations.}
    \end{enumerate}

    \item Type $B$: 
    \begin{enumerate}
        \item For $G=G^\ast$, by Arthur \cite[Theorem 1.5.1]{Arthur},
        \item For general $G$, by Ishimoto \cite[Theorem 1.2]{Ishimoto}.
    \end{enumerate}

    \item Type $C_2$:
    \begin{enumerate}
        \item For $G=G^{\ast}=\operatorname{GSp}_{4,F}$ by \cite[Main theorem]{GanTakeda}. 
        \item For general $G$, by \cite[Main theorem]{GanTantono}. 
    \end{enumerate}
\end{itemize}

\begin{Prop}\label{Prop: KnownLL}
        Let $F$, $G$, $G^\ast$ be as in \Cref{Sec:ABC-llc-list}. If $G^{\ast}$ is of type $A$ or $B$, then the local Langlands correspondences $\{\LL_{M^\ast}\}_{M^\ast\subseteq G^\ast}$ from the above list satisfy Assumptions \ref{ass:one}, \ref{ass:two} and \ref{ass:four}. If $G^{\ast}$ is of type $C_{2}$, then the local Langlands correspondences $\{\LL_{M^\ast}\}_{M^\ast\subseteq G^\ast}$ satisfy Assumptions \ref{ass:one}, \ref{ass:two}, \ref{ass:foura} and \ref{ass:fourprime}. For all non quasi-split $G$ as above, the local Langlands correspondences $\{\LL_{M}\}_{M\subseteq G}$ from the above list satisfy Assumption \ref{ass:five}.
\end{Prop}

\begin{proof}
We first discuss type $A$ and $B$: For Assumption \ref{ass:one}, this is \cite[Proposition 2.4.3]{Peng}, cf. \cite[Proposition 7.9]{DHKMFamilies}. That Assumption \ref{ass:two} is satisfied is due to Moeglin \cite[Theorem 8.4.4]{MoeglinUnitary}, Moeglin \cite[Theorem 1.5.1]{MoeglinOrthogonal}, cf. \cite[Theorem 3.3]{XuOrthogonal}, see also \cite[Corollary 2.5.2]{Peng}. For Assumption \ref{ass:four}, this is \cite[Proposition 5.1.2]{Varma}, see \Cref{Rem:VarmaHypo}.
  When $G$ is non-quasi-split, Assumption \ref{ass:five} is satisfied by \cite[Theorem 1.6.1.4]{KMSW}, \cite[Theorem 3.11.3]{Ishimoto}. In type $C_{2}$, Assumption \ref{ass:one} can be verified directly from the explicit description of the $L$-parameters for non-supercuspidal representations, see \cite[Proposition 13.1]{GanTakedaExplicit}. Assumption \ref{ass:two} is proved in \cite[discussion after Remark 2.2]{HamannGSp4}. Assumption \ref{ass:foura} is \cite[Main theorem]{ChanGan} and Assumption \ref{ass:fourprime} follows from \cite[Proposition 8.1.(ii)]{HamannGSp4} for non supercuspidal $\phi$; for supercuspidal $\phi$ Assumption \ref{ass:fourprime} is vacuous since then the $L$-packet $\Pi_{\phi}(G^{\ast})$ has cardinality one. Assumption \ref{ass:five} is \cite[Proposition 11.1.(i)]{ChanGan}.
\end{proof}

\begin{Rem}
The reader might have noticed that our treatment of $\operatorname{GSp}_{4}$ alongside the odd orthogonal and unitary cases is slightly artificial. For instance, for $G$ of type $C_{2}$, the compatibility between $\operatorname{LL}_{G}$ and $\operatorname{LL}_G^{\mathrm{FS}}$ is known for those $\pi$ with $\operatorname{LL}_G(\pi)$ not supercuspidal, by \cite[Proposition 8.1.(ii)]{HamannGSp4}. Nevertheless, for the remaining $\pi$ our argument is the same as in type $A$ and $B$, and thus we have opted for the current presentation. 
\end{Rem}

Before we prove the desired compatibility, we need some preparations. The first is a lemma about globalization. This is essentially \cite[Lemma 6.2.1, 6.2.2]{Arthur}, and \cite[Lemma 7.2.1, 7.2.3]{Mok}, except that we need to globalize the group in a way that it gives rise to a compact Shimura variety. 

\begin{Lem}\label{Lem: Globalization}
    Let $F$, $G$ be as before. Let $\pi\in \Pi_\temp(G)$ be a supercuspidal representation. Then there exists a tuple $(\ml, \mlp, v,\mathsf{H},\mathsf{X}, \Pi)$, where
    \begin{itemize}
        \item $\mlp \neq \mathbb{Q}$ is a totally real field,
        \item $v$ is a $p$-adic place of $\mlp$,
        \item $\mathsf{H}$ is a connected reductive group over $\mlp$,
        \item $\mathsf{X}$ is a Shimura datum for $\g=\operatorname{Res}_{\mlp/\mathbb{Q}} \mathsf{H}$ with reflex field $\ml$, an extension of $\mlp$ of degree $\le 2$,
        \item $\Pi$ is a cuspidal automorphic representation of $\mathsf{H}(\mathbb{A}_{\mlp})$;
    \end{itemize}
    such that $\mlp_v\simeq F$, $\mathsf{H}_v\simeq G$, such that $\gx$ gives one of the Shimura data of Section \ref{sub:examples}, and $\Pi_v\simeq \pi$, satisfies the conditions \ref{ass:coh-reg}, \ref{ass:ssc}, \ref{ass:simgen} and \ref{ass:Stprime}. 
\end{Lem}

\begin{proof}
We first focus on types $A$ and $B$. By \cite[Lemma 6.2.1]{Arthur}, \cite[Lemma 7.2.1]{Mok}, we can find a totally real number field $\mlp$ (and a CM extension $\ml/\mlp$ in the unitary case) with $\mlp_v\simeq F$ (and $\ml_{v'} \simeq F'$ for a unique place $v'$ above $v$) for a $p$-adic place $v$ of $\mlp$. Let $\mathsf{H}^{\ast}$ be the split odd special orthogonal group $\SO_{2n+1}$ or a quasi-split unitary group $\mathrm{U}_{\ml/\mlp}(n)$ attached to $\ml$ over $\mlp$. Consider the short exact sequence 
\[
    H^1(\mlp,\mathsf{H}^{\ast,\mathrm{ad}}) \to \bigoplus_{w} H^1(\mlp_{w},\mathsf{H}^{\ast,\mathrm{ad}}_w) \to \pi_{1}(\mathsf{H}^{\ast\mathrm{ad}})_{\gal_{\mlp}} \to 1,
\]
which comes from \cite[Corollary 2.5, Proposition 2.6]{KottwitzSTFelliptic}. It follows from this that there is an inner form $\mathsf{H}$ of $\mathsf{H}^{\ast}$ such that:
\begin{itemize}
    \item There is exactly one infinite place $\tau_0$ of $\mlp$ such that $\mathsf{H} \otimes_{\mlp,{\tau_0}}\mathbb{R}$ has non-compact $\mathbb{R}$-points. In the case of type B, $\mathsf{H}$ is isomorphic to $\operatorname{SO}(2,2n-1)$ at $\tau_0$, and in the case of type A, $\mathsf{H}$ is isomorphic to $\operatorname{U}(1,n-1)$ at $\tau_0$; 
    
    \item At the distinguished $p$-adic place $v$ the group $\mathsf{H} \otimes_{\mlp} \mlp_v$ is isomorphic to $G$.
\end{itemize}
Let $\mathsf{G}=\operatorname{Res}_{\mlp/\mathbb{Q}}\h$, which is clearly $\mathbb{Q}$-anisotropic modulo center. It is equipped with a Shimura datum $\mathsf{X}$ that is nontrivial precisely at the factor of $\mathsf{H}_{\R}$ corresponding to $\tau_0$, such that $\gx$ is a Shimura datum. By construction, it corresponds to the setup in type $A$ or $B$ of Section \ref{sub:examples}. \smallskip 

To globalize the representation, we apply \cite[Theorem 4.8]{ShinPlancherel}. More precisely, we apply \cite[Lemma 4.2.1, Lemma 4.3.1]{KMSW} in type $A$ and \cite[Lemma 6.4, Lemma 6.7]{Ishimoto} in type $B$ to construct a cuspidal automorphic representation $\Pi^{\ast}$ of $\mathsf{H}^{\ast}(\mathbb{A}_{\mlp})$ satisfying \ref{ass:coh-reg}, \ref{ass:ssc}, \ref{ass:Stprime}, \ref{ass:simgen} and $\Pi_{v} \simeq \pi$. Here we use the fact that for an auxiliary place $w \nmid 2$ of $\mlp$ where $\h_{w}$ is quasi-split and unramified, the group $\h_{w}$ admits supercuspidal representations satisfying \ref{ass:ssc}, see \cite[Theorem 1.1]{OIUnitary} in type $A$ and \cite[Theorem 1.1]{OiSimple} in type $B$. We are moreover using the fact that simple generic as in \cite[Lemma 4.3.1]{KMSW} and \cite[Lemma 6.7]{Ishimoto} is the same as \ref{ass:simgen}. We then transfer $\Pi^{\ast}$ to a cuspidal automorphic representation $\Pi$ of $\mathsf{H}(\mathbb{A}_{\mlp})$ satisfying \ref{ass:coh-reg}, \ref{ass:ssc}, \ref{ass:Stprime}, \ref{ass:simgen}, using \cite[Theorem 4.3.5]{Peng}. \smallskip 

In type $C_2$, we can let $\h$ be as in type $C$ of Section \ref{sub:examples} with $g=2$, assuming that $\h$ is quasi-split at all finite places except possibly for one place called $v_{\mathrm{st}}$ (see \cite[Section 7]{KretShinSymplectic} for the existence of such $\h$). The globalization can be done using a routine application of \cite[Theorem 4.8]{ShinPlancherel}. 
\end{proof}

\subsubsection{} Let $\operatorname{Std}:\widehat{G} \to \operatorname{GL}(V)$ be the standard representation of the dual group, which we take to be $\mathrm{GSpin}_5 \simeq  \mathrm{GSp}_4 \to \operatorname{GL}_{4}$ in type $C_{2}$. The following lemma is well known. 
\begin{Lem} \label{Lem:Multiplicity}
Let $\phi$ be a discrete and semisimple $L$-parameter. Then $\operatorname{Std} \circ \phi \vert_{W_{F'}}$ is multiplicity-free as a representation of $W_{F'}$.
\end{Lem}
\begin{proof}
In type $C_{2}$, this is \cite[Lemma 6.2.(ii)]{GanTakeda}. In type $A$ and $B$, this is a direct consequence of the descriptions of $S_{\phi}$ in \cite[Section 4]{GGP} as we will now explain: Write $M=\operatorname{Std} \circ \phi \vert_{W_{F'}}$, which is conjugate self dual in type $A$, and self dual in type $B$. Following the notation of \cite[Section 4]{GGP}, we write
\begin{align*}
    M = \bigoplus M_i^{\oplus m_i} \oplus  \bigoplus_I N_i^{\oplus 2 n_i} \oplus \bigoplus (P_i \oplus (P_i^{\sigma})^{\vee})^{\oplus p_i},
\end{align*}
where each $M_i,N_i,P_i$ is irreducible and the $M_i,N_i,P_i$ are pairwise distinct. The group $S_{\phi}$ is then described in \cite[Section 4]{GGP} (there it is denoted $C$) as
\begin{align*}
    \prod_{i} \operatorname{O}(m_i) \times \prod_i \operatorname{Sp}(2n_i) \times \prod_i \operatorname{GL}(p_i).
\end{align*}
Since $\phi$ is discrete, it follows that $S_{\phi}$ is finite and thus that all the $p_i$ and $n_i$ are zero, and that all the $m_i$ are equal to one; the lemma follows. 
\end{proof}

\begin{Rem}
If $\pi$ is supercuspidal, then $\phi=\operatorname{LL}_G(\pi)$ is a discrete $L$-parameter. However, if $\phi$ is not supercuspidal, then it is possible that $\phi^{\mathrm{ss}}$ is not discrete and thus that $\operatorname{Std} \circ \restr{\phi}{W_{F'}}$ is not multiplicity-free. This shows that we cannot directly prove Assumption \ref{ass:three} for such $\pi$ using a globalization argument and Corollary \ref{Cor:StrongCompatibilityII}, as happens for supercuspidal $\phi$ in the proof of Theorem \ref{Thm:UnitaryOddOrthogonalCompatibility} below.
\end{Rem}

\begin{Thm} \label{Thm:UnitaryOddOrthogonalCompatibility}
    Let $F$, $G$, $G^\ast$ be as in \Cref{Sec:ABC-llc-list}. The local Langlands correspondences $\{\LL_M\}_{M\subseteq G}$ from the above list are compatible with the Fargues--Scholze local Langlands correspondence.
\end{Thm}

\begin{proof}
By Propositions~\ref{prop: SemisimpleLLQuasiSplit}, \ref{prop: SemisimpleLLNonQuasiSplit} and \ref{Prop: KnownLL}, it suffices to check Assumption \ref{ass:three}, namely the compatibility between the constructions on supercuspidal $L$-packets for the quasi-split inner form $G^\ast$ (which is attached to some quadratic extension $E/F$ in the unitary case). Hence we may and do assume $G=G^\ast$. Let $\phi\in \Phi_\temp(G)$ be a supercuspidal $L$-parameter and let $\pi\in \Pi_\phi(G)$ be a member of its $L$-packet. By Assumption \ref{ass:two}, it is a supercuspidal representation. 

We globalize the triple $({F}, {G}, \pi)$ using \Cref{Lem: Globalization}. In particular, we get a tuple $(\ml,\mlp,v,\mathsf{G}, \mathsf{X}, \Pi)$ as in the conclusion of that lemma ($\ml=\mlp$ when $G$ is a special orthogonal group). The compatibility is then \Cref{Cor:StrongCompatibilityII}, using the multiplicity-freeness of \Cref{Lem:Multiplicity}. 

Now note that any proper Levi subgroup $M\subset G$ is a product of groups from the same list provided at the beginning of this subsection, or inner forms of general linear groups, which are of strictly smaller ranks. Hence we may conclude the compatibility of $\{\LL_M\}_{M\subseteq G}$ with the Fargues--Scholze local Langlands correspondence as desired.
\end{proof}

\subsection{Even special orthogonal groups} Let $F$ be a $p$-adic local field as before. Let $V'$ be a $2n$-dimensional quadratic space over $F$ with quasi-split orthogonal group $G^{\ast}$, and let $G$ be a pure inner form of $G^{\ast}$. Note that $G \simeq \operatorname{SO}(V)$ for some quadratic space $V$ over $F$, and that the action of $\gal_{F}$ on the absolute root datum of $G$ is of order $1$ or $2$.\footnote{Pure inner forms of $\operatorname{SO}(V')$ are isomorphic to $\operatorname{SO}(V)$ with $V$ of the same discriminant and dimension as $V'$, see \cite[Proposition 2.8 on page 73]{PR}. The fact that an order $3$ automorphism cannot occur follows from \cite[Proposition 2.20 on page 90]{PR}.} Write $G=G_b$ as an extended pure inner form of $G^{\ast}$. Note that $\widehat{G}=\operatorname{SO}_{2n}$, so that $\operatorname{O}_{2n}$ acts by conjugation on $\widehat{G}$. We define
\begin{align*}
    \widetilde{\Phi}_{\operatorname{temp}}(G) =  \Phi_{\operatorname{temp}}(G) / \operatorname{O}_{2n}(\mathbb{C}),
\end{align*}
and similarly define $\widetilde{\Phi}^{\mathrm{ss}}(G)$ and consider $(-)^{\mathrm{ss}}:\widetilde{\Phi}_{\operatorname{temp}}(G) \to \widetilde{\Phi}^{\mathrm{ss}}(G)$. We consider the map
\begin{align*}
    \widetilde{\operatorname{LL}}_{G}: \Pi_{\operatorname{temp}}(G) \to \widetilde{\Phi}_{\operatorname{temp}}(G)
\end{align*}
constructed by Arthur \cite[Theorem 1.5.1]{Arthur} in the quasi-split case and by Chen--Zou \cite[Theorem A.2]{ChenZou} in the non quasi-split case. We also consider $\widetilde{\operatorname{LL}}_{G}^{\operatorname{FS}}$ constructed from $\operatorname{LL}_{G}^{\operatorname{FS}}$ by composing with $\Phi^{\mathrm{ss}}(G) \to \widetilde{\Phi}^{\mathrm{ss}}(G)$. This way we get $L$-packets $\widetilde{\Pi}_{\phi}(G)$ and $\widetilde{\Pi}^{\operatorname{FS}}_{\phi}(G)$. 

\begin{Thm} \label{Thm:EvenOrthogonalCompatibility}
We have the equality $(-)^{\mathrm{ss}} \circ \widetilde{\operatorname{LL}}_G = \widetilde{\operatorname{LL}}_{G}^{\operatorname{FS}}$.
\end{Thm}

\subsubsection{} \label{subsub:GeometricOrthogonalParameters} Before we prove the theorem, it is useful to reinterpret the set $\widetilde{\Phi}_{\operatorname{temp}}^{\mathrm{ss}}(G)$ geometrically. Recall that $Z^1(W_{F}, \operatorname{SO}_{2n})$ is the moduli space of $\operatorname{SO}_{2n}$-valued $L$-parameters of $W_{F}$, see \cite[Theorem VIII.1.3]{FarguesScholze}. It is equipped with a natural conjugation action by $\operatorname{O}_{2n}$, and we consider the diagram
\begin{equation*}
    \begin{tikzcd}
        \left[Z^1(W_{F}, \operatorname{SO}_{2n})/\operatorname{SO}_{2n} \right] \arrow{r} \arrow{d} & Z^1(W_{F}, \operatorname{SO}_{2n}) \sslash \operatorname{SO}_{2n} \arrow{d} \\
        \left[Z^1(W_{F}, \operatorname{SO}_{2n})/\operatorname{O}_{2n} \right]
        \arrow{r} & Z^1(W_{F}, \operatorname{SO}_{2n}) \sslash \operatorname{O}_{2n}.
    \end{tikzcd}
\end{equation*}
It follows from \cite[Proposition VIII.3.2]{FarguesScholze} that $\operatorname{SO}_{2n}$-conjugacy classes of semisimple $L$-parameters $W_{F} \to \operatorname{SO}_{2n}(\qlbar)$ are in bijection with closed $\operatorname{SO}_{2n}$-orbits in $Z^1(W_{F}, \operatorname{SO}_{2n})(\qlbar)$ and also with closed points in $(Z^1(W_{F}, \operatorname{SO}_{2n}) \sslash \operatorname{SO}_{2n})(\qlbar)$. The proof of that theorem also establishes that $\operatorname{O}_{2n}$-conjugacy classes of semisimple $L$-parameters $W_{F} \to \operatorname{SO}_{2n}(\qlbar)$ are in bijection with closed $\operatorname{O}_{2n}$-orbits in $Z^1(W_{F}, \operatorname{SO}_{2n})(\qlbar)$ and also with closed points in $(Z^1(W_{F}, \operatorname{SO}_{2n}) \sslash \operatorname{O}_{2n})(\qlbar)$. \smallskip

Recall that $\mathcal{Z}^{\mathrm{spec}}(\operatorname{SO}_{2n})$ is the ring of functions of $Z^1(W_{F}, \operatorname{SO}_{2n}) \sslash \operatorname{SO}_{2n}$ and we define $\widetilde{\mathcal{Z}}^{\mathrm{spec}}(\operatorname{SO}_{2n}) \subset \mathcal{Z}^{\mathrm{spec}}(\operatorname{SO}_{2n})$ to be the ring of functions of $Z^1(W_{F}, \operatorname{SO}_{2n}) \sslash \operatorname{O}_{2n}$. Note that $\widetilde{\mathcal{Z}}^{\mathrm{spec}}(\operatorname{SO}_{2n}) \subset \mathcal{Z}^{\mathrm{spec}}(\operatorname{SO}_{2n})$ is just the subring of functions fixed by the involution induced by the $\operatorname{O}_{2n}$-action.

\begin{proof}[Proof of \Cref{Thm:EvenOrthogonalCompatibility}]
We first deal with the case that $G$ is quasi-split. Since all Levi subgroups $M \subset G$ are products of general linear groups with (quasi-split) special orthogonal groups of even rank, we in fact have maps $\widetilde{\operatorname{LL}}_{M}$ for all $M \subset G$.
These satisfy the obvious analogue of Assumption \ref{ass:one} by \cite[Proposition 2.4.3]{Peng}, and the obvious analogue of Assumption \ref{ass:two} by work of M{\oe}glin \cite[Theorem 1.5.1]{MoeglinOrthogonal}, cf. \cite[Theorem 3.3]{XuOrthogonal}, see \cite[Corollary 2.5.2]{Peng}. We can prove the obvious analogue of Assumption \ref{ass:three} by a globalization argument as in the proof of Theorem \ref{Thm:UnitaryOddOrthogonalCompatibility}, using Corollary \ref{Cor:StrongCompatibilityII}. The globalization argument there can be adapted as follows: The quadratic extension $K/F$ over which $G^{\ast}$ becomes split (corresponding to the discriminant of the quadratic space) should be globalized to a totally real\footnote{So that we get a Shimura datum of type $D^{\mathbb{R}}$.} extension $\mathsf{K}/\mlp$, and we should choose $\mathsf{H}^{\ast}$ to be an even special orthogonal group over $\mlp$ such that $\gal(\mathsf{K}/\mlp)$ acts nontrivially on its root datum (for example by choosing a quadratic space with discriminant corresponding to $\mathsf{K}$). Then $\mathsf{H}^{\ast}$ is split over $\mathbb{R}$ and we can choose an appropriate inner form $\mathsf{H}$ as in the proof of Lemma \ref{Lem: Globalization} (thus in particular, we will globalize to type $D^{\mathbb{R}}$). To globalize, we will use \cite[Lemma 6.14]{ChenZouMultiplicity} together with the existence of supercuspidal representations satisfying \ref{ass:ssc} for an auxiliary place $w \nmid 2$ of $\mlp$ which splits in $\mathsf{K}$, see \cite[Theorem 1.1]{EvenOrthogonalSSC}.\footnote{More precisely, they show the existence of irreducible $2n$-dimensional orthogonal representations of $W_F$, which correspond under $\operatorname{LLC}$ to the desired supercuspidal representations.} Finally, we note that Lemma \ref{Lem:Multiplicity} generalizes to the case of even special orthogonal groups with the same proof (keeping in mind Section \ref{subsub:NonSplitEvenOrthogonal}). \smallskip

  \Cref{Thm:EvenOrthogonalCompatibility} in the case that $G$ is quasi-split
  then follows verbatim from the proof of \Cref{prop: SemisimpleLLQuasiSplit}
  using \Cref{Cor:StrongCompatibilityII}, once we can establish the constancy of
  $\widetilde{\LL}_{G}^{\mathrm{FS}}$ on the $L$-packets
  $\widetilde{\Pi}_{\phi}(G)$ for all even special orthogonal groups $G$. For
  this, write $G' \supset G$ for the orthogonal group containing the special
  orthogonal group $G$. Note that $G'(F) \supsetneq G(F)$ is a proper inclusion because reflections
  have determinant $-1$. We have the following analogue of Assumption \ref{ass:four}, which holds by \cite{Arthur}, see
  \cite[Proposition~5.1.2, Remark~2.5.10]{Varma}: For $\phi \in
  \widetilde{\Phi}_\mathrm{temp}(G)$, there is a stable virtual character
  \[
    \Theta^1_{\phi}=\sum_{\pi\in \widetilde{\Pi}_{\phi}(G)} a_\pi \Theta_\pi,
  \]
  that is $G^\prime(F)$-invariant with all coefficients $a_{\pi} \neq 0$, and
  moreover atomic with respect to this property. By
  \cite[Theorem~1.1]{HansenStable} the Fargues--Scholze map factors as
  \[
    \Psi_G^\FS \colon \mathcal{Z}^\mathrm{spec}(G) \to
    \mathcal{Z}^\mathrm{vst}(G) \subseteq \mathcal{Z}(G).
  \]
  On the other hand, by functoriality of $\Psi_G^\FS$ under isomorphisms, it is
  naturally $G'(F)$-equivariant, hence induces
  \[
    \Psi_G^\FS \colon \widetilde{\mathcal{Z}}^\mathrm{spec}(G) \to
    \mathcal{Z}^\mathrm{vst}(G)^{G^\prime(F)}.
  \]
  This means that for every $z \in \widetilde{\mathcal{Z}}^\mathrm{spec}(G)$ the
  convolution $z \ast \Theta_\phi^1$ is a linear combination of $\Theta_\pi$ for
  $\pi \in \widetilde{\Pi}_\phi(G)$ that is stable and $G^\prime(F)$-invariant,
  and hence a constant multiple of $\Theta_\phi^1$. This implies that
  $\widetilde{\mathrm{LL}}_G^\FS$ is constant on $\widetilde{\Pi}_\phi(G)$.

To deal with non quasi-split $G$, we can combine the above argument with the proof of Proposition \ref{prop: SemisimpleLLNonQuasiSplit}, if we can prove that for all matching functions $f^\ast\in \mathcal{C}^\infty_c(G^\ast(F))$, $f\in \mathcal{C}_c^\infty(G(F))$, there is an equality
    \[\Theta^\mathbf{1}_{\phi}(f^\ast)=\sum_{\pi\in \widetilde{\Pi}_\phi(G)}c_\pi\Theta_\pi(f),\]
    for some coefficients $c_\pi \in \mathbb{C}^\times$ (depending on $\mathfrak{w}$, but not on $f, f^\ast$). This holds by \cite[Theorem A]{PengECR}. 
\end{proof}
We can now prove Theorem \ref{Thm:IntroCompatibility}.

\begin{proof}[Proof of Theorem \ref{Thm:IntroCompatibility}]
By Lemma \ref{Lem:TemperedToFull}, it suffices to check the compatibility for tempered irreducible smooth representations $\pi$. In this case, it follows from Theorem \ref{Thm:UnitaryOddOrthogonalCompatibility} and Theorem \ref{Thm:EvenOrthogonalCompatibility}.
\end{proof}

\subsubsection{}
We have the following analogue of \cite[Theorem 7.1.1]{Peng}, which is Corollary \ref{Cor:IntroActualEvenOrthogonalLL}. Let $F$ be a $p$-adic local field as before. Let $V'$ be a $2n$-dimensional quadratic space over $F$ with quasi-split special orthogonal group $G^{\ast}=\operatorname{SO}(V)$, and let $G$ be a pure inner form of $G^{\ast}$.

\begin{Thm} \label{Thm:ActualEvenOrthogonalLanglands}
There is a unique local Langlands correspondence $\operatorname{LL}_{G}:\Pi_{\mathrm{temp}}(G) \to \Phi_{\mathrm{temp}}(G)$ lifting $\widetilde{\operatorname{LL}}_{G}$ such that $(-)^{\mathrm{ss}} \circ \operatorname{LL}_{G}= \operatorname{LL}_{G}^{\FS}$. Moreover, it satisfies the conditions listed in \cite[Theorem 7.1.1]{Peng}.
\end{Thm}

\begin{proof}
Given Theorem \ref{Thm:EvenOrthogonalCompatibility}, the proof of \cite[Theorem 7.1.1]{Peng} goes through verbatim. 
\end{proof}
Note that \Cref{Thm:ActualEvenOrthogonalLanglands} is proved in \cite{Peng} if $p>2$ and $F$ is unramified, see \cite[Theorem 7.1.1]{Peng}.
}

{\appendix

\renewcommand\thesubsubsection{\thesection.\arabic{subsubsection}}
\def\vStk{\operatorname{vStk}}
\def\Det{\mathcal{D}_\mathrm{\acute{e}t}}

\section{Canonical Frobenius descents revisited} \label{sub:SixFunctor}
The goal of this appendix is to extend \cite[Section~8.2]{DvHKZIgusaStacks} regarding canonical Frobenius descent data, using the theory of six functor formalisms of \cite{MannThesis, HeyerMann} and results of \cite{DauserKuijper}. The arguments here were suggested to us by Lucas Mann, although all mistakes should be attributed to us. In this appendix we will use the word ordinary category to mean category, the word category to mean $(\infty,1)$-category, and the word symmetric monoidal category to mean symmetric monoidal $(\infty,1)$-category. Keeping this convention in mind, we let $\operatorname{Cat}$ denote the category of categories, see \cite[Example~D.1.4]{HeyerMann}. We will let $\operatorname{Cat}^{\times}$ denote the symmetric monoidal category given by the cartesian monoidal structure, see \cite[Example~B.1.9]{HeyerMann}. 

\subsubsection{}
We consider $\vStk$, the category of small v-stacks. Let $\Lambda$ be a $\zl$-algebra in which $\ell$ is nilpotent. By \cite[Proposition 3.16]{MannNuclear}, there is a functor
\[
  \mathcal{D}^{\ast} \colon \vStk^{\mathrm{op}} \to \operatorname{Cat}, \quad X
  \mapsto \Det(X, \Lambda),
\]
sending $f \colon X \to Y$ to $f^{\ast} \colon \Det(Y, \Lambda) \to \Det(X,
\Lambda)$. There is a natural transformation $\operatorname{1}_{\vStk} \to
\operatorname{1}_{\vStk}$ sending $X$ to the morphism $\phi \colon X \to X$.
This induces a natural transformation $\mathcal{D}^{\ast}(\phi) \colon
\mathcal{D}^{\ast} \to \mathcal{D}^{\ast}$.

\begin{Lem} \label{Lem:PhiDescentFunctorialI}
  There is an isomorphism $\mathcal{D}^{\ast}(\phi) \simeq
  \id_{\mathcal{D}^\ast}$ of natural transformations.
\end{Lem}

\begin{proof}
  \def\stdisc{\mathrm{StrTotDisc}}
  Let $\stdisc \subset \vStk$ be the full subcategory of strictly totally
  disconnected affinoid perfectoid spaces. Consider $\operatorname{HShv}(\vStk,
  \operatorname{Cat}) \subset \operatorname{Fun}(\vStk^{\mathrm{op}},
  \operatorname{Cat})$ the full subcategory of hypersheaves for the v-topology
  (in the $\infty$-categorical sense). Then by
  \cite[Proposition~3.16]{MannNuclear}, we have that $\mathcal{D}^{\ast}$ lies
  in this full subcategory. By hyperdescent, the restriction functor
  \[
    \operatorname{HShv}(\vStk, \operatorname{Cat}) \xrightarrow{\sim}
    \operatorname{HShv}(\stdisc, \operatorname{Cat})
  \]
  is an equivalence as every v-stack admits a v-hypercover by strictly totally
  disconnected perfectoid spaces, and so it suffices to construct
  $\mathcal{D}^\ast(\phi) \simeq \id_{\mathcal{D}^\ast}$ on this subcategory. We
  now note that if $X \in \stdisc$ then $\phi = \id$ as maps on underlying
  topological spaces $\lvert X \rvert \to \lvert X \rvert$. On the other hand,
  $\Det(X, \Lambda) = \mathcal{D}(\operatorname{Sh}(\lvert X \rvert, \Lambda))$
  by construction, see \cite[Remark~14.14]{EtCohDiam}. This shows that there is
  a natural equivalence $\phi^\ast \simeq \id$ of functors $\Det(X, \Lambda) \to
  \Det(X, \Lambda)$, functorial in $X \in \stdisc$.
\end{proof}

\subsubsection{} \label{Sec:DPhiAndDBZ}
There is a functor $\mathcal{D}^{\phi,\ast} \colon \vStk^\mathrm{op} \to
\operatorname{Cat}$ sending $X$ to the category
$\mathcal{D}([X/\phi^{\mathbb{Z}}])$. Indeed, we are just precomposing
$\mathcal{D}^{\ast}$ with $\vStk^{\mathrm{op}} \to \vStk^{\mathrm{op}}$ given by
$X \mapsto [X/\phi^{\mathbb{Z}}]$. Let $B\mathbb{Z}$ be the group of integers
considered as a category with one object. There is a functor $\operatorname{Cat}
\to \operatorname{Cat}$ sending $\mathcal{E} \mapsto \operatorname{Fun}(B
\mathbb{Z}, \mathcal{E}) \coloneqq \mathcal{E}^{B \mathbb{Z}}$. Let us write
$\mathcal{D}^{B \mathbb{Z},\ast} \colon \vStk^{\mathrm{op}} \to
\operatorname{Cat}$ for the composition of $\mathcal{D}^{\ast}$ with this
functor.

\begin{Lem} \label{Lem:PhiDescentFunctorialII}
  There is an isomorphism of functors $\mathcal{D}^{\phi,\ast}
  \xrightarrow{\sim} \mathcal{D}^{B \mathbb{Z},\ast}$.
\end{Lem}

\begin{proof}
  We note that by descent (see \cite[Lemma~8.2.7]{DvHKZIgusaStacks}), there is a
  canonical equivalence
  \[
    \mathcal{D}([X/\phi^\mathbb{Z}]) \simeq \mathcal{D}(X)^{B\phi^{\mathbb{Z}}},
  \]
  where the right hand side is the category $\mathbb{Z}$-equivariant objects
  with $\mathbb{Z}$ acting via $\phi^\ast$ on $\mathcal{D}(X)$. Applying
  \Cref{Lem:PhiDescentFunctorialI} to functorially identify the action of
  $\phi^\ast$ with the trivial action, we deduce that $\mathcal{D}^{\phi,\ast}$
  is isomorphic to $\mathcal{D}^{B\mathbb{Z},\ast}$.
\end{proof}

\subsubsection{} Recall the notion of geometric setup from \cite[Definition~2.1.1]{HeyerMann}. We consider the geometric setup $(\mathcal{C}, E)=(\vStk^{\mathrm{op}}, \ell\mathrm{-fine})$, where $\ell\mathrm{-fine}$ is the set of $\ell$-fine maps between small v-stacks in the sense of \cite[Definition~5.8]{MannNuclear}. Note that there is a natural morphism of geometric setups
\[
  \alpha:(\mathcal{C}, E) \to (\mathcal{C}, E), \quad X \mapsto [X/\phi^{\mathbb{Z}}]
\]
that follows from \'etale descent for $\ell$-fine maps, see \cite[Lemma~5.10]{MannNuclear}.

\subsubsection{} Associated to $(\mathcal{C},E)$ is the symmetric monoidal
category of correspondences $\operatorname{Corr}(\mathcal{C},E)^{\otimes}$, see
\cite[Definition~2.2.10]{HeyerMann}, whose underlying category receives a
functor $h \colon \mathcal{C}^\mathrm{op} \to
\operatorname{Corr}(\mathcal{C},E)$, see \cite[Remark~2.2.12]{HeyerMann}. Recall
the notion of $6$-functor formalism from \cite[Definition~3.2.1]{HeyerMann}: In
our setting these are $3$-functor formalisms, that is, lax symmetric monoidal
functors, satisfying certain properties. By \cite[Theorem~5.11]{MannNuclear},
there is a $6$-functor formalism
\[
  \mathcal{D} \colon \operatorname{Corr}(\mathcal{C},E)^{\otimes} \to
  \operatorname{Cat}^{\times}
\]
whose restriction along $h$ is isomorphic to the functor $\mathcal{D}^{\ast}$
discussed above. We similarly have $\mathcal{D}^{\phi}$, which is the
composition of the morphism of geometric setups $\alpha \colon (\mathcal{C}, E)
\to (\mathcal{C}, E)$ with the six functor formalism $\mathcal{D}$, and thus
itself a $6$-functor formalism. We also consider the $3$-functor formalism
$\mathcal{D}^{B \mathbb{Z}}:\operatorname{Corr}(\mathcal{C},E)^{\otimes} \to
\operatorname{Cat}^{\times}$ given by composing $\mathcal{D}$ with the symmetric
monoidal functor $\mathcal{E} \mapsto \mathcal{E}^{B \mathbb{Z}}$.

\begin{Prop} \label{Prop:CanonicalPhiDescentIII}
  There is an isomorphism of $3$-functor formalisms
  \[
    \mathcal{D}^{\phi} \xrightarrow{\sim} \mathcal{D}^{B \mathbb{Z}}.
  \]
  Because being a $6$-functor formalism is a property, see
  \cite[Definition~3.2.1]{HeyerMann}, it is in particular an isomorphism of
  $6$-functor formalisms.
\end{Prop}

\begin{proof}
As in the proof of \cite[Proposition~5.6]{MannNuclear}, we
consider the auxiliary collections of edges (in the notation of loc.\ cit.) $E_1=\mathrm{fdcqc} \supset E_2=\mathrm{fdcsqc} \subset E_3=\mathrm{fdcss} \subset E_4=\mathrm{fdcs} \subset E_5=E$. We also consider the full subcategories $\mathcal{C}_i \subset \mathcal{C}$ which equals $\mathcal{C}$ unless $i=3$, in which case it equals the category of separated locally spatial diamonds. For $i=1,\dotsc,5$, we write $\mathcal{D}_i^{\phi}, \mathcal{D}_i^{B \mathbb{Z}}$ for the restriction of $\mathcal{D}^{\phi}, \mathcal{D}^{B \mathbb{Z}}$ to $\operatorname{Corr}(\mathcal{C}_i, E_i)$. As in \cite[proof of Proposition 5.6]{MannNuclear}, we define $I \subset E_1$ to be the class of open immersions and $P \subset E_1$ to be the class of proper fdcqc maps. \smallskip

It is explained in loc. cit. that the pair $I,P \subset E_1$ is a suitable decomposition as in \cite[Definition A.5.9]{MannThesis}, which implies that $(\mathcal{C}_1, E_1, I,P)$ is a Nagata setup in the sense of \cite[Definition 2.1]{DauserKuijper}. It now follows from \cite[Theorem 3.3]{DauserKuijper} that the $3$-functor formalisms $\mathcal{D}_1^{\phi}, \mathcal{D}_1^{B \mathbb{Z}}$ are uniquely determined by $\mathcal{D}^{\phi,\ast}$ and $\mathcal{D}^{B \mathbb{Z}, \ast}$, and thus the isomorphism of Lemma \ref{Lem:PhiDescentFunctorialII} induces an isomorphism of $3$-functor formalisms $\mathcal{D}_1^{\phi} \xrightarrow{\sim} \mathcal{D}_1^{B \mathbb{Z}}$.\footnote{Note that the isomorphism of Lemma \ref{Lem:PhiDescentFunctorialII} is automatically an isomorphism of monoidal functors, because we are using the Cartesian monoidal structure.} Here we are using that morphisms in $P$ are cohomologically proper in $\mathcal{D}_1^{\phi}, \mathcal{D}_1^{B \mathbb{Z}}$, see \cite[Lemma 9.8]{MannNuclear}, and morphisms in $I$ are cohomologically \'etale in $\mathcal{D}_1^{\phi}, \mathcal{D}_1^{B \mathbb{Z}}$, see \cite[Proposition 5.6.(ii)]{MannNuclear}, to verify the hypotheses of \cite[Theorem 3.3]{DauserKuijper}. \smallskip 

We now take the isomorphism of $3$-functor formalisms $\mathcal{D}_1^{\phi} \xrightarrow{\sim} \mathcal{D}_1^{B \mathbb{Z}}$ and show that it induces isomorphisms of $3$-functor formalisms $\mathcal{D}_i^{\phi} \xrightarrow{\sim} \mathcal{D}_i^{B \mathbb{Z}}$ for $i=2,3,4,5$. Note that going from $i=1$ to $i=2$ is vacuous since $E_2 \subset E_1$. Going from $i=2$ to $i=3$ follows from the uniqueness statements of \cite[Proposition A.5.12, A.5.14]{MannThesis}; the assumptions of these propositions hold as explained in the proof of \cite[Proposition 5.6]{MannThesis}. Going from $i=3$ to $i=4$ follows directly from (the uniqueness part of) \cite[Proposition 3.4.2]{HeyerMann}, because $E_4$ consists precisely of edges representable in $E_3$. To go from $i=4$ to $i=5$, we note that the notion of $\ell$-fine morphisms $f:X \to Y$ agrees with the notion of morphisms that are $\mathcal{D}^{!}$-locally on the source fdcs, in the sense of \cite[Definition 3.4.10]{HeyerMann}, which is precisely condition (ii) of \cite[Proposition 3.4.8]{HeyerMann}. Thus we can go from $i=4$ to $i=5$ using the uniqueness part of \cite[Proposition 3.4.8]{HeyerMann}.
\end{proof}

\subsubsection{} \label{subsub:CanonicalDescent}
The natural maps $X \to [X/\phi^{\mathbb{Z}}]$ induce a morphism
$\mathcal{D}^{\phi} \to \mathcal{D}$ of $3$-functor formalisms, which
corresponds to ``forgetting the $\phi$-descent datum''. Under the isomorphism of
\Cref{Prop:CanonicalPhiDescentIII}, this corresponds to the ``forgetful map''
$\mathcal{D}^{B \mathbb{Z}} \to \mathcal{D}$. The map $\mathcal{D}^{B\mathbb{Z}}
\to \mathcal{D}$ has a section corresponding to ``taking the trivial
$\mathbb{Z}$-action'', which under our identification induces a section
$\operatorname{can} \colon \mathcal{D} \to \mathcal{D}^{\phi}$ of $3$-functor
formalisms corresponding to ``taking the canonical $\phi$-descent datum'': This
is the canonical $\phi$-descent of \cite[Lemma~8.2.4]{DvHKZIgusaStacks}.

\begin{Prop} \label{Prop:CanonicalDescentPushforward}
  Let $f \colon X \to Y$ be a morphism of small v-stacks and consider the
  associated induced map of v-stacks
  \[
    f_\phi \colon [X/\phi_X^\mathbb{Z}] \to [Y/\phi_Y^\mathbb{Z}].
  \]
  \begin{enumerate}[$($i$\,)$]
    \item There is a natural isomorphism $\operatorname{can}_Y \circ R f_\ast
      \xrightarrow{\sim} R f_{\phi,\ast} \circ \operatorname{can}_X$ of functors
      $\mathcal{D}(X) \to \mathcal{D}([Y/\phi^{\mathbb{Z}}])$.
    \item If $f$ is $\ell$-fine, then there is a natural isomorphism
      $\operatorname{can}_Y \circ f_! \xrightarrow{\sim} f_{\phi,!} \circ
      \operatorname{can}_X$ of functors $\mathcal{D}(X) \to
      \mathcal{D}([Y/\phi^{\mathbb{Z}}])$.
  \end{enumerate}
\end{Prop}
\begin{proof}
  By \Cref{Prop:CanonicalPhiDescentIII}, we may compute $Rf_{\phi,\ast}$ and
  $f_{\phi,!}$ using the $\mathcal{D}^{B\mathbb{Z}}$-theory instead of the
  $\mathcal{D}^\phi$-theory. That is, it suffices to produce natural
  isomorphisms
  \[
    Rf_{B\mathbb{Z},\ast}(A, \id_A) \cong (Rf_\ast A, \id_{Rf_\ast A}), \quad
    f_{B\mathbb{Z},!}(A, \id_A) \cong (f_! A, \id_{f_! A}),
  \]
  where by $Rf_{B\mathbb{Z},\ast}$ and $f_{B\mathbb{Z},!}$ we mean the
  corresponding functors extracted from $\mathcal{D}^{B\mathbb{Z}} \colon
  \operatorname{Corr}(\mathcal{C},E)^\otimes \to \operatorname{Cat}^\times$ as
  in \cite[Definition~3.1.4, 3.2.1]{HeyerMann}.

  For $f_!$, recall from \Cref{Sec:DPhiAndDBZ} that the $3$-functor formalism
  $\mathcal{D}^{B\mathbb{Z}}$ is simply defined as a composition of
  $\mathcal{D}$ with the endofunctor $(-)^{B\mathbb{Z}} \colon
  \operatorname{Cat} \to \operatorname{Cat}$. This endofunctor sends a functor
  $F \colon \mathcal{E} \to \mathcal{F}$ to $(A, \alpha \in
  \operatorname{Aut}(A)) \mapsto (F(A), F(\alpha))$, and therefore
  \[
    f_{B\mathbb{Z},!}(A, \id_A) = (f_! A, f_! \id_A) = (f_! A, \id_{f_! A}).
  \]
  Similarly, we see that
  \[
    f^{B\mathbb{Z},\ast}(A, \alpha) = (f^\ast A, f^\ast \alpha).
  \]

  To access $f_{B\mathbb{Z},\ast}$, we need to compute the right adjoint of
  $f^{B\mathbb{Z},\ast}$. By Lemma \ref{Lem:CategoryTheory}, the right adjoint of
  $f^{B\mathbb{Z},\ast} = (f^\ast)^{B\mathbb{Z}}$ is given by
  $(Rf_\ast)^{B\mathbb{Z}}$,  and therefore
  \[
    Rf_{B\mathbb{Z},\ast}(A, \alpha) = (Rf_\ast)^{B\mathbb{Z}}(A, \alpha) =
    (Rf_\ast A, Rf_\ast \alpha).
  \]
  Applying this to $\alpha = \id_A$, we obtain the desired identification.
\end{proof}

\begin{Lem} \label{Lem:CategoryTheory}
  Let $F \colon \mathcal{C} \to \mathcal{D}$ be a functor of categories with
  right adjoint $G$, and let $K$ be a simplicial set. Then $F^K \colon
  \mathcal{C}^K \to \mathcal{D}^K$ is left adjoint to $G^K$.
\end{Lem}

\begin{proof}
  As we see from \cite[Definition~5.2.2.1]{HigherTopos}, we obtain an inner
  fibration $p \colon \mathcal{M} \to \Delta^1$ that is both Cartesian and
  coCartesian, together with equivalences $\mathcal{C} \to p^{-1}\{0\}$ and
  $\mathcal{D} \to p^{-1}\{1\}$. By \cite[Proposition~3.1.2.1]{HigherTopos}, the
  projection
  \[
    q \colon \mathcal{M}^K \times_{(\Delta^1)^K} \Delta^1 \to \Delta^1
  \]
  is both Cartesian and coCartesian, and moreover $q$-Cartesian edges are
  precisely those $K \times \Delta^1 \to \mathcal{M}$ whose restriction to each
  $\{k\} \times \Delta^1$ is $p$-Cartesian and similarly for $q$-coCartesian
  edges. This shows that both $F^K$ and $G^K$ are associated to $q$.
\end{proof}

\subsubsection{} Now let $L$ be a finite extension of $\qp$. Let $C$ be a completed algebraic closure of $L$ and $\breve{L} \subset C$ the completion of the maximal unramified extension. Let $k_{L} \subset k$ be the extension on residue fields induced by $L \subset \breve{L}$, and let $r=[k_{L}:\fp]$. As in \cite[Section 8.2.13]{DvHKZIgusaStacks}, we consider the action of $\ul{\gal_L} \times \ul{\mathbb{Z}}$ on $\spd C$, where $\ul{\mathbb{Z}}$ acts by $\phi_C$. We also consider the Weil group $W_L \subset \gal_{L}$, which admits a twisted embedding $\iota:W_L \to \gal_{L} \times \mathbb{Z}$. We consider (see \cite[Lemma 8.2.12]{DvHKZIgusaStacks})
\begin{align*}
    \spd C / \ul{\iota(W_L)} \xrightarrow{\sim} \left[\spd L/ \phi_{L}^{r \mathbb{Z}}\right] \times_{\spd k_{L}} \spd k .
\end{align*}
\subsubsection{} For any small v-stack $X$ we can consider $\mathcal{D}(X)$ as a condensed $\infty$-category by considering the functor taking a profinite set $S$ to $\mathcal{D}(X \times \ul{S})$. This allows us to consider $\mathcal{D}(X)^{BH}$ for a locally profinite group $H$, see \cite[Section 8.2.8]{DvHKZIgusaStacks}. By \cite[Lemma 8.2.10]{DvHKZIgusaStacks}, there are natural identifications
\begin{align*}
    \mathcal{D}(\spd L) \xrightarrow{\sim} \mathcal{D}(\spd C)^{B \gal_{L}} \\
    \mathcal{D}(\spd C / \ul{\iota(W_L)}) \xrightarrow{\sim} \mathcal{D}(\spd C)^{B W_L},
\end{align*}
and so restricting along $W_L \to \gal_{L}$ defines a natural map 
\begin{align*}
    \rho:\mathcal{D}(\spd L) \to \mathcal{D}(\spd C / \ul{\iota(W_L)}).
\end{align*}

\subsubsection{} Let $f:X \to \spd L$ be a morphism of small v-stacks, and consider the morphism   
\begin{align*}
    f_{\phi,k}:\left[ X/ \phi_X^{r \mathbb{Z}}\right] \times_{\spd k_{L}} \spd k \to \left[\spd L/ \phi_{L}^{r \mathbb{Z}}\right] \times_{\spd k_{L}} \spd k,
\end{align*}
note that it is $\ell$-fine if $f$ is, by \'etale descent for $\ell$-fine morphisms. Note that the map $f_{\phi,k}$ is the base change of $f_\phi$ to $k$ over $\spd \fp$.

\begin{Prop} \label{Prop:GaloisRestriction}
  Let $f$ and $f_{\phi,k}$ be as above, and consider the natural projection map
  $\psi \colon [X/\phi_X^\mathbb{Z}] \times \spd k \to [X/\phi_X^\mathbb{Z}]$.
  \begin{enumerate}[$($i$\,)$]
    \item If $f$ is qcqs, then there is a natural isomorphism $\rho \circ
      Rf_\ast \simeq Rf_{\phi,k,\ast} \circ \psi^\ast \circ \mathrm{can}_X$ of
      functors $\mathcal{D}^+(X) \to \mathcal{D}^+([\spd C / \ul{\iota(W_L)}])$.
    \item If $f$ is $\ell$-fine, then there is a natural isomorphism $\rho \circ
      f_! \simeq f_{\phi,k,!} \circ \psi^\ast \circ \mathrm{can}_X$ of functors
      $\mathcal{D}(X) \to \mathcal{D}([\spd C / \ul{\iota(W_L)}])$.
    \item If $X$ is the analytification of a smooth variety $X^\mathrm{alg}/L$
      and if $A$ is the analytification of a locally constant sheaf of
      $\Lambda$-modules $A^\mathrm{alg}$ on
      $X^\mathrm{alg}_\mathrm{\acute{e}t}$, then there is a natural isomorphism
      $\rho (Rf_\ast A) \cong Rf_{\phi,k,\ast} \psi^\ast \mathrm{can}_X(A)$.
  \end{enumerate}
\end{Prop}

\begin{proof}
  The proof is the same as \cite[Proposition~8.2.15]{DvHKZIgusaStacks}. We
  consider the cartesian diagram
  \[ \begin{tikzcd}[column sep = small]
    \lbrack X / \phi^\mathbb{Z} \rbrack \times \spd k \arrow{r}{\psi}
    \arrow{d}{f_{\phi,k}} & \lbrack X / \phi^\mathbb{Z} \rbrack
    \arrow{d}{f_\phi} \\ \lbrack \spd L / \phi^\mathbb{Z} \rbrack \times \spd k
    = \lbrack \spd C / \underline{\iota(W_L)} \rbrack \arrow{r}{g} & \lbrack
    \spd L / \phi^\mathbb{Z} \rbrack = \lbrack \spd C / (\underline{\gal_L}
    \times \underline{\mathbb{Z}}) \rbrack,
  \end{tikzcd} \]
  and apply \Cref{Prop:CanonicalDescentPushforward} to identify $R
  f_{\phi,\ast} \circ \mathrm{can}_X \simeq \mathrm{can}_L \circ R f_{\ast}$
  and similarly $f_{\phi,!} \circ \mathrm{can}_X \simeq \mathrm{can}_L \circ
  f_!$ for (ii). We then use qcqs base change \cite[Proposition~17.6]{EtCohDiam}
  to identify $Rf_{\phi,k,\ast} \psi^\ast \cong g^\ast Rf_{\phi,\ast}$ for (i)
  and proper base change \cite[Proposition~22.19]{EtCohDiam} to identify
  $f_{\phi,k,!} \psi^\ast \cong g^\ast f_{\phi,!}$ for (ii). Finally, we note
  that $\rho = g^{\ast} \circ \mathrm{can}$, completing the proof.

  For (iii), it remains to verify that the base change morphism
  \[
    g^\ast Rf_{\phi,\ast} A_\phi \to Rf_{\phi,k,\ast} \psi^\ast A_\phi
  \]
  is an isomorphism, where $A_\phi \in \mathcal{D}([X/\phi^\mathbb{Z}])$ is the
  canonical descent of $A$. As isomorphisms can be detected on a v-cover, we may
  use smooth base change \cite[Proposition~23.16.(ii)]{EtCohDiam} along the
  separated \'{e}tale maps $\spd \breve{L} \to [\spd L/\phi^\mathbb{Z}] \times
  \spd k$ and $\spd L \to [\spd L/\phi^\mathbb{Z}]$ to reduce to checking that
  the base change morphism
  \[
    \tilde{g}^\ast Rf_\ast A \to Rf_{\breve{L}\ast} \tilde{\psi}^\ast A
  \]
  is an isomorphism in the diagram
  \[ \begin{tikzcd}[row sep=small]
    X_{\breve{L}} \arrow{d}[']{f_{\breve{L}}} \arrow{r}{\tilde{\psi}} & X
    \arrow{d}{f} \\ \spd \breve{L} \arrow{r}{\tilde{g}} & \spd L.
  \end{tikzcd} \]
  This follows from comparison with the algebraic theory. By \Cref{Lem:HuberComparison} below, we may instead verify that
  $\tilde{g}^{\mathrm{alg},\ast} Rf^\mathrm{alg}_\ast A^\mathrm{alg} \cong
  Rf^\mathrm{alg}_{\breve{L},\ast} \tilde{\psi}^{\mathrm{alg},\ast}
  A^\mathrm{alg}$ as \'{e}tale sheaves on $\spec \breve{L}$. This is now
  \cite[Corollary~XVI.1.6]{SGA4-3}.
\end{proof}

\begin{Lem} \label{Lem:HuberComparison}
Let $L$ be a complete nonarchimedean discretely valued field extension of $\qp$, and let $f \colon X \to \spec L$ be a smooth separated algebraic variety with analytification $f^{\diamondsuit} \colon X^\diamondsuit \to \spd L$. Let $\Lambda$ be a ring in which $\ell$ is nilpotent, and let $A$ be a locally constant sheaf of $\Lambda$-modules on the \'etale site of $X$. Then
\begin{align*}
    (R f_{\ast} A)^{\diamondsuit} \xrightarrow{\sim} R f^{\diamondsuit}_{\ast} A^{\diamondsuit}
\end{align*}
\end{Lem}

\begin{proof}
  Because the comparison theorem \cite[Proposition 27.5]{EtCohDiam}
  is only stated for constructible $\mathbb{Z}/\ell^n\mathbb{Z}$-modules, we
  make the following roundabout argument. First we reduce to the case when
  $A$ is trivial. We may as well assume that $X$ is
  connected. Because $A$ comes from a locally constant sheaf on
  $X$, there is a finite Galois \'{e}tale cover $\pi \colon
  \tilde{X} \to X$ with Galois group $G$ for which
  $\pi^\ast A = M$ is a constant sheaf. Then $Rf_\ast A$ is the
  derived $G$-fixed points of $R(f \circ \pi)_\ast M$, see
  \cite[Corollary~6.2.14]{LiuZheng}, and similarly for $Rf_\ast^\diamondsuit$. Therefore we
  may replace $(X, A)$ with $(\tilde{X}, M)$ and
  assume that $A$ is trivial.

  \def\Zln{\mathbb{Z}/\ell^n\mathbb{Z}}
  \def\sHom{\mathscr{H}om}
  We now use Verdier duality, \cite[Theorem~23.3.(i)]{EtCohDiam} and
  \cite[Proposition~6.2.4.(4)]{LiuZheng}. Choose $n$ large enough so that
  $\ell^n \Lambda = 0$. Since $M$ is constant, we can write $M=f^{\ast} N$. Then $M = R\sHom_{\Zln}(\Zln, Rf^! N(-d)[-2d])$ both
  algebraically and analytically, and it thus suffices to show that
  \[
    R\sHom_{\Zln}(Rf_! (\Zln), N(-d)[-2d])^\diamondsuit \cong
    R\sHom_{\Zln}(Rf^{\diamondsuit}_! (\Zln), N^{\diamondsuit}(-d)[-2d]).
  \]
  Since analytification induces an equivalence $\mathcal{D}^+(\spec L, \Zln)
  \cong \mathcal{D}^+(\spd L, \Zln)$ by \cite[Proposition~14.15]{EtCohDiam} applied to $Y=\spd L$, whose \'etale site agrees with the \'etale site of $\spa L$ by \cite[Lemma 15.6]{EtCohDiam} and thus with the \'etale site of $\spec L$, it
  is enough to show that $(Rf_! (\Zln))^\diamondsuit \cong Rf_!^\diamondsuit
  (\Zln)$. This now follows from \cite[Proposition 27.5]{EtCohDiam}.
\end{proof}
}

\renewcommand{\VAN}[3]{#3}
\renewcommand{\VANDEN}[3]{#3}
\renewcommand\MR[1]{}
\bibliographystyle{amsalpha}
\bibliography{references}
\end{document}